\documentclass[11pt]{article}
\usepackage{amsmath, exscale, epsfig,amssymb, 
amsthm,
 makeidx}

\usepackage{bm}
\usepackage{stmaryrd}
\usepackage{paralist}
\usepackage{titlesec}
\titlespacing*{\section}{0pt}{0.4\baselineskip}{0.3\baselineskip}
\titlespacing*{\subsection}{0pt}{0.4\baselineskip}{0.3\baselineskip}

\usepackage[pdftex]{hyperref} %%not compatible with showkeys
\hypersetup{
    unicode      = false,     % non-Latin characters in AcrobatÃÂ¢ÃÂÃÂs bookmarks
    pdftoolbar   = true,      % show AcrobatÃÂ¢ÃÂÃÂs toolbar?
    pdfmenubar   = true,      % show AcrobatÃÂ¢ÃÂÃÂs menu?
    pdffitwindow = true,      % page fit to window when opened
    pdfnewwindow = true,      % links in new window
    colorlinks   = true,      % false: boxed links; true: colored links
    linkcolor    = blue,      % color of internal links
    citecolor    = red,      % color of links to bibliography
    filecolor    = blue,      % color of file links
    urlcolor     = green       % color of external links
}

\small
\normalsize
\usepackage[utf8]{inputenc}
\usepackage[T1]{fontenc}
\usepackage[english]{babel}
\usepackage{enumerate,vmargin}

\usepackage{times}

\usepackage{amsfonts}
\usepackage{amsbsy}
\usepackage{amscd}

\usepackage{multicol}
\allowdisplaybreaks

\usepackage[all]{xy}

\usepackage{stmaryrd}
\usepackage{graphicx}
\usepackage{paralist}
\usepackage{amsfonts}
\usepackage{amssymb,mathrsfs}

\setmarginsrb{2.9cm}{2.6cm}{2.9cm}{1.7cm}{0cm}{0mm}{0cm}{9mm}

\def\build#1_#2^#3{\mathrel{\mathop{\kern 0pt#1}\limits_{#2}^{#3}}}
\def\noi{{\noindent}}

\def\llbracket{[\![}
\def\rrbracket{ ]\!]}

\def\cq{$\hfill \square$}
\def\un{{\bf 1}}

\newcommand{\bbT}{\mathbb{T}}

\newcommand{\bD}{\mathbb{D}}

\newcommand{\bE}{{\bf E}}

\newcommand{\bbM}{\mathbb{M}}
\newcommand{\bN}{\mathbf{N}}

\newcommand{\bbZ}{\mathbb{Z}}
\newcommand{\bbQ}{\mathbb{Q}}
\newcommand{\bP}{{\bf P}}
\newcommand{\bR}{\mathbb{R}}

\newcommand{\bT}{\mathbb{T}}
\newcommand{\bbU}{\mathbb{U}}

\newcommand{\cC}{\mathscr{C}}
\newcommand{\cI}{\mathcal{I}}

\newcommand{\cF}{\mathcal{F}}

\newcommand{\cH}{\mathcal{H}}
\newcommand{\cL}{\mathcal{L}}
\newcommand{\cM}{\mathcal{M}}

\newcommand{\cR}{\mathcal{R}}

\newcommand{\cT}{\mathcal{T}}

\def\cV{{\mathcal V}}
\def\varep{\varepsilon}

\def\be{\begin{equation}}

\def\ee{\end{equation}}
\def\ba{\begin{eqnarray*}}
\def\ea{\end{eqnarray*}}

\def\da{\downarrow}

\def\noi{\noindent}

\newcommand{\lgeo}{[\![}
\newcommand{\rgeo}{]\!]}
\def\cqfd{ \hfill $\blacksquare$ }

\newcommand{\eqo}{\! = \! }
\newcommand{\geqo}{\! \geq \! }
\newcommand{\leqo}{\! \leq \! }
\newcommand{\ino}{\! \in \! }

\newcommand{\bbR}{\mathbb{R}}
\newcommand{\bbN}{\mathbb{N}}
\newcommand{\leko}{\! < \! }
\newcommand{\geko}{\! > \! }

\newcommand{\epp}{\varepsilon}

\newtheoremstyle{thmstyl}
{3.5pt} % space above
{2.5pt} % space below
{\em} % Body font
{} % Indent amount
{\bfseries} % Theorem head font
{.} % Punctuation after theorem head
{.5em} % space after theorem head
{} % theorem head spec 
\theoremstyle{thmstyl}

\newtheorem{theorem}{Theorem}[section]

\newtheorem{lemma}[theorem]{Lemma}
\newtheorem{proposition}[theorem]{Proposition}
\newtheorem{corollary}[theorem]{Corollary}
\newtheoremstyle{dfstyl}
{3.5pt} % space above
{2.5pt} % space below
{} % Body font
{} % Indent amount
{\bfseries} % Theorem head font
{.} % Punctuation after theorem head
{.5em} % space after theorem head
{} % theorem head spec 
\theoremstyle{dfstyl}

\newtheorem{definition}[theorem]{Definition}

\newtheorem{remark}[theorem]{Remark}
\newtheorem{example}[theorem]{Example}

\renewcommand{\noi}{\noindent}

\renewcommand{\lgeo}{\llbracket}
\renewcommand{\rgeo}{\rrbracket}

\renewcommand{\un}{{\bf 1}}
\newcommand{\I}[1]{\boldsymbol{1}_{\{#1\}}}

\renewcommand{\da}{\downarrow}

\newcommand{\fTheta}{\boldsymbol \Theta}

\newcommand{\bbC}{\mathbb C}

\newcommand{\boo}{\mathtt o}
\newcommand{\booo}{\varrho}

\newcommand{\ccF}{\mathscr F}
\newcommand{\ccG}{\mathscr G}

\def\cq{$\hfill \square$}
\def\cqfd{$\hfill \blacksquare$}

\def\be{\mathbf{e}}

\def\bR{\mathbf{R}}
\def\bC{\mathbf{C}}
\def\bD{\mathbf{D}}

\def\bT{\mathbf{T}}

\def\cH{\mathcal{H}}

\def\ccL{\mathscr{L}}

\newcommand{\bdelta}{{\boldsymbol{\delta}}} 

\newcommand{\bcR}{\boldsymbol{\mathcal R}}
\newcommand{\bcT}{\boldsymbol{\mathcal T}}

\newcommand{\ftau}{\boldsymbol{\tau}}
\def\ino{ \! \in \! }
\def\leqo{\! \leq \! }
\def\geqo{\! \geq \! }
\def\eqo{\! = \! }
\def\leko{\! < \! }
\def\geko{\! > \! }
\def\bbZZ{\bbZ \! \cup \! \{\!  - \infty \}} 
\def\obbU{\overline{\bbU}}
\def\obbT{\overline{\bbT}}
\def\bbn{\boldsymbol{|}\!\! |}
\def\lll{| \! |}

\def\fbeta{\boldsymbol{\beta}}

\def\mnu{\mathtt{m}_\nu}
\def\Loo{L_{\mathtt o}}

\def\fftree{t}
\def\varep{\varepsilon}
\def\tta{\mathtt{a}}
\def\ttb{\mathtt{b}}
\def\ttc{\mathtt{c}}
\def\ttx{\mathtt{x}}
\def\ttu{\mathtt{u}}
\def\ttv{\mathtt{v}}

\def\ttz{\mathtt{z}}
\def\ttm{\mathtt{m}}
\def\cW{\mathcal{W}}
\def\cR{\mathcal{R}}
\def\ccC{\mathscr{C}}

\def\cC{\mathcal{C}}

\def\bz{\mathbf{z}}

\def\MMM{\mathbf{M}}
\def\MMT{\mathbf{MT}}
\newcommand{\fSigma}{\boldsymbol{\Sigma}}

\def\bCqq{\bC^{_0}_{^{q}}}
\def\bCpl{\bC^{_0}_{^{\!+}}}

\DeclareBoldMathCommand\blvert{\left\lVert}
\DeclareBoldMathCommand\brvert{\right\rVert}

\newcommand{\eqnsection}{
\renewcommand{\theequation}{\arabic{section}.\arabic{equation}}
    \makeatletter
    \csname  @addtoreset\endcsname{equation}{section}
    \makeatother}
\eqnsection

\begin{document}

\setlength{\abovedisplayskip}{5pt}
\setlength{\belowdisplayskip}{5pt}

\title{\textbf{Scaling limit of the range of tree-valued branching random walks in random environment}}
\date{}
\author{Thomas \textsc{Duquesne}
\thanks{Sorbonne Universit\'e, Campus Pierre et Marie Curie, 
LPSM, Case courrier 158, 4, place Jussieu, 
75252 Paris Cedex 05 
France. Email: thomas.duquesne@sorbonne-universite.fr} 
\and
Robin \textsc{Khanfir}
\thanks{Department of Mathematics and Statistics, McGill University, Montr{\'e}al, Qu{\'e}bec, Canada. Email: robin.khanfir@mcgill.ca} 
}

\maketitle

\begin{abstract}
We study a branching random walk (BRW) taking its values in a random tree $\bT$ (seen as a family tree) with an infinite line of ancestors that is a variant of a supercritical Galton--Watson (GW) tree with offspring distribution $\nu$. 
The transition probabilities of the BRW are those of a critical biased random walk on $\bT$: namely,  
the probability to move from $x$ to one of its $k_x$ children is $1/(\mathtt{m}_\nu+k_x)$ and the probability to move from $x$ to the direct parent of $x$ is $\mathtt{m}_\nu/(\mathtt{m}_\nu+k_x)$. Here $\ttm_\nu$ stands for the mean of $\nu$.  
The BRW is indexed by a critical GW tree conditioned to have $n$ {vertices} and whose offspring distribution is in the domain of attraction of an $\alpha$-stable law with $\alpha \ino (1, 2]$. We denote by $\cR_n$ the range of the BRW, i.e., ~the set of all sites in $\bT$ visited by the BRW. Under a moment assumption for $\nu$, we prove that if we view $\cR_n$ as a random subtree of $\bT$ equipped with its graph distance $d_{\mathtt{gr}}$ and with its occupation measure $\ttm^{_{(n)}}_{{\mathtt{occ}}}$ 
then there exists a scaling sequence $s_n \! \to \! \infty$ such that 
conditionally given the environment $\bT$, the measured metric space $(\cR_n, s_n^{-1}d_{\mathtt{gr}} , \frac{_1}{^n}\ttm^{_{(n)}}_{{\mathtt{occ}}} )$ weakly converges in the Gromov--Hausdorff--Prokhorov sense to a random measured 
compact real tree introduced by Curien, Le Gall \& Miermont in \cite{CuLGMi13} called the Brownian cactus with $\alpha$-stable branching mechanism. This work extends in random environment the result from D., K., Lin  \& Torri \cite{DuKhLiTo22} which deals with the case where $\bT$ is a regular tree.

\medskip

\noi
\textbf{Keywords} $\, $ Branching random walks $\cdot$ Galton--Watson tree $\cdot$ Random environment $\cdot$ Scaling limit $\cdot$ Superprocess $\cdot$ Brownian snake $\cdot$ Brownian cactus
$\cdot$ Real tree

\medskip

\noi
\textbf{Mathematics Subject Classification} $\, $ 60J80  $\cdot$ 60G50 $\cdot$ 60G52 $\cdot$ 60F17

\end{abstract}

\section{Introduction}
\label{introsec}

Branching random walks (BRW) represent a simple model to study a branching behavior coupled with some spatial displacements. As such, they drew a lot of interest and are naturally linked with various areas of research. Let us mention at least superprocesses, traveling wave solutions of semi-linear partial differential equations (FKPP), and some statistical mechanics models such as Mandelbrot cascades, the contact process, the Gaussian free field, or the voter model. We refer to Shi~\cite{Shi15} for an overview of this topic and its connections. We also cite the study of extreme particles of real-valued BRWs by Biggins~\cite{Big77} and Bramson~\cite{Bra78,Bra78bis}, the works of Gouëzel, Hueter, Lalley, Sellke~\cite{LalSel97, LalGou13,LalHue00, Lal06} about BRWs with values in hyperbolic spaces, and the study of BRWs on trees of 
Benjamini \& Müller~\cite{BenMue12} and Liggett~\cite{Lig96}.

In contrast to those works where the indexing tree is infinite, D., K., Lin \& Torri~\cite{DuKhLiTo22} consider a BRW that is indexed by a critical Galton--Watson (GW) tree conditioned to have $n$ vertices. In \cite{DuKhLiTo22}, the BRW moves like a nearest-neighbor null-recurrent random walk (RW) on a deterministic rooted regular tree: it jumps to the parent of the current position with probability $\frac{1}{2}$, or it chooses uniformly one of the sites above otherwise. With this setting, D., K., Lin \& Torri~\cite{DuKhLiTo22} describe the asymptotic behavior of the range of the BRW (i.e.,the set of all visited sites) as $n\! \to\! \infty$ by showing that its scaling limit is a variant of the Brownian cactus introduced by Curien, Le Gall \& Miermont~\cite{CuLGMi13}. Here, we obtain a similar result when the environment is instead a random tree with infinite depth.

We now give an informal description of the environment and the transitions of the BRWs that we consider in this paper. We fix a probability distribution $\nu=(\nu(k))_{k\in\bbN}$ on $\bbN$ that we assume to be supercritical with a finite moment of order $2$:
\begin{equation}
\label{defmusiintro}
\mathtt{m}_\nu:=\sum_{k\in\bbN}k\nu(k)\in(1,\infty)\quad\textrm{and}\quad\sigma_\nu^2:=\frac{1}{\mathtt{m}_\nu(\mathtt{m}_\nu-1)}\sum_{k\in\bbN}k(k-1)\nu(k)\in(0,\infty).
\end{equation}
We denote by $\bT$ the environment and by $\boo\in\bT$ the starting point of the BRW. One can understand $(\bT, \boo)$ as a GW tree with offspring distribution $\nu$ that is modified so that the distinguished point $\boo$ has an infinite line of ancestors: informally,  
the numbers of children $k_x(\bT)$ of the vertices $x\in\bT$ are independent; outside of the ancestral line of $\boo$, their common law is $\nu$; on the ancestral line of $\boo$, their law is the biased probability 
distribution $(k\nu (k)/ \ttm_\nu)_{k\in \bbN}$.  
We say that $(\bT,\boo)$ is an \emph{infinite GW tree with offspring distribution $\nu$} and we refer to Definition~\ref{GWptdef} for a more precise construction. 
For all $x\in\bT$, let $\overleftarrow{x}$ stand for the parent of $x$ (which is always well-defined since $\bT$ has an infinite line of ancestors). 

  The BRW on $(\bT,\boo)$ we are interested in shares its spatial displacements with the so-called critical biased RW. For all $\lambda\in\bbR_+$ and conditionally given $(\bT,\boo)$, the $\lambda$-biased RW on $(\bT,\boo)$ is the Markov chain  whose transition probabilities are expressed by
\begin{equation}
\label{biaRWintro}
p_{\bT,\lambda}(x,y)=\I{\overleftarrow{y}=x}\frac{1}{\lambda+k_{x}(\bT)}+\I{\overleftarrow{x}=y}\frac{\lambda}{\lambda+k_x(\bT)}
\end{equation}
for all $x,y\in\bT$. 
This model has been introduced by Lyons~\cite{Lyons90} who also showed that the behavior of biased RWs is subject to a phase transition concerning the parameter $\lambda$: we refer to Section~\ref{biaRWsec} for a more detailed account. We shall only consider the case $\lambda=\mathtt{m}_\nu$ which is critical in the sense that a 
$\ttm_\nu$-biased RW on an ordinary GW tree with offspring distribution $\nu$ is null-recurrent. 

The present paper is related to earlier works about scaling limits for the range of tree-valued critical or near-critical biased RW. First, D.~\cite{Du05} studies {near-critical biased} RWs on regular trees and show that their rescaled ranges converge to a CRT coded by Brownian motions with drift. 
Then, Peres \& Zeitouni~\cite{PerZei08} prove that the process of the distances to the root of a critical biased RW in a supercritical GW tree without leaves is diffusive. Similarly, Dembo \& Sun~\cite{DeSu12} treat the cases of critical biased RW on multi-type GW trees. Furthermore, Aïdékon \& de Raphélis~\cite{AidRap17} strengthen Peres \& Zeitouni's result and show that the rescaled range of the same RW converges to a variant of the Brownian CRT. We also mention Chen \& Miermont~\cite{ChMi17} who show, using previous results due to Bougerol \& Jeulin~\cite{BoJe99}, that the scaling limits of {Brownian} bridges and loops in hyperbolic spaces are the Brownian CRT. Independently, Stewart proves in his Ph.D.~thesis~\cite{MR3697600} that the rescaled simple RW bridges on a $\mathtt{b}$-regular tree (with $\mathtt{b}\geq 3$) converge to the Brownian CRT.
\medskip

Now, let us discuss the assumptions we make on the genealogy of the BRWs, which is always supposed to be independent of the environment $(\bT,\boo)$. We fix another probability measure $\mu\eqo (\mu(k))_{k\in\bbN}$ on $\bbN$. We denote by $\tau$ a (rooted and ordered) GW tree with offspring distribution $\mu$ that satisfies the following with $\alpha \ino (1, 2]$:
\begin{equation} 
\label{hypostaintro} 
\begin{cases}
(\mathrm{H}_1):&  \;  \mu (0) + \mu (1) <  1 \quad \textrm{and} \quad \sum_{k\in \bbN} k\mu(k)= 1 , \\
 (\mathrm{H}_2):& \; \textrm{$\mu$ is aperiodic, {(namely, $\mu$ is not supported by a proper subgroup of $\bbZ$),}} \\
 (\mathrm{H}_3):& \; \textrm{$\mu$ is in the domain of attraction of an $\alpha$-stable law. }
\end{cases}
\end{equation}
The assumption $(\mathrm{H}_1)$ implies that the number of vertices $\#\tau$ of $\tau$ is almost surely finite. The assumption $(\mathrm{H}_2)$ is convenient to work with because it implies that $\bP(\#\tau \eqo n)\geko 0$ for all integers $n$ sufficiently large. We recall that $(\mathrm{H}_3)$ translates into the following assertion: if $(V_n)_{n\in \bbN}$ is a RW on $\bbZ$ starting at $V_0\eqo 0$ and whose jump law is given by $\bP(V_1\eqo k)\eqo \mu (k+1)$ for all $k\ino \{ -1\}\! \cup \! \bbN$, then there exists a $\frac{\alpha-1}{\alpha}$-regularly varying sequence $(a_n)_{n\in \bbN}$ such that 
\begin{equation}
\label{cvlawRW_bis}
\tfrac{a_n}{n}V_n \xrightarrow[n\to \infty]{\textrm{(law)}}X_1\; , \quad\textrm{where $\log \bE[\exp(-\lambda X_1)]= \lambda^\alpha$ for all $\lambda\ino \bbR_+$}\; .
\end{equation}
We refer to Section~\ref{stabletree} for more details on these assumptions. 

We next define the BRW we consider as follows. 
For all sufficiently large $n$, we denote by $\tau_n$ a tree with the same law as a GW tree $\tau$ with offspring distribution $\mu$ under $\bP (\, \cdot \, \big| \,  \# \tau \eqo  n \big)$. We denote by $\varnothing$ the root of $\tau_n$ and to each vertex $u\ino \tau_n$, we associate a spatial position $Y^{_{(n)}}_{^{\! u}}\ino \bT$ whose conditional joint law given $(\bT, \boo)$ and $\tau_n$ is 
defined by 

\begin{equation}
\label{lawBRWintro} \bP \big(\,  \forall u\in \tau_n , \; Y^{(n)}_u \eqo  x_u\, \big|\,  (\bT, \boo), \tau_n \, \big)= \un_{\{ Y^{(n)}_\varnothing = \boo \}  } \!\!\! \prod_{u\in \tau_n \backslash \{ \varnothing \}}\!\!\!\!  p_{\bT, \ttm_\nu} ( x_{\overleftarrow{u}}, x_u) 
\end{equation}
for all $x_u\ino \bT$ and $u\ino \tau_n$ (here $\overleftarrow{u}$ stands for the direct parent of $u$ in $\tau_n$). 

It is convenient to see $\tau_n$ as a family tree whose ancestor corresponds to the root $\varnothing$ 
and where siblings are ordered by birth-rank. The depth-first exploration of $\tau_n$ is 
the sequence of vertices $(u_j)_{0\leq j < \# \tau}$ that is defined recursively as follows: 
$u_0$ is the root $\varnothing$ and for all $j\ino \{ 0, \ldots, \#\tau \! -\! 2\}$, if $v$ is the most recent ancestor of $u_j$ having at least one unexplored child then $u_{j+1}$ is the unexplored child of $v$ with least birth-rank (see Section~\ref{rootreesec} for a more precise construction).
\medskip

Our main result concerns a quenched limit of the following rescaled random pointed measured metric spaces  
\begin{equation}
\label{spacdefintro} \ftau_n := \big( \tau_n , \tfrac{1}{a_n} d_{\mathtt{gr}} , \varnothing , \tfrac{1}{n}\ttm_n \big)  \quad \textrm{and} \quad  \bcR_n := \big( \mathcal R_n ,  \tfrac{\sigma_\nu}{\sqrt{a_n}} d_{\mathtt{gr}} , \boo,  \tfrac{1}{n}  \ttm_{\mathtt{occ}}^{(n)}\big)
\end{equation}
where $d_{\mathtt{gr}}$ stands for the graph distance in $\tau_n$ and in $\bT$, where we recall from (\ref{defmusiintro}) the definition of $\sigma_\nu$ and where we have set: 
$$  \mathcal R_n = \big\{ Y_v \, ; \, v\ino \tau_n \big\} , \quad \ttm_n \eqo \sum_{v \in \tau_n } \delta_v ,   
\quad \textrm{and} \quad \ttm_{\mathtt{occ}}^{(n)}= \sum_{v\in \tau_n } \delta_{Y^{(n)}_v} . $$
Let us first recall the scaling limit of $\ftau_n$;  $\ftau_n$ is encoded by its \emph{height process} $(H_{s}^{_{(n)}})_{s\in[0,n]}$, i.e., the continuous interpolation of the discrete-time process $H^{_{(n)}}_j\eqo d_{\mathtt{gr}} (\varnothing, u_j)$ which is the height in $\tau_n$ of the $j$-th vertex of 
$\tau_n$ in the depth-first exploration. 
Then, Theorem 3.1 in D.~\cite{Du03} asserts the following: 
\textit{we assume that (\ref{lawBRWintro}) holds and let $(a_n)$ be as in (\ref{cvlawRW_bis}); then there is a nonnegative continuous process $H\eqo (H_{s})_{s\in [0, 1]}$, whose law only depends on $\alpha$ and such that the convergence}
\begin{equation}
\label{heightcvintro}
\big(\tfrac{1}{a_n}H_{ns}^{(n)}\big)_{s\in[0,1]}\xrightarrow[n\to\infty]{} (H_s)_{s\in [0, 1]}
\end{equation}
\textit{holds weakly for the topology of uniform convergence}. Under the stronger assumption that the variance of $\mu$ is finite, this result is due to Aldous~\cite[Theorem 23]{Al93}. The process $H$ is the normalized excursion of the $\alpha$-stable height process, which is a local-time function of an $\alpha$-stable spectrally positive L\'evy process. In the case $\alpha \! = \! 2$, $H$ is more precisely the normalized Brownian excursion. We refer to Section \ref{stabletree} for precise definitions.

As in the discrete setting, $H$ actually encodes a random pointed compact measured metric space that is defined as follows: we define a pseudometric $d_H$ on $[0,1]$ from $H$ by setting
\begin{equation}
\forall 0\leq s_1\leq s_2\leq 1,\quad d_H (s_1, s_2) \! = \! d_H(s_2,s_1) \! = \! H_{s_1}+ H_{s_2} -2\min_{s\in[s_1,s_2]} H_s\, , 
\end{equation}
and we introduce the relation $\sim_H $ on $[0, 1]$ by setting $s_1 \sim_H s_2$ {if and only if} $d_H (s_1, s_2)\! = \! 0$; 
clearly, $\sim_H$ is an equivalence relation and the normalized $\alpha$-stable L\'evy tree is taken as the quotient space $ T_H \! = \! [0, 1 ] / \! \sim_H $, equipped with the distance induced by $d_H$ that we keep denoting $d_H$. We denote by $p_H \! : \! [0, 1] \rightarrow T_H$ the canonical projection. Note that $p_H$ is continuous; therefore $T_H$ is compact and connected{.} Moreover, $T_H$ is a real tree, namely, a metric space such that all pairs of points are joined by a unique simple arc that turns out to be a geodesic (see Section \ref{GHPsec} for more details). We set {$\mathtt r_H\! := \! p_H (0)$} that is viewed as the root of $T_H$ and we equip $T_H$ with the measure $\ttm_H$ that is the image of the Lebesgue measure on $[0, 1]$ via $p_H$, namely, $\int_{T_H} \! f \, d\ttm_H\! = \! \int_0^1 \! f(p_H(s)) \, ds$, for all 
continuous $f\! : \! T_H \! \rightarrow \! \bbR$. The convergence (\ref{heightcvintro}) 
then implies the following one:
\begin{equation}
\label{Strobl}
\ftau_n =\big( \tau_n, \tfrac{1}{a_n} d_{\mathtt{gr}} , \varnothing, \tfrac{1}{n} \mathtt{m}_n  \big) 
\underset{n\rightarrow \infty}{- \!\!\! - \!\!\! - \!\!\! \longrightarrow } \; (T_H, d_H, \mathtt r_H, \ttm_H) \, .
\end{equation}
Here the convergence holds weakly on the space $\bbM$ of isometry classes of pointed measured compact metric spaces equipped with the Gromov--Hausdorff--Prokhorov distance 
$\bdelta_{\mathtt{GHP}}$ that makes it a Polish space, as proved in Abraham, Delmas \& Hoscheit in \cite{AbDeHo13}, Theorem 2.5 (see Section \ref{GHPsec} for precise definitions).

Then the limit of  $\bcR_n$ is constructed as 
follows: as proved in {D.~\& Le Gall}~\cite{DuLG05} (Lemma 6.4 p.~600) 
conditionally given $H$, there exists a Hölder-continuous centered Gaussian process $\sigma \ino T_H \! \longmapsto  \! W_\sigma \ino \bbR$ whose covariance is characterized by 
$\bE \big[ \big| W_{\sigma_1} \! - W_{\sigma_2} \big|^2 \big| \, H \big] \!  = \! 
d_H (\sigma_1, \sigma_2)$, for all $\sigma_1, \sigma_2 \ino T_H$. Then, we set 
\begin{equation*}
\forall \sigma_1, \sigma_2\ino T_H, \quad d_{H, W} ( \sigma_1, \sigma_2) = W_{\sigma_1}  +  W_{\sigma_2} -2 \!\!\!\! \!\! \!   \min_{ \quad \sigma \in \lgeo \sigma_1, \sigma_2 \rgeo} \!\! \!\! \!\! \! W_\sigma ,
\end{equation*}
where $\lgeo \sigma_1, \sigma_2 \rgeo $ is the unique geodesic that joins $\sigma_1$ to 
$\sigma_2$ in $T_H$. 
As proved in D., K., Lin \& Torri \cite{DuKhLiTo22},  $d_{H, W}$ is a pseudometric on $T_H$;  
we then define the equivalence relation $\sim_{H, W} $ on $T_H$ by setting $\sigma_1 \sim_{H, W}  \sigma_2$ {if and only if}
$d_{H, W} ( \sigma_1, \sigma_2) \! = \! 0$ and we denote by $T_{H, W}\! = \! T_H / \! \sim_{H, W}$ the quotient metric space and we keep denoting by 
$d_{H, W}$ the resulting metric; we denote by $\pi_{H, W}\! : \! T_H \! \rightarrow \! T_{H, W}$ the canonical projection that is continuous. Thus $T_{H, W}$ is compact and connected, and 
$(T_{H, W}, d_{H, W})$ is a real tree (see Section \ref{trebrRWsnasec} for more details). 
It turns out that this kind of spaces has been introduced in Curien, Le Gall \& Miermont~\cite{CuLGMi13} (see also Le Gall~\cite{LG15} for a different purpose); 
they coined the name Brownian cactus, so we call $(T_{H, W}, d_{H,W})$ the normalized Brownian cactus with $\alpha$-stable branching mechanism.  
We next set $\mathtt r_{H, W}\! = \! \pi_{H, W} (\mathtt r_H)$ that is viewed as the root of $T_{H, W}$ 
and we equip $T_{H, W}$ with the measure $\ttm_{H, W}$ that is the image 
of $\ttm_H$ via $\pi_{H, W}$: namely, $\int_{T_{H, W}} \! f \, d\ttm_{H, W}\! = \! \int_{T_H} \! f(\pi_{H, W}(\sigma)) \, \ttm_H (d\sigma) $, for all continuous $f\! : \! T_{H, W} \! \rightarrow \! \bbR$. Our main result is the following limit theorem.

\begin{theorem}
\label{QumetlimbrRW}
We keep the above notation: recall that $(\bT,\boo)$ is an infinite GW tree whose offspring distribution $\nu$ is supercritical, that we denote by $\tau_n$ is a GW tree with offspring distribution $\mu$ conditioned to have $n$ vertices for all sufficiently large $n$, and that $(\bT,\boo)$ and $\tau_n$ are independent. We assume that $\mu$ satisfies (\ref{hypostaintro}) and that $\nu$ satisfies (\ref{defmusiintro}).  We furthermore assume that there exists $\ttb \ino \big(\tfrac{2\alpha}{\alpha -1} , \infty \big)$ such that $\sum_{k\in \bbN} k^{1+ 2\ttb} \nu (k) \leko \infty$. 
Recall from (\ref{spacdefintro}) the definition of the spaces $\ftau_n$ and $\bcR_n$. 
We denote by $(T_H, d_H, \mathtt r_h, \ttm_H)$ the normalized $\alpha$-stable L\'evy tree and by 
$(T_{H, W} , d_{H, W}, \mathtt r_{H, W}, \ttm_{H, W})$ the normalized Brownian cactus with $\alpha$-stable branching mechanism as defined above. Recall that $\bbM$ is the space 
of isometry classes of pointed measured compact metric spaces equipped with the Gromov--Hausdorff--Prokhorov distance. 
Then, \emph{conditionally given the environment} $(\bT, \boo)$, the following joint convergence 
\begin{equation}
\label{metrcv}
\big( \ftau_n, \bcR_n \big) \underset{n\rightarrow \infty}{- \!\!\! - \!\!\! - \!\!\! \longrightarrow }  \Big(\big(T_H,d_H,\mathtt r_H,\ttm_H\big)\, ;\, \big(T_{H,W},d_{H,W},\mathtt r_{H,W},\ttm_{H,W}\big)\Big)
\end{equation}
holds weaky on $\bbM^2$ equipped with the product topology. 
\end{theorem}
\medskip

The proof of Theorem~\ref{QumetlimbrRW} is done by proving an analog of (\ref{heightcvintro}) for the BRW. Namely, we introduce the discrete snake associated with the BRW $(Y^{(n)}_{u})_{u\in \tau_n}$: this process $(\widehat{W}^{(n)}_s  )_{s\in [0, n]}$ is the continuous interpolation of the discrete-time process 
$$  \widehat{W}^{(n)}_j = \; \textrm{generation in $\bT$ of $Y_{u_j}^{(n)}$ } $$
where $(u_j)_{0\leq j< n}$ is the depth-first exploration of $\tau_n$. Here, by convention, the generation of $\boo$ is taken equal to $0$, so the generation of $x$ in $\bT$ is a possibly negative integer. Then the following is proved in Theorem \ref{cvsnacondi}: conditionally given $(\bT, \boo)$ the following convergence
\begin{equation}
\label{BRWscaling_limit_these_envt_INTR}
\big(\tfrac{1}{a_n}H_{ns}^{(n)},\tfrac{\sigma_\nu}{\sqrt{a_n}}\widehat{W}_{ns}^{(n)}\big)_{s\in[0,1]} \underset{n\rightarrow \infty}{- \!\!\! - \!\!\! - \!\!\! \longrightarrow } (H,\widehat{W})
\end{equation}
holds weakly on the space $\bC([0,1],\bbR^2)$ of continuous functions equipped with the topology of uniform convergence. Here $H$ is the $\alpha$-stable height process as introduced above and $\widehat{W}$ is the (endpoint process of the) one-dimensional Brownian snake whose lifetime process is $H$: namely, it is continuous centered real Gaussian process $(\widehat{W}_s)_{s\in[0,1]}$ whose covariances are characterized by 
$\bE \big[ \big| \widehat{W}_{s_1} \! - \widehat{W}_{s_2} \big|^2 \big| \, H \big] \!  = \! 
d_H (s_1, s_2)$ for all $s_1,s_2 \ino [0,1]$ (see {D.~\& Le Gall}~\cite[Lemma 6.4 p.~600]{DuLG05}). We refer to Section \ref{Cosnake} for a precise definition.  

To prove (\ref{BRWscaling_limit_these_envt_INTR}) and more generally Theorem \ref{cvsnacondi}, 
we rely on the so-called harmonic coordinates used by Peres \& Zeitouni~\cite{PerZei08}. These are a family $(S_x)_{x\in\bT}$ of real random variables (r.v.s) depending on the environment $(\bT,\boo)$ which approximates the relative generation of $x\ino \bT$ on the one hand, and on the other hand such that $(S_{X_n})_{n\in\bbN}$ is a martingale when $(X_n)$ is a $\ttm_\nu$-biased RW on $\bT$. 
This tool allows us to adapt the method of Janson \& Marckert~\cite{JanMar05} and Marzouk~\cite{Mar20}: we obtain the tightness thanks to a maximal inequality for martingales, together with a control of the variations of the height process of the genealogy, and we check the convergence of the finite-dimensional marginals via a conditional central limit theorem.

Theorem~\ref{QumetlimbrRW} echoes Theorem~1.2 in D., K., Lin \& Torri~\cite{DuKhLiTo22}. As in this article, our general idea 
consists in deriving the metric convergence in Theorem~\ref{QumetlimbrRW} from the snake convergence (\ref{BRWscaling_limit_these_envt_INTR}) (more precisely from Theorem \ref{cvsnacondi}) 
thanks to general arguments that are in part discussed in 
\cite{DuKhLiTo22}. Namely, we show that $\bcR_n$ is close in some sense to the discrete snake-tree coded by 
$\widehat{W}^{(n)}$: see Section \ref{trebrRWsnasec} and in particular Lemma \ref{graphvssna} for precise statements.

\subsection*{Organization of the paper}

In Section~\ref{backgrdsec} we present the main objects and tools that are involved in this article. Among others, this includes the formalism of trees with an infinite line of ancestors, some properties about biased RWs on trees, and definitions and limit theorems for snakes and the real trees they code. In Section~\ref{tauinftysec}, we prove Theorem \ref{cvsnakes}, which a limit theorem for the snake associted with a BRW indexed by a forest of unconditioned GW trees. It is the main technical step of the artcile.  Section~\ref{QumetlimbrRWpfsec} is devoted to the proof of Theorem~\ref{cvsnacondi}, which then completes the demonstration of Theorem~\ref{QumetlimbrRW}

\section{Tree-valued BRWs and snakes: definitions, previous results}
\label{backgrdsec}

In this section, we present the main objects and tools which are used  in this article. Among others, this includes the formalism of trees with an infinite line of ancestors, some properties about biased RWs on trees, and definitions and limit theorems for snakes and the real trees they encode. 
\subsection{Rooted ordered trees and forests, coding functions, discrete snakes.}
\label{rootreesec}
In this section, we recall the framework of (\emph{rooted ordered trees} (within Ulam's formalism) and their encoding functions:  \emph{Lukasiewicz paths}, \emph{height} and \emph{contour} processes. We then briefly recall the definition of Galton--Watson trees and forest (GW trees and forest) within this framework and finally we recall the definition of \emph{discrete snakes} which  encode real-valued branching walks.

First, let us briefly recall Ulam's formalism on rooted ordered trees. We first introduce the set of finite words written with positive integers: 
$$  \bbU\! = \! \bigcup_{n\in \bbN} (\bbN^*)^n\; .$$
Here, $\bbN\! = \! \{ 0, 1, 2, \ldots\}$, $\bbN^*\! = \! \bbN \backslash \{ 0\}$ and 
$(\bbN^*)^0$ stands for $\{ \varnothing \}$. The set $\bbU$ is totally ordered by the \textit{lexicographical order} $\leq_{\mathtt{lex}}$ (the strict order is denoted by $<_\mathtt{lex}$). 
It is convenient to think of $\bbU$ as a family tree whose $\varnothing$ is the ancestor. 
Let $u\! = \! (j_1, \ldots , j_n)\ino \bbU\backslash \{ \varnothing\}$. 
We use the notation $|u|\! = \! n$ for the \textit{height} of $u$ (with the convention $|\varnothing |\! = \! 0$) and we 
set $\overleftarrow{u}\! = \! (j_1, \ldots, j_{n-1})$ that is interpreted as the (direct) \textit{parent of $u$}. More generally, for all $p\ino \bbN^*$, we set $u_{|p} \eqo  (j_1, \ldots , j_{n\wedge p})$ and $u_{|0}\eqo \varnothing$. Namely, $u_{|p}$ is the \emph{ancestor of $u$ at the generation $p$} and we observe that $u_{|p}\eqo u$ if $p\geqo |u|$ and that 
$u_{|p}\eqo \overleftarrow{u}$ if $p\eqo |u|\! -\! 1$. 
Let $v\! = \! (k_1, \ldots, k_m)$, we denote by $u\ast v$ the concatenated word $(j_1, \ldots, j_n, k_1, \ldots, k_m)$. 
For all $u, v\ino \bbU$, we then denote by $u \wedge v$ the \emph{most recent common ancestor} of $u$ and $v$, which is defined by $u\wedge v\eqo u_{|p} \eqo v_{|p}$ where $p\eqo \max \{ q\ino \bbN:   u_{|p} \eqo v_{|p}\}$. 

We can think of $\bbU$ as a graph whose (undirected) edges are $\{ u, \overleftarrow{u} \}$, $u\ino \bbU\backslash \{ \varnothing\}$. Then, for all $u,v\ino \bbU$, we denote by 
$\lgeo u, v\rgeo  \! \subset \! \bbU$ the shortest path in the graph $\bbU$ joining $u$ to $v$. We
 also use the following notation 
$\, \rgeo u , v\rgeo \!   = \! \lgeo u , v\rgeo  \backslash \{ u\}$, $\lgeo u,v\lgeo  \,  = \! \lgeo u, v\rgeo  \backslash \{ v\}$ and $\, \rgeo u,v\lgeo  \,  = \! \lgeo u,v \rgeo  \backslash \{u,v \}$. 
Note that for integers $p \leqo q$, we also use the notation $\lgeo p, q\rgeo $ to denote the set of integers $[p,q] \cap \bbZ$ (along with the other obvious pieces of notation $\lgeo p, q \lgeo\, $,$\, \rgeo p, q \rgeo$, and $\, \rgeo p, q \lgeo\, $). We do not think that there is any conflict in the notation here, as the context always removes any ambiguity.

\begin{definition}
\label{Ulamtree}  \textbf{(Rooted ordered trees, subtrees, forests)}
$(\textbf{a})\, $ A \textit{rooted ordered tree} is a subset $ \fftree\! \subset \! \bbU$ satisfying the following conditions: 
\begin{compactenum}
\item[] \textbf{(1):} $\varnothing \ino \fftree$. \textbf{(2):} If $u\ino  \fftree\backslash \{ \varnothing\}$ 
then $\overleftarrow{u}\ino  \fftree$. \textbf{(3):} For all $u\ino  \fftree$, there exists $k_u ( \fftree)\ino \bbN\cup \! \{ \infty\}$ such that: $u \! \ast \!  (j)\ino  \fftree$  if and only if  $1\! \leq \! j \! \leq \! k_u ( \fftree)$. 
\end{compactenum}
Here, $k_u ( \fftree)$ is interpreted as the \textit{number of children of $u$} and if 
$1\! \leq \! j \! \leq \! k_u ( \fftree)$, then $u\ast (j) $ is the \emph{$j$-th child of $u$.}
We denote by $\bbT$ the set of rooted ordered trees.

\smallskip

\noindent
$(\textbf{b})\, $
Let $ \fftree\ino \bbT$ and $u\ino t$. Then 
$\theta_u  \fftree\!  = \! \{ v\ino \bbU: u\ast v\ino  \fftree \}$ is the \emph{subtree of $ \fftree$ stemming from $u$}. 
Note that $\theta_u t\ino \bbT$. 

\smallskip

\noindent
$(\textbf{c})\, $ A \emph{fores}t is a sequence of trees 
$(t^{_{(k)}}_{^{\,}} )_{k\in \bbN^*}$. It is sometimes convenient to view it as a single tree $t\ino \bbT$ with $k_\varnothing(t)\eqo\infty$ such that 
$t^{_{(k)}}_{^{\,}}$ is the subtree stemming from the $k$-th child of $t$: $t^{_{(k)}}_{^{\,}}\eqo \theta_{(k)}t$. \cq 
\end{definition}

\begin{definition}
\label{Lukadef} (\textbf{Lukasiewicz path, height process}) 
$(\textbf{a})\, $ Let $\fftree \ino \bbT$ be \emph{finite} and let $(u_l)_{0\leq l<\# t}$ be the vertices of $\fftree$ listed in 
increasing lexicographical order. By convenience we set $u_{\# t} \eqo \varnothing$ 
and we define $V(\fftree)\eqo (V_l(\fftree))_{0\leq l\leq  \# t}$ and $H(\fftree)\eqo (H_l(\fftree))_{0\leq l \leq  \# t}$ by setting $V_{0}(t)\eqo 0$, $H_{\# t} (t)\eqo 0$  and for all $0\leqo l \leko \# t$,
\begin{equation}
\label{coding}
V_{l+1} (\fftree) \! = \! V_{l}(\fftree)  + k_{u_{l}} (\fftree)\! -\! 1 \quad \textrm{and} \quad H_l (\fftree) \! =\!  |u_{l}|. 
\end{equation}
$V(\fftree)$ is the \emph{Lukasiewicz path} of $\fftree$ and $H(\fftree)$ is the \emph{height process} of $\fftree$. We extend $H(t)$ continuously on $[0, \#t]$ as follows: for all integer $0\leqo l\leko \#t $ and all $s\ino [0,1]$, we set 
\begin{equation}
\label{ctrvshght}
H_{l+s}(t)  = \left\{  \begin{array}{ll}
H_l (t)\! - \! \big(H_l (t) \! -\! |u_l \! \wedge \! u_{l+1}| \big) (2s\! \wedge \! 1)  + (2s\! -\! 1 )_+   &  \textrm{if $|u_{l+1}| \! \leq \! |u_l |$ } \\
H_l (t) + s   & \textrm{if  $|u_{l+1}| \geko | u_l | $.}
\end{array} \right.
\end{equation}
(see Remark \ref{bijecfor} $(\textbf{c})$ for a motivation).  

\smallskip

\noi
$(\textbf{b})\, $ Let $(t^{_{(k)}}_{^{\, }})_{k\in \bbN^*}$ be a forest of finite trees and $t\ino \bbT$ its associated tree as in Definition \ref{Ulamtree} $(\textbf{c})$. The Lukasiewicz path and the height process of the forest are denoted resp.~by $V(t)$ and $H(t)$ and they are obtained by concatenating the $V(t^{_{(k)}}_{^{\, }}) $ and $H(t^{_{(k)}}_{^{\, }}) $ in the following way. We set 
$\sigma_0\eqo 0$ and for all $k\ino \bbN^*$, we set $\sigma_k\! = \#  t^{_{(1)}}_{^{\, }}+ \ldots+  \# t^{_{(k)}}_{^{\, }}$. Moreover, for all $l \ino \lgeo \sigma_{k-1}, \sigma_k\lgeo \,$ and for all $ s\ino [ \sigma_{k-1}, \sigma_k] $, we set 
\begin{equation}
\label{concate}
V_l (t)\eqo V_{l-\sigma_{k-1}} (t^{_{(k)}}_{^{\, }}) \! -\! (k\! -\! 1) \quad \textrm{and} \quad  H_s (t)\eqo H_{s-\sigma_{k-1}} (t^{_{(k)}}_{^{\, }}). \qquad \square
\end{equation}
\end{definition}
\begin{remark}
\label{bijecfor} 
$(\textbf{a})\, $ Let $t$ be the tree associated with the forest of finite trees $(t^{_{(k)}}_{^{\, }})_{k\in \bbN^*}$ as in Definition \ref{Ulamtree} $(\textbf{c})$. Let $(u_l)_{l\in \bbN\cup \{ -1\}}$ be the sequence of vertices of $\fftree$ in increasing lexicographical order: $ u_{-1} \eqo \varnothing \! <_{\mathtt{lex}} \! u_0 \!   <_{\mathtt{lex}}\!  \ldots \! <_{\mathtt{lex}} \! u_l \! <_{\mathtt{lex}} u_{l+1} \! <_{\mathtt{lex}}\!  \ldots $. Let $V(t)$ and $H(t)$ be as in Definition~\ref{Lukadef} $(\textbf{b})$. 
Then $V_{l+1} (\fftree) \! - \! V_{l}(\fftree)  \eqo k_{u_{l}} (\fftree)\! -\! 1$ and $H_l (\fftree) \! =\!  |u_{l}|\! -\! 1$ for all $l\ino \bbN$. 
Then observe that $s\ino [\sigma_k , \sigma_{k+1} )$ if and only if $\inf_{\lgeo 0, \lfloor s \rfloor\rgeo } V(t)\eqo -k$.

\smallskip

\noi
$(\textbf{b})\, $ Observe that if $t$ is a finite tree or the tree associated to a forest of finite trees, then we can recover $\fftree$ from $H(\fftree)$ or from $V(t)$. Moreover, $H(\fftree) $ is an adapted functional of $V(\fftree)$, as shown by the following equation that holds for finite trees of forest of finite trees. 
\begin{equation}
\label{codheight}
\forall l \ino \bbN^*, \quad  H_l (\fftree)= \# \big\{ m\ino \{ 0, \ldots , l\! -\! 1\}: V_m (\fftree) \! = \! \inf_{m\leq j\leq l} V_j(\fftree)  \big\}. 
\end{equation}
See Le Gall \& Le Jan~\cite{LGLJ98} for more details.

\smallskip 

\noi
$(\textbf{c})\, $ Let $t\ino \bbT$ be a finite tree and let $(u_l)_{0\leq l<\# t}$ be the vertices of $\fftree$ listed in 
increasing lexicographical order. Let $l, l'$ be integers such that $0\leqo l\leqo l'\leko \#t$. Then the height of the most recent common ancestor of $u_l$ and $u_{l'}$ may differ from $\inf_{j\in \lgeo l, l' \rgeo} H_j(t)$ by $1$. The 
specific continuous extension of $H(t)$ implies however that $|u_l\wedge u_{l'}\big|  \eqo \min_{s\in [l, l']} H_s (t)$. 

Similarly, if $t$ is the tree corresponding to a forest of finite trees and if $(u_l)_{l\in \bbN\cup \{ -1\}}$ stands for its sequence of vertices listed in lexicographical order, then we also check that $|u_l\wedge u_{l'}| -1 \eqo \min_{s\in [l, l']} H_s (t)$, for all integers $l'\geqo l\geq 0$.  \cq 
 \end{remark}

\begin{definition}
\label{contourdef} (\textbf{Contour process}) $(\textbf{a})\, $ Let $\fftree\ino \bbT$, possibly infinite.  
Let $(v_k)_{0\leq k \leq 2(\# t-1)}$ (with an obvious convention if $\# t \eqo \infty$) 
the sequence of vertices of $\fftree$ be defined recursively as follows.

\smallskip

\begin{compactenum}
\item[] \emph{We set $v_{0}\eqo \varnothing$ and $R_k \eqo\{ w\ino  \fftree\backslash \{v_0,  \ldots, v_k\} : \overleftarrow{w}\eqo v_k \}$; if $R_k$ is empty, then we set 
$v_{k+1} \eqo \overleftarrow{v}_{\! \! k}$; otherwise we take $v_{k+1}$ as the $<_{\mathtt{lex}}$-least vertex of $R_k$.}
\end{compactenum}

\smallskip
\noi
We call $(v_k)$ the exploration of $t$ in the \emph{contour order}.  

\smallskip 

\noi
$(\textbf{b})\, $ Let $t\ino \bbT$ be finite. For all $k\ino \lgeo 0, 2(\# t \! -\! 1)\rgeo$ we set $C_k (t)\eqo |v_k|$ and we conveniently set $C_{2\# t \! -\! 1} (t)\eqo  C_{2\# t} (t)\eqo 0$. We next extends $C(t)$ continuously on $[0, 2\#t]$ by setting for all $k\ino \lgeo 0, 2\# t \! -\! 1\rgeo$ and all $s\ino [0, 1]$: $ C_{k+s}(\fftree) \eqo C_k (\fftree)+ s (C_{k+1} (\fftree) \! -\! C_k (\fftree))$. 
The resulting continuous function $s\ino [0, 2\#t ] \! \mapsto  \! C_s(\fftree)$ is called the \emph{contour process} of $\fftree$.

\smallskip 

\noi
$(\textbf{c})\, $ Let $t\ino \bbT$ be the tree associated with the forest of finite trees 
$(t^{_{(k)}}_{^{\,}})_{k\in \bbN^*}$ (Definition \ref{Ulamtree} $(\textbf{c})$). The contour process of the forest is denoted by $C(t)$. It is obtained by concatenating the $C(t^{_{(k)}}_{^{\,}}) $ in the following way. Recall notation $\sigma_k$ from Definition \ref{Lukadef} $(\textbf{b})$. Then for all $k\ino \bbN^*$, all $ s\ino [ 2\sigma_{k-1}, 2\sigma_k] $, we set $C_{s} (t)\eqo 
C_{s-2\sigma_{k-1}} (t^{_{(k)}}_{^{\,}})$.  \cq 
\end{definition}

\begin{remark}
\label{contord} Let $t$ be the tree associated with the forest of finite trees $(t^{_{(k)}}_{^{\,}})_{k\in \bbN^*}$ as in Definition \ref{Ulamtree} $(\textbf{c})$. Observe that the exploration of $t$ in the contour order visits all the vertices of $t$. In what follows we denote by $(v_k)_{l\in \bbN\cup \{ -1\}}$ the sequence of vertices of $\fftree$ in the contour order starting at $v_{-1} \eqo \varnothing$ (then $v_0\eqo (1)$, etc).

\medskip

\noi
\textbf{(a)} We first note that $C_k(t) \eqo (|v_k| \! -\! 1)_+$ for all $k\ino \bbN$. 
 
\smallskip

\noi
\textbf{(b)} As already mentioned, $t\eqo \{v_k; k\ino \bbN\cup \{ -1\} \}$. More precisely, during the exploration of $\fftree$ in the contour order, each edge is crossed exactly twice (upwards first and then downwards). Namely, for all $u\ino \fftree\backslash \{ \varnothing\}$, there is a unique pair of integers $k,k^\prime$ such that  
$k\leko k^\prime$ and $(\overleftarrow{u},u)\eqo (v_k, v_{k+1})\eqo (v_{k^\prime+1} , v_{k^\prime})$. 
It is also easy to check recursively 
for all $k\ino \bbN$ that $\{ v_j; -1\leqo j\leqo k\} \cap \{ v_j ; j\geqo k\} \eqo \lgeo \varnothing , v_k\rgeo$.  

\smallskip

\noi
\textbf{(c)} The contour process $C(\fftree)$ completely encodes $\fftree$ and it is closely related to the height process $H(\fftree)$. Indeed, let $(u_l)_{l\in \bbN\cup \{ -1\}}$ stands for the sequence of vertices of $\fftree$ listed in increasing lexicographical order as in Remark \ref{bijecfor}. Then, we observe that 
\begin{equation}
\label{Klcontour}
\forall l\ino \bbN, \quad K_l:= 2l \! -\! H_{l} (\fftree) \eqo \inf \{ k\ino \bbN: v_k\eqo u_l\}
\end{equation}
which necessarily increases with $l$. Thus for all $l\ino \bbN$, 
$C_{K_l} (\fftree) \eqo H_{l} (\fftree)$ and furthermore, for all $s\ino [K_l, K_{l+1} ]$, we get
\begin{equation}
\label{ctrvshght_prime}
C_s (\fftree) = \left\{  \begin{array}{ll}
(H_{l}(\fftree) \! -\! s +K_l)_+ &  \textrm{if $s\ino [K_l, K_{l+1} \! -\! 1)$ } \\
(H_{l+1}(\fftree) \! -\! K_{l+1} +s )_+ & \textrm{if $s\ino [K_{l+1} \! -\! 1, K_{l+1}] $.}
 \end{array} \right.
\end{equation}
( See D.~\& Le Gall \cite[Section 2.4]{DuLG02} for more details).

\smallskip

\noi
$\textbf{(d)} \; $ We can derive $H(t)$ from $C(t)$ by the time-change $\phi_t \! : \! \bbR_+ \! \to \! \bbR_+$ that is defined as follows: 
we set $\phi_t(0)\eqo 0$ and for all $l\ino \bbN$ and all $s\ino [0, 1]$, 
\begin{equation}
\label{phitdef}
\phi_t(l+s) \eqo K_l + (K_{l+1} \! -\! 1\! -\!  K_l) 
(2s\! \wedge \! 1) + (2s\! -\! 1)_+. 
\end{equation} 
Then, $\phi_t$ is an increasing bijection from $\bbR_+$ onto itself and for all $s\ino \bbR_+$
\begin{equation}
\label{timechange}
H_s (t)= C_{\phi_t(s)} (t) \; .
\end{equation}     
We easily check that for all $l\ino \bbN^*$, 
\begin{equation}
\label{controphit}
\max_{s\in [0, l]} \big| \tfrac{1}{2} \phi_t(s) \! -\! s\big| \leq 1+ 2 \sup_{s\in [0, l] } H_s (t) \; .
\end{equation}

When $t\ino \bbT$ is finite, then $H(t)$ is also obtained from $C(t)$ by a continuous increasing time-change 
$\phi_t$ from $ [0, \#t] $ onto $ [0, 2\#t] $ by (\ref{phitdef}): then (\ref{timechange}) holds true for all $s\ino [0, \#t]$ and (\ref{controphit}) holds true for all integers $l \leqo \#t$.

\smallskip
\noi
\textbf{(e)} As in Remark~\ref{bijecfor} $(\textbf{c})$, we observe that
$(|v_k\wedge v_{k'}|  \! -\! 1)_+ \eqo \min_{s\in [k, k']} C_s (t)$ for all integers $k'\geqo k$.
If $t$ is a finite tree, then  $|v_k\wedge v_{k'}|  \eqo \min_{s\in [k, k']} C_s (t)$ is obtained instead, for all integers 
$0\leqo k\leqo k'\leqo 2(\# t\! -\! 1)$. \cq 
\end{remark}

\noi
\textbf{GW trees and forests.} We next recall below a convenient recursive definition of Galton--Watson trees and forests (\emph{GW trees} and \emph{GW forests}, for short). 
To that end, we equip $\bbT$ with the sigma-field $\ccF(\bbT)$ generated by the sets $\{ t\ino \bbT\! : \! u\ino t\}$, $u\ino \bbU$.
Unless otherwise specified, all the r.v.s that are mentioned in this paper are defined on the same probability space $(\Omega, \ccF, \bP)$. 
\begin{definition}
\label{GWfordef} (\textbf{GW trees and forests}) Let $\mu\eqo (\mu(k))_{k\in \bbN}$ 
be a probability measure on $\bbN$. 

\smallskip

\noi
$(\textbf{a})\, $ A \textit{GW tree with offspring distribution $\mu$} (a \textit{GW($\mu$)-tree}, for short) is a random tree $\tau$  that satisfies the following. 
\begin{compactenum}
\item[] \textbf{(1):} $k_\varnothing (\tau)$ has law $\mu$. \textbf{(2):} For all $k\ino \bbN^*$ such that $\mu(k)\! >\! 0$, the subtrees $\theta_{(1)} \tau, \ldots , \theta_{(k)} \tau$ under $\bP (\, \cdot \, | k_\varnothing (\tau)\! = \! k)$ are independent with the same law as $\tau$ under $\bP$. 
\end{compactenum}

\smallskip

\noi
$(\textbf{b})\, $ Let $(\tau^{_{(k)}}_{^{\!}})_{k\in \bbN^*}$ be i.i.d.~GW($\mu$)-trees. We generically denote by $\tau_\infty$ the tree associated with the forest $(\tau^{_{(k)}}_{^{\!}})_{k\in \bbN^*}$. We call $\tau_\infty$ a GW($\mu$)-forest.

\smallskip

\noi
$(\textbf{c})\, $ We say that $\tau$ is a \emph{grafted GW($\mu$)-tree} if $k_\varnothing (\tau)\eqo 1$ and if $\theta_{(1)}\tau$ is a GW($\mu$)-tree. \cq 
\end{definition}

\begin{remark}
\label{GWrem} $\textbf{(a)}\, $ We shall always consider GW($\mu$)-trees with a non-trivial offspring distribution. Namely, $\mu (0)\! + \! \mu (1) \leko 1$. Let us recall that a GW($\mu$)-tree is a.s.~finite if and only if 
$\mu$ is critical or subcritical, namely if and only if $\sum_{k\geq 1} k\mu (k)\! \leq \! 1$.

\smallskip

\noi
$\textbf{(b)}\, $ Let us suppose that $\mu$ is non-trivial  and let $\tau_\infty$ be a GW($\mu$)-forest. 
Then its Lukasiewicz path $(V_l(\tau_\infty))_{l\in \bbN}$ RWis a RW  starting at $0$ whose jump law is $\widetilde{\mu} (k)\! :=\!  \mu (k+1)$, $k\ino \bbN\cup \{ -1\}$ (according to Spitzer's terminology, it is a \emph{left-continuous RW}). See Le Gall \& Le Jan \cite{LGLJ98} for more details. \cq 
\end{remark}

\noi
\textbf{Real valued branching walks and discrete snakes.} Here we consider a population represented by a rooted ordered tree 
that moves in $\bbR$ (a \emph{$\bbR$-valued branching walk}): each individual has a spatial position. We discuss 
two ways to encode the spatial positions of the individuals of the population: namely, via the \emph{contour snake} and 
via the \emph{height snake} that are related to the exploration of the family tree respectively in 
contour order and in lexicographical order. Here,  
 $\bC^0(\bbR_+, \bbR)$ stands for the space of the continuous functions from $\bbR_+$ to $\bbR$ equipped with the topology of the uniform convergence on all compact intervals.
\begin{definition}
\label{brwlkdef} (\textbf{Discrete snakes}) $(\textbf{a}) \, $
Let $\fftree\ino \bbT$ be finite. Let $(v_k)_{0\leq k\leq 2(\# t-1)}$ be the vertices of $t$ in contour order. It is convenient to set $v_{2\# t-1}\eqo  v_{2\# t}\eqo \varnothing$. 
Let $\mathtt M\! : \! u \ino t \mapsto\!  \mathtt M_u \ino \bbR$ be a $\fftree$-indexed and $\bbR$-valued walk.   
 For all $k \ino \lgeo 0,  2\# t \rgeo$ and all $p\ino \bbN$, we set 
 $W_{\! k} (t; p)\eqo \mathtt M_{(v_k)_{|p}}$, where we recall that $(v_k)_{|p}$ is the ancestor of $v_k$ at generation $p$ (and that $(v_k)_{|p}\eqo v_k$ if $p\geqo |v_k|$). 
We next set 
$$\forall r\ino \bbR_+ , \quad W_{\! k} (t ; r)\, =    W_{\! k} (t; \lfloor r \rfloor ) + (r \! -\! \lfloor r \rfloor ) \big(W_{\! k} (t; \lfloor r \rfloor \! +\! 1)  \! -\! W_{\! k} (t; \lfloor r \rfloor )  \big).$$ 
We note that $W_{\! k} (t;  \cdot\, ) \ino \bC^0(\bbR_+, \bbR)$ for all $k \ino \lgeo 0, 2\# t \rgeo$. We recall the contour process $C_{ \cdot} (\fftree)$ from Definition \ref{contourdef}. 
Then the \emph{contour snake} associated with the $\bbR$-valued branching walk $(\mathtt M_u)_{u\in \fftree}$ is the $\bC^0(\bbR_+, \bbR)$-valued continuous process $(W_{\! s} (t, \cdot) )_{s\in [0, 2\# t] }$ given by 
$$ \forall s \ino  [0, 2\# t), \quad W_{\! s}  (t\, ; \cdot)   = \left\{ \begin{array}{ll}
  W_{\! \lfloor s \rfloor +1}  \big(t\, ;   \cdot \wedge C_{\! s} (\fftree) \big) & \textrm{if $\; |v_{\lfloor s \rfloor +1} | \eqo 1+ |v_{\lfloor s \rfloor }|$,}\\
 W_{\! \lfloor s \rfloor}  \big( t\, ;   \cdot \wedge C_{\! s} (\fftree) \big)  & \textrm{if $\; |v_{\lfloor s \rfloor} | \eqo 1+ |v_{\lfloor s \rfloor+1 }|$.}
\end{array} \right. $$
We also set $\widehat{W}_{\! s} (t) \eqo  W_{\! s} (t; C_{\! s} (\fftree))$, $s\ino [0, 2\# t] $, which is 
the \emph{endpoint process of the contour snake}.

\smallskip

\noi 
$(\textbf{b}) \, $ Let $\phi_t \! : \! [0,\#t] \! \to \! [0, 2\#t]$ be the continuous increasing time-change such that $H_s (t)\eqo C_{\phi_t(s)} (t)$ as in Remark \ref{contord} $(\textbf{d})$. Then we define the \emph{height snake} as the $\bC^0(\bbR_+, \bbR)$-valued continuous process $(W^*_{\! s} (t, \cdot) )_{s\in [0,\# t] }$ given by 
$$ \forall s\ino [0\, , \#t], \quad W^*_s (t\, ; \cdot)= W_{\phi_t(s)} (t\, ; \cdot) \; .$$
If $(u_l)_{0\leq l< \#t}$ stands for the vertices of $t$ listed in lexicographical order, then $W^*_l (t; p) \eqo \mathtt M_{(u_l)_{|p}}$, for all $l\ino \lgeo 0, \# t \! -\! 1\rgeo$ and all $p\ino \bbN$. 
We also set $\widehat{W}^*_{\! s} (t) \eqo  W^*_{\! s} (t;  H_{\! s} (\fftree))\eqo \widehat{W}_{\phi_t (s)} (t) $, $s\ino [0, \# t] $, which is 
the \emph{endpoint process of the height snake}.

\smallskip

\noi 
$(\textbf{c}) \, $ Let $(t^{_{(k)}}_{^{\!}})_{k\in \bbN^*}$ be a forest of finite trees 
and let $t\ino \bbT$ be 
its associated tree as in Definition \ref{Ulamtree} $(\textbf{c})$. For all $k\ino \bbN^*$, let $R^{_{(k)}}_{^{\!}} \eqo 
(M^{_{(k)}}_{u})_{ u\in t^{{(k)}}}$ be a $\bbR$-valued $t^{_{(k)}}_{^{\!}}$-indexed branching walk. 
The contour snake  of the branching walks $( M^{_{(k)}}_{^{\! }} )_{k\in \bbN^*}$ is denoted $(W_s (t; \cdot))_{s\in \bbR_+} $ 
and is obtained by concatenating and interpolating the $W_\cdot (t^{_{(k)}}_{^{\!}}; \cdot ) $ in the following way. Recall notation $\sigma_k$ from Definition \ref{Lukadef} $(\textbf{b})$. Then for all $r, s\ino \bbR_+$ 
\begin{equation}
\label{concatebis}
W_{\! s}  (t; r)   = \left\{ \begin{array}{ll}
 W_{s-2\sigma_{k-1}} (t^{_{(k)}}_{^{\!}}; r)   & \textrm{if $\; s\ino [ 2\sigma_{k-1}, 2(\sigma_k\! -\!1)]$ ,}\\
 M^{_{(k)}}_{^\varnothing}  + \frac{_{_1}}{^{^2}} (s+ 2\! -\! 2\sigma_{k}) \big(  M^{_{(k+1)}}_{^\varnothing} \!\!  -\!  M^{_{(k)}}_{^\varnothing} \big)  & \textrm{if $s\ino [ 2(\sigma_k\! -\!1), 2\sigma_k]$.}
\end{array} \right. 
\end{equation}
Namely, if $s\ino [ 2(\sigma_k\! -\!1), 2\sigma_k]$, then $W_s (t;\cdot)$ is a constant function and its value is the affine interpolation between  $M^{_{(k)}}_{^\varnothing}$ and $M^{_{(k+1)}}_{^\varnothing}$. 
We also set $\widehat{W}_{\! s} (t) \eqo  W_{\! s} (t; C_{\! s} (\fftree))$, $s\ino \bbR_+ $, which is 
the \emph{endpoint process of the contour snake}.

\smallskip

\noi 
$(\textbf{d}) \, $ We keep the notation of $(\textbf{c})$ and for all $s\ino \bbR_+$, we set 
$W^*_s (t\, ; \cdot)\eqo W_{\phi_t (s)} (t\, ;\cdot)$, where 
$\phi_t $ is the continuous increasing bijection from $\bbR_+$ onto itself such that 
$H_s (t)\eqo C_{\phi_t(s)} (t)$ as in Remark \ref{contord} $(\textbf{d})$. 
The resulting $\bC^0 (\bbR_+, \bbR)$-valued continuous 
process $(W^*_s (t\, ; \cdot))_{s\in \bbR_+}$ is the \emph{height snake} associated with the branching walks $( M^{_{(k)}}_{^{\!}} )_{k\in \bbN^*}$. 
We also set $\widehat{W}^*_{\! s} (t) \eqo  W^*_{\! s} (t; H_{\! s} (\fftree))\eqo \widehat{W}_{\phi_t (s)} (t) $, $s\ino \bbR_+ $, which is the \emph{endpoint process of the height snake}. \cq 
\end{definition}

\subsection{Trees with an infinite line of ancestors.}
\label{inflinetreesec}
We extend Ulam's formalism to define ordered trees that may have 
an infinite line of ancestors. To that end, it is convenient to introduce \textit{bilateral words} as follows. 

\begin{definition}
\label{bilaworddef}(\textbf{Bilateral words, relative height, shifts})
$(\textbf{a})$ A \emph{bilateral word} is a sequence 
$u\! =\!  (a_k)_{k_1 <k\leq k_2}$ where $k_1\ino \bbZZ$, $k_2 \ino \bbZ$ and $a_k \ino \bbN^*$, $k_1 \leko k \leqo k_2$.   
If $k_1\! \geq \! k_2$, then $u\! = \! \varnothing$. We denote by $\overline{\bbU}$ the set of bilateral words. 

\smallskip

\noi
$(\textbf{b})$ Let $u\eqo (a_{k})_{k_1 < k\leq k_2 } \ino \obbU$, as above. Then 
\begin{equation}  
\label{extrem1def}
\bbn u\bbn_-\! =\!  k_1 , \quad \textrm{and} \quad \bbn u \bbn\! = \! k_2 .
\end{equation}
Here $\bbn u \bbn$ is called the \textit{relative height} 
of $u$ and $\bbn u \bbn_-$  its \emph{depth}. 

\smallskip

\noi
$(\textbf{c})$  Let $u\eqo (a_{k})_{k_1 < k\leq k_2 } \ino \obbU$ as above. 
Let $l \ino \bbZ$. We denote by $\varphi_l (u)$ the \emph{$l$-shifted word $u$}. Namely, 
\begin{equation}  
\label{shiftdef}
\varphi_{l} (u) = (a_{k+l})_{k_1 -l < k\leq k_2 -l }, 
\end{equation}
Therefore, $\bbn \varphi_{l} (u) \bbn_- \! = \! \bbn  u\bbn _- \! -\! l$,  $\bbn  \varphi_{l} (u) \bbn \! = \! \bbn  u \bbn \! -\! l$.  We adopt the convention $\varphi_l (\varnothing)\! = \! \varnothing$. Observe that the $\varphi_l\! : \! \obbU \! \to \! \obbU$ are bijections such that $\varphi_l \! \circ \! \varphi_{l^\prime} \! = \! \varphi_{l+ l^\prime}$, $l,l'\ino \bbZ$. For all $k \ino \bbZ$, we then set 
$$ \overline{\bbU}_k = \varphi_{-k} (\bbU) \; .$$
Note that $\overline{\bbU}_k \backslash \{ \varnothing \}$ is the set of bilateral words $u$ such that 
$\bbn u \bbn_- \eqo k$.  We shall also denote by $\overline{\bbU}_{-\infty}$ the set of bilateral words with an infinite depth.  \cq
\end{definition}

 Note that $\bbU \! \subset \! \overline{\bbU}$ and if $u\ino \bbU$, then $\bbn u \bbn \eqo |u| $ (the relative height in this case is the height of the word) and $\bbn u\bbn_-\eqo 0$. Note for all $k\ino \bbZ$ that $\overline{\bbU}_k$ is countable. However $\overline{\bbU}_{-\infty}$ is not countable. 

Let $u\eqo (a_{k})_{k_1 < k\leq k_2 } \ino \obbU$ be distinct from $\varnothing $. We introduce the following notation.  
\begin{equation}  
\label{extremdef}
\mathtt{end} (u)= a_{k_2}, \quad 
\overleftarrow{u}\! = \! (a_k)_{k_1 < k\leq k_2-1} \quad u_{ |\!|_{(l_1, l_2]}} \! =\!  (a_k)_{k_1 \vee l_1 <k \leq k_2 \wedge l_2},  
\end{equation}
for all $l_1\ino \bbZZ $, $l_2 \ino \bbZ$ such that $l_1\leqo l_2$. 
We still call $\overleftarrow{u}$ the (direct) \textit{parent} of $u$. Note that if $\bbn u \bbn_-+1\eqo \bbn u \bbn$, then $\overleftarrow{u}\eqo \varnothing$.  
Here, we insist on the following convention: viewed as a bilateral word, $\varnothing$ 
has neither a relative height nor a depth. However, we set $\varnothing_{ |\!|_{(l_1, l_2]}}=\varnothing$ by convention.

For all $u \ino \obbU$, we next denote by $L_u$ the \emph{ancestral line of $u$} and by $\preceq$ the \emph{genealogical order}, which is a partial order on $\obbU$. Namely 
\begin{equation}
\label{ancestral}
L_u = \big\{ u_{\lll_{(-\infty, p] }} \, ; \, p \ino\bbZ\big\} \quad \textrm{and} \quad \forall u,v\ino \obbU,  
\quad  \Big( u\preceq v \Big) \Longleftrightarrow \Big( u\ino L_v \Big)  \; .
\end{equation}

\begin{definition}
\label{deftreeee}(\textbf{Bilateral trees}) 
$(\textbf{a})$ A non-empty subset $T \! \subset \! \overline{\bbU}$ is a \emph{bilateral tree} if it satisfies the following conditions. 
 
 \begin{compactenum}

\smallskip

 \item[$(\textbf{a1})$] There is $k\ino\bbZZ$ such that $T \backslash \{ \varnothing\} \! \subset \! \overline{\bbU}_k$.
 
\smallskip

\item[$(\textbf{a2})$] For all $u\ino T\backslash \{ \varnothing \}$, $\overleftarrow{u} \ino T$.   

\smallskip

 \item[$(\textbf{a3})$]  For all $u\ino T$, 
 $k_u (T)\!\!  : = \! \# \{ v\ino T\! : \! \overleftarrow{v}\eqo u\} \leko  \infty$ 
 and $\{ 1, \ldots , k_u (T)\}  \! =\! \{ \mathtt{end} (v)\, ; \, v\ino T\!   : \!    \overleftarrow{v}\! = \! u\}$ if $k_u (T) \! \geq \! 1$. 

\smallskip

 \item[$(\textbf{a4})$] For all $u,v\ino T $, there is $l\ino \bbZ$ such that $u_{\mathbf{|\! |}_{(-\infty, l]}} \! \! = \! v_{\mathbf{|\! |} _{(-\infty,  l]}}$. 
 
\end{compactenum}  
 
\smallskip 

\noi
$(\textbf{b})$  We denote by $\overline{\bbT}$ the set of all bilateral trees. For all $T\ino\overline{\bbT}$ distinct from $\{ \varnothing \}$, we define the \emph{depth $\bbn T \bbn_- $ of $T$} as the unique $k\ino\bbZZ$ such that $T\backslash\{\varnothing\} \! \subset \! \overline{\bbU}_k$.

\smallskip

\noi
$(\textbf{c})$ A \emph{pointed bilateral tree} is a bilateral tree $T\ino \obbT$ and a distinguished vertex $o\ino T$ that is interpreted as an origin when $T$ is viewed as a space (or a distinguished individual when $T$ is viewed as a family tree: in this case it is rather denoted by $\varrho$). Recall from (\ref{ancestral}) the notation $L_o$ for the ancestral line of $o$. For all integers $p \ino \bbZ$, we shall use the following notation.
\begin{equation}
\label{spinedfbis} 
 o (p) = o_{\lll_{(-\infty , -p]}}, 
\quad  \textrm{and} \quad 
\partial L_o  \! = \! \big\{ u\ino T \backslash \Loo : \overleftarrow{u}\ino \Loo  \big\}. 
\end{equation}
Here $\partial L_o$ is called the \emph{spine} of $(T, o)$. We denote by $\obbT^\bullet$ the set of pointed bilateral trees. \cq 
\end{definition}

\begin{rems}
\label{bilatreerem}
\noi
\textbf{(a)} We view $T\ino \obbT$ as a graph: its set of vertices is $T$ and its edges  (simple and non oriented) are the unordered pair $\{ v, \overleftarrow{v} \}$, $v\ino T \backslash \{ \varnothing\}$. 

\smallskip

\noi
\textbf{(b)} Let $T \ino \overline{\bbT}$ be distinct from $\{ \varnothing \}$. If $\bbn T\bbn_- \! \neq \! -\infty$, then $1+\bbn T\bbn_-\! = \! \min_{u\in T \backslash \{ \varnothing \}} \bbn u \bbn$ and Definition \ref{deftreeee} $(b)$ implies that $\varnothing \ino T$ that is viewed as the root (or the ancestor) of $T$. However, if 
$\bbn T\bbn_-\! = \! -\infty$, then $\varnothing \! \notin \! T$ and $T$ has no proper ancestor but an infinite line of ancestors.   \cq 
\end{rems}

As for rooted ordered trees, we introduce a formal definition of 
\emph{subtree, most recent common ancestor and lexicographical order} as follows.

\smallskip

\noi
$\bullet$ \textit{Subtrees.} Let us first specify what we mean by concatenation in the context of bilateral words. For all
$u\! =\!  (a_k)_{k_1 <k\leq k_2} \ino \overline{\bbU}$ and all $v\! = \! (b_k)_{1\leq k\leq k_3} \ino \bbU$. 
\begin{equation} 
\label{concdef}
u \ast v\!  = \! (c_l)_{k_1< k\leq k_2+k_3} \quad \textrm{where} \quad c_k 
= \left\{ 
\begin{array}{ll}
a_k &  \textrm{If $k_1 \! < \! k \! \leq \! k_2$,} \\
b_{k-k_2} &  \textrm{If $k_2 \! < \! k \! \leq \! k_2+ k_3$.}
\end{array} \right.
\end{equation}
The bilateral word $u\ast v$ is the concatenation of a bilateral word on the left with 
a null depth word $v$ on the right. Note that $\bbn u\ast v \bbn_-\! = \! \bbn u \bbn_-$, that  
$\bbn u\ast v \bbn \! = \! \bbn u \bbn + \bbn v\bbn $ and that $u\ast \varnothing\! = \! u$. 
Then for all $T\ino \overline{\bbT}$ such that $u\in T$, the \emph{subtree $\theta_u T$ stemming from $u$} is defined by 
\begin{equation}
\label{curshitr} 
\theta_u T\! = \! \{ v\ino \bbU: u\ast v \ino T \} , 
\end{equation}
Note that either $\theta_u T\! = \! \{\varnothing\}$ or $\bbn \theta_u T \bbn_-\! = \! 0$, namely: $\theta_u T \ino \bbT$.

\smallskip

\noi
$\bullet$ \textit{ Most recent common ancestor.}
Let $T\ino \overline{\bbT}$. We define the \textit{most recent common ancestor} (or more simply the common ancestor) 
of $u, v\ino T$ by 
\begin{equation}  
\label{comancdef}
u \! \wedge \! v \!  := \! u_{\lll_{(-\infty, b(u,v)]}} \!\! =  v_{\lll_{(-\infty, b(u,v)]}} \;\,  \textrm{where} \;  \, 
b(u,v)\! :=\! \max \big\{ l  \leqo \bbn u\bbn,\bbn v\bbn : \, u_{\lll_{(-\infty, l]}} \!\! = \!  v_{\lll_{(-\infty, l]}} \big\}, 
\end{equation}
which is well-defined thanks to Definition \ref{deftreeee} $(\textbf{a4})$.

\smallskip

\noi
$\bullet$ \textit{Lexicographical order.} Let $T\ino \overline{\bbT}$. The vertices of $T$ are totally ordered by the \textit{lexicographical order} $\leq_{T}$ that is formally defined as follows: if $\varnothing \ino T$, then $\varnothing$ is the $\leq_{T}$-least element of $T$ and  
for all $u,v\ino T\backslash \{ \varnothing\}$, 
\begin{equation}
\label{lextredef}
u \leq_{T} v  \quad \textrm{iff} \quad \I{b(u,v)<\bbn u\bbn}\mathtt{end} (u_{|_{] -\infty, b(u,v) +1]}}) \leq  \I{b(u,v)<\bbn v\bbn}\mathtt{end} (v_{|_{] -\infty, b(u,v) +1]}})\; .
\end{equation}
Note that there are $u',v'\ino\bbU$ such that $u\! =\! (u\wedge v)\ast u'$ and $v\! = \! (u\wedge v)\ast v'$. Then, it holds that $u\leq_T v$ if and only if $u' \leq_{\mathtt{lex}} v'$.
We denote by $<_T $ the strict order associated with $\leq_{T}$.

\medskip

\noi
\textbf{Pointed labelled trees.} We view BRWs and more generally branching Markov chains as labelled pointed bilateral trees. Let $(E, d_E)$ be a Polish metric space that is the 
space of labels, i.e., ,~the state space where Markovian spatial motion takes place. A $E$-labelled tree is given by a pointed tree $(T, o)\ino \obbT^\bullet$ and labels $z_v\ino E$, $v\ino T$ that are interpreted as follows: $T$ is a family tree of a population (that possibly goes back to an infinite line of ancestors) and 
$z_v\ino E$ is the spatial position of individual $v$. 

We introduce the following notation 
for the \emph{space of pointed $E$-labelled trees}.
$$ \overline{\bbT}^\bullet  (E) \! = \! \{ \partial \} \! \cup \! \Big\{  \Theta \! =\!  \big(T, o \, ; \mathbf z \! =\!  (z_v)_{v\in T} \big)\, ;  \, (T, o)\ino \overline{\bbT}^\bullet,  \, z_v\ino E, v\ino T  \Big\}. $$
Here $\partial$ stands for a cemetery point.

\medskip

\noi
\textbf{Truncation and local convergence.} 
We shall consider, the ergodicity of certain transformations on $\obbT^\bullet$ and $\obbT^\bullet \! (E)$ on the one hand, and $\obbT^\bullet\!  (E)$-valued r.v.s and various statements involving their conditional laws on the other hand. For this reason, we equip $\obbT^\bullet\!  (E)$ with the Polish topology inherited from $E$ and from the local convergence on $\obbT^\bullet\! $. More precisely, we introduce the following truncation procedure.

\begin{definition}
\label{truncdef} 
Let $ \Theta \! =\!  \big(T, \boo \, ; \mathbf z \! =\!  (z_v)_{v\in T} \big) \! \ino \, \overline{\bbT}^\bullet \!   (E)$ and 
let $p\ino \bbZ$ and $q\ino \bbZ  \cup\{ \infty \}$ such that $p\! < \!  q$. If $T\! =\! \{\varnothing\}$ then we set $[\Theta]_p^q\! =\! \Theta$. When $T$ is distinct from $\{\varnothing\}$, we use the convention that if $\boo\! =\! \varnothing$ then $\bbn \boo\bbn\! =\! \bbn T\bbn_-$. Then, if $q\leqo  \bbn T\bbn_{_{ -}}$ or if $p\geko \bbn \boo \bbn$, then we adopt the convention that $[\Theta ]^q_p\! = \!  \partial$. Otherwise, we define $[\Theta ]_p^q$ as the $E$-labelled tree $\big( T^\prime, o^\prime ; \mathbf z^\prime \eqo   (z'_{\! v'})_{v'\in T'} \big)$ given by 
$$ o^\prime \! =\!  o_{\lll_{( p, q]}} , \quad T^\prime \! = \! \big\{  v_{\lll_{( p,q]}}\,  ;\;  v\ino T :  v_{\lll_{( -\infty ,  p]}}\! =
\! \boo_{\lll_{( - \infty, p]}} \big\} \quad  \textrm{and} \quad  z^\prime_{v^\prime}\! = \! z_v, $$
where $v$ is the vertex of $T$ such that $\bbn v\bbn\leqo q$, $v' \! = \!  v_{\lll_{( p,q]}}$ and $v_{\lll_{( - \infty, p]}}\! =
\! o_{\lll_{( - \infty, p]}}$. 
We adopt the convention that $[\partial]_p^q\eqo \partial$ and we call $[\Theta ]_p^q$ the \emph{$(p,q)$-truncation of $\Theta$}.
  
If there is no label, we use the same notation $[(T, o)]_{p}^q$ for the $(p,q)$-truncation of the pointed tree $(T, o) \ino \overline{\bbT}^\bullet$. \cq 
\end{definition}

Let us define local convergence on $\overline{\bbT}^\bullet$ and $\overline{\bbT}^\bullet(E)$. 
To that end, we first introduce for all $\Theta  \! = \! (T,o; \mathbf z)$, $\Theta ^\prime \!= \! (T^\prime, o^\prime; \bz^\prime)$ in 
$\overline{\bbT}^\bullet \! (E)$,
$$   
\Delta_E (\Theta , \Theta ^\prime)\! = \! \un_{\{ (T, o)\neq (T^\prime, o^\prime)\}} + 
\un_{\{ (T, o)= (T^\prime, o^\prime) \}} \max_{v\in T} \big(1\! \wedge \! d_E (z_v, z^\prime_v) \big) $$
with the conventions $\Delta_E (\Theta, \partial)\! = \! 1$ and $\Delta_E (\partial, \partial)\! = \! 0$. 
We easily check that $\Delta_E$ is a (non-separable) metric on $\overline{\bbT}^\bullet \! (E)$. Then, \emph{local convergence on $\obbT^\bullet\! (E)$} is defined as follows. 
\begin{definition}
\label{loccvdef} For all $n\ino \bbN$, let $\Theta _n\! = \! (T^{(n)} , o^{(n)} ; \bz^{(n)})\ino \obbT^\bullet \! (E)$. 

\smallskip

\noi
$(\textbf{a})$ $(\Theta_n)_{n\in \bbN}$ \emph{locally converges} to $\Theta  \! = \! (T, o; \bz) \ino  \obbT^\bullet\!  (E)$ if for all $\epp \ino (0, 1)$ and for all $q\ino \bbN$, there is $n_{q, \epp} \ino \bbN$ such that for all integers $n\! \geq \! n_{q, \epp}$, $\Delta_E \big( [\Theta^{(n)}]_{-q}^q,  [\Theta]_{-q}^q \big) \! < \! \epp$. 

\smallskip

\noi
$(\textbf{b})$ 
The $\obbT^\bullet\!\! $-valued sequence $\big( (T^{(n)} , o^{(n)}) \big)_{n\in \bbN}$  \emph{locally converges} to 
$(T, o)$ if for all $\epp \ino (0, 1)$ and for all $q\ino \bbN$, there is $n_{q, \epp} \ino \bbN$ such that for all integers $n\! \geq \! n_{q, \epp}$, $[(T^{(n)} , o^{(n)}) ]_{-q}^q\eqo  [(T,o)]_{-q}^q$. \cq 
\end{definition}

\begin{remark}
\label{poltoploc}
Local convergence corresponds for instance to the following metric.  
\begin{equation}
\label{disdelloc}
\forall \Theta, \Theta^\prime \ino  \obbT^\bullet\! (E) , \quad \bdelta^E_{\mathtt{loc}} \big(\Theta, \Theta^\prime \big)= \sum_{q\in \bbN} 2^{-q-1} \Delta_E \big(   [\Theta]_{-q}^q, [\Theta^{\prime}]_{-q}^q \big), 
\end{equation}
with $\bdelta^E_{\mathtt{loc}} (\Theta, \partial)\! = \! 1$ and $\bdelta^E_{\mathtt{loc}}(\partial, \partial)\! = \! 0$. We easily check that  
the topology of $(\obbT^\bullet\! (E), \bdelta^E_{\mathtt{loc}})$ is Polish. Note that $\Theta \mapsto [\Theta]_p^q$ is (by definition) continuous 
and that $\lim_{p,q\to \infty}   \bdelta^E_{\mathtt{loc}} (\Theta, [\Theta]_{-p}^q)\eqo 0$. 

Similarly, local convergence on $ \obbT^\bullet\! $ corresponds to a Polish topology. generated by the distance 
$\bdelta_{\mathtt{loc}} \big((T,o),  (T^\prime,o^\prime)\ \big)= \sum_{q\in \bbN} 2^{-q-1}\un_{\{ [(T,o)]_{-q}^q \neq 
 [(T', o')]_{-q}^q  \}}$, for all $(T, o), (T', o')  \ino \obbT^\bullet\! $, with the same conventions $\bdelta_{\mathtt{loc}} ((T, o), \partial)\! = \! 1$ and $\bdelta_{\mathtt{loc}}(\partial, \partial)\! = \! 0$.
 Let us mention that in most parts of the paper, the space $E$ of the labels is the space of pointed bilateral trees 
 $\obbT^\bullet\! $ itself.   \cq 
\end{remark}

\noi
\textbf{Centering function, infinite GW trees.} We next introduce a fundamental function on 
$\overline{\bbT}^\bullet$ that allow us to define dynamics: the centering function. Right after, we introduce infinite GW trees that are invariant under such transforms and whose law has nice related ergodic properties.

\begin{definition}
\label{succdef} Let $(T, o)\ino \overline{\bbT}^\bullet$. When $\varnothing\ino T$ but $T\! \neq\!\{\varnothing\}$, we use the convention $\bbn \varnothing \bbn\! =\! \bbn T\bbn_-$. If $T\!=\!\{\varnothing\}$ then we set $\mathtt{cent}(T,o)\! =\! (T,o)$, and otherwise, we set
\begin{equation}
\label{centerTdef}
 \mathtt{cent}(T, o)\eqo  \big( \varphi_{\bbn o \bbn }(T), \varphi_{\bbn o \bbn } (o) \big) 
\end{equation}
The function $\mathtt{cent}\! : \! \overline{\bbT}^\bullet \!\to  \! \overline{\bbT}^\bullet$ is called the \emph{centering} function. \cq 
\end{definition}
\begin{lemma}
\label{consucc} The centering application $\mathtt{cent} (\cdot)$ is 
locally continuous on $ \obbT^\bullet\!$.
\end{lemma}
\noi
\textbf{Proof.} See D., K., Lin \& Torri \cite[Lemma. 2.7]{DuKhLiTo22}. 
\cqfd

\begin{definition}
\label{GWptdef} (\textbf{Infinite GW trees})  Let $\mathbf r \eqo (r(j,k))_{k\geq j\geq 1}$ be a probability measure 
on $\{ (j,k)\ino (\bbN^*)^2\! : \! j\! \leq \! k\}$. 
Let $\mu$ be a probability measure on $\bbN$ whose mean $\ttm_\mu \! :=\! \sum_{k\in \bbN} k\mu (k)$ is finite and positive. 
 Let $(\bT \! , \boo)\! :\!  \Omega \! \rightarrow \!   \overline{\bbT}^\bullet\! $ be a 
Borel-measurable random pointed tree such that 
$$ \textrm{$\bP$-a.s.} \quad  \bbn \boo \bbn \! = \! 0 \quad \textrm{and} \quad \bbn \bT \bbn_-\!\! = \! -\infty \; .$$ 
Recall from (\ref{ancestral}) that the notation $L_\boo\eqo \{ \boo (p) ; p\ino \bbN\}$ for the 
the ancestral line of $\boo$ and recall from (\ref{spinedfbis}) the definition of the spine $\partial L_\boo$ of $(\bT, \boo)$.  
Recall from (\ref{extremdef}) the notation $\mathtt{end} (\cdot)$.  

\smallskip

\noi
$(\textbf{a})$ We say that $(\bT, \boo)$ is an \emph{infinite GW tree with offspring distribution $\mu$ and dispatching measure $\mathbf r$} if: 
\begin{compactenum}

\smallskip

\item[$-$]  the 
r.v.s $S\! :=\! \big( \mathtt{end} (\boo (p)), k_{\boo (p +1)} (\bT)\big)_{\! p\in \bbN}$ are i.i.d.~with law $\mathbf r$, 

\smallskip

\item[$-$] conditionally given $S$, 
the subtrees $(\theta_{u} \bT \, ; \, u\ino \partial L_\boo )$ are independent GW($\mu$)-trees (as in Definition \ref{GWfordef} $(\mathbf a)$) 

\smallskip

\end{compactenum}

\noi
More specifically:

\smallskip

\noi
$(\textbf{b})$ $(\bT, \boo)$ is an \emph{infinite GW tree with offspring distribution $\mu$} if
$r(j,k)\! = \! \frac{\mu (k)}{\ttm_\mu}$, for all $k\! \geq \! j\! \geq \! 1$.  \cq 
\end{definition}

\noi
\textbf{Many-to-one.} Infinite (pointed) GW trees are related to GW trees via the \textit{many-to-one} principle that asserts 
the following. 
\emph{Let $\bT'$ be a GW($\mu$)-tree 
and let $(\bT, \boo)$ be an infinite GW($\mu$)-tree. 
Then for all Borel-measurable functions $F \! : \! \bbN \times \overline{\bbT}^\bullet \! \rightarrow \! \bbR_+$,   }
\begin{equation}
\label{sizebiased}
\bE \Big[\sum_{v\in \bT'} F \big( |v| ;  \mathtt{cent}(\bT', v)  \big) \Big]= \sum_{p\geq 0} \ttm_\mu^p\, \bE\big[ F (p\, ;  [(\bT, \boo)]^\infty_{-p})\big] \; , 
\end{equation}
where $\ttm_\mu \eqo \sum_{k\in \bbN} k\mu (k) \leko \infty$ and where 
we recall from Definition \ref{truncdef} that $ [(\bT, \boo)]^\infty_{-p}$ is the tree $(\bT, \boo)$ 
truncated above $\boo (p)$, a vertex which is situated at relative height $-p$.

\subsection{Biased RWs on trees}
\label{biaRWsec}
In this section, we recall the definition of \emph{biased RWs on trees} as introduced by R.~Lyons \cite{Lyons90}. This model depends on one parameter $\lambda$ ranging in $(0, \infty)$ and we focus on the critical parameter when the environment is a GW tree $\bT$. Instead of viewing the RW as successive positions $X_n$ in $\bT$, $n\in \bbN$, we view it as the $\overline{\bbT}^\bullet$-valued sequence $\overline{X}_n \eqo (\bT, X_n)$. We then recall the definition of the law of an invariant GW tree, that is a law on $\overline{\bbT}^\bullet$ invariant (up to centering) 
for the biased RW $(\overline{X}_n )_{n\in \bbN}$. We next recall ergodic properties of biased RW 
on GW trees which have been originally investigated in Zeitouni \& Peres~\cite{PerZei08}.

\medskip

\noi
\textbf{$\boldsymbol{\lambda}$-biased RWs on trees.} We fix $\lambda \ino (0, \infty)$ and $T\ino \obbT$, a bilateral tree not reduced to $\{\varnothing\}$. The transition probabilities of a $\lambda$-biased RW on $T$ are defined as follows: $p_{T, \lambda} (x,y)\eqo 0$ if $x$ and $y$ are not adjacent and otherwise for all $x\ino T \backslash \{ \varnothing\} $, 
\begin{equation}
\label{deflambRW}
 p_{T, \lambda} (\overleftarrow{x}, x)= \I{\overleftarrow{x}\neq \varnothing}\frac{1}{\lambda + k_{\overleftarrow{x}} (T)}+\I{\overleftarrow{x}=\varnothing}\frac{1}{k_\varnothing(T)} \quad \textrm{and} \quad p_{T, \lambda} ( x,\overleftarrow{x})= \frac{\lambda}{\lambda + k_x (T)}
\end{equation}
(we recall that $\varnothing \! \notin \! T$ if $\bbn T \bbn_- \eqo -\infty$).

A $\lambda$-RW on $T$ is the RW associated with the conductances $C_{\! \{ x, \overleftarrow{x}\}}\eqo \lambda^{-\bbn x \bbn}$. As proved in R.~Lyons \cite{Lyons90} when $T \ino \bbT$, there is a quantity $\mathtt{br} (T)\ino [0, \infty]$ called the \emph{branching number} such that the $\lambda$-biased RW is recurrent if $\lambda \geko\mathtt{br} (T)$, and transient if $\lambda\leko \mathtt{br} (T)$ (see also Lyons \& Peres \cite[Theorem 3.5]{LyoPerbook} for more details).

In this article, \emph{we want to view $\lambda$-biased RWs on a tree $T\ino  \overline{\bbT}$ as 
a $ \overline{\bbT}^\bullet$-valued sequences of r.v.~$\overline{X}_n \eqo (T, X_n)$, $n\ino \bbN$.}
The only reason why we adopt this point of view is technical: we want to apply ergodic theory, to work with regular versions of conditional laws of the RW when $(T, o)$ is random, etc.  
The transition kernels of the Markov chain $(\overline{X}_n)_{n\in \bbN}$ have to be modified accordingly as explained in the following definition. 
To fit in our framework of labelled bilateral trees, we generically denote elements of $\obbT^\bullet\! $ by $\overline{x}\eqo (T, x)$, $\overline{x}^\prime\eqo (T^\prime, x^\prime)$, etc and $\obbT^\bullet\!$-valued sequences by $\overline{x}_n \eqo (T_n, x_n)$, $n\ino \bbN$. 

\begin{definition}
\label{biaRWtreedef} Let $(T, o) \ino \overline{\bbT}^\bullet$. Let $\lambda \ino (0, \infty)$ and let $p_{T, \lambda} (\cdot, \cdot)$ be the transition probabilities of $\lambda$-RWs on $T$ as defined in (\ref{deflambRW}).  
The transition kernel on $\obbT^\bullet\! $ 
of the $\lambda$-biased RW is then defined for all $\overline{x}\!:= \! (T,x) \ino \obbT^\bullet\! $ by 
\begin{equation}
\label{defcritRW} 
\overline{\mathtt p}_{\lambda} ( \overline{x}, d\overline{x}^\prime  ) = \sum_{y\in T} p_{T, \lambda } (x, y) \delta_{(T, y)} ( d\overline{x}^\prime) \, .
\end{equation}
We then denote by $P_{T,o}$ the law of the $\lambda$-biased RW
on $(\obbT^\bullet\! )^\bbN$ (that is 
the space of the $\obbT^\bullet \! $-valued sequences equipped with the product topology, which makes it Polish) starting at $(T,o)$. 
Note that $P_{T,o}$ depends on $\lambda$ although it does not appear in the notation. \cq 
\end{definition}
\begin{remark}
\label{toregcondilaw}
It is easy to check that 
\begin{equation}
\label{contbiasRW}
 (T,o) \ino  \obbT^\bullet \! \longmapsto P_{T,o} \quad \textrm{is continuous}
\end{equation}
with respect to local convergence on $\obbT^\bullet \! $ and to the weak convergence on the space of probability measures on $(\obbT^\bullet\! )^\bbN$. If $(\bT, \boo)$ is a  $\obbT^\bullet$-valued r.v., then $P_{\bT,\boo}$ stands for the (regular version of the) conditional law conditionally given $(\bT, \boo)$ of the $\lambda$-biased RW on $(\bT, \boo)$ denoted by $\overline{X}_n\eqo (\bT, X_n)$, $n\ino \bbN$, starting at $\overline{X}_0\eqo (\bT, \boo)$. 
These technical points being cleared, when such a formal point of view is not essential to our arguments, 
we rather speak of $(X_n)_{n\in \bbN}$ as to the $\lambda$-biased RW on $\bT$ starting at $\boo$. \cq  
\end{remark}

Let  $\overline{X}_n\eqo (T, X_n)$, $n\ino \bbN$, be a $\lambda$-biased RW whose initial value is $\overline{X}_0 \eqo (T, x)$. We use the following notation for the hitting times of $\overline{X}$: 
\begin{equation}
\label{ttitime}
\forall y \in T, \quad \mathtt H_y\! = \! \inf \{ n\ino \bbN: X_n\! = \! y \} \quad \textrm{and} \quad     \mathtt H^+_y\! = \! \inf \{ n\ino \bbN^*: X_n\! = \! y \}, 
\end{equation}
with the convention that $\inf \emptyset \eqo \infty$. 
The next lemma recalls  basic results
on hitting times and the Green function of $\lambda$-biased RWs. 

\begin{lemma}
\label{dlczidec} Let $\lambda\ino(1,\infty)$ and let $\overline{x}\eqo (T,x) \ino \obbT^\bullet\! $, whose depth $\bbn T\bbn_-$ is possibly infinite. Let $\overline{X}_n\eqo (T, X_n)$, $n\ino \bbN$, be a $\lambda$-biased RW with initial state $(T,x)$. We assume that 
\begin{equation}
\label{eschyp} 
 \forall x\ino T \backslash \{ \varnothing \}, \quad P_{(T,x)} \big(  \mathtt{H}_{\overleftarrow{x}} \leko \infty \big) \eqo 1 \; .  
\end{equation}
Let $\boo \ino T \backslash \{ \varnothing \}$. Recall from (\ref{ancestral}) that $L_\boo$ stands for 
the ancestral line of $\boo$ and let us denote by $ \boo(p)$ the vertex of $L_\boo$ at distance $p$ from $\boo$, which makes sense if and only if $0\leq p \leqo \bbn o\bbn -\bbn T\bbn_{ -}$. Namely 
\begin{equation}
\label{zmvkb}
 \boo(p)\! = \! \boo_{\lll_{\mathbf{(}  -\infty\, ,\,   \bbn \boo \bbn -p \mathbf{]} }}, \quad   0\leqo p \leqo \bbn \boo\bbn\!  -\! \bbn T\bbn_{ -} 
\end{equation}
We also introduce the successive times of visit of $\overline{X}$ in $L_\boo$ that are recursively defined for all $k\ino \bbN$ by $\mathbf{r}^\boo_{k}\! = \! \inf \big\{ n \! >\! \mathbf{r}^\boo_{k-1} : X_n \ino L_\boo \big\}$, 
with the convention that $\mathbf r^\boo_{-1}\eqo -1$ and $\inf \emptyset \eqo \infty$.  
Then, the following holds true. 
\begin{compactenum}

\smallskip

\item[$(i)$] For all $x\ino T$, $P_{(T, x)}$-a.s.~for all $k\ino \bbN$, $\mathbf{r}^\boo_k \! < \! \infty$ and  
there exists a birth-and-death Markov chain $(Z_k)_{k\in \bbN}$ that takes its values in $\bbN \cap [0, \bbn \boo\bbn\!  -\! \bbn T\bbn_{ -} ]$ and such that $\boo (Z_k)\! = \! X_{\mathbf{r}_k}$, $k\ino \bbN$. 
If $T$ has a finite depth, then it is reflected at $\bbn \boo\bbn\!  -\! \bbn T\bbn_{ -}$; otherwise, its transition probabilities are given for all $1\leqo p\leko \bbn \boo\bbn\!  -\! \bbn T\bbn_{ -}$ by 
\begin{equation}
\label{fgfhfj}
\rho_{p, p+1} \! = \! \frac{\lambda }{\lambda  +  k_{\boo (p) } (T)  } , \; \,  \rho_{p, p-1}\! = \! \frac{1}{ \lambda + k_{\boo (p) } (T)}, \; \,  \rho_{p, p}\! =\! 
\frac{k_{\boo (p) } (T)\! -\! 1 }{ \lambda +  k_{\boo (p) } (T)} \; ,
\end{equation}
and $\rho_{0, 0}\!= \! 1\! -\! \rho_{0, 1} \eqo\frac{k_{\boo} (T)}{ \lambda  + k_{\boo  } (T)}$.  
Moreover, if $T$ has an infinite depth, then $Z$ is transient which implies that $\bbn \boo  \wedge \! X_n \bbn \! \rightarrow \! -\infty$ and that $(X_n)_{n\in \bbN}$ is transient. 

\smallskip

\item[$(ii)$] Let $x,y, z\ino T$ such that $z \preceq y$. Then 
\end{compactenum}
\begin{equation}
\label{lbdev}P_{(T \!\!, \,  x)} \big(\mathtt H_{y} \! <\! \mathtt H_{z}  \big)\! =  \! 
\frac{(\lambda^{\bbn x\wedge y\bbn }\! -\! \lambda^{\bbn z\bbn })_+}{\lambda^{\bbn  y\bbn }\!  -\! 
\lambda^{\bbn z \bbn }} \; .
\end{equation}
\begin{compactenum}
\item[$(iii)$]  We assume now that $T$ has an infinite depth. For all $x,y\ino T$, we first get 
\end{compactenum}
\begin{equation}
\label{finhit}
P_{(T \!\!, \,  x)} \big(\mathtt H_{y} \! <\! \infty  \big) \! = \! \lambda^{- (\bbn y \bbn -\bbn x\wedge y\bbn) }  \quad  \textrm{and} \quad P_{(T \!\!, \,  x)} \big(\mathtt H^+_{x} \! <\! \infty  \big) \! = \! \tfrac{1+k_x(T)}{\lambda + k_x (T)} \; .
\end{equation}
\begin{compactenum}
\item[] We denote 
the Green function by $G_T(x,y)\! = \! E_{(T,x)} \big[ \sum_{n\in \bbN}\un_{\{ X_n =  y\} } \big]$. Then 
\end{compactenum}

\begin{equation}
\label{greenoj}
G_T(x, y) = G_T(x\wedge y, y)= \tfrac{\lambda + k_y (T)}{\lambda  -1} \, \lambda^{- (\bbn y \bbn -\bbn x\wedge y\bbn) }.
\end{equation}
\begin{compactenum}
\item[] If $x\preceq y$, we also get $E_{(T,x)} \big[ \sum_{0\leq n < \mathtt H_{\overleftarrow{x}}} \un_{\{ X_n =  y\} } \big]\eqo (\lambda \! + \! k_y (T)) \lambda^{- (\bbn y \bbn- \bbn x \bbn +1)}$. 
\end{compactenum}
\end{lemma}
\noi
\textbf{Proof.} 
We easily see that Assumption (\ref{eschyp}) implies that $\mathbf{r}^\boo_k \! < \! \infty$ a.s.~for all $k \ino \bbN$. 
Markov property at $\mathbf{r}^\boo_k$ implies that $(Z_k)_{k\in \bbN}$ is a $\bbN$-valued birth-and-death Markov chain with transition probabilities as in (\ref{fgfhfj}). We use the previous result with $\boo\eqo y$: Gambler's ruin for birth and death chains (see e.g.~Durrett \cite[Theorem.~6.4.6 p.~249]{Durbook}) implies the first equality in (\ref{lbdev}). When $T$ has an infinite depth, it also implies that $\bbn \boo  \wedge \! X_n \bbn \! \rightarrow \! -\infty$ 
(the RW is transient) and the first equality in (\ref{finhit}). To prove the second equality in (\ref{finhit}), 
we first deduce from (\ref{eschyp}) that $P_{(T \!\!, \,  x)} (\mathtt H^+_{x} \! <\! \infty )\eqo \frac{k_x(T)}{\lambda+ k_x(T)}+\frac{\lambda}{\lambda+ k_x(T)}P_{(T \!\!, \,  \overleftarrow{x})} (\mathtt H_{x} \! <\! \infty ) $ and we use the first equality in (\ref{finhit}).

 We now compute the Green function when $T$ has an infinite depth: by the strong Markov property at the first time of visit of $x\wedge y$, we get $G_T(x, y) \eqo G_T(x\wedge y, y)$. So we only need to consider the cases where $x\! \preceq \! y$. By the Markov property, we observe that 
 $G_T(x,y)\eqo P_{T,x} (\mathtt H_y\leko \infty) G_T(y,y)$ and that  $E_{T,y} [\sum_{n\in \bbN} \un_{\{ X_n=y\} } | (X_0, \ldots , X_{\mathtt H_{\overleftarrow{y}}} )] \eqo 1+ Z+ G_T(\overleftarrow{y}, y)$ where $P_{T, y} (Z\eqo j)\eqo p(1\! -\! p)^j $, $j\ino \bbN$, with $p\eqo p_{T, \lambda} ( y, \overleftarrow{y})$. Thus, $G_T(x,y)\eqo  P_{T,x} (\mathtt H_y\leko \infty) \big( p^{-1}+ G_T(\overleftarrow{y}, y))$. We find the value of $G_T(\overleftarrow{y}, y)$ by taking $x\eqo \overleftarrow{y}$ and (\ref{greenoj}) follows. 
 
We then prove the last point of $(iii)$ by counting how many times the RW visits $y$ before and after 
$\mathtt H_{\overleftarrow{x}}$ to get  $G_T(x, y)\! -\!  G_T(\overleftarrow{x}, y) \eqo 
E_{(T,x)} \big[ \sum_{0\leq n < \mathtt H_{\overleftarrow{x}}} \un_{\{ X_n =  y\} } \big]$, by Markov, which 
implies the desired result. \cqfd

\medskip

\noi
\textbf{Critical biased RWs on GW trees.}We restrict our attention to \emph{critical} biased RWs on supercritical GW trees and variants of such trees. More precisely, let $\bT'$ be GW($\nu$)-tree whose offspring distribution $\nu$ is supercritical: namely $\mathtt m_\nu \! := \! \sum_{k\in \bbN} k \nu (k) \in (1, \infty)$. 
In Lyons~\cite[Proposition 6.4]{Lyons90}) it is proved that 
$\bP (\, \cdot \, | \, \# \bT'\eqo \infty)$-a.s.~$\mathtt{br} (\bT')\eqo \ttm_\nu$ and  and that $\ttm_\nu$- biased RWs on $\bT'$ are null-recurrent. We shall refer to this case as to the \emph{critical case}.  
This lemma shows in particular that the estimates of Lemma \ref{dlczidec} apply to critical biased RW on GW trees. 
Recurrence here does not contradict Lemma \ref{dlczidec} $(i)$ where the tree is supposed to be infinitely deep. 
\begin{lemma}
\label{critGWimpl} Let $\nu$ be a supercritical offspring distribution whose mean is denoted by $\ttm_\nu$. Let $(\bT, \boo)$ be a GW($\nu$)-tree  (Definition \ref{GWfordef}) or an infinite GW($\nu$)-tree (Definition \ref{GWptdef}). 
Let $\overline{X}_n \eqo (\bT, X_n)$ be a $\ttm_\nu$-biased RW on $\bT$. 
Then, $\bP$-a.s.~for all $x\ino \bT\backslash\{\varnothing\}$, $P_{(\bT,x)} \big(  \mathtt{H}_{\overleftarrow{x}} \leko \infty \big) \eqo 1$. 
\end{lemma} 
\noi
\textbf{Proof.} By e.g.~Lyons \& Peres \cite[Theorem.~3.5]{LyoPerbook}), as already mentioned 
if $\bT$ is a GW($\nu$)-tree, then the $\mathtt m_\nu$-biased RW is recurrent which implies 
the lemma in this case. 
 
 Suppose that $(\bT, \boo)$ is an infinite GW($\nu$)-tree: conditionally given $\partial L_\boo$, the subtrees $\theta_y \bT$, $y\ino \partial L_\boo$, are i.i.d.~GW($\nu$)-trees and $\bP$-a.s.~for all $x\notin L_\boo$, $P_{\bT, x} (\mathtt H_{\overleftarrow{x}} \leko \infty) \eqo 1$ by the previous case. If $x \ino  L_\boo$ and $\mathtt H_{\overleftarrow{x}} \eqo \infty$, then the previous cases imply that the $\mathtt m_\nu$-biased RW has to visit infinitely many often 
each vertex of the subtree above $x$ because it is the union of a finite number of GW($\nu$)-trees and of a finite line. 
However, each time the RW returns to $x$, it has an independent possibility to jump to $\overleftarrow{x}$ with positive probability $p_{T, \ttm_\nu} (x, \overleftarrow{x})$. Thus 
$\bP$-a.s.~$P_{\bT, x}(\mathtt H_{\overleftarrow{x}} \eqo \infty)\eqo 0$, 
which completes the proof of the lemma.  \cqfd

\medskip

Let $(X_n )_{n\in \bbN}$ be a $\ttm_\nu$-biased RW on a GW($\nu$)-tree $\bT'$. Then 
the centered trees $\mathtt{cent} (\bT', X_n)$ weakly locally converge in $\overline{\bbT}^\bullet$ to a law which is recalled in the following definition, and which is in some sense invariant and reversible for $\ttm_\nu$-biased RWs as explained right after. 
\begin{definition}
\label{definvRW}\textbf{(Invariant GW trees)} Let $\nu$ be a supercritical offspring distribution whose mean is denoted by 
$\mathtt m_\nu$. 
Let $(\bT, \boo)$ be $\obbT^\bullet$-valued r.v.~It is distributed as an  
\emph{invariant GW($\nu$)-tree} if for all bounded Borel measurable function $F$ on $\obbT^\bullet$,
\begin{equation}
\label{InvGWnu}
\bE \big[ F \big(\bT, \boo \big)\big]= \bE \Big[ \tfrac{\mathtt m_\nu + k_{\boo^\prime} (\bT^\prime) }{2\mathtt m_\nu  }F\big(\bT^\prime, \boo^\prime \big) \Big], 
\end{equation}
where $(\bT^\prime, \boo^\prime)$ is distributed as an infinite  GW($\nu$)-tree as in Definition \ref{GWptdef}. \cq 
\end{definition}

  Let $(\bT, \boo)$ be an invariant GW($\nu$)-tree and let $\overline{X}_n \eqo (\bT, X_n)$, $n\ino \bbN$, be a $\ttm_\nu$-biased RW on $\bT$ starting at $\overline{X}_0\eqo (\bT, \boo)$. 
As proved in Zeitouni \& Peres \cite[Lemma 2]{PerZei08}, the law of invariant GW($\nu$)-trees is invariant and reversible in the following sense: 
 \begin{equation}
 \label{invbias}
 (\overline{X}_0, \overline{X}_1) \overset{\textrm{(law)}}{=} \big( \mathtt{cent}(\overline{X}_1) , \varphi_{\bbn X_1\bbn} (\overline{X}_0) \big)
 \end{equation}
where we recall $\mathtt{cent} (\cdot)\! :\!  \obbT^\bullet \! \to  \obbT^\bullet \! $ from Definition \ref{succdef}.  
This allows us to extend this Markov chain into a \emph{stationnary process} $(\overline{X}_n)_{n\in \bbZ}$ such that for all $n_0\ino \bbZ$:
\begin{equation}
\label{sdkjvc}
\big(\varphi_{\bbn X_{n_0}\bbn } (\overline{X}_{n_0+ n}) \big)_{n\in \bbZ} \overset{\textrm{(law)}}{=}  \big(\varphi_{\bbn X_{n_0}\bbn} (\overline{X}_{n_0 -n}) \big)_{n\in \bbZ} \overset{\textrm{(law)}}{=}  (\overline{X}_n)_{n\in \bbZ} . 
\end{equation}
Moreover, this process is ergodic as stated in the following proposition. 
\begin{proposition}
\label{rwergo}  Let $\nu$ be a supercritical offspring distribution whose mean is denoted by $\ttm_\nu$ and let $\overline{\mathtt p}_{\ttm_\nu}$ be as in Definition \ref{biaRWtreedef}.  Let $(\bT, \boo)$ be an invariant GW($\nu$)-tree as in Definition \ref{definvRW}. Let $\overline{X}_n\eqo (\bT, X_n)$, $n\ino \bbN$, 
be a $\mathtt m_{\nu}$-biased RW on $(\bT, \boo)$ starting at $\overline{X}_0\eqo (\bT, \boo)$. 
Denote by $\Pi_\nu$ the law of $(\overline{X}_n)_{n\in \bbN}$ on 
$(\obbT^\bullet\! )^{\bbN}$. 
Then, $ \vartheta \! : \! (\overline{x}_n)_{n\in \bbN} \! \mapsto \! 
(\varphi_{\bbn x_1\bbn} (\overline{x}_{n+1}))_{n\in \bbN} $ is $\Pi_\nu$-ergodic, where $\overline{x}_n\eqo (T_n, x_n)$, $n\ino \bbN$, stands for the canonical sequence on $(\obbT^\bullet\! )^{\bbN}$. 
 \end{proposition}
\noi
\textbf{Proof.} See Dembo \& Sun~\cite[Theorem 1.4 and Section 2.1]{DeSu12}. In their notation, here, $A$ is the single-entry matrix $\ttm_\nu$, which is of course irreductible, and $\rho=\ttm_\nu$. \cqfd 

\medskip

We shall apply Proposition \ref{rwergo} in the following specific way. 
\begin{corollary}
\label{rwergobis} We keep the assumptions and the notation of Proposition \ref{rwergo}. Let $F \! :\! \obbT^\bullet\! \to \bbR$ be 
a measurable function 
such that $\bE [|F(\bT, \boo)|] \leko \infty$ and such that $F(T, \boo) \eqo F(\mathtt{cent} (T, o))$ for all $(T, o)\ino   \obbT^\bullet\!$. Then 
\begin{equation}
\label{ergoexpli}
\textrm{$\bP$-a.s.~} \quad \lim_{n\to \infty} \frac{1}{n} \sum_{0\leq k<n} F \big( \bT, X_k \big) = \bE \big[ F(\bT, \boo)\big] \; . 
\end{equation} 
Moreover, if  $F$ is nonnegative and if there is $\ttb\ino (1, \infty)$ such that $\bE [F(\bT, \boo)^\ttb] \leko \infty$, then 
\begin{equation}
\label{wienerconseq}
\textrm{$\bP$-a.s.~} \quad  \sup_{n\in \bbN^*} \frac{1}{n} \sum_{0\leq k<n} E_{\bT, \boo} \big[ F \big( \bT, X_k \big) \big]  <\infty \; .
\end{equation} 
\end{corollary}  
\textbf{Proof.} Denote by $\overline{x}_\cdot \eqo (\overline{x}_n)_{n\in \bbN}$ the canonical $\obbT^\bullet\! $-sequence. We write $G (\overline{x}_\cdot) \eqo F( \mathtt{cent} (\overline{x}_0 ))$. Observe that  
$F \big( \bT, X_k \big)\eqo (G\circ \vartheta^{\circ k})( (\overline{X}_n)_{n\in \bbN})$. It first implies that 
$(F \big( \bT, X_k \big))_{k\in \bbN}$ under $\bP$ has the same law as $ (G\circ \vartheta^{\circ k})_{k\in \bbN}$ under 
$\Pi_\nu$. By Birkhoff ergodic theorem (see e.g.~Krengel \cite[Theorem 2.3 p.~9]{Kr85}) $\Pi_\nu$-a.s.~$\lim_{n\to \infty} 
\frac{1}{n} \sum_{0\leq k<n} G\circ \vartheta^{\circ k}\eqo \int \! G\,  d\Pi_\nu\eqo  \bE \big[ F(\bT, \boo)\big] $, hence (\ref{ergoexpli}). 

We next suppose $F$ (and thus $G$) nonnegative and such that 
$\bE [F(\bT, \boo)^\ttb] \eqo \int G^\ttb d\Pi_\nu \leko \infty$, and we proceed to the proof of (\ref{wienerconseq}). To that end, 
we use Wiener $L^\ttb$-bound for the maximal function $S^*\!\!  :=\! \sup_{n\in \bbN^*} \frac{1}{n} \sum_{0\leq k<n} G\circ \vartheta^{\circ k}$: this theorem asserts that $\int (S^*)^\ttb d\Pi_\nu\leq \big(\tfrac{\ttb}{\ttb-1} \big)^\ttb \int G^\ttb d\Pi_\nu$ 
(see Krengel \cite[Theorem 6.3 p.~52]{Kr85}). Thus,  
 $\bE [S^*(\overline{X}_\cdot)^\ttb ] \leko \infty$, which implies in particular that $S^*(\overline{X}_\cdot)$ is $\bP$-integrable. Then, for any $\sigma$-field $\mathcal G \! \subset \! \ccF$, we $\bP$-a.s.~get 
$\bE \big[ S^*(\overline{X}_\cdot) | \mathcal G] \leko \infty$ since $S^*$ is nonnegative. 
Next, observe that for all $n\ino \bbN^*$, $\bP$-a.s.~$\frac{1}{n} \sum_{0\leq k<n}  F ( \bT, X_k ) \leqo S^*(\overline{X}_\cdot)$, 
Therefore, we $\bP$-a.s.~get 
$ \sup_{n\in \bbN^*}\frac{1}{n} \sum_{0\leq k<n} \bE [  F ( \bT, X_k )| \mathcal G ] \leqo \bE \big[ S^*(\overline{X}_\cdot) | \mathcal G] \leko \infty$, which entails  (\ref{wienerconseq}) when $\mathcal G$ is the $\sigma$-field generated by $(\bT, \boo)$ since, in this case, $\bE [F ( \bT, X_k )| \mathcal G ] \eqo E_{\bT, \boo} \big[ F \big( \bT, X_k \big) \big] $ by definition of $P_{\bT, \boo}$ as explained in Remark \ref{toregcondilaw}. \cqfd

\subsection{Tree-valued BRWs}
\label{trerbrRWsec}
In this section, we introduce BRWs with values in trees and more specifically RWs that are \emph{indexed by GW trees} whose offspring distribution is critical 
(or a variant of such trees) and that \emph{take their values in an invariant GW tree} whose offspring distribution 
is supercritical.  
Formally, these BRWs are viewed as trees that are labeled by pointed trees, i.e., ,~$\overline{\bbT}^\bullet (\overline{\bbT}^\bullet)$-valued r.v. After some background on such branching Markov chains, we extend the \emph{centering} function to 
$\overline{\bbT}^\bullet (\overline{\bbT}^\bullet)$.

\medskip

\noi
\textbf{Branching Markov chains.} We first set some notation for branching Markov chains indexed by bilateral trees and that take their values in a complete and separable metric space $(E, d_E)$. 
 We denote by $(\mathtt q(y,dy^\prime))_{ y\in E}$ a Borel measurable transition kernel on $E$: for all $y\ino E$, $\mathtt q(y,dy^\prime)$ is a Borel probability measure on $E$ and for all Borel sets $B$, the function $y\ino E \mapsto \mathtt q(y, B)$ is measurable. 
We equip $\obbT^\bullet \! (E)$ with the Polish topology of local convergence and the associated Borel $\sigma$-field. 
For all pointed tree $(t, u)\ino \overline{\bbT}^\bullet\! \! $, we briefly recall the definition of the law $Q^{_y}_{^{\mathtt q, (t,u)}}$ on 
$\obbT^\bullet \! (E)$ of the \emph{$t$-indexed $\mathtt q$-branching Markov chain} such that the spatial position of $u$ is $y$. 

To that end, we introduce the following notation. We equip $t$ with its \emph{graph distance} $ d_{\mathtt{gr}}$; for all $v, w \ino t$, we denote by $\lgeo v, w\rgeo$ the geodesic path joining $v$ to $w$; if $w\! \neq \! v$, then we denote by 
$\overleftarrow{w}^v$ the unique $v^\prime\ino \lgeo v, w \rgeo $ such that $d_{\mathtt{gr}} (v^\prime , w)\! = \! 1$. Namely 
\begin{equation}
\label{geotree}
\lgeo v, w\rgeo\! = \!\!  \big\{  v' \! \ino t : d_{\mathtt{gr}} (v,v^\prime) \! + \! d_{\mathtt{gr}} (v^\prime , w)\! = \! d_{\mathtt{gr}} (v,w) \big\},  \quad \overleftarrow{w}^v\! \ino \lgeo v, w\rgeo \; \, \textrm{and} \;\,  d_{\mathtt{gr}} \big( w, \overleftarrow{w}^v \big)\eqo 1 . 
\end{equation}

\noi
We easily see that $Q^{_y}_{^{\mathtt q, (t,u)}}$ is characterized by the following properties. 

\smallskip

\begin{compactenum}

\item[$(i)$] Recall that $\partial$ is a cemetery point in $\overline{\bbT}^\bullet$. We 
first agree on the convention $Q^{_y}_{^{\mathtt q,\partial}} \eqo \delta_{\partial}$. 
\item[$(ii)$] If $(t,u) \ino \overline{\bbT}^\bullet$ is finite, then $Q^{_y}_{^{\mathtt q, (t,u)}}$ is the law of the $\obbT^\bullet \! (E)$-valued r.v.~$\fTheta\! :=\! \big(t,u; (Y_v)_{v\in t}\big)$, where the joint law of the $Y_v$, $v\ino t$, on $E^t$ is
\begin{equation}
\label{brMCfinite}
 \delta_y(dy_u) \!\! \!  \prod_{v\in t\backslash \{ u\}} \!\!\! \mathtt q( y_{\overleftarrow{v}^u}, dy_v). 
\end{equation}
\item[$(iii)$] Let $(t,u) \ino \overline{\bbT}^\bullet$ be finite. If $\fTheta$ has law 
$Q^{_y}_{^{\mathtt q, (t,u)}}$, then for all $p,q \ino \bbZ$ such that $p\!\leq  \! \bbn u \bbn \leqo  q$, 
$[\fTheta]_p^q$ has law $Q^{{y}}_{^{\mathtt q, [(t,u)]_p^q}}$. 
\end{compactenum}

\smallskip

\noi
\emph{Indeed}, when $(t, u)  \ino \overline{\bbT}^\bullet$ is infinite, $Q^{_y}_{^{\mathtt q, (t,u)}}$ is defined thanks to $(iii)$ and to 
a simple projective argument. Moreover, in general 
for all $p, q\ino \bbZ$, such that $p\leko \bbn u  \bbn \leko  q$, 
\begin{equation}
\label{dfljbvs}
\Big( \fTheta \overset{\textrm{(law)}}{=} Q^{_y}_{^{\mathtt q, (t,u)}} \Big) \Longrightarrow \Big( [\fTheta]_p^q \overset{\textrm{(law)}}{=} 
Q^{y}_{^{\mathtt q , [(t,u)]_p^q}} \Big)\; .
\end{equation}
It is easy to check for all Borel subsets $A$ of 
$ \obbT^\bullet \! (E) $ that
\begin{equation}
\label{measQ}
\big( y, (t,u) \big) \! \in \! E \! \times \! \obbT^\bullet \! \longmapsto  Q^{_y}_{^{\mathtt q, (t,u)}} (A) \; \, \textrm{is Borel measurable.} 
\end{equation}
Moreover, let us equip the space of Borel probability measures on $E$ and $ \obbT^\bullet \! (E)$ with the (Polish) topology of weak convergence. If we assume that $y\ino E \mapsto \mathtt q(y, \cdot)$ is weakly continuous, then it is also easy to check that 
\begin{equation}
\label{C0brinit}
\big( y, (t,u) \big) \in  E \times \obbT^\bullet \!\longmapsto  Q^{_y}_{^{\mathtt q, (t,u)}}\; \, \textrm{is weakly continuous.}  
\end{equation}
In (\ref{measQ}) and (\ref{C0brinit}), $ E \! \times \!  \obbT^\bullet  $ is equipped with the product topology.

\medskip

\noi
\textbf{Tree-valued branching walks.} Here we fix a \emph{genealogical tree} $(t, \varrho)\ino \obbT^\bullet$ and a pointed tree 
$(T, o)\ino \obbT^\bullet$ that is viewed as an \emph{environment}: here, each individual $v \ino t$ has a position $y_v \ino T$. To fit into our framework of labelled pointed trees, we view the $t$-indexed and $T$-valued branching walk as the $\obbT^\bullet $-labelled pointed tree 
$$ \Theta\eqo \big( t,\varrho\, ; \, \overline{y}_v\! := \! (T, y_v)  , v\ino t  \big) \in \obbT^\bullet (\obbT^\bullet).$$ 
Note that $\Theta$ is not a generic element of $\obbT^\bullet (\obbT^\bullet)$ because all the $ \overline{y}_v$ have the same base-tree $T$. We extend the definition (\ref{centerTdef}) of the centering of pointed trees to $\obbT^\bullet (\obbT^\bullet)$ by setting
\begin{equation}
\label{centerbrTdef}
\mathtt{cent} (\Theta)\! = \! \Big( \varphi_{\bbn \varrho \bbn }(t), \varphi_{\bbn \varrho \bbn } (\varrho ) ; \varphi_{\bbn y_\varrho \bbn} (\overline{y}_{\varphi_{-\bbn \varrho \bbn} (v)}) , v\ino  \varphi_{\bbn \varrho \bbn}(t) \Big),
\end{equation}
using the same conventions. We observe here that spatial marks (that are pointed trees) are shifted too. By convention we set $\mathtt{cent} (\partial) \eqo \partial$. 

\begin{lemma}
\label{consuccbr} The function $\mathtt{cent} (\cdot)$ is
locally continuous on $\obbT^\bullet\!(\obbT^\bullet )$.  
\end{lemma}
\noi
\textbf{Proof.} This is easily derived from Lemma \ref{consucc}: we leave the details to the reader. \cqfd 

\begin{definition}
\label{biasedbfRWdef} (\textbf{Biased BRWs}) We next fix $\lambda \ino (0, \infty)$, and 
$(t, \varrho), (T, o)\ino \overline{\bbT}^\bullet$. 
We recall from (\ref{deflambRW}) the definition of the transition probabilities $p_{T, \lambda} (\cdot, \cdot) $ of the 
$\lambda$-biased RW on $T$ and we recall from (\ref{defcritRW}) in Definition \ref{biaRWtreedef} the related transition 
kernel $\overline{\mathtt p}_{\lambda} (\overline{y}, d\overline{y}')$ on $\overline{\bbT}^\bullet$. A \emph{$t$-indexed, 
$T$-valued $\lambda$-biased BRW with initial position $o$} is a  $\overline{\bbT}^\bullet(\overline{\bbT}^\bullet)$-valued 
r.v.~$\fTheta\eqo \big(t, \varrho\, ; \overline{Y}_{\!\! v} \eqo (T, Y_{\! v}), v\ino t\big)$ whose law is 
$Q_{^{\overline{\mathtt p}_{^{\! \lambda} }\! , (t, \varrho)}}^{_{(T,o)}}$, which implies that 
$Y_{\! \varrho}\eqo o$. 

\smallskip

\noi
$\bullet$ When $t$ is finite and $\varrho \eqo \varnothing$ then for all $y_v \ino T$, $v\ino t$, 
\begin{equation}
\label{}
\bP \big( \forall v\ino t , \; Y_{\! v} \eqo y_v \big)\eqo \un_{\{ y_\varrho = o\}} \prod_{v\in t\backslash \{ \varnothing\} } p_{T, \lambda} (y_{\overleftarrow{v}} , y_v) \; .
\end{equation}

\noi
$\bullet$ When there is no ambiguity on the parameter $\lambda$, we simply write 
$Q_{^{t, \varrho}}^{_{T,o}}$ instead of $Q_{^{\overline{\mathtt p}_{^{\! \lambda} }\! , (t, \varrho)}}^{_{(T,o)}}$.  \cq
\end{definition}
\begin{remark}
\label{brrems}
We easily check that $\overline{y} \ino\overline{\bbT}^\bullet \mapsto \overline{\mathtt p}_{\lambda} (\overline{y}, d\overline{y}')$ is weakly continuous. Thus by (\ref{C0brinit}), 
$$\big( (t,\varrho), (T, o) \big) \ino \overline{\bbT}^\bullet \! \times\!  \overline{\bbT}^\bullet \longmapsto 
Q_{^{\overline{\mathtt p}_{^{\! \lambda} }\! , (t, \varrho)}}^{_{(T,o)}}$$
is also weakly continuous on the space of Borel probability measures on $\overline{\bbT}^\bullet(\overline{\bbT}^\bullet)$.  \cq 
\end{remark}

\subsection{Stable continuous height processes.}
\label{stabletree}

In this section, we recall (mostly from \cite{DuLG02} and \cite{DuLG05}) various results concerning stable height processes and stable spectrally positive Lévy processes that are used throughout the paper.  

\medskip

\noi
\textbf{${\boldsymbol \alpha}$-stable spectrally positive Lévy processes, related height processes.}
We fix $\alpha \in (1, 2]$. We denote by $X\! = \! (X_s )_{s\in \bbR_+}$ an \emph{$\alpha$-stable and spectrally positive L\'evy process}. Its law is characterized by its Laplace exponent
\begin{equation}
\label{Xlaw}
\forall s, \lambda \ino \bbR_+, \quad \bE \big[ e^{-\lambda X_s}  \big] = e^{ s\lambda^\alpha } \; . 
\end{equation} 
Note that \emph{$X_s$ is integrable} and that $\bE [X_s] \eqo  0$, which easily implies that \emph{$X$ oscillates when $s\to \infty$}. 
Moreover, $\bP$-a.s.$\;$the path $X$ has \emph{infinite variation sample paths} (for more details, see Bertoin~\cite{Be} Chapters VII and VIII). 

In the more general context of spectrally positive L\'evy processes, it has been proved in Le Gall \& Le Jan \cite{LGLJ98} and in D.~\& Le Gall~\cite[Chapter 1]{DuLG02} that there exists a continuous process $H\eqo  (H_s)_{s\in \bbR_+} $ such that for any $s\ino \bbR_+$, the following limit holds in $\bP$-probability:  
\begin{equation}
\label{Hlimit}
H_s:=\lim_{\varepsilon\to 0} \frac{1}{\epp}\int_0^s {\bf 1}_{\{ X_r <  \varep+ \inf_{[r, s]} X \}}\,dr  \; .
\end{equation}
The process $H\eqo  (H_s)_{s\in \bbR_+} $ is called the \emph{$\alpha$-stable height process}. It is an adapted functional of 
 $X$ and we shall sometimes use the notation $H_s \eqo H_s(X)$ to indicate that. The $\alpha$-stable height process 
 provides a way to explore the genealogy of an $\alpha$-stable CSBP. We refer to Le Gall \& Le Jan~\cite{LGLJ98} for a careful explanation of (\ref{Hlimit}) in the discrete setting. Compare with (\ref{codheight}) in Remark \ref{bijecfor} in the discrete setting.

\smallskip 

In Section~\ref{pfharmosec}, to prove the quenched convergence of the finite-dimensional marginals 
of snakes processes coding the range of BRWs, we shall need the following lemma. Its proof, which is technical, is postponed in Appendix.

\begin{lemma}
\label{branchmass}
Let $\alpha \ino (1, 2]$ and let $H$ be the height process associated with the $\alpha$-stable spectrally positive Lévy process $X$ as above. Let $s_1,s_2$ be two real numbers such that $0\leko s_1\leko s_2$. Then $\bP$-almost surely, it holds that  $H_{s_1} \wedge H_{s_2} \geko \min_{[s_1, s_2]} H$.
\end{lemma} 
\noi
\textbf{Proof.} See Section \ref{branchmasspfsec}.  \cqfd 
 
\medskip

\noi
\textbf{Excursions of the ${\boldsymbol \alpha}$-stable height process.} For all $s\ino \bbR_+$, we set $I_s\eqo \inf_{r\in [0, s]} X_r$, the \emph{infimum process of $X$}, which has continuous sample paths since $X$ is spectrally positive.   
Recall that $X$ has infinite variation sample paths. Basic results on fluctuation theory (see e.g.~Bertoin \cite{Be} Chapter VI.1 and VII.1) entail that $X\! -\! I$ is a strong Markov process in $\bbR_+$ and that $0$ is regular for 
$(0, \infty)$ and recurrent with respect to this Markov process. Moreover, $-I$ is a local time at $0$ for $X\! -\! I$ (see Theorem VII.1 \cite{Be}). We denote by $\bN$ the corresponding 
\emph{excursion measure} of $X\! -\! I$ above $0$ and denote by $(l_j, r_j)$, $j\ino  \cI$, the excursion intervals of $X\! -\! I$ above $0$ and by $X^j_\cdot \eqo X_{(l_j + \cdot )\wedge r_j}-I_{l_j}$, $j\ino \cI$, the corresponding excursions. 
Then, the point measure $\sum_{j\in \cI} \delta_{(-I_{l_j}, X^j)}$ is a \emph{Poisson point measure on $\bbR_+\times \bD(\bbR_+, \bbR)$ with intensity 
$dx \otimes \bN$}. Here $\bD(\bbR_+, \bbR)$ stands for the space of c{\`a}dl{\`a}g functions from $\bbR_+$ to $\bbR$ equipped with Skorokhod's topology. 
Now, observe that (\ref{Hlimit}) implies that the value of $H_s$ only depends on the excursion of $X\! -\! I$ straddling $s$ and that 
\begin{equation}
\label{zeroHX}  \big\{ s \ino \bbR_+: X_s \geko I_s \big\}\eqo \bigcup_{^{j\in \cI}} (l_j, r_j)= \big\{ s \ino \bbR_+: H_s \geko 0 \big\}\; .
\end{equation}
This allows defining $H$ under $\bN$ as $H(X)$ when $X$ is under the excursion measure $\bN$. 
More precisely the processes $H^j\! :=\! H_{(l_j+\cdot )\wedge r_j}\eqo H(X^j)$, $j\ino  \cI$, are 
the excursions of $H$ above $0$, and the point measure 
\begin{equation}
\label{Poissheight}
\sum_{j\in \cI} \delta_{(-I_{l_j}, H^j)}
\end{equation}
is distributed under $\bP$ as a Poisson point measure on $\bbR_+\times \bC^0(\bbR_+, \bbR)$ with intensity $dx  \otimes \bN (dH)$, (with a slight abuse of notation for using $\bN (dH)$ as the `distribution' of $H(X)$ under the excursion measure 
$\bN (dX)$ ). Recall here $\bC^0(\bbR_+, \bbR)$ is the space of continuous functions from $\bbR_+ $ to $\bbR$ equipped with the topology of uniform convergence on every compact interval. We refer to the comments in D.~\& Le Gall \cite[Section 3.2 ]{DuLG02} for more details. 

Under $\bN$, $X$ and $H$ have the same lifetime $\zeta$: namely, we have $\bN$-a.e.~$ \zeta \leko \infty$, $X_s\eqo X_0\eqo H_0\! =\! H_{s}\! =\! 0$ for all $s \ino [\zeta, \infty)$ and $H_s$ and $X_s \! >\! 0$ for all $s \! \in \! (0, \zeta)$. 
Basic results of fluctuation theory (see Bertoin \cite{Be}, Chapter VII) also entail that 
$\bN \big[ 1\! -\! e^{-\lambda \zeta}  \big] \eqo \lambda^{1/\alpha}$, $\lambda \ino \bbR_+$. 
Therefore $\bN(\zeta \geko  r) \! =\! c_\alpha 
r^{-\frac{_1}{^\alpha}}$ where $c_\alpha = 1/ \Gamma_{\! \mathtt e} \big(\frac{_{\alpha -1}}{^\alpha} \big)$ and where  
$\Gamma_{\! \mathtt e}$ stands for Euler's Gamma function.

\begin{remark}
\label{gundundef}
Denote by $(\mathbf g_1, \mathbf d_1)$ the first excursion interval of $H$ (or $X\! -\! I$) above $0$  that is longer than $1$.  
Namely, for all $s\ino (0, \infty)$, introduce $r(s)\eqo \inf \{ r\ino (s, \infty): H_r \eqo 0\}$. Then 
$\mathbf g_1 \eqo \inf \{ s\ino \bbR_+: r(s) \! -\! s \geko 1 \}$ and $\mathbf d_1\eqo r( \mathbf g_1)$. Then, standard results on Poisson point processes imply that 
\begin{equation}
\label{zeta>1}
\big( X_{(\mathbf g_1 + \cdot) \wedge \mathbf d_1}, H_{(\mathbf g_1 + \cdot) \wedge \mathbf d_1}  \big)  \quad \textrm{under $\bP$} \quad \overset{\textrm{(law)}}{=} (X, H) \quad \textrm{under $\bN (\, \cdot \, | \, \zeta \geko 1)$.} \qquad \square
\end{equation}     
\end{remark}

Let us next define the law of the \emph{normalized $\alpha$-stable height process} $\bN (\, \cdot \, | \, \zeta \eqo 1)$.  
To that end, recall that the $\alpha$-stable spectrally positive Lévy process 
$X$ enjoys the following \emph{scaling property}: for all $r\ino (0,\infty)$, under $\bP$ the rescaled process 
$(r^{-1/\alpha}X_{rs})_{s\in \bbR_+}$ has the same law as $X$. We easily derive from (\ref{Hlimit}) that $H$ also enjoys a scaling property: namely, under $\bP$, 
$( r^{-(\alpha -1)/\alpha}  H_{rs} \big)_{s\in \bbR_+} $ and $H$ have the same law. Therefore, 
\begin{equation}
\label{HscalingN}
\big( r^{-\frac{1}{\alpha}}X_{rs} ,  r^{-\frac{\alpha-1}{\alpha}}\! H_{rs} \big)_{s\in \bbR_+} \;  \textrm{under} \; \, 
r^{\frac{1}{\alpha}}   \bN \quad 
\overset{\textrm{(law)}}{=} \quad  (X, H ) \; \textrm{under} \,  \; \bN \; . 
\end{equation}
(with a slight abuse of the word `law'). This scaling property allows us to define a regular version of the excursion conditioned to last $r$. Namely, there exists a family of laws on $\bD(\bbR_+, \bbR)$ denoted by  
$\bN(\, \cdot \, | \, \zeta \! = \! r)$, $r \! \in \! (0, \infty)$, such that 
$r \mapsto \bN(\, \cdot \, | \, \zeta \! =\!  r)$ is weakly continuous on 
$\bD(\bbR_+, \bbR)$, 
such that $\bN(\, \cdot \, | \, \zeta \! = \! r)$-a.s.~$\zeta \! = \! r$ and such that $\bN \eqo  \int_0^\infty \bN(\, \cdot \, | \, \zeta \! =\!  r) \, \bN (\zeta \! \in \! d r)$. 
We call $\bN(\, \cdot \,| \, \zeta \! =\!  1 )$ the \textit{normalized law of the 
$\alpha$-stable spectrally positive Lévy process above its infimum} (or of \emph{the $\alpha$-stable height process}). By (\ref{HscalingN}), we see that 
\begin{equation}
\label{echtscal}
\bN \big[ F(X, H)\big]= \alpha^{-1} c_\alpha \int_0^\infty \!\!\! \! dr  \, r^{-1-\frac{1}{\alpha}}\, 
\bN \Big[ F\Big( \big( r^{\frac{1}{\alpha}}\! X_{s/r} , r^{\frac{\alpha-1}{\alpha}}\! H_{s/r} \big)_{\! s\in \bbR_+}\Big) \big| \zeta \eqo 1\Big] \; . 
\end{equation}

\begin{remark}
\label{anscontzeta} Let us fix $s\ino (0, 1)$. Then the laws of $(H_{\cdot \wedge s} ,X_{\cdot \wedge s})$ under $\bN( \, \cdot \, | \zeta \geko 1)$ and $\bN( \, \cdot \, | \zeta \eqo 1)$ are absolutely continuous. 
More precisely, there exists a continuous function $( s ,x) \ino (0, 1)\! \times \bbR_+ \mapsto D_s(x)$ such that for all measurable nonnegative functions $F$,
\begin{equation}
\label{abscont}
\bN \Big[ F\big(  H_{\cdot \wedge s} ,X_{\cdot \wedge s} \big)  \big| \zeta \eqo 1\Big]= 
\bN \Big[ F\big(  H_{\cdot \wedge s} ,X_{\cdot \wedge s} \big) D_s (X_s)  \big| \zeta \geko 1\Big]. 
\end{equation}
See Kortchemski \cite{Kor13} for more details. A discrete version of this identity is used in the proof of our limit theorem. \cq
\end{remark}

\smallskip

\noi
\textbf{Limit theorems.} We next recall mostly from D.~\& Le Gall \cite[Chapter 2]{DuLG02}, D.~\cite{Du03} and Kortchemski \cite{Kor13} the convergence of the rescaled contour and height processes of GW forests whose offspring distribution belongs to the domain of attraction of an $\alpha$-stable law. 
More precisely, we fix $\alpha \ino (1, 2]$. Then, all the genealogical trees that we consider are (variants of) GW trees whose offspring distribution $\mu$ is fixed and satisfies (\ref{hypostaintro}).    
We denote by $(V_n)_{n\in \bbN}$ a RW whose jump law is $\widetilde{\mu}(k)\eqo \mu (k+1)$, $k\ino \{ -1\}\! \cup \! \bbN$. 
By definition, Assumption (H3) in (\ref{hypostaintro}) implies that there exists a $(0, \infty)$-valued sequence $(b_n)_{n\in \bbN}$ tending to $\infty$ such that 
\begin{equation}
\label{cvlawRW}
\tfrac{1}{b_n}V_n \underset{n\to \infty}{\overset{\textrm{(law)}}{-\!\!\! -\!\!\! \longrightarrow}} X_1 \; .
\end{equation}
where $(X_s)_{s\in \bbR_+}$ stands for the $\alpha$-stable spectrally positive L\'evy process whose Laplace exponent is 
$\lambda^\alpha$. Here, the sequence $(b_n)_{n\in \bbN}$ is $\tfrac{1}{\alpha}$-regularly varying, i.e., there is 
a slowly varying function  $L\! :\!  (0, \infty)\!  \rightarrow\!  (0, \infty)$ at $\infty$ such that 
\begin{equation}
\label{bnBiGoTe}
b_n = n^{1/\alpha} L(n) \; . 
\end{equation}
By e.g.~Bingham, Goldies \& Teugels \cite[Theorem 8.3.1]{BiGoTe}, Assumption ($\textrm{H}_3$) in (\ref{hypostaintro}) implies that there exists 
a slowly varying function  $L^*\! :\!  (0, \infty)\!  \rightarrow\!  (0, \infty)$ at $\infty$ such that 
\begin{equation}
\label{mutail}
\mu ([n ,\infty))\!  \sim \! 
n^{-\alpha} L^*(n) \; \, \textrm{if}\;  \alpha \!  \in \! (1, 2) \quad \textrm{and} \quad \sum_{0\leq k\leq n} \!\!\!\! k^2\mu (k) \! - \! 1\, \sim \, 2L^*(n) \; \textrm{if} \; \alpha \eqo 2. 
\end{equation}
$\sum_{0\leq k\leq n} k^2\mu (k) -1$ is ultimately positive by  ($\textrm{H}_1$) in (\ref{hypostaintro}). 
By e.g.~Bingham, Goldies \& Teugels \cite[Theorem 8.1.6]{BiGoTe}, 
\begin{equation}
\label{Laplmu}
\psi_\mu (s) \underset{0}{\sim}  c_\alpha' s^{\alpha} L^*\! \big( \tfrac{1}{s} \big) \; \,  \textrm{where} \; \,  \psi_\mu (s)\eqo  g_\mu (1\!  -\! s) \! -\! (1\! - \! s) \; \textrm{and} \; g_\mu (s) \eqo   \sum_{k\in \bbN} s^k \mu(k), \, s\ino [0, 1]  .
\end{equation}

\vspace{-2mm}

\noi
where $c_\alpha'\eqo \frac{\alpha-1}{\Gamma_{\! \mathtt e} (2-\alpha)}$ if $\alpha\leko 2$ and $c_2'\eqo 1$. 
Then $L$ and $L^*$ are related as follows. 
\begin{equation}
\label{slowlinks}
n L(n)^\alpha \sim b_n^\alpha \sim n L^* (b_n) \sim n L^* \big(n^{1/\alpha} L(n) \big)  \; .
\end{equation}

Theorem \ref{VHCcvstable} below  gathers several limiting results 
for GW trees and their coding processes that  are used  in our proofs. 
To that end, let us first introduce the following notation. 
Let $\tau_\infty$ be a GW($\mu$)-forest as in Definition \ref{GWfordef} $(\textbf{b})$ which is associated with the sequence of i.i.d.~GW($\mu$)-trees $(\tau^{_{(k)}}_{^{\!}})_{k\in \bbN^*}$ as in Definition \ref{Ulamtree} $(\textbf{c})$. Recall that $V(\tau_\infty)$, $H(\tau_\infty)$ and $C(\tau_\infty)$ stands for resp.~its Lukasiewicz path, its height process and its contour process. 
For all integers $n$, we then set 
 \begin{equation}
\label{JGDdef}
\kappa_n= \inf \{ k\ino \bbN^*:   \# \tau^{_{(k)}}_{^{\!}}\geqo n \big\}, \quad \frak g_n \eqo \sum_{1\leq k<\kappa_n} \! \! \# \tau^{_{(k)}}_{^{\!}} 
\quad \textrm{and} \quad \frak d_n \eqo \frak  g_n + \# \tau^{_{(\kappa_n)}}_{^{\!}}
\end{equation}
with the convention that $\frak g_n \eqo 0$ if $\kappa_n \eqo 1$. 
Namely, $\frak g_n$ is the time when $H(\tau_\infty)$ begins its first excursion above $0$ lasting at least $n$ units of time. 
Note that $\tau^{_{(\kappa_n)}}_{^{\!}}$ 
has the same law as a single GW($\mu$)-tree $\tau$ under 
$\bP (\, \cdot \, | \# \tau \geqo n)$.

\begin{theorem}
\label{VHCcvstable} Let $\alpha\ino (1, 2]$ and let $\mu$ satisfy (\ref{hypostaintro}). 
Let $\tau_\infty$ be a GW($\mu$)-forest and let $\tau$ be a GW($\mu$)-tree (Definition \ref{GWfordef}). Then for all (sufficiently large) $n\ino \bbN^*$, we denote by $\tau_{\geq n}$ a tree distributed as $\tau$ under $\bP (\, \cdot\, | \, \# \tau \geqo n)$ and by $\tau_n$ a  tree distributed as $\tau$ under $\bP (\, \cdot\, | \, \# \tau \eqo n)$. 
Recall $\frak g_n$ and $\frak d_n$ from (\ref{JGDdef}) and recall $\mathbf g_1$ and $\mathbf d_1$ from Remark \ref{gundundef}. 
Let $(b_n)_{n\in \bbN^*}$ be as in (\ref{cvlawRW}). For all $n\ino \bbN^*$, we set 
\begin{equation}
\label{aennedef}
 a_n = n/b_n , \quad n\ino \bbN^* \; .
\end{equation} 
Then, the following holds true. 
\begin{compactenum}

\smallskip

\item[$(i)$] The joint convergence 
\end{compactenum}
\begin{equation}
\label{VHCGDcv_i}
\mathscr Q_n = \Big(  \big(\tfrac{1}{b_n} V_{\lfloor n\cdot \rfloor} (\tau_\infty) ,   \tfrac{1}{a_n} H_{ n\cdot } (\tau_\infty)  ,    \tfrac{1}{a_n} C_{2n\cdot } (\tau_\infty) ,  \tfrac{1}{n}\frak g_n   ,   \tfrac{1}{n}\frak d_n \Big) \underset{n\to \infty}{-\!\!\!-\!\!\!\longrightarrow} 
\big(  X, H, H , \mathbf g_1  , \mathbf d_1 \big)
\end{equation}
\begin{compactenum}
\item[]  holds weakly on $\bD (\bbR_+, \bbR) \! \times \! \bC^0 (\bbR_+, \bbR)^2 \! \times \! (0, \infty)^2$ equipped with the product topology.

\smallskip

\item[$(ii)$] For all integers $n\geqo k\geq 0$, there exists a function $D^{_{(n)}}_{^k}\! : \! \{ -1\} \! \cup\!  \bbN \! \to \bbR_+$ such that for all measurable nonnegative function $F$
\end{compactenum}
\begin{equation}
\label{absCdiscr}
\bE \big[  F( V_{\cdot \wedge \, k} (\tau_n)  ) \big]= \bE \big[  F( V_{\cdot \wedge\,  k} (\tau_{\geq n})  )D^{_{(n)}}_{^k} \big(V_k (\tau_{\geq n}) \big)\big].
\end{equation}
\begin{compactenum}
\item[$(iii)$] We recall from Remark \ref{anscontzeta} the definition of the continuous function $( s, x) \ino (0, 1)\! \times \bbR_+ \mapsto D_s(x)$. We fix $( s,x, y) \ino (0, 1)\! \times (0, \infty)^2 $ with $x\leqo y$. Let $s_n\ino \bbN^*$, $n\ino \bbN$, be such that $s_n /n \! \to \! s$. Then 
\end{compactenum}
\begin{equation}
\label{limlocgned}
\lim_{n\to \infty} \max_{ p \in \lgeo \lfloor b_n x\rfloor, \lfloor b_n y \rfloor \rgeo } \big| D^{_{(n)}}_{{s_n}} (p) \! -\! D_s (p/b_n)  \big|\eqo 0 \; .
\end{equation}
\begin{compactenum}
\item[] Moreover, the r.v.~$\big( D^{_{(n)}}_{^{s_n}} ( V_{s_n} (\tau_{\geq n}) ) \big)_{\! n\in \bbN^*}$ are uniformly integrable: 
\end{compactenum}
\begin{equation}
\label{Scheff}
\lim_{c\to \infty} \sup_{n\in \bbN^*} \bE \Big[ D^{_{(n)}}_{^{s_n}} ( V_{s_n} (\tau_{\geq n}) ) \un_{ \big\{ D^{_{(n)}}_{^{s_n}} ( V_{s_n} (\tau_{\geq n}) ) \geq c \big\}} \Big] = 0\; . 
\end{equation}
\begin{compactenum}
\item[$(iv)$] The joint convergence 
\end{compactenum}
\begin{equation}
\label{VHCGDcv}
\mathscr Q_n'= \Big(  \tfrac{1}{b_n} V_{\lfloor n\cdot \rfloor} (\tau_n) ,   \tfrac{1}{a_n} H_{ n\cdot } (\tau_n)  ,    \tfrac{1}{a_n} C_{2n\cdot } (\tau_n)  \Big) \underset{n\to \infty}{-\!\!\!-\!\!\!\longrightarrow} 
\big(  X', H', H'  \big)
\end{equation}
\begin{compactenum}
\item[]  holds weakly on $\bD ([0, 1], \bbR)  \times  \bC ([0, 1] , \bbR)^2$ equipped with the product topology. Here 
$(X', H')$ has the same law as $(X, H(X))$ under $\bN (dX \, | \zeta \eqo 1)$.  
\end{compactenum}
\end{theorem} 
\noi
\textbf{Proof.} In $(i)$, for a proof of the joint convergence of the Lukasiewicz path, the height process and the contour process: see D.~\& Le Gall~\cite[Thm 2.3.2 and Corollary 2.5.1]{DuLG02}. The joint convergence with $(\frak g_n /n, \frak d_n /n)$ follows from arguments in Le Gall~\cite[Thm 5.1]{LG10}. 
For $(ii)$ see Le Gall \& Miermont~\cite[Lemma 10]{LGMi11} or Kortchemski~\cite[Lemma 2]{Kor13}. 
For $(iii)$, we refer to the arguments in Kortchemski~\cite[Section 3.2]{Kor13}. For $(iv)$: see D.~\cite{Du03}.   
\cqfd

\medskip

We shall use the previous Theorem to derive limits of coding processes of $\tau_n$ from limits of coding processes of $\tau_{\geq n}$. More precisely, we shall use the following lemma, whose proof is given in Appendix. 
\begin{lemma}
\label{absClem} We keep the assumptions and notation of Theorem \ref{VHCcvstable}. Let $\mathscr G\! \subset \! \ccF$ be $\sigma$-field. Let $E$ and $E'$ be Polish spaces equipped with their respective Borel $\sigma$-field. The product spaces below 
are equipped with product topology. For all $n\ino \bbN^*$, let $r_n \ino \bbN$, let  $U_n$ be a $E$-valued r.v.~and let $Z_n \! : \bD (\bbR_+, \bbR) \! \times \! E \! \to E'$ be a measurable function. We assume for all $n\ino \bbN^*$ 
that $\tau_{\geq n} $ and $\tau_n$ are independent from $(\mathscr G,U_n)$, that  
$r_n/n \! \to \! r\ino (0, 1)$ and that conditionally given $\ccG$, the joint convergence 
\begin{equation}
\label{supenncnd} M_{\geq n}:= \Big(  \big( \tfrac{1}{b_n}V_{ (ns) \wedge r_n} (\tau_{\geq n})\big)_{\! s\in \bbR_+}\,  ,\,  Z_n \big(V_{ \cdot \wedge r_n} (\tau_{\geq n}), U_n \big) \Big) \underset{n\to \infty}{-\!\!\!-\!\!\!\longrightarrow}  \big( X^{_{(r)}}_{^{\!}}, Z \big)
\end{equation}
holds weakly on $\bD (\bbR_+, \bbR)\! \times \! E'$. Here we furthermore assume that 
$X^{_{(r)}}_{^{\!}}$ has the same law as $X_{\cdot \wedge r}$ under $\bN (\, \cdot \, \big| \, \zeta \geko 1)$ and that 
$( X^{_{(r)}}_{^{\!}}, Z)$ is independent from $\mathscr G$. Then, conditionally given $\mathscr G$, the joint convergence 
\begin{equation}
\label{eqenncnd}
M_{n}:= \Big( \big( \tfrac{1}{b_n}V_{ (ns) \wedge r_n} (\tau_{ n})\big)_{\! s\in \bbR_+}\,  ,\,Z_n \big(V_{ \cdot \wedge r_n} (\tau_{n}), U_n \big) \Big) \underset{n\to \infty}{-\!\!\!-\!\!\!\longrightarrow}  \big( Y^{_{(r)}}_{^{\!}}, Z' )
\end{equation}
holds weakly on $\bD (\bbR_+, \bbR)\! \times \! E'$ where the r.v.~$( Y^{_{(r)}}_{^{\!}} \!\!  , Z' )$ is independent from $\mathscr G$ 
and such that  
$\bE [ F(  Y^{_{(r)}}_{^{\!}}\!\!  , Z')] \eqo \bE [ F( X^{_{(r)}}_{^{\!}}\!\! , Z ) D_r(X^{_{(r)}}_{^{r}}) ]$ for all bounded measurable functions $F$ on 
$\bD (\bbR_+, \bbR)\! \times \! E'$. 
\end{lemma}
\noi
\textbf{Proof.} See Section \ref{absClempfsec}. \cqfd

\medskip

  We complete this section by recalling several standard estimates on the distribution of $\# \tau$ when $\tau$ is a GW($\mu$)-tree. Let $(V_n)_{n\in \bbN}$ be the RW as in (\ref{cvlawRW}). By standard results on GW trees and RWs,  
\begin{equation}
\label{Tmoinsp}
 \# \tau   \overset{\textrm{(law)}}{=} T_{-1} \quad \textrm{where} \quad  T_{-p} \! :=\!  \inf \big\{ n\ino \bbN: V_n \eqo -p\big\}, \quad p\ino \bbN.
\end{equation}
Note that $T_{-p}$ is a.s.~finite by ($\textrm{H}_1$) in (\ref{hypostaintro}). By the strong Markov property, $(T_{-p})_{p\in \bbN}$ is a RW. 
\begin{lemma}
\label{numbtau} 
We assume that $\mu$ satisfies (\ref{hypostaintro}). Recall from (\ref{bnBiGoTe}) the definition of the slowly varying function $L\! : \! (0, \infty) \! \to \! (0, \infty)$.  
For all $s\ino [0, 1]$, we set $\varphi_\mu (s) \eqo \bE [s^{\# \tau} ]$. Then, there is a constant 
$c_\alpha\ino (0, \infty)$ that only depends on $\alpha$, such that 
\begin{equation}
\label{numtauber}
n^{\frac{1}{\alpha}}L(n) \bP \big(\# \tau \geqo n  \big) \to  c_\alpha \quad \textrm{and} \quad 1\! -\! \varphi_\mu (e^{-\lambda}) \sim_{0^+} c_\alpha \Gamma_{\!\! \texttt e} \big( 1\! -\! \tfrac{1}{\alpha}\big) \lambda^{\frac{1}{\alpha}} L\big( \tfrac{1}{\lambda}\big)^{-1} 
\end{equation}
(here $\Gamma_{\!\! \texttt e} $ is Euler's gamma function). Furthermore, $n^{1+\frac{1}{\alpha}}L(n) \bP \big(\# \tau \eqo n  \big) \to  \alpha^{-1} c_\alpha $. 
\end{lemma}
\noi
\textbf{Proof.} The first limit in (\ref{numtauber}) implies the second one by Karamata's Abelian Theorem for Laplace transforms  
 see e.g.~Bingham, Goldies \& Teugels \cite[Theorem 1.7.6]{BiGoTe}. By Kemperman's identity (see e.g.~\cite[Theorem 8.9.15]{BiGoTe}), we first get 
 $\bP (T_{-1} \eqo n )\eqo \frac{1}{n} \bP (V_n \eqo -1)$. By Gnedenko's local limit theorem 
(see e.g.~\cite[Theorem 8.4.1]{BiGoTe}), we next get $ \bP (V_n \eqo -1) \sim \textrm{Cst}/ b_n$. Thus 
$\bP ( \# \tau \eqo n) \sim \textrm{Cst}/ (nb_n)  $. By Karamata's Theorem for tails (see e.g.~\cite[Proposition 1.5.10]{BiGoTe}), we get $\bP ( \# \tau \geqo  n) \sim \alpha \textrm{Cst}/ b_n  $, which entails the lemma. \cqfd

\subsection{Real trees, convergences of metric spaces and pseudometrics.}
\label{GHPsec} 
In this section, we first recall the definition of a specific class of metric spaces called real trees. They appear as limits of rescaled discrete trees. Here the limit holds in the sense of   
Gromov--Hausdorff--Prokhorov (GHP for short). GHP-convergence actually concerns isometry classes of pointed measured compact metric spaces
and, although it provides a neat and intrinsic framework for formulating 
convergences of renormalized random graphs, it is more convenient to proceed with a certain type of encoding of these spaces which provides pseudometrics viewed as $\bbR_+^2$-indexed processes. 
We briefly recall from D., K., Lin \& Torri~\cite{DuKhLiTo22} 
the results on the space of pseudometrics that are used in our proofs.

\medskip

\noi
\textbf{Real trees.} Informally, real trees are obtained by glueing compact intervals without creating cycles. More precisely, they are defined as follows.  
\begin{definition}
\label{errtredef} (\textbf{Real trees}) 
Let $(T, d)$ be a metric space; it is a {\it real tree} {if and only if} the following holds true.

\smallskip

\noi
\textbf{(a)} For any $\sigma_1, \sigma_2 \! \in\!  T$, there is a unique isometry 
$f:[0,d(\sigma_1,\sigma_2)] \! \rightarrow \! T$ such
that $f(0)\!=\! \sigma_1$ and $f(d(\sigma_1,\sigma_2))\! =\! \sigma_2$. Then, we set 
$\lgeo \sigma_1,\sigma_2\rgeo \! :=\! f([0,d (\sigma_1,\sigma_2)])$. 

\smallskip

\noi
\textbf{(b)} For any continuous injective function 
$g: [0, 1] \! \rightarrow \! T$, such that $g(0)=\sigma_1$ and $g(1)=\sigma_2$,  $g([0,1]) \! = \! \lgeo \sigma_1,\sigma_2\rgeo$. \cq 
\end{definition}
Among connected metric spaces, real trees are characterized by the so-called \emph{four points inequality}. Namely,  if $(T, d)$ is a connected metric space, then $(T, d)$ is a real tree {if and only if} for any $\sigma_1,  \sigma_2,  \sigma_3,  \sigma_4  \in T$,
\begin{equation}
\label{4points}
d(\sigma_1, \sigma_2) + d(\sigma_3, \sigma_4) \leq \big(d(\sigma_1, \sigma_3) + d(\sigma_2, \sigma_4)\big) \vee  \big( d(\sigma_1, \sigma_4) + d(\sigma_2, \sigma_3)  \big) . 
\end{equation}
We refer to Evans \cite{Ev08} or to Dress, Moulton \& Terhalle~\cite{DrMoTe96} for a detailed account on this property. 
In this paper, we shall only consider \emph{compact pointed measured real trees}:  
namely, compact real trees $(T,d)$ equipped with a distinguished point  $r\ino T$ that is viewed as a  \textit{root}, and with a finite measure $\mu$ defined on the Borel $\sigma$-field of $T$.

\medskip

\noi
\textbf{Gromov--Hausdorff--Prokhorov metric.} In this paper, we prove that rescaled specific random graph trees equipped with a root and (variants of) the counting measure converge to random compact pointed measured real trees. The convergence of these metric spaces holds weakly in the Gromov--Hausdorff--Prokhorov sense as recalled here. 

Let $(E_1, d_1, r_1, \mu_1)$ and $(E_2, d_2, r_2, \mu_2)$ be two pointed measured compact metric spaces: here $\mu_1$ and $\mu_2$ are finite measures on the respective Borel $\sigma$-fields of $E_1$ and $E_2$, and $r_1\ino E_1$ and $r_2 \ino E_2$ are distinguished points. The pointed \textit{Gromov--Hausdorff--Prokhorov distance} (the \textit{GHP-distance} for short) between $E_1$ and $E_2$ is defined by
\begin{multline}
\label{defGHP}
\bdelta_{\mathtt{GHP}} (E_1, E_2)\! = \! \inf \!  \Big\{ d_{E}^{\textrm{Haus}} \! \big(\phi_1 (E_1), \phi_2 (E_2)\big)  \\ 
+ 
d_E (\phi_1 (r_1), \phi_2 (r_2))  + d_{E}^{\textrm{Prok}} \big(\mu_1 \! \circ \! \phi^{-1}_1\! \! , \mu_2 \! \circ \! \phi^{-1}_2\big)  \Big\}. 
\end{multline}
Here, the infimum is taken over all Polish spaces $(E, d_E)$ and all isometric embeddings $\phi_i: E_i
\hookrightarrow E$, $i\ino \{ 1, 2\}$;
$d_{E}^{\textrm{Haus}}$ stands for the Hausdorff distance on the space of compact subsets of $E$ (namely, 
$d_{E}^{\textrm{Haus}}(K_1, K_2)\! = \! \inf \{ \epp \ino (0, \infty): K_1 \! \subset \!  K_2^{(\epp)} \; \textrm{and} \;   
K_2 \! \subset \!  K_1^{(\epp)} \}$, where $A^{(\epp)}\! = \! \{ y\ino E: d_E(y, A) \! \leq \! \epp \}$ for all non-empty 
$A\! \subset \! E$); 
$d_{E}^{\textrm{Prok}}$ stands for the Prokhorov distance on the space of finite Borel measures on $E$ (namely, 
$d_{E}^{\textrm{Prok}}(\mu, \nu)\! = \! \inf \{ \epp \ino (0, \infty) \! : \! \mu (K)\! \leq \! \nu(K^{(\epp)}) + \epp  \; \textrm{and} \;   
\nu (K)\! \leq \! \mu(K^{(\epp)}) + \epp , \; \forall K \! \subset \! E \; \textrm{compact} \}$); 
for all 
$i\ino \{ 1, 2\}$, $\mu_i \! \circ \! \phi^{-1}_i$ stands for the pushforward measure of $\mu_i$ via $\phi_i$. 

\smallskip

We next recall from Theorem 2.5 in Abraham, Delmas \& Hoscheit \cite{AbDeHo13}
the following assertions: $\bdelta_{\mathtt{GHP}} $ is a pseudometric (i.e., , ~it is symmetric and it satisfies the triangle inequality) and $\bdelta_{\mathtt{GHP}} (E_1, E_2)\! = \! 0$ {if and only if} $E_1$ and $E_2$ are \textit{isometric}, namely {if and only if} there exists a bijective isometry $\phi \! : \! E_1 \! \rightarrow \! E_2$ such that $\phi(r_1)\! = \! r_2$ and such that $\mu_2 \! = \! \mu_1 \circ  \phi^{-1}$. 
Then denote by $\bbM$ \emph{the space of isometry classes of pointed measured compact metric spaces.} 
Theorem 2.5 in Abraham, Delmas \& Hoscheit \cite{AbDeHo13} asserts that 
\begin{equation}
\label{rappGHP} \textrm{$(\bbM, \bdelta_{\mathtt{GHP}})$ is a complete and separable metric space. }
\end{equation}

It is easy to prove that the four points inequality (\ref{4points}) is preserved under GHP-convergence. Therefore, \emph{the measure preserving pointed isometry classes of compact real trees form a $\bdelta_{\mathtt{GHP}}$-closed subspace of $\bbM$}. 

\medskip

\noi
\textbf{Limits of continuous pseudometrics.} We introduce here the space of pseudometrics as continuous functions on real intervals, which is an alternative point of view to the intrinsic approach of GHP convergence of metric spaces: this allows a more concrete manipulation of metric spaces and it makes possible to use tools of weak convergence of 
processes.

\begin{definition}
\label{psddef} Let $\zeta\ino (0, \infty)$. We denote by $\bC ([0, \zeta]^2, \bbR)$ the space of continuous 
functions from $[0, \zeta]^2$ to $\bbR$ that is a Banach space when equipped with 
uniform norm $\lVert \cdot \rVert$. 

\smallskip

\noi
$\textbf{(a)}$ We denote by $\MMM([0, \zeta])$ the subspace of $\bC ([0, \zeta]^2, \bbR)$ of \emph{continuous pseudometrics} on $[0, \zeta]$. Namely, the functions  
$d (\cdot, \cdot) \ino  \bC ([0, \zeta]^2, \bbR)$ such that for all $s_1, s_2, s_3\ino [0, \zeta]$, 
$$d(s_1, s_2) \! \geq \! 0, \; \,  d(s_1, s_1)\! = \! 0, \; \, d(s_1, s_2)\! = \! d(s_2, s_1) \; \, \mathrm{and} \;\,  d(s_1, s_3)\! \leq \! d (s_1, s_2) + d(s_2, s_3). $$
$\textbf{(b)}$ We also denote by $\MMT([0, \zeta])$ the space {of continuous real tree pseudometrics}. 
Namely, the functions $d (\cdot, \cdot) \ino \MMM([0, \zeta])$ such that for all $s_1, s_2, s_3, s_4 \ino [0, \zeta]$, 
\begin{equation}
\label{fourpoints}
 d(s_1, s_2) + d(s_3, s_4) \leq \max \big( d(s_1, s_3) + d(s_2, s_4) \, ; \,d(s_1, s_4) + d(s_2, s_3) \big) \; .
\end{equation}
We easily check that $\MMT ([0, \zeta])$ and $\MMM ([0, \zeta])$ are closed subsets of 
$(\bC ([0, \zeta]^2, \bbR), \lVert \cdot \rVert)$. \cq 
\end{definition}

Let $d \ino \MMM ([0, \zeta])$. It induces a connected compact pointed measured metric space as follows: 
for all $s_1  , s_2\ino [0, \zeta]$, set $s_1 \sim_d s_2$ {if and only if} $d(s_1, s_2)\! = \! 0$. Clearly, $\sim_d$ is an equivalence relation. We then define the quotient space:  
\begin{equation}
\label{quotient}
E_d = [0, \zeta] / \!\! \sim_d , \quad \mathtt{proj}_d: [0, \zeta] \! \rightarrow \! E_d,\; \textrm{the canonical projection} , \; r_d = \mathtt{proj}_d (0){.}
\end{equation}
We keep denoting $d$ the metric induced by $d$ on $E_d$. Since $d$ is continuous on $[0, \zeta]^2$, $\mathtt{proj}_d$ is continuous and $(E_d, d, r_d)$ is a pointed compact and connected metric space. 
We also equip $E_d$ with the pushforward measure $\mu_d$ of the Lebesgue measure on $[0, \zeta]$ via the canonical projection: namely, for all nonnegative measurable functions $f$ on $E_d$, 
\begin{equation}
\label{induleb}
\int_{E_d}\!\! \! f(x) \, \mu_d (dx)\! = \! \int_0^\zeta \!\!  f(\mathtt{proj}_d (s))\,  ds \; .
\end{equation}
Note that $\mu_d $ is a finite measure with total mass $\zeta$. We next recall the following result from D., K., Lin \& Torri \cite{DuKhLiTo22}. 
\begin{proposition} 
\label{groweak} Let $d, d^\prime\ino \MMM([0, \zeta])$ and let $E_d$ and $E_{d^\prime}$ be the related 
pointed measured compact metric spaces as defined by (\ref{quotient}) and (\ref{induleb}). Then
\begin{equation}
\label{ineqw}
\bdelta_{\mathtt{GHP}} (E_{d},E_{d^\prime}) \leq \tfrac{3}{2} \lVert d \! -\! d^\prime \rVert \; .
\end{equation}
\end{proposition}
\noi
\textbf{Proof.} See D., K., Lin \& Torri \cite[Proposition 4.14]{DuKhLiTo22}.\cqfd 
 
\begin{example}
\label{rltrcdng} (\textit{Real trees coded by continuous functions}) Let $\zeta \ino (0, \infty)$ and 
let $h\! : \! [0, \zeta] \! \rightarrow \! \bbR$ be continuous. For all $s_1, s_2\ino [0, \zeta]$, we set 
\begin{equation} 
\label{djhczkk}
m_h (s_1, s_2) = \!\! \!\!\!\!\!\! \!\!\!\! \!\!\!\!\inf_{\,  \qquad s\in [s_1 \wedge s_2, s_1 \vee s_2] }\! \!\!\!\!\!\!\!\!\!\!\!\! \!\!\!\! h (s) \qquad  \textrm{and} \quad d_h (s_1, s_2)= h(s_1)+ h(s_2) - 2 m_h (s_1, s_2) \; .
\end{equation}
We easily check that $d_h \ino \MMT ([0, \zeta])$ and to simplify we denote by $(T_h, d_h, r_h, \mu_h)$ 
the induced  metric space {$(E_{d_h}, d_h, r_{d_h}, \mu_{d_h})$} as defined in (\ref{quotient}) 
and (\ref{induleb}); 
we also denote by $p_h \! : \! [0, \zeta] \! \rightarrow \! T_h$ the canonical projection (instead of $\mathtt{proj}_{d_h}$); $(T_h, d_h, r_h, \mu_h)$ is a pointed measured compact real tree that is the \textit{tree coded by the function $h$}.  
Let $h^\prime \! : \! [0, \zeta] \! \rightarrow \! \bbR$ be continuous. Then observe that 
\begin{equation}
\label{contitreed}
 \forall s_1, s_2, \ino [0, \zeta], \quad \big| d_h (s_1, s_2)\! -\! d_{h^\prime} (s_1, s_2)\big| \leq  4 \sup_{s\in [0, \zeta] } |h(s)\! -\! h^\prime (s)| \; , 
\end{equation}
which shows the continuity in $ \MMT ([0, \zeta])$ of the application $h\mapsto d_h$ and therefore the continuity of $h\mapsto T_h$ with respect to GHP-convergence by Proposition \ref{groweak}.  \cq 
\end{example}

\subsection{Continuous snakes and related metrics.}
\label{Cosnake}

Here we recall the definition of Brownian snakes that have been introduced by J-F.~Le Gall in \cite{LG93bis}. They are continuous analogs of discrete snakes introduced in Definition \ref{brwlkdef}. One-dimensional Brownian snakes play a key role in our result: they encode a random metric space called the \emph{Brownian cactus} (see Curien, Le Gall \& Miermont  \cite{CuLGMi13}) which is the limit of the rescaled range of the tree-valued BRWs studied in this article.  
In this section, we recall the definition of the Brownian cactus with $\alpha$-stable branching mechanism, along with technical results on limits of snake-pseudometrics, some come from D., K., Lin \& Torri~\cite{DuKhLiTo22}, 
others are new.

\medskip
\noi
\textbf{General definition of snakes.} Let us first introduce some notation and conventions on spaces of continuous functions.
 If $(E, d_E)$ is a complete and separable metric space (a Polish space), then $\bC^0(\bbR_+, E)$ stands for the space of continuous functions from $\bbR_+$ to $E$ that is also Polish when equipped with the topology of uniform convergence on all compact intervals, which is metrized e.g.~by $\delta (f,g)$ $\eqo$ $\sum_{N\in \bbN}$  $2^{-N} \min (1, \sup_{s\in [0,N]} $ $d_E (f(s), g(s)) )$, $f, g\ino \bC^0(\bbR_+, E)$. 

The Euclidian space $\bbR^q$, $q\ino \bbN^*\! $, is equipped with its  
canonical basis denoted by $(\mathtt{e}_1, \ldots, \mathtt{e_q})$, with its canonical scalar product denoted by $\langle \cdot, \cdot \rangle $ and with its associated Euclidian norm denoted by $\lvert \,  \cdot \, \rvert$.
When it is convenient, \emph{we shorten the $\bC^0(\bbR_+, \bbR^q)$ notation to $\bC^{_0}_{^q}$ and to simplify notation, we set $\bC^{_0}_{^{\!+}}\eqo \bC^0 (\bbR_+, \bbR_+)$.}

\begin{definition} (\textbf{Continuous snakes)} 
\label{defC0sna}
Let $(E, d_E)$ be complete and separable. 
We equip the set
$ \bCpl  \times    \bC^0(\bbR_+, \bCqq)$ with the product topology, which is Polish. 
Then $(h(s), w_s(\cdot))_{s\in \bbR_+}\! \ino  \bCpl  \times    \bC^0(\bbR_+, \bCqq)$ is a \emph{snake} if the following conditions are satisfied.
\begin{compactenum}

\smallskip

\item[$(a)$] For all $s\ino \bbR_+$ and for all $r\ino [h(s), \infty)$, $w_s (r)\! = \! w_s (h(s))\! =: \! \widehat{w}_s $. 

\smallskip

\item[$(b)$] For all $s_1, s_2 \ino \bbR_+$, 
$w_{s_1} (r)\! = \! w_{s_2} (r)$ for all  $r\ino [0,m_h (s_1, s_2)]$ where we have set 
$m_h (s_1, s_2)\!   =\!  \inf_{s\in [s_1\wedge s_2, s_1\vee s_2]} h(s)$. 

\smallskip

\end{compactenum}
\noi
$\bullet$ We refer to $(b)$ as to the \textit{snake property}. The function $h$ is called the \textit{lifetime process} of the snake and the function $(\widehat{w}_s)_{s\in \bbR_+}$ the \textit{endpoint process} of the snake.

\noi
$\bullet$ We denote by $\fSigma_q (\bbR_+)$ the space of snakes of $ \bCpl \! \times \!   \bC^0(\bbR_+, \bCqq)$. For all 
$\zeta \ino (0, \infty)$, we denote by $\fSigma_q ([0,\zeta])$ the snake with duration $\zeta$: namely snakes $(h,w)\ino \fSigma_q (\bbR_+)$ such that $h(s)\eqo 0$ for all $s\ino [\zeta, \infty)$. 

\noi
$\bullet$ To simplify notation, we also set $\fSigma (\bbR_+)\! :=\! \fSigma_1 (\bbR_+)$ and 
$ \fSigma ([0,\zeta]) \! := \! \fSigma_1 ([0,\zeta])$.  \cq  
\end{definition}
\begin{remark}
\label{easy} 
\textbf{(a)} We easily check that $\fSigma_q(\bbR_+)$ is a closed subset of  $\bCpl \! \times \!   \bC^0(\bbR_+, \bCqq)$ .

\smallskip

\noi
\textbf{(b)} If $(h,w)\ino \fSigma_q(\bbR_+)$, then note that $\widehat{w}\ino \bCqq$ and observe that $(h,w)\mapsto \widehat{w}$ is continuous from the Polish spaces $\fSigma_q(\bbR_+)$ to $\bCqq$.

\smallskip

\noi
\textbf{(c)} Let $(h,w)\ino \fSigma_q(\bbR_+)$ and let $s, r\ino \bbR_+$. If $r\leqo m_h (0, s)$, the snake property implies that $w_s(r)\eqo w_0 (r)$ and if $r\geqo m_h (0, s)$, $w_s(r)\eqo \widehat{w}_{\alpha _{r,s}}$ 
where $\alpha_{r,s}\eqo \sup \{ s^\prime\ino [0, s]: h(s^\prime)\leqo r\wedge h(s)\}$ by the snake property again. 
Thus, there is a function $G$ from $\bCpl\! \times \! (\bCqq)^2$ to $\fSigma_q(\bbR_+)$ such that $G(h,w_0, \widehat{w})\eqo (h, w)$ for all $(h,w)\ino \fSigma_q(\bbR_+)$. Note that actually $G$ can be chosen to be Borel-measurable (this fact is not needed in what follows).

\smallskip

\noi
\textbf{(d)} Let $(M_u)_{u\in \fftree}$ be a $\bbR$-valued branching walk and let $W$ stand for its associated discrete contour snake as in Definition~\ref{brwlkdef}. Then observe that $(C(\fftree), W_\cdot (t; \cdot) )\ino \fSigma (\bbR_+)$. Therefore, the contour process $C(\fftree) $ of $\fftree$ is the lifetime process of $W_\cdot (t; \cdot)$.

\smallskip

\noi
\textbf{(e)} Let $(h,w)\ino \fSigma_q(\bbR_+)$ and $N\ino\bbN^*$. We define $R_N(h,w)\eqo (h^\prime,w^\prime)$ as follows: $h^\prime(s)\eqo h(s)$ and $w_{s}^\prime\eqo w_s$ for all $s\ino[0,N]$, $h^\prime(s)\eqo h(2N\!-\!s)$ and $w_{s}^\prime\eqo w_{2N\!-\!s}$ for all $s\ino[N,2N]$, and $h^\prime(s)\eqo (2N\!+\!1\!-\!s)_+ h(0)$ and $w_s^\prime\eqo w_0(\cdot\wedge h^\prime(s))$ for all $s\ino [2N, \infty)$. Thus, $R_N(h,w)\ino\fSigma_q([0,2N\!+\!1])$, $\widehat{w}^\prime_s\eqo \widehat{w}_s$ if $s\ino[0,N]$, $\widehat{w}^\prime_s\eqo \widehat{w}_{2N\!-\!s}$ if $s\ino[N,2N]$ and $\widehat{w}^\prime_s\eqo w_0((2N\!+\!1\!-\!s)_+\, h(0))$ if $s\geko 2N$.

\smallskip

\noi
\textbf{(f)} Let $(h,w)\ino \fSigma_q(\bbR_+)$. We set $\underline{h}(s) \eqo h(s)\! -\! m_h(0,s)$ and $\underline{w}_s(r)\eqo w_s(r\! +\! m_h(0,s)) \! -\! w_0(m_h(0,s))$, $s, r\ino \bbR_+$, and we observe that $(\underline{h},\underline{w})\ino \fSigma_q(\bbR_+)$ and that $\widehat{\underline{w}}_s \eqo \widehat{w}_s \! -\! w_0(m_h(0,s))$, $s\ino\bbR_+$. 
 \cq 
\end{remark}

\noi
\textbf{Finite-dimensional marginal laws of random snakes.} Here we consider Brownian snakes. To define their finite-dimensional marginals, it is convenient to introduce the following. 
For all $w_0, w_1\ino \bCqq$ and all $m\ino \bbR_+$, we define $w_2\eqo w_0 \! \oplus_m \! w_1\ino \bCqq$ by setting $w_2(s)\eqo w_0(s)$ if $s\ino [0, m]$ and $w_2(s)\eqo w_0(m)+ w_1(s\! -\! m)\!-\! w_1(0)$ if $s\ino [m, \infty)$. We easily check that 
\begin{equation}
\label{contoplus_prel}
(w_0, w_1, m, r)\ino (\bCqq)^2\! \times \! \bbR_+^2\longmapsto  (w_0 \! \oplus_m\!  w_1) (\cdot \wedge r) \ino \bCqq \quad \textrm{is continuous.}
\end{equation}
 Let $h\ino \bCpl$, let $p\ino \bbN^*$, let $s_0\eqo 0\leqo s_1\leko \ldots \leko s_p$ be real numbers 
and let $w_0, \ldots, w_p\ino \bCqq$. We define $F_p (h;  (w_j)_{ 0\leq j\leq p}\, ;  (s_j)_{1\leq j\leq p})\eqo (w^\prime_j)_{0\leq j\leq p} \ino (\bCqq)^{p+1}$ by setting 
\begin{equation}
\label{defFpmarg}
w^\prime_0\eqo w_0(\cdot\wedge h(0)) \quad \textrm{and} \quad w^\prime_{j+1} \eqo \big( w^\prime_j \oplus_{m_h(s_j, s_{j+1})} w_{j+1} \big) (\cdot \wedge h(s_{j+1})) , \quad 0\leq j\leko p\; ,
\end{equation}
where we recall that $m_h (s_j, s_{j+1})\!  =\!  \inf_{s\in [s_j, s_{j+1}]} h(s)$. By (\ref{contoplus_prel}) 
and the continuity of the function $(h, (s_j)_{0\leq j\leq p})\! \mapsto \! (h(0), (h(s_{j+1}), m_h (s_j, s_{j+1}))_{0\leq j<p})$, we see that 
\begin{equation}
\label{contoplus}
F_p : \bCpl \! \times \! ( \bCqq)^{p+1}\!  \!\times \! \bbR_{^+}^{_{\da p}} \longmapsto (\bCqq)^{p+1}  \quad \textrm{is continuous,}
\end{equation}
where $\bbR_{^+}^{_{\da p}}\eqo \{ (s_j)_{1\leq j \leq p} \ino \bbR_+^p: s_1 \leko \ldots \leko s_p\}$. The functions $(F_p)_{p\in \bbN^*}$ 
provide a convenient way of expressing the finite-dimensional marginal laws 
of Brownian snakes. Recall that unless otherwise specified, all the r.v.s which are considered in this article 
are defined on the same probability space $(\Omega, \mathscr F, \bP)$. 
\begin{definition}
\label{Brosnadef} 
\textbf{(Brownian snakes)} 
A $q$-dimensional Brownian snake is a r.v. $(H, W)\! :\!  \Omega\!  \rightarrow \! \fSigma_q(\bbR_+)$ that is measurable with respect to $\mathscr F$ and to the Borel $\sigma$-field of $\fSigma_q(\bbR_+)$, such that for all $p\ino \bbN^*$, and all real numbers 
$0\leqo s_1\leko \ldots \leko s_p$, 
\begin{equation}
\label{margBrosna}
\big(W_0,W_{\! s_1}, \ldots , W_{\! s_p} \big) \; \overset{\textrm{(law)}}{=} \;  F_p \big( H ; (B^{(j)}) _{0\leq j\leq p}  ;  (s_j)_{1\leq j\leq p} \big)\; , 
\end{equation}
where $B^{(0)}\eqo W_0$, where $(B^{(j)})_{1\leq j\leq p}$ are independent standard $q$-dimensional Brownian motions that are moreover independent from $W_0$ and $H$. \cq 
\end{definition}
\begin{remark}
\label{Bsnarem}
 
\textbf{(a)} First, it is easy to check that the laws provided by the left member in (\ref{margBrosna}) actually form 
a consistent family of finite-dimensional laws (so that the previous definition makes sense). 
Next, since $\bC^0 (\bbR_+, \bCqq)$ is Polish, standard results ensure the existence of a regular version of the conditional law of $W$ given $(W_0, H)$, i.e., , ~a measurable transition kernel $P_{w_0, h} (dW)$ from $\bCqq \! \times \! \bCpl$ to $\bC^0 (\bbR_+, \bCqq)$ such that $\bE [G(W_0, H)\un_{\{ W\in A\}}]\eqo \bE [G(W_0, H) P_{W_0, H}(A)]$ for all Borel subsets $A$ of $\bC^0 (\bbR_+, \bCqq)$ and all nonnegative measurable functions $G$. 

More precisely, the continuity (\ref{contoplus}) of $F_p$ entails the following. Denote by $\Pi^{w_0,h}_{s_1, \ldots, s_p} $ the law of $F_p (h; (B^{(j)}) _{0\leq j\leq p}  ;  (s_j)_{1\leq j\leq p} )$ where $B^{(0)}\eqo w_0$ and the $(B^{(j)})_{1\leq j\leq p}$ are independent standard $q$-dimensional Brownian motions. 
Then, there exists a Borel subset $B$ of $\bCqq\! \times \! \bCpl$ such that $\bP ((W_0, H)\ino B)\eqo 1$ 
and such that for all $(w_0, h)\ino B$, for all $p\ino \bbN^*$ and for all real numbers $0\leqo$ $s_1$  $\leko$ $\ldots$ $\leko$ $s_p$, 
\begin{equation}
\label{regcondsna}
\textrm{$\Pi^{w_0,h}_{s_1, \ldots, s_p}$ is the law of $(W_0,W_{s_1}, \ldots, W_{s_p}) $ under $P_{w_0, h} (dW) $.} 
\end{equation}

\smallskip

\noi
\textbf{(b)} The recursive construction (\ref{defFpmarg}) entails that conditionally given $(W_0, H)$, $W$ is a (time inhomogeneous) Markov process. More precisely, keeping notation from \textbf{(a)}, for all $(w_0, h)\ino B$ we check that $W$ under $P_{\! w_0, h}$ is an inhomogeneous Markov process with transition probabilities $P_{\! w_0, h} (w_{s^\prime} \ino dw^\prime \, | \, w_s\eqo w) \! := \!  \Pi^{w, h (\cdot\,  +s)}_{s^\prime -s} (dw^\prime)$.  \cq
\end{remark}

Let us mention the following: given a continuous lifetime process $H$ and an initial path $W_0$, \emph{there is not necessarily a continuous snake} $(H, W)$. However, when $H$ stands for the $\alpha$-stable height process, there is a 
continuous version of the Brownian snake as shown in the following proposition, which is a direct consequence of a result in D.~\& Le Gall \cite{DuLG02}.  
\begin{proposition}
\label{Holdsnake} Let $\alpha\ino (1, 2]$. Let $(H_s)_{s\in \bbR_+}$ be the $\alpha$-stable height process as defined by (\ref{Hlimit}) (resp.~under the normalized excursion law $\bN (\, \cdot \, | \, \zeta \geko 1)$ or $\bN (\, \cdot \, | \, \zeta \eqo 1)$). Then there exists a continuous one-dimensional Brownian snake $W$ whose lifetime process is $H$ which is called the \emph{one-dimensional Brownian snake with $\alpha$-stable branching mechanism}. 
 \end{proposition}
\noi
\textbf{Proof.} See \cite[Proposition 4.4.1]{DuLG02} when $H$ is under $\bP$. Standard arguments entail the result for $H$ under $\bN (\, \cdot \, | \, \zeta \geko 1)$, then we use (\ref{abscont}) to derive it under $\bN (\, \cdot \, | \, \zeta \eqo 1)$. See \cite[Lemma 4.30]{DuKhLiTo22} for a more explicit discussion of the same ideas.  \cqfd

\medskip

\noi
\textbf{Snake-pseudometrics and snake-trees.} Snakes encode the metric of the rescaled range of the BRWs that we consider here. 
\begin{definition}
\label{defsnkme} (\textbf{Snake-pseudometrics}) Let $\zeta \ino (0, \infty)$ and let $(h,w)\ino \fSigma ([0, \zeta])$. Recall from (\ref{djhczkk}) the definition of $m_h (\cdot, \cdot)$ and recall  that $\widehat{w}_s\eqo w_s( h(s))$, $s\ino \bbR_+$. 
For all $s_1, s_2\ino [0, \zeta]$, we set
\begin{eqnarray}
\label{snkmeff}
M_{h,w} (s_1, s_2) \!\! \! \! \!& =&\!\!  \! \!\!  \min \! \Big( \!  \min \! \big\{  w_{s_1} (r) ; r\geqo m_h (s_1, s_2) \big\}  ,  \min \! \big\{ w_{s_2} (r) ; r\geqo m_h (s_1, s_2)\big\} \! \Big)  \nonumber \\
 &   &\!\!\!\! \!\!\!\! \!\!\!\! \!\!\!\!  \textrm{and} \quad d_{h,w} (s_1, s_2)\! = \! \widehat{w}_{s_1}+  \widehat{w}_{s_2} -2M_{h,w} (s_1, s_2) \; .
\end{eqnarray}

We call $d_{h,w}$ the \textit{snake-pseudometric associated with $(h,w)$} (see Lemma~\ref{Hermagor_fix} below). \cq
\end{definition}
\begin{lemma}
\label{Hermagor_fix} Let $\zeta \ino (0, \infty)$, let $(h,w)\ino \fSigma ([0, \zeta])$ and let $d_{h,w}$ be as in Definition \ref{defsnkme}. Then, $d_{h,w}$ is a continuous pseudometric on $[0, \zeta]$ satisfying (\ref{fourpoints}): namely $d_{h,w} \ino \MMT([0, \zeta])$.
\end{lemma}
\noi
\textbf{Proof.} See \cite[Lemma 4.22]{DuKhLiTo22}. \cqfd 

\begin{definition}
\label{notasnkm} (\textbf{Snake-trees, ${\boldsymbol \alpha}$-stable Brownian cactus}) Let $\zeta \ino (0, \infty)$, let $(h,w)\ino \fSigma ([0, \zeta])$ and $d_{h,w} \ino \MMT([0, \zeta])$ be as in Definition~\ref{defsnkme}.

\noi
$(\mathbf a)$ We denote by 
$(T_{h,w}, d_{h,w}, r_{h,w}, \mu_{h,w})$ the pointed measured compact metric space induced by the pseudometric $d_{h,w}$ as defined in (\ref{quotient}). We call it the \textit{snake-tree} associated with $(h,w)$. To simplify, we denote by $p_{h,w}: [0, \zeta] \! \rightarrow \! T_{h,w}$ the canonical projection.

\smallskip

\noi
$(\mathbf b)$ Let $\alpha\ino (1, 2]$ and let $(H, W)$ be a one-dimensional Brownian snake with $\alpha$-stable branching mechanism as defined in Proposition \ref{Holdsnake}. Then the random real tree 
$(T_{H,W}, d_{H,W}, r_{H,W}, \mu_{H,W})$ is called a \emph{Brownian cactus with $\alpha$-stable branching mechanism}.  \cq   
\end{definition}

 As already mentioned, the Brownian Cactus has been introduced in Curien, Le Gall \& Miermont \cite{CuLGMi13} to study planar maps: roughly speaking the Brownian cactus {corresponds} to the case of a quadratic branching mechanism $\alpha\! = \! 2$. In this case, they prove that {a.s.~}its upper-local density for typical points is $4$ (Proposition 5.1 {\cite{CuLGMi13}} p.~364) and that its Hausdorff dimension is $4$ (Corollary 5.3 {\cite{CuLGMi13}} p.~365). See also Le Gall \cite{LG15} where the level sets of the Brownian cactus are studied to derive results on the Brownian maps. 
In D., K., Lin \& Torri \cite[Lemma 4.34]{DuKhLiTo22}, it is proved that the Brownian cactus with $\alpha$-stable mechanism is a.s.~a continuum real tree (i.e., , ~$\mu_{H,W}$ is diffuse and supported by the set of leaves of $T_{H,W}$) and that it is a binary real tree with a (necessarily countable) dense set of branch points: namely, a.s.~for all 
$x \ino  T_{H,W}$, $T_{H,W}\backslash \{ x\}$ has one, two, or three connected components and the set of the $x$ where  $T_{H,W}\backslash \{ x\}$ has three connected components is dense in $T_{H,W}$.

\medskip

\noi
\textbf{Limits of snakes and snake-pseudometrics.} We first present a general result showing that convergence of snakes boils down to the convergence of the corresponding endpoint processes. This is a slight extension of the homeomorphism theorem of Marckert \& Mokkadem~\cite[Theorem 2.1]{MaMo03}. Then we recall from D., K., Lin \& Torri~\cite{DuKhLiTo22} a result showing that the convergence of snakes implies the convergence of snake-pseudometrics, which leads to a pratical criterion to prove weak convergence of random snake-trees.

If $(h^{_{\! (n)}}_{^{}}\!\! , w^{_{ \! (n)}}_{^{}})\ino \fSigma_q(\bbR_+)$, $n\ino \bbN$, is such that $(h^{_{\! (n)}}_{^{}}\!\! ,\widehat{w}^{_{\! (n)}}_{^{}})\! \rightarrow \! (h, f)$ in $\bCpl\! \times \! \bCqq$, then $f$ is not necessarily the endpoint process of a continuous snake and the $w^{_{\! (n)}}_{^{}}$ do not necessarily converge in the space $\bC^0 (\bbR_+, \bCqq)$. 
In this direction, we however 
prove the following result. 
\begin{proposition}
\label{endvssna} Let $(h^{_{\! (n)}}_{^{}}\!\! , w^{_{ \! (n)}}_{^{}} )_{ n\in \bbN}$ be a sequence of 
$q$-dimensional continuous snakes. We assume that there exists a $q$-dimensional continuous snake $(h,w)$ such that $\lim_{n\rightarrow \infty} h^{_{\! (n)}}_{^{}}\! \eqo h$ in $\bCpl$, such that $\lim_{n\rightarrow \infty} w_{^0}^{_{(n)}}\eqo w_0$ and such that $\lim_{n\rightarrow \infty} \widehat{w}^{_{(n)}}_{^{}}\! \eqo \widehat{w}$ in $\bCqq$. Then $\lim_{n\rightarrow \infty} w^{_{(n)}}_{^{}}\!\eqo w$ in $ \bC^0(\bbR_+, \bCqq)$. 
\end{proposition}

\noi
\textbf{Proof.} When $h^{_{\! (n)}}_{^{}}(0)\eqo 0$ and $(h^{_{\! (n)}}_{^{}}\!\! , w^{_{ \! (n)}}_{^{}})\ino\fSigma_q([0,1])$ for all $n\ino \bbN$, the result follows from a direct application of Marckert \& Mokkadem~\cite[Theorem 2.1]{MaMo03}. By time-scaling, this still holds when there is some $N\ino\bbN$ such that $h^{_{\! (n)}}_{^{}}(0)\eqo 0$ and $(h^{_{\! (n)}}_{^{}}\!\! , w^{_{ \! (n)}}_{^{}})\ino\fSigma_q([0,2N\!+\!1])$ for all $n\ino \bbN$. Now, we claim that the desired result holds when $h^{_{\! (n)}}_{^{}}(0)\eqo 0$ for all $n\ino\bbN$. Indeed, using Remark~\ref{easy} \textbf{(e)}, the previous particular case implies that $\lim_{n\to\infty}R_N(h^{_{\! (n)}}_{^{}}\!\! , w^{_{ \! (n)}}_{^{}} )\eqo R_N(h,w)$ in $\fSigma_q([0,2N\!+\!1])$ for any $N\ino\bbN$, so $w^{_{(n)}}_{^{}}\!\! \to \! w$ uniformly on $[0,N]$. Now, we prove the desired result without any additional assumptions thanks to Remark~\ref{easy} \textbf{(f)}. We easily see that $\lim_{n\rightarrow \infty} \underline{h}^{_{\! (n)}}_{^{}} \eqo \underline{h}$ in $\bCpl$ and $\lim_{n\rightarrow \infty} \widehat{\underline{w}}^{_{(n)}}_{^{}} \eqo \widehat{\underline{w}}$ in $\bCqq$, so $\lim_{n\rightarrow \infty} \underline{w}^{_{(n)}}_{^{}}\eqo \underline{w}$ in $ \bC^0(\bbR_+, \bCqq)$. Then, we observe that
\[\forall s,r\in\bbR_+,\quad w_s(r)=w_0\big(r\wedge m_h(0,s)\big)+\underline{w}_s\big((r-m_h(0,s))_+\big),\]
which completes the proof. \cqfd 

\medskip

The following lemma recalled from D., K., Lin \& Torri \cite{DuKhLiTo22}. It shows that the convergence of one-dimensional snakes implies the convergence of the associated pseudometrics. Recall that $\fSigma ([0,\zeta])$ stands for the set of one-dimensional snakes. For all $(h,w)\ino \fSigma ([0,\zeta]) $, recall from Example~\ref{rltrcdng} the definition of the real-tree pseudometric $d_h$ and recall from Definition~\ref{defsnkme} the snake-pseudometric $d_{h,w}$. Recall that $\lVert \cdot \rVert$ stands for the uniform norm on 
$\bC([0, \zeta]^2, \bbR)$. 
\begin{lemma} 
\label{Hermagor} Let $\zeta \ino \bbR_+$. We recall that $\fSigma ([0,\zeta])$ stands for the set of one-dimensional snakes. 
For all $(h,w)\ino \fSigma ([0,\zeta]) $, we recall from Example \ref{rltrcdng} the definition of the real-tree pseudometric $d_h$ and we recall from Definition \ref{defsnkme} the snake-pseudometric $d_{h,w}$. 
We denote by $\lVert \cdot \rVert$ the supremum norm on $\bC([0, \zeta]^2, \bbR)$. 
For all $n\ino \bbN$, let $(h^{_{\! (n)}}_{^{}}\!\! , w^{_{ \! (n)}}_{^{}} )\ino \fSigma ([0,\zeta])$. We suppose that $(h^{_{\! (n)}}_{^{}}\!\! , w^{_{ \! (n)}}_{^{}} ) \! \to \! (h, w)$ in $\bCpl \! \times \! \bC([0, \zeta ], \bC^{_0}_{^{1}})$. Then, 
$ \lim_{n\rightarrow \infty}  \lVert d_{h^{(n)}} \! -\!  d_{h} \rVert\! = \!  \lim_{n\rightarrow \infty}  \lVert d_{h^{(n)} , w^{(n)}} \! -\!  d_{h,w} \rVert \eqo 0 $. 
\end{lemma}
\noi
\textbf{Proof.} See D., K., Lin \& Torri~\cite[Lemma 4.20]{DuKhLiTo22}.  \cqfd 

\subsection{Tree-valued branching walks and snake-pseudometrics}
\label{trebrRWsnasec}

We fix an environment $(T, o)\ino \overline{\bbT}^\bullet$.
We then fix a finite tree $t\ino \bbT$ and a branching walk
$\Theta\! : = \! (t, \varnothing; \overline{Y}_u \eqo (T, Y_u), u\ino t) $ such that: $(a)$ $Y_\varnothing\eqo o$; $(b)$ if $v,v^\prime\ino t$ are neighbors, so are $Y_v$ and $Y_{v^\prime}$ in $T$.
In this section, we consider the range $R\eqo \{ Y_u; u\ino t\}$ that is a graph subtree of $T$ and we explain that $R$ is close in some sense to the snake-tree associated with the contour snake of the relative heights of the branching walk, i.e., , ~the contour snake associated with the $\bbR$-branching walk $(\bbn Y_u \bbn)_{u\in t}$.

To that end, we denote by $(\widetilde{t}, d_{\widetilde{t}})$
the real tree obtained by joining adjacent vertices of $t$ 
by a 
segment isometric to $[0, 1]$. We assume that $t \! \subset \! \widetilde{t}$.  
Thus $d_{\mathtt{gr}} (u,v) \eqo d_{\widetilde{t}} (u,v) $ for all $u ,v\ino t$, where $ d_{\mathtt{gr}}$ stands for the graph distance in $t$. For all $\sigma, \sigma'\ino \widetilde{t} $, we denote by  $\lgeo \sigma, \sigma' \rgeo_{\widetilde{t}}$  the geodesic path joining $\sigma$ and $\sigma'$ in $\widetilde{t}$.

We then introduce the \emph{contour exploration of $\widetilde{t}$} as follows. We denote by $(v_k)_{0\leq k \leq 2(\# t-1)}$ the vertices of $t$ listed in the contour order.
We easily check that there exists a continuous function $\widetilde{v}\! : \! [0, 2(\# t\! -\! 1)] \! \rightarrow \! \widetilde{t}$ such that for all $0\leqo k \leko 2(\# t\! -\! 1)$ and for all $s\ino [0, 1]$, 
\begin{equation}
\label{explotilde}     
\widetilde{v} (k+s)\ino \lgeo v_k, v_{k+1} \rgeo_{\widetilde{t}} \quad  \textrm{is such that} \quad  
d_{\widetilde{t}}(v_k ,\widetilde{v} (k+s))\! = \! s. 
\end{equation}
It is convenient to set $\widetilde{v} (s) \eqo \varnothing $ for all $s\ino [  2(\# t\! -\! 1), 2\# t]$. 
The continuous path $(\widetilde{v}(s))_{s\in [0,2\#t]}$ is the contour exploration of $\widetilde{t} $. Namely for all $s,s'\ino [0, 2\# t]$, 
$$ d_{\widetilde{t} } \big( (\widetilde{v}(s), (\widetilde{v}(s') \big) = d_{C(t)} (s,s') $$
where $d_{C(t)}$ is the real tree distance associated with the contour function $(C_s(t))_{s\in [0, 2\#t ]}$ as introduced in Example \ref{rltrcdng}.  
\begin{remark}
\label{interpo1} It is easy to see that the tree $(T_{C(t)}, d_{C(t)} , \mathtt r_{C(t)}, \ttm_{C(t)})$ coded by $C_\cdot (t)$ as introduced in Example \ref{rltrcdng} is isometric to 
$(\widetilde{t} , d_{\widetilde{t} }, \varnothing, \widetilde{\ttm} )$ where $\widetilde{\ttm}$ is the occupation measure of the exploration process $\widetilde{v} $. As a consequence of D., K., Lin \& Torri~\cite[Lemma 5.2]{DuKhLiTo22}, we get that $d^{\widetilde{t}}_{\mathtt{Prok}} \big(  \ttm, \tfrac{_1}{^2}   \widetilde{\ttm} \big) \leqo 2$ where $\ttm \eqo \sum_{v\in t} \delta_v$, which implies that $ \bdelta_{\mathtt{GHP}} \big( (t, d_{\mathtt{gr}}, \varnothing, 
\ttm) , (\widetilde{t}, d_{\widetilde{t}}, \varnothing , \frac{_1}{^2}\widetilde{\ttm}\big)\big) \leqo 3$.\cq 
\end{remark} 
 
Similarly, we denote by $(\widetilde{T}, d_{\widetilde{T}})$
 the real tree obtained by joining adjacent vertices of $T$ 
 by a segment isometric to $[0, 1]$ and we extend  relative heights on $\widetilde{T}$ by setting 
 $$\forall x\ino T, \quad  \forall \gamma \in \lgeo  \overleftarrow{x} , x \rgeo_{\widetilde{T}}, \quad \bbn \gamma \bbn \eqo \bbn \overleftarrow{x} \bbn + d_{\widetilde{T}} (\overleftarrow{x}, \gamma) \; .$$
For all $\gamma, \gamma' \ino \widetilde{T}$, we also denote by $\gamma \wedge \gamma'$ the most recent common ancestor of $\gamma$ and $\gamma'$ that is defined as the unique point of $\lgeo \gamma, \gamma' \rgeo_{\widetilde{T}}$ with minimal relative height. 
We finally define a continuous function $\widetilde{Y}\! :\! \widetilde{t}\! \to \! \widetilde{T}$ that extends $Y$ as follows: 
for all $\sigma \ino \widetilde{t}$ there exist two neighboring vertices $v, v^\prime \ino t$ such that $\sigma \ino \lgeo v, v^\prime\rgeo_{\widetilde{t}}$ and we define 
\begin{equation}
\label{Yexten}
\textrm{$\widetilde{Y}_\sigma$ as the only point of $\lgeo Y_v, Y_{v^\prime} \rgeo_{\widetilde{T}}$ such that 
$d_{\widetilde{T}} (Y_v, \widetilde{Y}_\sigma)\! = \! d_{\widetilde{t}} (v, \sigma)$.  }
\end{equation}
We denote by $\widetilde{R}$ the range of $\widetilde{Y}$ and we denote by $\widetilde{\ttm}_{\mathtt{occ}}$ the occupation measure on $\widetilde{R}$ yielded by the contour exploration of $\widetilde{t}$. Namely,  
\begin{equation}
\label{moccudef}
 \widetilde{R}\eqo \big\{ \widetilde{Y}_{\sigma} \,  ; \, \sigma \ino \widetilde{t} \big\} \quad \textrm{and} \quad \int_{ \widetilde{R}} \! f(\gamma) \, \widetilde{\ttm}_{\mathtt{occ}}(d\gamma)\, =\,  \int_{0}^{2\#t} \!\!\!  \!\!f\big(  \widetilde{Y}_{\widetilde{v} (s) } \big)  \, ds . 
\end{equation} 
We also denote by $\ttm_{\mathtt{occ}}$ the occuptation measure of $Y$ on $R\eqo \{ Y_u\, ; \, u\ino t\}$: 
$ \ttm_{\mathtt{occ}}=\eqo \sum_{u\in t} \delta_{Y_u} $. We note that $d^{\widetilde{T}}_{\mathtt{Haus}} \big( \widetilde{R}, R\big) \leqo 1$.  Then we recall from D.~, K.~, Lin \& Torri \cite[Lemma 5.2]{DuKhLiTo22} that 
\begin{equation}
\label{interpospa}
d^{\widetilde{T}}_{\mathtt{Prok}} \big(  \ttm_{\mathtt{occ}}, \tfrac{_1}{^2}   \widetilde{\ttm}_{\mathtt{occ}} \big)  \leqo 2 \quad \textrm{and thus} \quad  \bdelta_{\mathtt{GHP}} \big( (R, d_{\mathtt{gr}}, o, \ttm_{\mathtt{occ}}) , (\widetilde{R}, d_{\widetilde{T}}, o,
 \frac{_{_1}}{^{^2}}\widetilde{\ttm}_{\mathtt{occ}}) \big) \leqo 3.
\end{equation}
We remind that $d^{\widetilde{T}}_{\mathtt{Haus}}$  and $d^{\widetilde{T}}_{\mathtt{Prok}}$ stand respectively 
for the Hausdorff distance on compact subsets of $\widetilde{T}$ and for the Prokhorov distance on the space of finite Borel measures on of $\widetilde{T}$.

We next introduce the \emph{range pseudometric} $d_{\widetilde{R}} (\cdot, \cdot ) \ino \MMM ([0, 2\#t])$ by setting for all $s,s'\ino [0, 2\#t]$  
\begin{equation}
\label{rangepsdd}
d_{\widetilde{R}} (s, s' )= d_{\widetilde{T}} \big( \widetilde{Y}_{\widetilde{v} (s) } , \widetilde{Y}_{\widetilde{v} (s') }  \big) = \bbn \widetilde{Y}_{\widetilde{v} (s) } \bbn +  \bbn \widetilde{Y}_{\widetilde{v} (s') } \bbn -2 \bbn  \widetilde{Y}_{\widetilde{v} (s)} \wedge \widetilde{Y}_{\widetilde{v} (s')} \bbn .  
\end{equation}
We observe that $d_{\widetilde{R}}$ is a real tree pseudometric and that 
\begin{equation}
\label{ncdngRR}
\big( \widetilde{R}, d_{\widetilde{T}}, o, \widetilde{\ttm}_{\mathtt{occ}} \big)\quad  \textrm{is isometric to} \quad \big( E_{d_{\widetilde{R}}}, d_{\widetilde{R}}, \mathtt r_{d_{\widetilde{R}}}, \ttm_{d_{\widetilde{R}}}  \big)
\end{equation} 
where $E_{d_{\widetilde{R}}}$ is the real tree yielded by $ d_{\widetilde{R}}$ as explained in (\ref{quotient}) and (\ref{induleb}). 
\begin{remark}
\label{intrplcnt} The previous constructions on the real trees $\widetilde{T}$ and $\widetilde{t}$ only depend on the branching walk $\Theta\eqo  (t, \varnothing; \overline{Y}_u \eqo (T, Y_u), u\ino t) $ locally in a continuous way. Thus $\Theta \ino \overline{\bbT}^\bullet (\overline{\bbT}^\bullet) \! \mapsto \! d_{\widetilde{R}} (\cdot, \cdot) \ino \MMM([0, 2\# t])$ is continuous. \cq 
\end{remark}

We now compare $d_{\widetilde{R}} $ with the snake-pseudometric associated with the relative heights of the branching walk. 
We next denote by $(W_s (t; \cdot ))_{s\in [0, 2\# t]}$ the contour snake associated with the $\bbR$-valued branching walk $( \bbn Y_u \bbn )_{u\in t} $ (see Definition~\ref{brwlkdef}). We denote by $(\widehat{W}_s (t) )_{s\in [0, 2\# t]}$ its associated endpoint process 
and we recall that the corresponding lifetime process is the contour process $(C_s (t))_{s\in [0, 2\#]}$. We then easily check for all $s\ino [0, 2\# t]$ and $r\ino [0, C_s(t)]$ that 
\begin{equation}
\label{histosna}W_s(t;r) = \bbn \widetilde{Y}_{\sigma} \bbn \quad \textrm{where $\sigma\ino \lgeo \varnothing , \widetilde{v} (s)\rgeo_{\widetilde{t}}$ is such that $d_{\widetilde{t}} (\varnothing, \sigma)\eqo r$. } 
\end{equation}
In particular, note that $\widehat{W}_s (t)\eqo \bbn \widetilde{Y}_{\widetilde{v} (s)} \bbn$. 
We next recall from (\ref{snkmeff}) in Definition 
\ref{defsnkme} that $d_{C(t), W(t; \cdot)}$ stands for the snake-pseudometric associated with $(C(t),W (t ;\cdot))$. 
As a consequence of (\ref{histosna}), we see for all $s, s' \ino [0, 2\# t]$ that 
\begin{equation}
\label{Rdissnad}
d_{C(t), W(t; \cdot)} (s,s') =  \bbn \widetilde{Y}_{\widetilde{v} (s) } \bbn +  \bbn \widetilde{Y}_{\widetilde{v} (s') } \bbn -2 \min 
\big\{\,    \bbn \widetilde{Y}_{\sigma } \bbn  \,  ;  \sigma \ino \lgeo  \widetilde{v} (s),   \widetilde{v} (s')\rgeo_{\widetilde{t}}\,  \big\}.  
\end{equation} 
We then observe that 
\begin{equation}
\label{cntrlzsipp}
0\leq \tfrac{1}{2} \big( d_{C(t), W(t; \cdot)} (s,s') -d_{\widetilde{R}} (s,s') \big) = 
\bbn  \widetilde{Y}_{\widetilde{v} (s)} \wedge \widetilde{Y}_{\widetilde{v} (s')}  \bbn - \min 
\big\{   \bbn \widetilde{Y}_{\sigma } \bbn  \,   ; \sigma \ino \lgeo  \widetilde{v} (s),   \widetilde{v} (s')\rgeo_{\widetilde{t}} \big\}.  
\end{equation} 
When $Y$ is a BRW whose spatial motion is a biased RW, (\ref{cntrlzsipp}) implies that $d_{C(t), W(t; \cdot)}$ and $ d_{\widetilde{R}}$ are close because for the difference in (\ref{cntrlzsipp}) to be large, two independent and large parts of the BRW have to coincide. More precisely, we conclude this section by proving the following lemma that shows (for fixed times) that the difference in (\ref{cntrlzsipp}) is exponentially small 
when $Y$ is a (sub)critical biased BRW.

\begin{lemma}
\label{graphvssna} Let $\lambda\ino(1,\infty)$ and let $(T, o)\ino \overline{\bbT}^\bullet$ be such that $T\neq\{\varnothing\}$. Let $t\ino \bbT$ be finite. Let $\fTheta\eqo (t, \varnothing; \overline{Y}_{\! u}\eqo (T, Y_u), u\ino t )$ be a $t$-indexed $T$-valued $\lambda$-biased BRW whose law is $Q^{T,o}_{t, \varnothing}$ (Definition \ref{biasedbfRWdef}). If Assumption~(\ref{eschyp}) in Lemma~\ref{dlczidec} is verified, then for all $\xi \ino (1, \infty)$,
\begin{equation}
\label{nfdffrc}
\bP \Big( \!\!\! \!\!\! \max_{\quad s,s'\in [0, 2\# t]} \! \big| d_{C(t), W(t\, ;\,  \cdot)} (s,s') \! -\! d_{\widetilde{R}} (s,s') \big| \, >\, 2\xi   \Big) \leq \lambda^2 (\# t )^3 \, \lambda^{-\xi}.
\end{equation}
\end{lemma}
\noi
\textbf{Proof.} By (\ref{cntrlzsipp}) and the definition on $\widetilde{T}$ of relative heights and of most recent common ancestors, we see that if there are $s_1, s_2 \ino [0, 2\# t]$ such that 
$d_{C(t), W(t; \cdot)} (s_1,s_2) \! -\! d_{\widetilde{R}} (s_1,s_2) \geko 2\xi $, then there are $u^\prime, v^\prime\ino t$ such that 
$\bbn Y_{u^\prime} \wedge Y_{v^\prime}\bbn  \! -\! \min_{w\in\lgeo u^\prime,v^\prime\rgeo } \bbn Y_w \bbn \geko \xi  \!-\! 1\geko 0$. Then, let us choose $w^\prime \ino \lgeo u^\prime,v^\prime\rgeo$ such that $\bbn Y_{w^\prime}\bbn\eqo \min_{w\in\lgeo u^\prime,v^\prime\rgeo } \bbn Y_w \bbn$ and with the greatest height possible. Since $\min (\bbn Y_{\! u'}\bbn ,  \bbn Y_{\! v'}\bbn)\geqo \bbn  Y_{\! u'} \wedge Y_{\! v'}\bbn \geko \bbn Y_{\! w'}\bbn$, we see that $w'\! \notin \! \{ u',v'\}$. Therefore, 
without loss of generality, we can assume that $w^\prime\ino \lgeo u^\prime\wedge v^\prime,u^\prime\lgeo\, $ and we can find  $w\! \in \, \rgeo u^\prime\wedge v^\prime,u^\prime\rgeo$ such that $\overleftarrow{w}\eqo w^\prime$. By definition of $w^\prime$ and $w$, we observe that $\overleftarrow{Y_w}=Y_{w^\prime}$, that $Y_{u^\prime}\wedge Y_{v^\prime}\ino \lgeo Y_{w^\prime}, Y_{v^\prime}\rgeo$, and that $Y_{u^\prime}\wedge Y_{v^\prime}\ino \lgeo Y_{w}, Y_{u^\prime}\rgeo$. Thus, there are $u,v\ino t$ such that $v\wedge w\! \neq\! w$, $w \!\preceq\! u$ and $Y_u\eqo Y_v\eqo Y_{u^\prime}\wedge Y_{v^\prime}$. Consequently, (\ref{nfdffrc}) is derived from the following bound that holds for all $u, v,w\ino t$ with $v\wedge w$ and $w\! \neq\! w\preceq u$, and all $p\ino \bbN^*$,
\begin{equation}
\label{grasna}
\bP\Big(\, Y_u=Y_v\, ,\, \bbn Y_v\bbn-\bbn Y_w\bbn\geq p\, ,\, \overleftarrow{Y_{w}}\notin\{Y_a:a\ino\lgeo w,u\rgeo\}\, \Big)\leq \lambda^{-p} . 
\end{equation}

It remains to prove (\ref{grasna}). To that end, recall from Definition~\ref{biaRWtreedef} that $P_{T \!\!, \,  x}$ stands for the law of the $\lambda$-biased RW $\overline{X}_n \eqo (T, X_n)$, $n\ino \bbN$ on $T$ such that $X_0=x$. By reasoning conditionally given $(Y_w,Y_v)$, the very definition of BRWs implies that 
\begin{multline*}
\bP\Big(\, Y_u=Y_v\, ,\, \bbn Y_v\bbn-\bbn Y_w\bbn\geq p\, ,\, \overleftarrow{Y_{w}}\notin\{Y_a:a\ino\lgeo w,u\rgeo\}\, \Big)\\
=\bE\Big[\I{\bbn Y_v\bbn-\bbn Y_w\bbn\geq p}P_{T,Y_w}\big(X_{|u|-|w|}=Y_v\, , \, \overleftarrow{Y_w}\notin \{X_j:j\leq |u|-|w|\}\, \big|\, Y_w,Y_v\big)\Big].
\end{multline*}
For all $x,y\ino T\backslash\{\varnothing\}$, we set $f(x,y)\eqo P_{T \!\!, \,  x} \big(\mathtt H_{y} \! <\! \mathtt H_{\overleftarrow{x}}  \big)$, where $\mathtt{H}_y$ is the hitting time as in (\ref{ttitime}). Then the previous identity yields that
\begin{equation}
\label{translev}
\bP\Big( Y_u=Y_v\, ,\, \bbn Y_v\bbn-\bbn Y_w\bbn\geq p\, ,\, \overleftarrow{Y_{w}}\notin\{Y_a:a\ino\lgeo w,u\rgeo\} \Big)\ \leq\ \bE\Big[ \I{\bbn Y_v\bbn-\bbn Y_w\bbn\geq p}f(Y_w,Y_v) \Big]
\end{equation}
Thanks to Assumption~(\ref{eschyp}), we can use (\ref{lbdev}) to get $f(x,y)\leq \lambda^{\bbn x\bbn -\bbn y\bbn}$ for all $x,y\ino T\backslash\{\varnothing\}$. Together with (\ref{translev}), this immediately entails (\ref{grasna}), which completes the proof. \cqfd

\section{Limit of the $\tau_\infty$-indexed BRW and its snakes}
\label{tauinftysec}

\subsection{Harmonic coordinates and related estimates.}
\label{harmcoosec}
One key-ingredient of the proof in Peres \& Zeitouni \cite{PerZei08} of the quenched CLT for a critical biased RW 
$(X_n)_{n\in \bbN}$ in an invariant 
GW tree is the \emph{harmonic coordinates}, which are $\bbR$-valued 
measurable functions $(S_x)_{x \in \bT}$ of the environment $\bT$ such that on one hand $(S_{X_n})_{n\in \bbN}$ is a martingale to which conditional CLT applies, and on the other hand such that $S_{X_n} \sim \mathtt{Cst}\bbn X_n  \bbn$. This idea is actually adapted to BRWs to prove the main result of the paper. 
In this section, we recall from Peres \& Zeitouni \cite{PerZei08} 
the construction of the harmonic coordinates of (variants of) the invariant GW($\nu$)-tree $\bT$ (within our framework of ordered trees) and we prove several estimates that are used in our proofs. 
The following lemma summarizes definitions, notation, and the main properties of harmonic coordinates. Its 
proof is adapted from Peres \& Zeitouni~\cite{PerZei08} and Dembo \& Sun~\cite{DeSu12}. 
\begin{lemma}   
\label{harmcoo} Let $\nu$ be a supercritical offspring distribution whose mean is denoted by $\ttm_\nu$. 
 We assume that there exists $\ttb' \ino (1, \infty)$ such that $\mnu (\ttb' )\! := \!  \sum_{k\in \bbN} k^{\ttb '} \nu (k) \leko \infty$. 
For all rooted ordered tree $t\ino \bbT$ as in Definition \ref{Ulamtree}, we next set 
\begin{equation}
\label{Udefin}
\mathtt U (t)= \liminf_{n\to \infty} \;  \mnu^{-n} \, \# \big\{ u \ino t : |u|\eqo n\big\} \in [0, \infty] \; .
\end{equation}
Here, $(\bT, \boo)$ is either a GW($\nu$)-tree as in Definition~\ref{GWfordef} $(\textbf a)$ (in this case: $\boo\eqo \varnothing$), or an infinite GW($\nu$)-tree as in Definition~\ref{GWptdef} $(\textbf b)$, or an invariant GW($\nu$)-tree as in 
Definition~\ref{definvRW}. 
For all $x\ino \bT$, we 
recall from (\ref{curshitr}) that the subtree $\theta_x \bT $ stemming from $x$ belongs to $\bbT$ and we set $U_x = \mathtt U (\theta_x \bT)$. Then the following holds true. 
\begin{compactenum}

\smallskip

\item[$(i)$]  $\bP$-a.s.~ for all $x\ino \bT$, 
\end{compactenum}
\begin{equation}
\label{welldefcoo}
U_x \! = \! \lim_{n\to \infty} \;  \mnu^{-n} \, \# \big\{ y \ino \bT :  x\preceq y,\,\bbn y\bbn \eqo  \bbn x \bbn \! +\!  n \, \big\} \ino (0, \infty) \quad \textrm{and} \quad U_x = \frac{1}{\mnu}\!\!\!\! \!\! \!\! \!\!   \sum_{\qquad y\in \bT : \overleftarrow{y} = x} \!\! \!\! \!\! \!\!  \!\! U_y .
\end{equation}
\begin{compactenum}

\smallskip

\item[$(ii)$] Suppose that $\ttb\ino (1, \infty)$ is such that $\sum_{k\in \bbN} k^{\ttb} \nu (k) \leko \infty$. 
Then, 
$\bE [ U_\boo^\ttb]\leko \infty$ if $\bT$ is a GW($\nu$)-tree or an infinite  
GW($\nu$)-tree. If $\bT$ is 
an invariant GW($\nu$)-tree, then $\bE [ U_\boo^{\ttb-1}] \leko \infty$ if $\ttb \ino (2, \infty)$.

\item[$(iii)$] The harmonic coordinates are the $\bbR$-valued r.v.~$(S_x)_{x\in \bT}$ defined by the following  
\end{compactenum}
\begin{equation}
\label{defSx}
S_\boo \eqo 0 \quad \textrm{and} \quad \forall x \ino \bT\backslash \{ \varnothing\} , \quad S_x \! -\! S_{\overleftarrow{x}}\eqo U_x  ,  
\end{equation}
\begin{compactenum}
\item[] on the event where (\ref{welldefcoo}) holds true and by $S_x \eqo 0$, $x\ino \bT$, on the $\bP$-negligible complementary event. 
Let $\overline{X}_n \eqo (\bT, X_n)$, $n\ino \bbN$, be a $\mathtt m_\nu$-biased RW such that $X_0 \eqo \boo$. We denote by $\ccF_{\! n}$ the $\sigma$-field generated by $(\overline{X}_k)_{0\leq k\leq n}$. Then $(S_{X_n})_{n\in \bbN}$ is a $(\ccF_{\! n})_{n\in \bbN}$-martingale:
\end{compactenum}
\begin{equation}
\label{mgcoo}
\textrm{$\bP$-a.s.~for all $n\in \bbN^*$, } \quad \bE  \big[ S_{X_{n+1} } \, \big| \, \ccF_{\! n} \big] = S_{X_{n}} 
\end{equation}

\item[$(iv)$] We keep the previous notation and for all $n\ino \bbN$, we set $\Delta_n\eqo  S_{X_{n+1} } \! -\!  S_{X_{n} }$ and 
$\Phi_\ttb (\overline{X}_n) \eqo  \bE  \big[ |\Delta_{n}|^\ttb \big|  \ccF_{\! n}\big]$.  Then, 
\begin{equation}
\label{sigdeltamg}
\Phi_\ttb (\overline{X}_n)  \leq \frac{\ttm_\nu^\ttb +\ttm_\nu}{\ttm_\nu \! +\!  k_{X_n} \! (\bT)} U^\ttb_{X_n}  \quad \textrm{and} \quad \Phi_\ttb (\overline{X}_n) \eqo \Phi_\ttb \big( \mathtt{cent} ( \overline{X}_n)\big) \; .
\end{equation}
If $\ttb \ino (1, \infty)$ is such that $\sum_{k\in \bbN} k^{\ttb} \nu (k) \leko \infty$ and if $\bT$ is an invariant GW($\nu$)-tree, then $\bE[\Phi_\ttb (\overline{X}_n) ]\eqo \bE [ \Phi_\ttb (\overline{X}_0) ] \leko \infty$.  
In the specific case where $\ttb\eqo 2$, then
\begin{equation}
\label{vutchi}
\bE \big[ \Phi_2  ( \overline{X}_n)\big]\eqo \frac{\ttm_\nu(2)\! -\! \ttm_\nu}{\ttm_\nu^2\! -\! \ttm_\nu} = : \sigma_\nu^2.  
\end{equation}
\end{lemma}
\noi
\textbf{Proof.} We prove $(i)$. To that end, let us first consider the case where $\bT$ is GW($\nu$)-tree as in Definition \ref{GWfordef} $(\textbf a)$. By standard arguments about martingales, a $\ttb'$-moments for $\nu$ with $\ttb'\ino (1, \infty)$ implies 
that $\bP$-a.s.~$\mathtt U(\bT)\eqo \lim_{n\to \infty} \mnu^{-n} \# \{ x\ino \bT: |x|\eqo n\} \ino (0, \infty)$. Next observe that for all $x\ino \bbU$ such that $\bP (x\ino \bT) \geko 0$, the subtree $\theta_x \bT$ under $\bP (\, \cdot \, | \, x\ino \bT)$ has the same law as $\bT$. Since $\bbU$ is countable, $\bP$-a.s.~(\ref{welldefcoo}) holds true for all 
$x\ino \bT$. 
 
Suppose next that $(\bT, \boo)$ is an infinite GW($\nu$)-tree or an invariant GW($\nu$)-tree. 
In these cases, conditionally given $\partial \Loo$, the subtrees $\theta_y \bT$, $y\ino \bT \backslash \Loo$, are independent GW($\nu$)-trees. Therefore $\bP$-a.s.~(\ref{welldefcoo}) holds true for all $x\ino \bT \backslash \Loo$. Now  
note that (\ref{welldefcoo}) provides a unambiguous definition of the weight $U_x$ for $x\ino \Loo$, which completes the proof of $(i)$.

Let us prove $(ii)$. We first suppose that $\bT$ is a GW($\nu$)-tree or an infinite GW($\nu$)-tree: in both cases 
$\theta_\boo \bT$ is GW($\nu$)-tree and as proved in Dembo \& Sun \cite[Lemma 6.2]{DeSu12}, 
$\bE [U_\boo^\ttb ] \leko \infty$, which implies the first part of $(ii)$. 

Next, assume that $\bT$ is an invariant GW($\nu$)-tree and that $\ttb \ino (2, \infty)$. By convexity of $x\mapsto x^{\ttb -1}$ and the second equality in (\ref{welldefcoo}), $U_\boo^{\ttb-1}\leqo \mathtt{m}^{_{-(\ttb-1)}}_{^\nu} k_\boo (\bT)^{\ttb-2} \sum_{x\in \bT: \overleftarrow{x}= \boo} U_x^{\ttb-1}$. By (\ref{InvGWnu}) in the definition of invariant GW($\nu$)-trees, we easily get $\bE [U^{\ttb-1}_\boo| k_\boo (\bT)]\leqo \mathtt{m}^{_{-(\ttb-1)}}_{^\nu} \bE \big[ \mathtt U (\bT_0)^{\ttb-1}\big] \,  k_\boo (\bT)^{\ttb-1}$, where $\bT_0$ is a GW($\nu$)-tree. Thus 
$\bE [U^{\ttb -1}_{\! \boo} ] \leqo \tfrac{1}{2}\mnu^{-\ttb} \bE [\mathtt U (\bT_0)^{\ttb-1}] $ $\sum_{k\in \bbN} (\mnu + k) k^{\ttb-1} \nu(k) \leko \infty$, which completes the proof of $(ii)$.

To prove $(iii)$, we observe that 
\begin{eqnarray*}
 \bE \big[ S_{X_{n+1}}\!\!\!   -\!  S_{X_n}  \big|  \ccF_{\! n} \big]\!\!\! \! &=& \!\!\!\!  (S_{\overleftarrow{X}_{n}} \!\! \! -\! S_{X_n}) \bE \big[\un_{\{ X_{n+1}= \overleftarrow{X}_{n} \}} \big|  \ccF_{\! n} \big] + \!\!\!\!\!\!\!\!\!\!\!\!\!\!  \sum_{\qquad y\in \bT : \overleftarrow{y}= X_n}\!\!\!\!\!\!\!\!\!\!\!\!\!
 (S_{y}  \! -\! S_{X_n}) \bE \big[\un_{\{ X_{n+1}=y\}} \big|  \ccF_{\! n} \big] \\
 &= & -\frac{\mnu U_{X_n} }{\mnu + k_{X_n} (\bT)}+ \!\!\!\!\!\!\!\!\!\!\!  \sum_{\qquad y\in \bT : \overleftarrow{y}= X_n}\!\!\!\!\!\!\!\! \!\! \frac{U_y}{\mnu + k_{X_n} (\bT)} = 0
\end{eqnarray*}  
by the second equality in (\ref{welldefcoo}). This completes the proof of $(iii)$.  

Let us prove $(iv)$. We first observe that 
\begin{eqnarray*}
\Phi_\ttb (\overline{X}_n) &= &\frac{1}{\ttm_\nu \! +\!  k_{X_n} \! (\bT)}\Big( \ttm_\nu U_{X_n}^\ttb+ \sum_{y\in \bT} 
U_y^\ttb\un_{\{ \overleftarrow{y}= X_n \}} \Big) \\
&\leq & \frac{1}{\ttm_\nu \! +\!  k_{X_n} \! (\bT)}  \Big( \ttm_\nu U_{X_n}^\ttb + \Big( \sum_{y\in \bT} 
U_y\un_{\{ \overleftarrow{y}= X_n \} } \Big)^\ttb \Big) =  \frac{\ttm_\nu^\ttb +\ttm_\nu}{\ttm_\nu \! +\!  k_{X_n} \! (\bT)} U^\ttb_{X_n}
\end{eqnarray*}  
since $s^\ttb + r^\ttb \leq (s+r)^\ttb$ for all $s, r \ino \bbR_+$ and by the second equality in (\ref{welldefcoo}). This proves the inequality in (\ref{sigdeltamg}). 
Then, recall from (\ref{shiftdef}) the definition of the shifts $\varphi_l$, $l\ino \bbZ$, and $\mathtt{cent} (\cdot)$ from Definition \ref{succdef} and observe that for all $x\ino \bT$, $\mathtt U (\theta_x \bT)\eqo \mathtt U (\theta_{\varphi_l (x)} \varphi_{l} (\bT))$ and that $k_x(\bT)\eqo k_{\varphi_l (x)  } (\varphi_l (\bT))$, which easily implies the equality in (\ref{sigdeltamg}).

Next we suppose that $\bT$ is invariant. By Proposition \ref{rwergo}, the law of $\overline{X}_n$ is then invariant and 
$\bE[\Phi_\ttb (\overline{X}_n) ]\eqo \bE [ \Phi_\ttb (\overline{X}_0) ]$. By (\ref{sigdeltamg}), we then get 
$$\bE \big[  \Phi_\ttb (\overline{X}_0) \big] \leq  \tfrac{1}{2} (\ttm_\nu^{\ttb-1} \! +1)\bE \big[ \mathtt U(\bT_0)^\ttb \big] <\infty$$
where $\bT_0$ a GW($\nu$)-tree, by (\ref{InvGWnu}) in the definition of invariant GW($\nu$)-trees. 
In the special case $\ttb\eqo 2$, we then get 
\begin{eqnarray*} 
\bE\big[ \Phi_2 (\bT,\booo) \big] &\eqo & 
\tfrac{1}{2\ttm_\nu} \bE \Big[ \ttm_\nu \mathtt U (\bT_0)^2 +\!  \sum_{u\in \bT_0} \! \mathtt U (\theta_u \bT_0)^2\un_{\{ \overleftarrow{u} = \varnothing\} }  \Big] \\
&\eqo & \bE \big[ \mathtt U (\bT_0)^2\big] = \frac{\ttm_\nu(2)\! -\! \ttm_\nu}{\ttm_\nu^2\! -\! \ttm_\nu}, 
\end{eqnarray*} 
$$\bE\big[ \Phi_2 (\bT,\booo) \big] \eqo 
\tfrac{1}{2\ttm_\nu} \bE \Big[ \ttm_\nu \mathtt U (\bT_0)^2 +\!  \sum_{u\in \bT_0} \! \mathtt U (\theta_u \bT_0)^2\un_{\{ \overleftarrow{u} = \varnothing\} }  \Big] 
\eqo  \bE \big[ \mathtt U (\bT_0)^2\big] \eqo \frac{\ttm_\nu(2)\! -\! \ttm_\nu}{\ttm_\nu^2\! -\! \ttm_\nu}, $$
the last equality being a consequence of the explicit computation of the variance of GW processes: see e.g.~Athreya \& Ney \cite[Chapter I, (2) p.~4]{AtNe72}.  \cqfd 

\medskip

We next provides comparison bounds of harmonic coordinates and relative heights of vertices.  
\begin{lemma}
\label{harmcoo4} 
Let $\nu$ be a supercritical offspring distribution whose mean is denoted by $\ttm_\nu$. 
We assume that $\ttm_\nu (2)\eqo \sum_{k\in \bbN} k^{2} \nu(k) \leko \infty$. 
Let $(\bT, \boo)$ be an infinite tree GW($\nu$)-tree as in Definition \ref{GWptdef} 
or an invariant GW($\nu$)-tree as in Definition \ref{definvRW}. Let the r.v.s $U_x, S_x$, $x\ino \bT$, be as in (\ref{welldefcoo}) and  (\ref{defSx}). Recall from (\ref{spinedfbis}) the notation $\partial \Loo$ and $\boo (p)$, $p\ino \bbN$. 
Then the following holds true.

\smallskip

\begin{compactenum}
\item[$(i)$] For all $p\ino \bbN^*$ we set $\fbeta_p \eqo \mnu^{-1}\sum_{x\in \partial \Loo: \overleftarrow{x}= \boo (p)} U_x$. 
Then $\fbeta_p \eqo  U_{\boo (p)} \! - \! \mnu^{-1}U_{\boo (p-1)}$. Moreover, the r.v.s $(\fbeta_p)_{p\in \bbN^*}$ are 
i.i.d.~such that $\bE [\fbeta_1]\eqo \mnu^{-2} (\mnu (2) \! -\! \mnu)$. If $\ttb\ino (1, \infty)$ is such that 
$\ttm_{\nu} (\ttb +1)\! :=\! \sum_{k\in \bbN} k^{\ttb+1} \nu(k) \leko \infty$, then $\bE [\fbeta_1^\ttb] \leko \infty$.    

\smallskip

\item[$(ii)$] Recall $\sigma_\nu$ from (\ref{vutchi}) and assume that there exists $\ttb\ino (1, \infty)$ such that 
$\ttm_\nu (2\ttb+1)\eqo $ $\sum_{k\in \bbN} $ $k^{2\ttb+1}\nu(k)$ $\leko \infty$. Then, there is a constant $c\ino (0, \infty)$ that only depends on $\nu$ and $\ttb$ such that for all $p\ino \bbN$, 
$\bE \big[ \big| \tfrac{1}{\sigma^2_\nu} S_{\boo (p)}\!  -\! \bbn \boo (p) \bbn  \big|^{2\ttb}\big] \leqo 
c\, p^\ttb $. 
Moreover, for all $\beta\ino (\tfrac{1}{2} +\tfrac{1}{2\ttb} , 1)$, we set 
\begin{equation}
\label{LoRatiodef}
R_\beta= \sup_{p\in \bbN} \, \frac{\big| \tfrac{1}{\sigma^2_\nu} S_{\boo (p)} \!-\! \bbn \boo (p) \bbn  \big|}{1+ p^\beta} .
\end{equation} 
Then, there exists a constant $c'\ino (0, \infty)$ that only depends on $\nu$, $\ttb$ and $\beta$ such that for all $\lambda \ino (0, \infty)$, we get $\bP \big( R_\beta \geq \lambda \big)\leqo c' \lambda^{-2\ttb}$. 
\end{compactenum}
\end{lemma} 
\noi
\textbf{Proof.} 
The first point of $(i)$ is a direct consequence of the definition of invariant GW trees and of (\ref{welldefcoo}). 
Then, observe that $\bE [\fbeta_1| k_{\boo (1)} (\bT)] \eqo \ttm_\nu^{-1}\bE [ \mathtt U (\bT') ] (k_{\boo (1)} (\bT)\! -\! 1) \eqo  \ttm_\nu^{-1} (k_{\boo (1)} (\bT)\! -\! 1) $, where $\bT'$ stands for a GW($\nu$)-tree. Thus, $ \bE [\fbeta_1]\eqo    \ttm_\nu^{-2} \sum_{k\in \bbN} k(k\! -\! 1) \nu(k)\eqo \mnu^{-2} (\mnu (2) \! -\! \mnu)$. 
More generally, note that 
$\fbeta_1^\ttb \leqo  \ttm_\nu^{-\ttb}  (k_{\boo (1)} (\bT)\! -\! 1)^{\ttb -1} \sum_{x\in \partial \Loo: \overleftarrow{x}= \boo (1)} U^\ttb_x $ since $z\mapsto z^\ttb$ is convex. Then recall from Lemma \ref{harmcoo} $(ii)$ that $\bE [ \mathtt U^\ttb  (\bT') ]\leko \infty$ if $\ttm_\nu (\ttb) \leko \infty$. 
Thus $\bE [\fbeta_1^\ttb | k_{\boo (1)} (\bT)] \leqo \ttm_\nu^{-\ttb}\bE [ \mathtt U^\ttb (\bT') ] (k_{\boo (1)} (\bT)\! -\! 1)^\ttb$, which implies that $\bE [\fbeta_1^\ttb] \leqo \ttm_\nu^{-\ttb-1}\bE [ \mathtt U^\ttb (\bT') ] \ttm_{\nu} (\ttb+1)\leko \infty$. 

\medskip

Let us prove $(ii)$. As an immediate consequence of $(i)$, for all $p\ino \bbN^*$ we get  
\begin{equation}
\label{geosum}
 U_{\boo (p)} = \mnu^{-p} U_{\boo} +\!\!  \sum_{1\leq k \leq p} \! \mnu^{-(p-k)} \fbeta_k \; =  \mnu^{-p} U_{\boo} +(1\! -\! \mnu^{-1})^{-1} \bE \big[ \fbeta_{p-\mathtt G} \un_{\{ 0\leq \mathtt G < p\}} \big| (\bT, \boo) \big] 
\end{equation}
where $\mathtt G$ stands for a r.v.~that is independent from $(\bT, \boo)$ and such that $\bP (\mathtt G\eqo k)\eqo (1\! -\! \mnu^{-1}) \mnu^{-k}$, $k\ino \bbN$. Since $S_{\boo (p-1)} \! -\! S_{\boo (p)}\eqo U_{\boo (p-1)}$ (as a consequence of  (\ref{defSx})), we get by (\ref{geosum}) 
\begin{eqnarray*}
-S_{\boo (p)} \!\! \!\! &=& \!\! \!\! U_\boo + \!\!  \sum_{1\leq q < p} U_{\boo (q)} \, = \, \frac{1\! -\! \mnu^{-p}}{1\! -\! \mnu^{-1}} U_\boo \, + \!\!  \sum_{1\leq k < p} \!\!  \fbeta_k\!\!   \sum_{k\leq q<p} \mnu^{-(q-k)} \\
\!\! \!\!  & =&\!\! \!\!  \frac{1\! -\! \mnu^{-p}}{1\! -\! \mnu^{-1}} U_\boo\,  +\!\!  \sum_{1\leq k < p} \!\!  \fbeta_k \frac{1\! -\! \mnu^{-(p-k)}}{1\! -\! \mnu^{-1}} 
\, =   \frac{1\! -\! \mnu^{-p}}{1\! -\! \mnu^{-1}} U_\boo  + \!\!  \sum_{1\leq k < p}\!\!   \frac{\fbeta_k}{1\! -\! \mnu^{-1}} -\!\!   \sum_{1\leq k < p}  \!\!  \frac{ \mnu^{-(p-k)}\fbeta_k }{1\! -\! \mnu^{-1}} .
\end{eqnarray*}
To simplify notation, we set $ a\eqo \bE [\fbeta_1] $ and $\overline{\fbeta}_k\eqo \fbeta_k - a$, that is a centered r.v.~and we get 
\begin{eqnarray*} 
-(1\! -\! \mnu^{-1}) S_{\boo (p)} - ap  \!\! \!\! &= & \!\! \!\!  \Sigma_{p-1}+ R_p(1)-R_p(2)  \quad \textrm{where} \\ 
\Sigma_{p-1}  \!\! \!\! &=&  \!\! \!\! \sum_{1\leq k \leq p-1} \!\!  \overline{\fbeta}_k \, , \qquad R_p(1)= (1\! -\! \mnu^{-p})U_\boo \! -\! a \quad \textrm{and}\\
R_p(2)  \!\! \!\!&=& \!\! \!\! \sum_{1\leq k < p}  \!\! \mnu^{-(p-k)}\fbeta_k= (1\! -\! \mnu^{-1})^{-1} \bE \big[ \fbeta_{p-\mathtt G} \un_{\{ 0< \mathtt G < p \}} \big|  (\bT, \boo) \big]  . 
\end{eqnarray*}
Note from $(i)$ and (\ref{vutchi}) that $a^{-1}(1\! -\! \mnu^{-1})\eqo \sigma_\nu^{-2}$. Since $z\mapsto z^{2\ttb}$ is convex, we get 
$$ a^{2\ttb}\bE \Big[  \big| \tfrac{1}{\sigma^2_\nu} S_{\boo (p)}  \! + \!  p \, \big|^{2\ttb}\Big] \leq 3^{2\ttb-1} 
\big( \bE \big[ |R_p(1)|^{2\ttb} \big] +\bE \big[ |R_p(2)|^{2\ttb} \big]   + \bE \big[ |\Sigma_{p-1}|^{2\ttb} \big]  \big) \; .$$
First observe that $\bE [|R_p(1)|^{2\ttb} ] \leq 2^{2\ttb-1}  (\bE [U_\boo^{2\ttb} ] + a^{2\ttb} )$, which finite by Lemma \ref{harmcoo} $(ii)$. 

To bound $\bE \big[ |R_p(2)|^{2\ttb} \big] $, we use conditional Jensen inequality to get 
$$|R_p(2)|^{2\ttb} \leqo (1\! -\! \mnu^{-1})^{-2\ttb} \bE \big[ \fbeta^{2\ttb}_{p-\mathtt G} \un_{\{ 0< \mathtt G < p \}} \big| (\bT, \boo) \big]\; , $$
which entails $\bE \big[ |R_p(2)|^{2\ttb}\big] \leqo (1\! -\! \mnu^{-1})^{-2\ttb}  \bE \big[ \fbeta^{2\ttb}_{1} \big]$, which is finite by $(i)$. 

Let us bound $\bE \big[ |\Sigma_{p-1}|^{2\ttb} \big] $. 
Since the $\overline{\fbeta}_k$ are i.i.d.~centered and such that $\bE \big[ |\overline{\fbeta}_k|^{2\ttb} ] \leko \infty$ by $(i)$, we can apply Burkholder-Davis-Gundy inequality (see see e.g.~Burkholder \cite[Theorem 9 p.~1502]{Bu66}) to get a constant $c_0\ino (0, \infty)$ that only depends on $\ttb$ such that 
$$ \bE \big[ |\Sigma_{p-1}|^{2\ttb} \big] \leq c_0 \, \bE \Big[ \Big(\!\!\!\!\!\!  \sum_{\quad 1\leq k<p} \!\!\!\!\!  \overline{\fbeta}_k^{2} \Big)^{\ttb}\Big] . $$
Since $z\mapsto z^\ttb$ is convex, 
$$ \bE \Big[ \Big(\!\!\!\!\!\!  \sum_{\quad 1\leq k<p} \!\!\!\!\! \overline{\fbeta}_k^{2} \Big)^{\ttb}\Big] \leq (p\! -\! 1)^{\ttb-1} \!\!\! \sum_{1\leq k<p} \!\!\! \bE \big[ \,  |\overline{\fbeta}_k|^{2\ttb}\big]= (p\! -\! 1)^{\ttb} p \bE \big[ \,  |\overline{\fbeta}_1|^{2\ttb}\big], $$
which proves the first inequality in $(ii)$. To complete the proof of $(ii)$ we simply observe that 
$$  \bP \big( R_\beta \geq \lambda \big)\leqo \sum_{p\in \bbN^*} \bP \Big( 
\big|  \tfrac{1}{\sigma^2_\nu} S_{\boo (p)} \!-\! \bbn \boo (p) \bbn\,  \big| \geqo \lambda( 1+ p^\beta) \Big) \leq   \sum_{p\in \bbN^*} 
\frac{cp^{\ttb}}{\lambda^{2\ttb} (1+ p^\beta)^{2\ttb}} < c' \lambda^{-2\ttb}, $$
by a Markov inequality and since $(2\beta -\! 1) \ttb \geko 1$. \cqfd

\medskip

The following lemma shows that BRWs do not visit `bad' vertices, namely vertices $x$ 
such that $\sigma_\nu^{-2} \times $ (the harmonic coordinate at $x$) is large compared to the relative height $\bbn x\bbn$ of $x$. 

\begin{lemma} 
\label{badharm}  Let $\tau$ be a GW($\mu$)-tree, where $\mu$ is a non-trivial critical offspring distribution.  
Let $\nu$ be a supercritical offspring distribution whose mean is denoted by $\ttm_\nu$. 
Let $\ttb\ino (1, \infty)$. We assume that $\sum_{k\in \bbN} k^{2\ttb +1}\nu(k) \leko \infty$. Let $(\bT, \boo)$ be an invariant 
GW($\nu$)-tree. Let $(Y'_u)_{u\in \tau}$ be a $\tau$-indexed, $\bT$-valued $\ttm_\nu$-biased BRW whose conditional law given $(\bT, \boo)$ and $\tau$ is $Q^{_{(\bT, \boo)}}_{^{(\tau, \varnothing)}}$ 
as in Definition \ref{biasedbfRWdef}. 
We recall from (\ref{spinedfbis}) 
the pieces of notation $\boo (p)$ and $L_\boo$. Let 
$ (S_x)_{x\in \bT}$ be the harmonic coordinates as in (\ref{defSx}) and we recall from (\ref{vutchi}) the notation $\sigma_\nu$. For all $q\ino \bbN$, all $\lambda \ino (0, \infty)$ and all $\beta \ino 
(\tfrac{1}{2} +\tfrac{1}{2\ttb} , 1)$, we define the following subset of $\bT$ 
$$ B_{q, \lambda, \beta}= \Big\{ x\ino \bT\backslash L_\boo\,  : \bbn \boo \wedge x\bbn \geqo -q \;\,  \textrm{and}  \;\,  \big| 
\frac{_{_1}}{^{^{\sigma^2_\nu}}} (S_x\! -\! S_{\boo \wedge x}) \! -\!  \big( \bbn x\bbn\! -\!  \bbn \boo \wedge x\bbn\big) \big| \geqo \lambda \big( 1+  \big( \bbn x\bbn\! -\!  \bbn \boo \wedge x\bbn\big)^\beta \big) \Big\}. $$
Then, there exists $c\ino (0,\infty)$, which only depends on $\nu, \ttb, \beta$, such that 
$\bP \big( \exists u\ino \tau: Y'_u \ino B_{q, \lambda, \beta}) \leqo cq \lambda^{-2\ttb}.$ 
\end{lemma}
\noi
\textbf{Proof.} Let $(\bT_0,\varnothing)$ be a GW($\nu$)-tree (Definition \ref{GWfordef} $(\textbf a)$). Denote by $(S^0_x)_{x\in \bT_0}$ its harmonic coordinates. We first bound the expectation of 
$$ Z(\bT_0)= \sum_{x\in \bT_0} \tfrac{\mnu + k_x (\bT_0) }{\mnu  - 1} \,   \mnu^{-|x|} \un_{ 
\! \big\{  \big|\sigma_\nu^{-2} S^0_x  - |x|  \big| \geq \lambda (1+|x|^\beta) \big\} } .$$
Let $(\bT', \boo')$ be an infinite GW($\nu$)-tree (Definition \ref{GWptdef}) whose harmonic coordinates  are denoted by  
$(S'_x)_{x\in \bT'}$. We recall that $S^0_\varnothing \eqo0$ and that for all $x\ino \bT_0$, $S^0_x\eqo \sum_{y\in \, \rgeo \varnothing , x \rgeo } \mathtt U(\theta_y \bT_0)$. 
Then, the many-to-one formula (\ref{sizebiased}) implies 
\begin{eqnarray}
\label{Singlesti}
 \bE \big[ Z(\bT_0) \big] &=& \sum_{p\in \bbN^*} 
 \bE \Big[ \tfrac{\mnu + k_\boo (\bT') }{\mnu  - 1} \,  \un_{ \! \big\{  \big|\sigma_\nu^{-2} S'_{\boo'(p)}  - \bbn \boo'(p) \bbn  \big| \geq \lambda (1+|p|^\beta) \big\} } \Big] \nonumber \\
\!\! \!\! \!  & = & \!\! \!\! \!  \tfrac{2\mnu}{\mnu-1} \!\!  \sum_{p\in \bbN^*} \bP \big(  \big|\sigma_\nu^{-2} S_{\boo(p)}  - \bbn \boo(p) \bbn  \big| \geq \lambda (1+|p|^\beta)  \big) \leq  
\sum_{p\in \bbN^*} 
\frac{ cp^{\ttb}}{\lambda^{2\ttb} (1+ p^\beta)^{2\ttb}} <  c'\lambda^{-2\ttb},
\end{eqnarray}
by the absolute continuity relation (\ref{InvGWnu}) between the laws of $\bT$ and $\bT'$, by Lemma \ref{harmcoo4} $(ii)$ and since $2\beta (\ttb\! -\! 1) \geko 1$. 

We now complete the proof of the lemma. To that end, we denote by $\overline{X}_n\eqo (\bT, X_n)$, $n\ino \bbN$, a $\ttm_\nu$-biased RW on $\bT$ starting at $X_0\eqo \boo$. By definition of BRW, we first get 
$$ \bP \big( \exists u\ino \tau: Y'_u \ino B_{q, \lambda, \beta} \big| \, (\bT, \boo), \tau \big)\leq \sum_{u\in \tau}\bE \big[ \un_{\{ Y'_u \in B_{q, \lambda, \beta} \}} \big| \,  (\bT, \boo), \tau\big] = \sum_{u\in \tau} P_{\bT, \boo} (X_{|u|} \ino B_{q, \lambda, \beta} \big)\; .$$
The many-to-one formula (\ref{sizebiased}) is applied to $\tau$ whose mean number of children is one and we get 
\begin{eqnarray*}
\bE \big[ \sum_{u\in \tau} P_{\bT, \boo} (X_{|u|} \ino B_{q, \lambda, \beta} \big)\big| (\bT, \boo) \Big] &=&  \sum_{n\in \bbN} P_{\bT, \boo} (X_n  \ino B_{q, \lambda, \beta} \big) =\sum_{x\in B_{q, \lambda, \beta} } G(\boo, x) \\
& =& \tfrac{1}{\mnu}\sum_{y\in  \partial L_\boo } \un_{\{ \bbn\boo \wedge y \bbn \geq -q  \}} Z (\theta_y \bT) .  
\end{eqnarray*}
Here, we used the fact that the Green function $G(\boo, x)$ is equal to $\tfrac{\mnu + k_x (\bT) }{\mnu  - 1} \,   \mnu^{-(\bbn x\bbn-\bbn \boo \wedge x\bbn) }$ by (\ref{greenoj}). Recall from the definition of invariant and infinite GW trees that conditionally on $\partial L_\boo$, the trees $(\theta_y \bT)_{y\in \partial L_\boo}$ are i.i.d.~GW($\nu$)-trees. Therefore, 
$$  \bP \big( \exists u\ino \tau: Y'_u \ino B_{q, \lambda, \beta} \big| \, \partial L_\boo \big)\leq 
\tfrac{1}{\mnu}\sum_{y\in  \partial L_\boo } \un_{\{ -\bbn\boo \wedge y \bbn \leq q  \}} \bE [ Z (\bT_0) ] \, \leq  \, 
\tfrac{1}{\mnu}\bE [ Z (\bT_0) ] \sum_{0\leq j\leq q}  k_{\boo (j)} (\bT) $$
which implies the desired result by (\ref{Singlesti}) and since $\bE[k_{\boo (j)} (\bT)]\eqo \ttm_\nu^{-1} \sum_{k\in \bbN} k^2 \nu(k)\leko \infty$ if $j\geqo 1$ and $\bE[ k_{\boo } (\bT)]\eqo (2\ttm_\nu)^{-1} \sum_{k\in \bbN} k(k+\ttm_\nu)\nu(k) \leko \infty$. \cqfd

\medskip

We next get upper bound of the (quenched moments of the) increments of the martingale $(S_{X_n})_{n\in \bbN}$. 
\begin{lemma}
\label{quoscillW} Let $\nu$ be a supercritical offspring distribution whose mean is denoted by $\ttm_\nu$. 
 Let $\ttb_0\ino (1, \infty)$. We assume that $\sum_{k\in \bbN} k^{2\ttb_0 +1}\nu(k) \leko \infty$. Let $(\bT, \boo)$ be an invariant GW($\nu$)-tree as in Definition \ref{definvRW}. 
Let the r.v.s $U_x$, $x\ino \bT$, be as in (\ref{welldefcoo}). Let $\overline{X}_n \eqo (\bT, X_n)$, $n\ino \bbN$, be a $\mathtt m_\nu$-biased RW on $\bT$ such that $X_0 \eqo \boo$, and recall that $P_{(\bT, \boo)}$ stands for the conditional law of $(X_n)_{n\in \bbN}$ given $(\bT, \boo)$. 
Then, for all $\ttb \ino (1, \ttb_0)$, there exists a $\bP$-a.s.~finite r.v. $K_\ttb (\bT)$ that is a measurable function of $\bT$ and $\ttb$ such that for all $m\ino \bbN$ and all $n\ino \bbN^*$,
\begin{equation}
\label{Burkholder}
\textrm{$\bP$-a.s.~}\qquad E_{\bT, \boo} \big[ |S_{X_{m+n}} \! \! -\! S_{X_m} |^{2\ttb}  \big] \leq K_\ttb (\bT)\,  n^{\ttb-1} (m+n) .
\end{equation} 
\end{lemma}
\noi
\textbf{Proof.} Recall the notation $\Delta_{k} \eqo S_{X_{k+1}} \! \! -\! S_{X_{k}}$ from Lemma~\ref{harmcoo} $(iv)$. Since the $\Delta_k$ are increments of martingales, the Burkholder-Davis-Gundy inequality 
(see e.g.~Burkholder \cite[Theorem 9 p.~1502]{Bu66}) implies that there is a constant $c(\ttb)\ino (0, \infty)$ that only depends on $\ttb$ such that 
$$ E_{\bT, \boo} \Big[ \max_{1\leq k\leq n} \big| S_{X_{m+k}} \! \! -\! S_{X_m}\big|^{2\ttb}\Big] \leq 
c(\ttb) \, E_{\bT, \boo} \Big[ \Big(  \!\!\! \sum_{\; \, 0\leq k< n} \!\!\! \Delta_{m+k}^2 \Big)^{\! \ttb}  \Big] .$$
Since $b\ino (1, \infty)$, $z\mapsto z^b$ is convex and we get 
$( \frac{1}{n}\sum_{\; 0\leq k< n} \Delta_{m+k}^2)^{ \ttb} \leq  \frac{1}{n}\sum_{\; 0\leq k< n} |\Delta_{m+k}|^{2\ttb} $.  
Then by (\ref{sigdeltamg}) in Lemma~\ref{harmcoo}, we get
\begin{align*} 
E_{\bT, \boo} \Big[ \max_{1\leq k\leq n} \big| S_{X_{m+k}} \! \! -\! S_{X_m}\big|^{2\ttb}\Big]&\leq 
 c(\ttb)  n^{\ttb -1} 
\!\! \sum_{\; \, 0\leq k< n} E_{\bT, \boo} \big[ |\Delta_{m+k}|^{2\ttb} \Big]\\
& \leq  c(\ttb)  n^{\ttb -1} \!\! 
\sum_{\; \, 0\leq k< n} E_{\bT, \boo }\Big[  \tfrac{\ttm_\nu^{2\ttb} +\ttm_\nu}{\ttm_\nu  +  k_{X_{m+k}} \! (\bT)} U^{2\ttb}_{X_{k+m}} \Big] \\
& \leq  
c(\ttb) (\ttm_\nu^{2\ttb} \! +\! \ttm_\nu) \, n^{\ttb-1}  (m+n)\,  \Big( \tfrac{1}{m+n} \! \! \! \! \! \!  \sum_{\quad 0\leq k < m+n} \! \! \! \! \! \! E_{\bT, \boo} \big[ F(\bT, X_k)  \big]  \Big)
\end{align*}
where we have set $F( T, x)\eqo \mathtt U (\theta_x T)^{2\ttb}  /  (\ttm_\nu  +  k_{x}  (T))$, for all $(T, x)\ino \overline{\bbT}^\bullet\! $ and where we recall the notation $\mathtt U (\cdot)$ from (\ref{Udefin}). We observe that $F(T,x)\eqo F(\mathtt{cent} (T,x))$. Then, we set 
$\ttb_1\eqo \ttb_0/\ttb\ino (1, \infty)$ and we observe that $\bE [F(\bT, \boo)^{\ttb_1} ] \leq \bE [U_\boo^{2\ttb_0}] \! <\! \infty$ since 
$ \sum_{k\in \bbN} k^{2\ttb_0+1} \nu(k) \leko \infty$ by Lemma \ref{harmcoo} $(ii)$. Therefore Corollary \ref{rwergobis} applies 
and asserts that 
$$\textrm{$\bP$-a.s.} \quad K_{\ttb} (\bT):= c(\ttb)\,  (\ttm_\nu^{2\ttb} \! +\! \ttm_\nu) \, \sup_{n\in \bbN^*}\,  \frac{1}{n} \! \sum_{0\leq k<n} \!\!  E_{\bT, \boo} \big[ F(\bT, X_k)  \big] 
\; < \infty, $$ 
which entails the desired result. \cqfd

\medskip

We next use a specific form of a conditional  CLT for martingales that is used to prove the finite-dimensional convergence of the harmonic coordinates of the BRW in Section~\ref{pfharmosec}. 
\begin{proposition}
\label{condiTCL} For all $n\ino \bbN$, let $(\cF_{n, k})_{k\in \bbN}$ be a filtration on $(\Omega, \ccF)$ and let $(D_{n, k})_{k\in \bbN^*}$ be a sequence of real valued r.v.s. Let $(s_n)_{n\in\bbN}$ be a sequence of integers and let $s\in\bbR_+$. We make the following assumptions. 
\begin{compactenum}
\item[$(a)$] For all $n,k$, $D_{n, k}$ is $\cF_{n,k}$-measurable, $\bE [ D_{n,k}^2] \leko \infty$ and $\bP$-a.s.~$\bE [D_{n, k} | \cF_{n, k-1}]\eqo 0$.
\item[$(b)$] 
There exists a filtration $(\mathcal G_j)_{j\in \bbN}$ on $(\Omega, \ccF)$ and a sequence of integers $(r_n)_{n\in \bbN}$ such that 
$\cF_{n, 0}\eqo \mathcal G_{r_n}$, $n\ino \bbN$. 
\item[$(c)$] For all $\epp \ino (0, 1)$, $\bP$-a.s.~$\lim_{n\to \infty} 
\sum_{1\leq k\leq s_n}\bE [D_{n,k}^2\un_{\{| D_{n, k}|>\epp  \}} | \cF_{n, k-1} ]= 0$. 
\item[$(d)$] $\bP$-a.s.~$\lim_{n\to \infty} \sum_{1\leq k\leq s_n}\bE [D_{n,k}^2 | \cF_{n, k-1} ]\eqo s$. 
\end{compactenum}

\smallskip
\noi
Then, we set $M_n\eqo \sum_{1\leq k\leq s_n}D_{n,k}$ and for all $\lambda \ino \bbR$, we get 
\begin{equation}
\label{fourcondi}
\textrm{$\bP$-a.s.~} \quad \lim_{n\to \infty}\bE \big[ e^{i\lambda M_n} \big| \cF_{n, 0} \big] = e^{-\frac{_1}{^2} s\lambda^2}\; .
\end{equation}    
\end{proposition}  
\noi
\textbf{Proof.} See Appendix~\ref{pfCLTmarti}. \cqfd

\begin{remark}
\label{liens}
In Assumption $(b)$, $\sigma$-fields are only partially nested 
(i.e., ~$\cF_{n, k}$ is not necessarily included in 
$\cF_{n+1, k}$ for all $k$). 
However, as mentioned in the remarks following Hall \& Heyde \cite[Corollary 3.1 pp.~58-59]{HaHe80}, 
since the limit in Assumption $(d)$ is deterministic, the conclusion of \cite[Corollary 3.1 pp.~58-59]{HaHe80} 
applies 
and $M_n$ converges in law to a centered Gaussian r.v.~with variance $s$. However, this sole convergence does not imply (\ref{fourcondi}). A related 
type of result can be found in Eagleson \cite[Corollary 2 p.~560]{Ea75}, which entails that $\bP$-a.s.~
$\lim_{n\to \infty}\bE \big[ e^{i\lambda M_n} \big| \mathcal G \big] \eqo e^{-\frac{1}{2} s\lambda^2}$ where $\mathcal G\eqo \eqo \bigcap_{n\in \bbN} \cF_{n, 0}$. However, it does not imply (\ref{fourcondi}) and the methods used in \cite{Ea75} seem difficult to adapt in order to obtain (\ref{fourcondi}). Since we have not found an accurate reference for Proposition \ref{condiTCL}, we provide in Appendix~\ref{pfCLTmarti} a proof which is a (relatively straightforward) adaptation of proof of the standard Lindeberg CLT for martingales. \cq      
\end{remark}

\begin{proposition}
\label{fourcondharmo}  Let $\nu$ be a supercritical offspring distribution whose mean is denoted by $\ttm_\nu$. 
 Let $\ttb\ino (2, \infty)$. We assume that $\sum_{k\in \bbN} k^{\ttb}\nu(k) \leko \infty$.
Let $(\bT, \boo)$ be an invariant GW($\nu$)-tree as in Definition \ref{definvRW}. 
Let the r.v.s $U_x, S_x$, $x\ino \bT$, be as in (\ref{welldefcoo}) and (\ref{defSx}). 
Let $\overline{X}_n \eqo (\bT, X_n)$, $n\ino \bbN$, be a $\mathtt m_\nu$-biased RW on $\bT$ such that $X_0 \eqo \boo$. 
We denote by $\ccF_n$ the $\sigma$-field generated by 
$(\overline{X}_k)_{0\leq k\leq n}$. Let $r_n , s_n\ino \bbN$, $n\ino \bbN$ be such that $\frac{1}{n}r_n \! \to \! r\ino \bbR_+ $ and $ \frac{1}{n}s_n \! \to \! s \ino \bbR_+$. Recall from (\ref{vutchi}) the definition of $\sigma_\nu$. 
Then for all $\lambda\ino \bbR$, 
\begin{equation}
\label{condiharmodiffu}
\textrm{$\bP$-a.s.~} \quad \lim_{n\to \infty}\bE \Big[ \exp \big( i\tfrac{\lambda}{\sqrt{n}} \big( S_{X_{r_n+s_n}} \! \! -\! S_{X_{r_n}}\big) \big) \Big| \ccF_{r_n} \Big]\eqo  e^{-\frac{_1}{^2} s\sigma^2_\nu \lambda^2}\; . 
\end{equation}
\end{proposition}
\noi
\textbf{Proof.} We want to apply Proposition \ref{condiTCL} with $\cF_{n, k}\eqo \ccF_{r_n +k}$ and $D_{n, k}\eqo n^{-1/2}\Delta_{r_n + k-1}$, where as in Lemma \ref{harmcoo} $(iv)$, we have set $\Delta_{m}\eqo S_{X_{m+1}}\! -\! S_{X_m}$. Since $\nu$ is square integrable, (\ref{vutchi}) in Lemma~\ref{harmcoo} $(iv)$ implies that $\bE [D_{n, k}^2] \eqo \tfrac{1}{n}\sigma_\nu^2 \leko \infty$ and Lemma \ref{harmcoo} $(iii)$ implies that Assumption $(a)$ of Proposition \ref{condiTCL} is satisfied. 
Next observe that $\cF_{n, 0}\eqo \ccF_{r_n}$:  Assumption $(b)$ of Proposition \ref{condiTCL} is satisfied.

We next prove that Assumption $(d)$ in Proposition \ref{condiTCL} is satisfied. 
Recall notation 
$\Phi_{\ttb} (\overline{X}_n)$ from Lemma~\ref{harmcoo} $(iv)$ and to simplify notation, we set $C_{\ttb, k} \eqo  \sum_{0\leq l<k} \Phi_{\ttb} (\overline{X}_l) $.  
First observe that 
$\tfrac{1}{n}\Phi_2 (\overline{X}_{r_n + k-1} )\eqo  \eqo \bE [D^2_{n, k}| \cF_{n, k-1}]$. 
Since $\Phi_2 (\overline{X}_{k} )\eqo \Phi_2 (\mathtt{cent} (\overline{X}_{k}))$ by (\ref{sigdeltamg}) and 
thanks to Corollary~\ref{rwergobis}, we get 
$\bP$-a.s.~$\lim_{k\to \infty} \tfrac{1}{k} C_{2, k} \eqo \sigma^2_\nu$. 
Recall that $\frac{1}{n}r_n \! \to \! r\ino \bbR_+ $. If $r\ino (0, \infty)$, then 
$\tfrac{1}{n}  C_{2, r_n} \eqo \tfrac{r_n}{n} \tfrac{1}{r_n} C_{2, r_n}  \to r\sigma^2_\nu$ a.s. 
If $r\eqo 0$, since $\Phi_2\geq 0$, we get for all $\eta\ino (0, 1)$, 
$\tfrac{1}{n} C_{2, r_n} \leqo \tfrac{r_n+\eta n}{n} \tfrac{1}{r_n+\eta n} C_{2, r_n + \eta n}\to \eta\sigma^2_\nu$ a.s.~and thus $\tfrac{1}{n} C_{2, r_n} \to 0$ a.s.~since $\eta$ can be arbitrarily small. Similarly, a.s.~$\tfrac{1}{n}C_{2, r_n+s_n}\to (r+s)\sigma^2_\nu$, which then implies that 
$\sum_{1\leq k\leq s_n} \bE [D^2_{n, k} | \cF_{n, k-1} ] \eqo \tfrac{1}{n} (C_{2, r_n+s_n} \! -\! C_{2, r_n}) \to s\sigma^2_\nu$ and Assumption $(d)$ in Proposition \ref{condiTCL} is satisfied.

It remains to prove Assumption $(c)$. We fix $\epp \ino (0, 1)$ and we first 
observe that 
$$ \sum_{1\leq k\leq s_n}\!\! \bE [D_{n,k}^2\un_{\{| D_{n, k}|>\epp  \}} | \cF_{n, k-1} ]\leq  \tfrac{1}{\epp^{\ttb -2}} \!\! \sum_{1\leq k\leq s_n}\!\! \bE [D_{n,k}^\ttb | \cF_{n, k-1} ] \leq \tfrac{1}{(\epp\sqrt{n})^{\ttb -2}} \frac{1}{n} C_{\ttb, r_n + s_n} $$
which a.s.~tends to $0$ since $\frac{1}{n} C_{\ttb, r_n + s_n}\to (r+s) \bE [\Phi_\ttb (\overline{X}_0)]$ a.s.~by the ergodic theorem (which applies as Corollary~\ref{rwergobis} since $\bE [\Phi_\ttb (\overline{X}_0)]\leko \infty$ by  Lemma \ref{harmcoo} $(iv)$). This proves that Assumption $(c)$ in Proposition \ref{condiTCL} is satisfied and (\ref{condiharmodiffu}) follows from (\ref{fourcondi}). \cqfd

\medskip

To conclude this section, let us recall from Peres \& Zeitouni \cite{PerZei08} annealed uniform bounds for the (relative) height of $\ttm_\nu$-biased RWs on GW trees which is used later in the proof of Proposition \ref{harmoclose} in Section \ref{harmoclosepfsec}.  
\begin{lemma} 
\label{PZCV}  Let $\nu$ be a supercritical offspring distribution whose mean is denoted by $\ttm_\nu$. 
We assume that $ \sum_{k\in \bbN} k^{2} \nu(k) \leko \infty$. Let $(\bT, \boo)$ be 
either a GW($\nu$)-tree as in Definition \ref{GWfordef} $(\textbf a)$ (in this case: $\boo\eqo \varnothing$) or an infinite GW($\nu$)-tree as in Definition \ref{GWptdef} or an invariant GW($\nu$)-tree as in 
Definition~\ref{definvRW}. Let $\overline{X}_n \eqo (\bT, X_n)$, $n\ino \bbN$, be a $\mathtt m_\nu$-biased RW on $\bT$ such that $X_0 \eqo \boo$. Then, for all $n, p\ino \bbN^*$, the following holds true. 
\begin{compactenum}

\smallskip

\item[$(i)$] If $\bT$ is a GW($\nu$)-tree, $\bP \big(\max_{1\leq k\leq n} |X_k| \geq p \big) \leq 4 ne^{-p^2/2n}$. 

\smallskip

\item[$(ii)$] If $\bT$ is an invariant GW($\nu$)-tree, $\bP \big(\max_{1\leq k\leq n} \big|  \bbn X_k \bbn   \big| \geq p \big) \leq 8 n^3e^{-(p-1)^2/2n}$.

\smallskip

\item[$(iii)$] If $\bT$ is an infinite GW($\nu$)-tree, $\bP \big(\max_{1\leq k\leq n} \big|   \bbn X_k \bbn  \big| \geq p \big) \leq 16 n^3e^{-(p-1)^2/2n}$. 
\end{compactenum}
\end{lemma}
\noi
\textbf{Proof.} For $(i)$, see Peres \& Zeitouni \cite[Lemma 5 p.~608]{PerZei08}. For $(ii)$ and $(iii)$, see \cite[Corollary 2 p.~608]{PerZei08}. \cqfd 

\subsection{Statements of the result, overview of its proof}
\label{overviewsec}
Before stating the main results of the section, let us fix some notation and let us specify our assumptions. 
We fix $\alpha \ino (1, 2]$ and we fix an offspring distribution $\mu$ that satisfies (\ref{hypostaintro}). 
We denote by 
$\tau_{\infty}$ a GW($\mu$)-forest corresponding to a sequence $(\tau^{_{(k)}}_{^{\!}})_{k\in \bbN^*}$ of i.i.d.~GW($\mu$)-trees (Definition \ref{GWfordef} $(\textbf{b})$). 
We denote by $V_l (\tau_\infty)$ its  Lukasiewicz path, by $(H_s (\tau_\infty))_{s\in \bbR_+}$ its height process (Definition \ref{Lukadef}) and by $(C_s (\tau_\infty))_{s\in \bbR_+}$ its contour process (Definition \ref{contourdef}). 
We next recall from (\ref{JGDdef}) the definition of $\frak g_n $ and $\frak d_n$ (roughly speaking $(\frak g_n , \frak d_n)$ is the first excursion interval of $H(\tau_\infty)$ above $0$ longer than $n$).   
Then we also 
recall from Theorem \ref{VHCcvstable} that $(a_n)_{n\in \bbN^*}$ and $(b_n)_{n\in \bbN^*}$ are scaling sequences such that $a_n\eqo n/b_n$ and such that the following limit \begin{equation}
\label{remindcv}
\mathscr Q_n := \big(  \tfrac{1}{n} \frak g_n ,\tfrac{1}{n} \frak d_n, \tfrac{1}{b_n} V_{\lfloor n\cdot \rfloor} (\tau_\infty) ,   \tfrac{1}{a_n} H_{ n\cdot } (\tau_\infty)  ,    \tfrac{1}{a_n} C_{2n\cdot} (\tau_\infty)  \big)  \underset{n\to \infty}{-\!\!\!-\!\!\!\longrightarrow} 
\big(  \mathbf g_1, \mathbf d_1, X, H, H \big)
\end{equation}
holds weakly on $\bbR^2\! \times \! \bD (\bbR_+, \bbR)\! \times \! \bC^0 (\bbR_+, \bbR)^2$ .
Here, $X$ is an $\alpha$-stable spectrally positive Lévy process characterized by (\ref{Xlaw}), $H$ is its associated height process as defined in (\ref{Hlimit}) and $(\mathbf g_1, \mathbf d_1)$ is the first excursion interval of $H$ longer than 
$1$ as defined in Remark \ref{gundundef}.  
 
 We next state a limit theorem for snakes associated with $\tau_\infty$-indexed BRWs that take their values in an invariant GW($\nu$)-tree (Definition \ref{definvRW} ). Here $\nu$ is a fixed \emph{supercritical} offspring distribution whose mean is denoted by $\ttm_\nu$ and which shall be subject to various moment assumptions.  

Our main objects of interest in this section are BRWs in invariant environment. More precisely, 
we denote by $(\bT, \boo)$ an invariant GW($\nu$)-tree as in Definition \ref{definvRW}. We assume that $(\bT, \boo)$ is \emph{independent from $\tau_\infty$} and  we introduce the sequence of 
$\bT$-valued BRWs $\fTheta_k\eqo \big( \tau^{_{(k)}}_{^{\,}}\!\! \! , \, \varnothing; \overline{Y}^{_{(k)}}_{u}\eqo (  \bT, Y^{_{(k)}}_{u}), u\ino \tau^{_{(k)}}_{^{\!}}  \big)$, $k\ino \bbN^*$ such that 
\begin{equation}
\label{lawYinfty}
 \textrm{conditionally on $(\bT, \boo)$ and $\tau_\infty$, the $(\fTheta_k)_{k\in \bbN^*}$ are independent and} \; \fTheta_k \overset{\textrm{(law)}}{=} 
 Q_{(\tau^{{(k)}}, \varnothing)}^{(\bT, \boo)} .
 \end{equation}
Here $ Q_{^{(\tau^{{(k)}}, \varnothing)}}^{_{(\bT, \boo)}}$ is as in Definition \ref{biasedbfRWdef} with $\lambda\eqo \ttm_\nu$.  
We recall from (\ref{defSx}) the definition of the harmonic coordinates $(S_x)_{x\in \bT}$. 
Then, for all $k\ino \bbN^*$, we associate with $\fTheta_k$ two $\bbR$-valued $\tau^{_{(k)}}_{^{\!}}$-indexed BRWs:  
$$ M_k := \big( \bbn Y^{_{(k)}}_{u}\bbn \big)_{u\in \tau^{_{(k)}}_{^{\!}}}  \quad \textrm{and} \quad \mathcal M_k := \big( S_{Y^{_{(k)}}_{u}}  \big)_{u\in \tau^{_{(k)}}_{^{\!}}} \; ,$$
and according to Definition \ref{brwlkdef}, we introduce the following four $\bC^0(\bbR_+, \bbR)$-valued continuous processes:  
\begin{compactenum}
\item[$-$] $\big( W_s(\tau_\infty;  \cdot)\big)_{\! s\in \bbR_+}$ is the contour snake of the BRWs  $( M^{_{(k)}}_{^{\! }} )_{k\in \bbN^*}$;
\item[$-$] $\big(W_s^*(\tau_\infty; \cdot)\big)_{\! s\in \bbR_+}$ is the height snake of the BRWs $( M^{_{(k)}}_{^{\! }} )_{k\in \bbN^*}$; 
\item[$-$] $\big(\cW_s(\tau_\infty; \cdot)\big)_{\! s\in \bbR_+}$ is the contour snake of the BRWs $( \mathcal M^{_{(k)}}_{^{\! }} )_{k\in \bbN^*}$; 
\item[$-$] $\big(\cW_s^*(\tau_\infty; \cdot)\big)_{\! s\in \bbR_+}$ is the height snake of the BRWs $( \mathcal M^{_{(k)}}_{^{\! }} )_{k\in \bbN^*}$. 
\end{compactenum}
We also recall from Definition \ref{brwlkdef} the notation $\widehat{W}_s(\tau_\infty)$, $\widehat{W}_s^*(\tau_\infty)$, $\widehat{\cW}_s(\tau_\infty)$ and $\widehat{\cW}_s^*(\tau_\infty)$, $s\ino \bbR_+$, for the corresponding endpoint-processes.

Finally, we denote by $(\widehat{W}_s)_{s\in \bbR_+}$ the endpoint process of the one-dimensional Brownian snake whose lifetime-process is the $\alpha$-stable height process $(H_s)_{s\in \bbR_+}$ as introduced in Proposition \ref{Holdsnake}.

\medskip

The goal of this section is to prove the following theorem that shows that the 
contour  and heigh-snakes in harmonic coordinates and in relative heights jointly converge to $W$ when suitably rescaled.
To simplify the proof we first admit several intermediate results, which are stated below and 
proved in Sections \ref{qtightsec}, \ref{pfharmosec} and \ref{harmoclosepfsec}, before proceeding 
to the proof of Theorem \ref{cvsnakes} at the end of the present section. 
\begin{theorem}
\label{cvsnakes} We fix $\alpha \ino (1, 2]$. Let $\mu$ satisfy (\ref{hypostaintro}). Let $\tau_\infty$ be a GW($\mu$)-forest  (Definition \ref{GWfordef} $(\textbf{b})$). Let $\nu$ be a supercritical offspring distribution whose mean is denoted by $\ttm_\nu$. We assume that there is  $\ttb\ino ( \frac{2\alpha}{\alpha-1}, \infty)$ such that $\sum_{k\in \bbN} k^{1+2\ttb} \nu(k)\leko \infty$. Let $(\bT, \boo)$ be an invariant GW($\nu$)-tree (Definition \ref{definvRW}) that is independent from $\tau_\infty$. 
We recall the definition of $\mathscr Q_n $ from the convergence (\ref{remindcv}). We recall 
from (\ref{vutchi}) the definition of $\sigma_\nu$. Otherwise, we keep the previous notation. 
Then, conditionally given $(\bT, \boo)$, the following joint convergence 
\begin{eqnarray}
\label{snakescv}
\Big(\mathscr Q_n ,  \tfrac{1}{ \sigma_\nu \! \sqrt{a_n}} \widehat{\cW}_{2n\cdot }(\tau_\infty),  \tfrac{1}{ \sigma_\nu \! \sqrt{a_n}}\widehat{\cW}^*_{n\cdot }(\tau_\infty)    \!\! \!\! \!\! \!\! \!\! \!\!  \!\!     &, &\!\! \!\! \!\! \!\! \!\! \!\!  \!\!    \tfrac{\sigma_\nu}{ \sqrt{a_n}} \widehat{W}_{2n\cdot }(\tau_\infty),\!  \tfrac{\sigma_\nu}{ \sqrt{a_n}} \widehat{W}^*_{n\cdot }(\tau_\infty)\Big) \nonumber  \\
&\underset{n\to \infty}{-\!\!\!-\!\!\!-\!\!\!\longrightarrow}& \big(  \mathbf g_1, \mathbf d_1, X, H, H, \widehat{W},  \widehat{W} , \widehat{W},  \widehat{W} \big)
\end{eqnarray}
a.s.~holds weakly on $\bbR^2\! \times \bD(\bbR_+, \bbR) \! \times \! \bC^0 (\bbR_+, \bbR)^6 $ equipped with the product topology. 
\end{theorem}
\noi
\textbf{Proof.} See the end of the section.  \cqfd

\medskip

\noi
\textbf{Overview of the proof of Theorem \ref{cvsnakes}.} To prove Theorem~\ref{cvsnakes}, we first prove the following proposition in Section \ref{pfharmosec}, which shows that the contour  and heigh-snakes in harmonic coordinates converge to $W$ when suitably rescaled.

\begin{proposition}
\label{harmocv} We fix $\alpha \ino (1, 2]$. Let $\mu$ satisfy (\ref{hypostaintro}). Let $\tau_\infty$ be a GW($\mu$)-forest  (Definition \ref{GWfordef} $(\textbf{b})$). Let $\nu$ be a supercritical offspring distribution whose mean is denoted by $\ttm_\nu$. We assume that there is  $\ttb_0\ino ( 1+\frac{\alpha}{\alpha-1}, \infty)$ such that $\sum_{k\in \bbN} k^{1+2\ttb_0} \nu(k)\leko \infty$. Let $(\bT, \boo)$ be an invariant GW($\nu$)-tree (Definition \ref{definvRW}) that is independent from $\tau_\infty$. 
We recall the definition of $\mathscr Q_n $ from the convergence (\ref{remindcv}). We recall 
from (\ref{vutchi}) the definition of $\sigma_\nu$. Otherwise, we keep the previous notation. 
Then, conditionally given $(\bT, \boo)$, the following convergence 
\begin{equation}
\label{cvharmo}
\Big(\mathscr Q_n ,  \tfrac{1}{ \sigma_\nu\sqrt{a_n}}\,  \widehat{\cW}_{2n\cdot }(\tau_\infty),  \tfrac{1}{ \sigma_\nu \sqrt{a_n}}\, \widehat{\cW}^*_{n\cdot }(\tau_\infty) \Big) 
 \underset{n\to \infty}{-\!\!\!-\!\!\!-\!\!\!\longrightarrow} \big( \mathbf g_1, \mathbf d_1, X,  H, H, \widehat{W},  \widehat{W} \big)
\end{equation}
a.s.~holds weakly on $\bR^2 \! \times \! \bD (\bbR_+, \bbR) \! \times \! \bC^0 (\bbR_+, \bbR)^4$ equipped with the product topology. 
\end{proposition}
\noi
\textbf{Proof.} See Section \ref{pfharmosec}. \cqfd

\medskip

To prove Proposition~\ref{harmocv}, we first rely on the following lemma, which is a conditional version of Kolmogorov continuity theorem and whose proof is a straightforward adaptation of the standard version of the theorem. To state it let us introduce the following notation for the $\gamma$-Hölder factor (here $\gamma \ino (0, 1]$) of a function $\mathtt x(\cdot) \ino \bC^0 (\bbR_+, \bbR)$ on the interval $[0, s_0]$ (here $s_0 \ino (0, \infty))$:  
\begin{equation}
\label{Holdnodef}
\mathtt{Hol}_{\gamma, s_0} (\mathtt x) = \sup \Big\{ \frac{|\mathtt x(s)\! -\! \mathtt x(s')|}{|s\! -\! s'|^\gamma}\, ; \, s, s'\ino [0, s_0], \; s\! \neq \! s' \Big\} 
\end{equation}  
which may be infinite. 
\begin{lemma}
\label{condKolcont} Let $\ccG\! \subset \! \ccF$ be a $\sigma$-field and let $(Z_s)_{s\in \bbR_+}$ be a $\bbR$-valued continuous process. Let $s_0, \ttb \ino (1, \infty)$ and $\ttc \ino (0, \infty)$. We assume that there is 
a $\ccG$-measurable and a.s.~finite r.v. $C$ such that for all $s, s'\ino [0,s_0]$,
\begin{equation}
\label{hypoKol}
\textrm{$\bP$-a.s.~} \quad \bE \Big[ \big| Z_s \! -\! Z_{s'}\big|^\ttb \Big| \, \ccG \Big] \leq C |s\! -\!s'|^{1+ \ttc}. 
\end{equation}
Then, for all $\gamma \ino (0, \frac{\ttc}{\ttb} )$, we get 
\begin{equation}
\label{Holdbound}
\textrm{$\bP$-a.s.~} \quad  \bE \big[  \big( \mathtt{Hol}_{\gamma, s_0} (Z) \big)^\ttb  \big| \ccG\big] \leq c_{\ttb, \ttc, \gamma, s_0} C, 
\end{equation}
where $c_{\ttb, \ttc, \gamma, s_0}\eqo 2^{2\ttb} (2s_0)^{1+c-\gamma b} \big(1 \! -\! 2^{-(\frac{\ttc}{\ttb}) +\gamma} \big)^{-\ttb} $. 
\end{lemma}
\noi
\textbf{Proof:} This is easily adapted from e.g.~Revuz \& Yor \cite[Theorem 2.1 p.26]{ReYo99}. \cqfd 

\medskip

More precisely, in Proposition \ref{HolderH} in Section~\ref{qtightsec}, we first show that the rescaled height process $a_n^{-1} H_{n\cdot }(\tau_\infty)$ satisfies an (unconditioned) assumption of type (\ref{hypoKol}). Then in Proposition \ref{HolderWW} in Section \ref{qtightsec}, we prove that it entails that 
$a_n^{-1/2}\widehat{\cW}^*_{n\cdot } (\tau_\infty)$ also satisfies an assumption of type (\ref{hypoKol}). This implies that conditionally given $(\bT, \boo)$, $a_n^{-1/2}\widehat{\cW}^*_{n\cdot } (\tau_\infty)$ is tight, thanks to the following general result (the proof given below makes precise what we mean here by conditional tightness). 
\begin{lemma}
\label{condtight} Let $\ccG\! \subset \! \ccF$ be a $\sigma$-field and let $(Z^{_{(n)}}_s)_{s\in \bbR_+}$, $n\ino \bbN^*$, be a sequence of $\bbR$-valued continuous processes such that $Z^{_{(n)}}_0\eqo 0$. Let $\ttb\ino (1, \infty)$ and $\gamma \ino (0, 1]$. Let us assume that for all $p\ino \bbN^*$, there exists a $\ccG$-measurable r.v.~$\mathtt C_p$ such that for all $n\ino \bbN^*$, 
\begin{equation}
\label{Holdcontr}
\textrm{$\bP$-a.s.~} \quad \bE \big[ \big( \mathtt{Hol}_{\gamma , \, p} (Z^{_{(n)}}_{^{\,}} )\big)^\ttb  \big| \ccG\big] \leq \mathtt C_p < \infty \; .
\end{equation}  
Then conditionally given $\ccG$, the laws of the processes $(Z^{_{(n)}}_{^{\,}}\! )_{n\in \bbN^*}$ are a.s.~tight on $\bC^{0} (\bbR_+, \bbR)$. 
\end{lemma}
\noi
\textbf{Proof.} We denote by $\cM_1 (\bC^{_{0}}_{^1})$ the space of Borel probability measures on $\bC^{_{0}}_{^1}\! :=\! \bC^{0} (\bbR_+, \bbR)$ equipped with the Polish topology of weak convergence. Denote by $\omega\ino \Omega  \! \mapsto\!  Q_n (\omega, \cdot)$ a regular version of the conditional law of $Z^{_{(n)}}_{^{\,}}$ given $\ccG$ (which exists 
since $\bC^{_{0}}_{^1}$ is Polish). Namely, 
\begin{compactenum}
\item[$-$] $Q_n$ is a $\cM_1 (\bC^{_{0}}_{^1} )$-valued r.v., measurable with respect to $\ccG$ and the Borel $\sigma$-field of $\cM_1 (\bC^{_{0}}_{^1} )$. 
\item[$-$] For all Borel subsets $B$ of 
$\bC^{_{0}}_{^1} $ we $\bP$-a.s.~have $Q_n (\cdot, B)\eqo \bE [\un_{\{ Z^{(n)} \in B \}} | \ccG ]$. 
\end{compactenum}
Our assumptions imply that there exists $\Omega_0\ino \ccG$ such that $\bP (\Omega_0)\eqo 1$ and such that for all $\omega\ino \Omega_0$ and for all $n, p\ino \bbN^*$,
$$ \bE \big[ \big( \mathtt{Hol}_{\gamma , \, p} (Z^{_{(n)}}_{^{\,}} )\big)^\ttb  \big| \ccG\big] (\omega) \, = \int_{\bC^0_1}   \!\!\! \big( \mathtt{Hol}_{\gamma , \, p} (\mathtt x  )\big)^\ttb Q_n (\omega , d \mathtt x) \leqo \mathtt C_p (\omega) < \infty \; .$$
For all $\mathtt x (\cdot) \ino \bC^{_{0}}_{^1} $ and $\delta \ino (0, \infty)$, we set $w_\delta (\mathtt x, p) \eqo \max \big\{ |\mathtt x(s)\! -\! \mathtt x (s') | \, ; \, s, s'\ino [0, p] : |s\! -\! s'| \leqo \delta \big\}$, which is the $\delta$-modulus of continuity of $\mathtt x$ on $[0, p]$. Then, we get for all $\omega \ino \Omega_0$, all $p,n\ino \bbN^*$ and all $\delta, \eta \ino (0, \infty)$,
$$ Q_n \big(\omega\, , \, \big\{ \mathtt x\ino \bC^{_{0}}_{^1}: w_{\delta} (\mathtt x, p) \geko \eta \big\} \big) \leq \frac{\delta^{\gamma \ttb}}{\eta^\ttb}   \int_{\bC^0_1}   \!\!\! \big( \mathtt{Hol}_{\gamma , \, p} (\mathtt x  )\big)^\ttb Q_n (\omega , d \mathtt x) \leqo \mathtt C_p (\omega) \frac{\delta^{\gamma \ttb}}{\eta^\ttb} , $$
which implies $\lim_{\delta\to 0^+} \limsup_{n\to \infty} Q_n \big(\omega,  \big\{ \mathtt x\ino \bC^{_{0}}_{^1}: w_{\delta} (\mathtt x, p) \geko \eta \big\} \big) \eqo 0$. Since $Z^{_{(n)}}_0\eqo 0$, a standard result (see e.g.~Billingsley \cite[Theorem 7.3 p.~82]{Bil68}) entails that for all $\omega\ino \Omega_0$, the laws $(Q_n (\omega, \cdot))_{n\in \bbN^*}$ are tight. \cqfd

\medskip

In Section~\ref{pfharmosec}, we prove that $a_n^{-1} H_{n\cdot }(\tau_\infty)$ and $\sigma^{-1}_{\nu} a_n^{-1/2}\widehat{\cW}^*_{n\cdot } (\tau_\infty)$ jointly converge in law conditionally given $(\bT, \boo)$ to $(H, \widehat{W})$. To do so, we compute the conditional Fourier transforms of finite-dimensional marginals of the height snake in harmonic coordinates and we use Proposition~\ref{fourcondharmo}.

\medskip

We next derive the convergence (\ref{cvharmo}) using the fact that $C(\tau_\infty)$ and $\widehat{\cW} (\tau_\infty)$ are derived 
from $H(\tau_\infty)$ and $\widehat{\cW}^* (\tau_\infty)$ by a time-change. 
More precisely, recall from Remark~\ref{contord} $(\textbf{d})$ the definition of the increasing continuous bijection $\phi_{\tau_\infty} \! :\! \bbR_+ \! \to \! \bbR_+$ such that $H_s (\tau_\infty)\eqo C_{\phi_{\tau_\infty} (s)} (\tau_\infty)$, $s\ino \bbR_+$. For all $s_0\ino (0, \infty)$, (\ref{controphit}) implies that 
$\max_{s\in [0, s_0]} |\frac{1}{2n}\phi_{\tau_\infty} (ns) \! -\! s | \to 0$ in probability because 
(\ref{remindcv}) implies that $\frac{1}{n} \max_{s\in [0, s_0]} H_{ns} (\tau_\infty) \to 0$ in $\bP$-probability since $a_n/n\to 0$. Therefore, if for all $s\ino \bbR_+$ we set $\varphi_n (s)\eqo \frac{1}{n} \phi^{-1}_{\tau_\infty} (2ns)$ (here, $\phi^{-1}_{\tau_\infty}$ stands for the inverse of $\phi_{\tau_\infty} $), then 
we have proved 
\begin{equation}
\label{liminvtc}
\forall s_0\ino (0, \infty) , \quad \lim_{n\to \infty}\max_{s\in (0, s_0] } \big| \varphi_n (s) \! -\! s \big| = 0 \quad \textrm{in $\bP$-probability.}
\end{equation}
We  derive the conditional joint convergence (\ref{cvharmo}) 

\smallskip

\noi
$-$ from 
the conditional joint convergence of  $a_n^{-1} H_{n\cdot }(\tau_\infty)$ and $\sigma_{\nu}^{-1} a_n^{-1/2}\widehat{\cW}^*_{n\cdot } (\tau_\infty)$, 

\smallskip

\noi
$-$ from the following equalities that holds for all $s\ino \bbR_+$
\begin{equation}
\label{intcCWW}
\tfrac{1}{a_n}C_{2ns} (\tau_\infty)\eqo \tfrac{1}{a_n}H_{n\varphi_n(s)} (\tau_\infty)\quad \textrm{and} \quad \tfrac{1}{ \sqrt{a_n}}\,  \widehat{\cW}_{2ns}(\tau_\infty)= 
\tfrac{1}{ \sqrt{a_n}}\, \widehat{\cW}^*_{n\varphi_n (s)} (\tau_\infty)\; , 
\end{equation}
$-$ from the following deterministic lemma. 
\begin{lemma}
\label{composko} Suppose that $\mathtt x_n  \! \rightarrow \! \mathtt x$ in $\bC^0 (\bbR_+, \bbR^q)$ and that 
$\varphi_n \! \rightarrow \! \varphi$ in $\bC^0 (\bbR_+, \bbR_+)$. Assume that 
$\varphi_n (0) \! = \! 0$ and that $\varphi_n$ is nondecreasing. Then, 
$\mathtt x_n \circ \varphi_n \! \rightarrow \! \mathtt x \circ \varphi$ in $\bC^0 (\bbR_+, \bbR^q)$. 
\end{lemma}
\noi
\textbf{Proof.} See Whitt \cite[Theorem 3.1, p.~75]{Wh80}. \cqfd 

\begin{proposition}
\label{harmoclose} We fix $\alpha \ino (1, 2]$. Let $\mu$ satisfy (\ref{hypostaintro}). Let $\tau_\infty$ be a GW($\mu$)-forest  (Definition \ref{GWfordef} $(\textbf{b})$). Let $\nu$ be a supercritical offspring distribution whose mean is denoted by $\ttm_\nu$. We assume that there is  $\ttb\ino ( \frac{2\alpha}{\alpha-1}, \infty)$ such that $\sum_{k\in \bbN} k^{1+2\ttb} \nu(k)\leko \infty$.
We keep the notation of Proposition \ref{harmocv}. Let $s_0 \in (0, \infty)$. Then there are two constants $c, \epp_1 \ino (0, \infty)$, which only depend on $\mu$, $\nu$, $\ttb$ and $s_0$, such that 
\begin{equation}
\label{closeharmo}
\textrm{$\bP$-a.s.~} \quad \lim_{n\to \infty} \bP \Big( \!\!\!\! \max_{\quad s\in [0, s_0] }\! \big| \tfrac{1}{\sigma^2_{\nu}} \widehat{\cW}^*_{ns} (\tau_\infty) \! -\!   \widehat{W}^*_{ns} (\tau_\infty)\big| \, > \, cn^{-\epp_1} \sqrt{a_n} \,  \Big| \, (\bT, \boo) \Big)= 0 \; .
\end{equation}
\end{proposition}
\noi
\textbf{Proof.} See Section \ref{harmoclosepfsec}. \cqfd 

\medskip

\noi
\textbf{Proof of Theorem \ref{cvsnakes}.} We admit Propositions  \ref{harmocv} and  \ref{harmoclose} and we prove Theorem \ref{cvsnakes}. First note that the moment 
assumption on $\nu$ is stronger in Proposition \ref{harmoclose} (and Theorem \ref{cvsnakes}) than in Proposition \ref{harmocv} since $\frac{2\alpha}{\alpha-1} \geko 1+ \frac{\alpha}{\alpha-1} $. Thus under the assumptions of Theorem~\ref{cvsnakes}, both propositions apply and they 
easily entail that conditionally given $(\bT, \boo)$, the following convergence
\begin{equation}
\label{cvharmo5/2}
\big(\mathscr Q_n ,   \tfrac{1}{a_n} H_{n\cdot } (\tau_\infty)  ,   \tfrac{1}{ \sigma_\nu \sqrt{a_n}}\, \widehat{\cW}^*_{n\cdot }(\tau_\infty) ,  \tfrac{\sigma_\nu}{  \sqrt{a_n}}\, \widehat{W}^*_{n\cdot }(\tau_\infty) \big) \underset{n\to \infty}{-\!\!\!-\!\!\!-\!\!\!\longrightarrow} \big(\mathbf g_1, \mathbf d_1, X, H, H,  H, \widehat{W} , \widehat{W} \big)
\end{equation}
holds weakly on $\bbR^2\! \times\! \bD(\bbR_+, \bbR) \! \times \! \bC^0 (\bbR_+, \bbR)^5$ equipped with the product topology. 
Then recall (\ref{liminvtc}), (\ref{intcCWW}) and recall (by definition of height snake) that for all $s\ino \bbR_+$,
$$\tfrac{1}{a_n} C_{2ns} (\tau_\infty)  = \tfrac{1}{a_n} H_{n\varphi_n (s)} (\tau_\infty)  \quad \textrm{and} \quad  \tfrac{\sigma_\nu}{ \sqrt{a_n}}\, \widehat{W}_{2ns}(\tau_\infty)= 
\tfrac{\sigma_\nu}{ \sqrt{a_n}}\, \widehat{W}^*_{n\varphi_n (s)} (\tau_\infty)\; .$$
This entails (\ref{snakescv}) by (\ref{cvharmo5/2}) and Lemma \ref{composko}, which completes the proof of Theorem \ref{cvsnakes}. \cqfd

\subsection{Quenched tightness of the height snake of $\tau_\infty$-indexed BRWs} 
\label{qtightsec}

We keep the notation and the general assumptions of the previous section. 
We first prove Hölder estimates for $H(\tau_\infty)$ by adapting to the discrete setting the proof of D.~\& Le Gall~\cite[Lemma 1.4.6]{DuLG02} which provides a similar result for the continuous height process. 

To that end, we use the fact that the Lukasiewicz path $V(\tau_{ \infty})$ of $\tau_\infty$ 
is a RW whose jump law is $\widetilde{\mu} (k)\eqo \mu (k+1)$, $k\ino \bbN\cup \{ -1\}$ (Remark \ref{GWrem} ($\mathbf b$)) and we use (\ref{codheight}) 
which shows that $H_n (\tau_\infty)$ counts the number of times the time-reversed RW $(V_n (\tau_\infty)\! -\! V_{n-k}  (\tau_\infty))_{0\leq k\leq n}$ 
reaches its supremum. Namely, the height process is related to the so-called \emph{weak ascending ladder variables} of $V(\tau_\infty)$ whose definition and basic properties are recalled from Spitzer \cite{SpiRW} below. 
Let $(V_n)_{n\in \bbN}$ be a RW whose jump law is $\widetilde{\mu}$. For all $n\ino \bbN$, we set 
\begin{equation}
\label{VbarL}
\overline{V}_{\! n} \eqo \max _{0\leq k\leq n} V_k \quad \textrm{and} \quad \mathtt \mathtt  L_n\eqo \# \{ k\ino \{ 1, \ldots, n\}: V_k \eqo \overline{V}_{\! k}  \}\; , 
\end{equation}
with the convention that $L_0\eqo 0$. 
The weak ascending ladder variables are then given by 
\begin{equation}
\label{laddvar} 
\cL_n = \inf \{ k\ino \bbN: L_k= n \} \quad \textrm{and} \quad \cV_n =  V_{\cL_n}=\overline{V}_{\!\!\! \cL_n} .
\end{equation}
We also recall from (\ref{Tmoinsp}) that $T_{-n}\eqo \inf \{ k\ino \bbN: V_k\eqo  -n \}$. We recall from (\ref{Laplmu}) that 
$g_\mu$ is the generating function of $\mu$ and we recall from Lemma \ref{numbtau} that $\varphi_\mu$ is the generating function of $T_{-1}$. Then $(\cL_n, \cV_n)_{n\in \bbN}$ is a $\bbN^2$-valued RW whose jump law is characterized by 
\begin{equation}
\label{lasacsexpli}
\forall r, s \ino [0, 1], \quad 1-\bE \big[ r^{\cL_1}  s^{\cV_1}\big]= \frac{s-rg_\mu(s)}{s-\varphi_\mu(r)} \; .
\end{equation}
For a proof we refer to Spitzer \cite{SpiRW}, Chapter IV, \textbf{P5} (b) p.~181 and more precisely to (3) p.~187 in the specific case of left-continuous RWs (with Spitzer's notation, $P\eqo g_\mu$ and $r\eqo \varphi_\mu$). We next provide estimates on the Laplace exponents of $\cL_1$, $\cV_1$ and $T_{-1}$.   
\begin{lemma} 
\label{ascladd} We keep the above notation. We fix $\alpha \ino (1, 2]$ and we assume that $\mu$ satisfies (\ref{hypostaintro}). For all $\lambda \ino \bbR_+$, we set $\Lambda_0(\lambda)\eqo -\log \bE [e^{-\lambda \cL_1}]$, $\Lambda_1 (\lambda)\eqo -\log \varphi_\mu (e^{-\lambda})$ and $\Lambda_2 (\lambda)\eqo -\log \bE[ e^{-\lambda \cV_1}] $. Then there are $c_\mu, c^\prime_\mu \ino (0, \infty)$ such that 
\begin{equation}
\label{estiLambs} 
\Lambda_0 (\lambda) \sim_{0^+}c_\mu \,  \lambda^{\frac{\! \alpha-1}{\alpha}}  L\big(\tfrac{1}{\lambda} \big) \quad \textrm{and} \quad \Lambda_2(\Lambda_1 (\lambda)) \sim_{0^+} c^\prime_\mu\,   \lambda^{\frac{\! \alpha-1}{\alpha}}  L\big(\tfrac{1}{\lambda} \big)
\end{equation} 
where $L$ is the slowly varying function appearing in (\ref{bnBiGoTe}). 
\end{lemma}
\noi
\textbf{Proof.} By (\ref{lasacsexpli}), $1\! -\! \bE [e^{-\lambda \cL_1}]\eqo (1\! -\! e^{-\lambda} ) / (1\! -\! \varphi_\mu (e^{-\lambda}))$. Then the first equivalence in (\ref{estiLambs}) follows from (\ref{numtauber}) in Lemma \ref{numbtau}. Next, we deduce from (\ref{lasacsexpli}) that $1\! -\! \bE [e^{-\lambda \cV_1}]\eqo \psi_\mu (s_\lambda) / s_\lambda$. where we recall 
$\psi_\mu$ from (\ref{Laplmu}) and where we have set $s_\lambda\eqo 1\! -\! e^{-\lambda}$. Then, $\Lambda_2 (\lambda)\sim_{0^+} c_\alpha'\lambda^{\alpha-1} L^* (1/\lambda)$ where $L^*$ is the slowly varying function in (\ref{mutail}).  
Then (\ref{numtauber}) in Lemma \ref{numbtau} entails that 
$$\Lambda_2(\Lambda_1 (\lambda)) \sim_{0^+} c_1 \lambda^{\frac{\alpha-1}{\alpha}} L\big( \tfrac{1}{\lambda}\big)^{-(\alpha-1)}L^* \big( c_2 \lambda^{-\frac{1}{\alpha}} L\big( \tfrac{1}{\lambda}\big)\big)\; , $$
where $c_1, c_2 \ino (0, \infty)$ are two constants which only depend on $\alpha$. This implies the second equivalent in (\ref{estiLambs}) by (\ref{slowlinks}), which relates $L$ and $L^*$. \cqfd

\begin{proposition}      
\label{momheight} We keep the above notation. We assume that $\mu$ is non-trivial and critical. 
Let $\ttb\in (1, \infty)$. Then for all $m\ino \bbN$ and $n\ino \bbN^*$,
\begin{eqnarray}
\label{heightmom}
\bE \Big[  \big| H_{m+n} (\tau_\infty) \! -\!\!\!\!\!\!\!\! \!\! \inf_{\; \quad \lgeo m, m+n \rgeo}\!\!\!\!\!\!\!\! H (\tau_\infty) \big|^\ttb  \Big] & \leqo  & e\ttb \Gamma_{\!\! \mathtt e} (\ttb)\,  \big( \Lambda_0 \big(\tfrac{1}{n} \big)\big)^{-\ttb} \qquad  \textrm{and} \nonumber\\ 
   \bE \Big[  \big| H_{m} (\tau_\infty) \! -\!\!\!\!\!\!\!\! \!\! \inf_{\; \quad \lgeo m, m+n \rgeo}\!\!\!\!\!\!\!\! H (\tau_\infty) \big|^\ttb  \Big] &\leqo&  e\ttb \Gamma_{\!\! \mathtt e} (\ttb)\,  \big( \Lambda_2\big( \Lambda_1  \big(\tfrac{1}{n} \big) \big)\big)^{-\ttb}. 
\end{eqnarray}
\end{proposition}      
\noi
\textbf{Proof.} We prove the first inequality in (\ref{heightmom}). Without loss of generality, we assume that $V$ is a RW whose jump law is $\widetilde{\mu}$ such that $V_k\eqo V_{m+n} (\tau_\infty) \! -\!  V_{m+n-k} (\tau_\infty)$, $0\leqo k\leqo m+n$. Recall from (\ref{VbarL}) the notation $L$. We easily deduce from (\ref{codheight}) that 
$L_n \eqo H_{m+n} (\tau_\infty) \! -\! \inf_{ \lgeo m, m+n \rgeo} H (\tau_\infty)$. Thus, 
\begin{eqnarray*}
\bE \Big[  \big| H_{m+n} (\tau_\infty) \! -\!\!\!\!\!\!\!\! \!\! \inf_{\; \quad \lgeo m, m+n \rgeo}\!\!\!\!\!\!\!\! H (\tau_\infty) \big|^\ttb  \Big] \!\!\!\! &=&\!\!\!  \int_{0}^\infty \!\! \ttb s^{\ttb-1} \bP \big( L_n \geko s\big) \, ds \, \leq \int_{0}^\infty\!\! \ttb s^{\ttb-1} \bP \big( \cL_{\lceil s \rceil} \leqo n \big) \, ds  \\
\!\!\! &\leq &\!\!\! 
\int_{0}^\infty\!\! \ttb s^{\ttb-1} e\cdot\bE \big[ \exp\big( \! -\! \tfrac{1}{n} \cL_{\lceil s \rceil}\big) \big] ds \leq \int_{0}^\infty\!\! \ttb s^{\ttb-1} e\cdot e^{ -s \Lambda_0 (\frac{1}{n}) } ds,
 \end{eqnarray*}
which implies the desired result. Here, we use Markov inequality at the second line and the fact that $\cL$ is a nonnegative RW whose Laplace exponent is $\Lambda_0$. 

The proof of the second upper bound in (\ref{heightmom}) is slightly more involved. To simplify notation we set $I_{k,l}\eqo \min_{\lgeo k, l\rgeo} V(\tau_\infty)$, for all integers $l\geqo k\geqo 0$. Denote by $m^\prime$ the smallest integer $k\ino \lgeo m,  m+n\rgeo$ such that $V_k(\tau_\infty)\eqo I_{m, m+n}$.  We deduce from (\ref{codheight}) that 
$$\!\!\!\! \!\!\!\!  \inf_{\; \quad \lgeo m, m+n \rgeo} \!\!\!\!\!\! \!\! H (\tau_\infty) \eqo H_{m^\prime} (\tau_\infty) \eqo \# \big\{ l \ino \lgeo 0, m-1\rgeo: V_l(\tau_\infty) \eqo I_{l,m} \; \textrm{and} \; I_{l,m} \leqo I_{m,m+n} \} .$$
We now assume that $V$ is such that 
$V_k\eqo V_{m} (\tau_\infty) \! -\!  V_{m-k} (\tau_\infty)$, $0\leqo k\leqo m$ and we set $V^\prime\eqo (V_{m+k} (\tau_\infty) \! -\! V_{m} (\tau_\infty))_{k\in \bbN}$ that is an independent copy of $V$. We set $I^\prime_n\eqo \min_{\lgeo 0, n\rgeo} V^\prime$. Then, (\ref{codheight}) implies that 
$$ H_{m} (\tau_\infty) \! -\!\!\!\!\!\!\!\! \!\! \inf_{\; \quad \lgeo m, m+n \rgeo}\!\!\!\!\!\!\!\! H (\tau_\infty) = \# \big\{ k \ino \lgeo 1, m\rgeo: V_k \eqo \overline{V}_{\! k}\leko -I_n^\prime \}= L_{m\wedge (\overline{V}^{_{-1}}_{\!\!\! -I_n^\prime} - 1)} \; ,$$
where we have set $\overline{V}^{_{-1}}_{\! k} \eqo \inf \{ l\ino \bbN: \overline{V}_{\! l}\geqo k \}$. We then set $T^\prime_{-k}\eqo \inf \{  l\ino \bbN: V^\prime_{l}\eqo -k \}$ and we deduce from the previous inequality that for all integer $k\ino \bbN^*$,
\begin{eqnarray*}
\bP \big(H_{m} (\tau_\infty) \! -\!\!\!\!\!\!\!\! \!\! \inf_{\; \quad \lgeo m, m+n \rgeo}\!\!\!\!\!\!\!\! H (\tau_\infty)\geqo k \big) & \leqo &  \bP \big( L_{ \overline{V}^{_{-1}}_{\!\!\! -I_n^\prime} - 1} \geqo k \big)  \eqo \bP \big(\! -\! I_n^\prime \geko \cV_k \big) \leqo  
\bP \big(T^\prime_{\! - \cV_k}\leqo n \big) \\ 
&\leq &e \cdot \bE \big[ \exp \big(\! -\! \tfrac{1}{n}T^\prime_{\! - \cV_k} \big)\big] = e \cdot \exp \big(\!  -\! k \Lambda_2\big( \Lambda_1 \big( \tfrac{1}{n}\big)\big) \big),   
\end{eqnarray*}
by the Markov inequality. It entails the following, 
$$ \bE \Big[  \big| H_{ m} (\tau_\infty) \! -\!\!\!\!\!\!\!\! \!\! \inf_{\; \quad \lgeo m, m+n \rgeo}\!\!\!\!\!\!\!\! H (\tau_\infty) \big|^\ttb  \Big] \leq  e \int_{0}^\infty \!\! \ttb s^{\ttb-1} e^{- \lceil s \rceil  \Lambda_2( \Lambda_1 (\frac{1}{n}))}  \, ds,  $$
which easily implies the second upper bound in (\ref{heightmom}). \cqfd 

\medskip

\begin{proposition}
\label{HolderH} We fix $\alpha \ino (1, 2]$ and we assume that $\mu$ satisfies (\ref{hypostaintro}). Recall from (\ref{aennedef}) the definition of $(a_n)_{n\in \bbN^*}$. 
Let $\ttb \ino \big(1 , \infty\big)$ and $s_0 \ino (1, \infty)$ and let $\epp \ino \big( 0, 1\! -\! \frac{1}{\alpha} \big) $. 
Then there exists $c\ino (1, \infty)$ that only depends on $\mu, \ttb, s_0$ and $\epp$ such that 
\begin{equation}
\label{HHolder}
\forall n \ino \bbN^*, \forall s,s^\prime\ino [0, s_0], \quad 
\bE \Big[ \big| \tfrac{1}{a_n}H_{ns^\prime} (\tau_\infty) \! -\! \tfrac{1}{a_n} H_{ns}  (\tau_\infty) \big|^\ttb \Big] 
\leq c \, \big| s^\prime \! -\! s\big|^{\ttb (1-\frac{1}{\alpha} -\epp)} \; .
\end{equation}
Moreover, for all $\gamma\ino \big( 0, 1\! -\! \tfrac{1}{\alpha}\big)$, there exists $c'\ino (1, \infty)$
that only depends on $\mu, \ttb, s_0$ and $\gamma$ such that 
\begin{equation}
\label{LbHolder}
\sup_{n\in \bbN^*} \bE \Big[ \Big( \mathtt{Hol}_{\gamma, s_0} \big(   \tfrac{1}{a_n}H_{ n \cdot } (\tau_\infty) \big) \Big)^{\! \ttb} \Big] \leq c' , 
\end{equation}
where we recall from (\ref{Holdnodef}) the definition of  $\mathtt{Hol}_{\gamma, s_0} (\mathtt x (\cdot))$, for all $\mathtt x \ino \bC^0 (\bbR_+, \bbR)$. In particular it implies that 
\begin{equation}
\label{Lbtotheight}
\sup_{n\in \bbN^*} \bE\Big[ \Big( \max_{s\in [0, s_0]}  \tfrac{1}{a_n}H_{ ns} (\tau_\infty) \Big)^{\! \ttb} \Big] \leq c' s_0^\gamma .
\end{equation}
\end{proposition}
\noi
\textbf{Proof.} We set $\tta (x)= x^{1-\frac{1}{\alpha}} / L(x)$ where $L$ is as in (\ref{bnBiGoTe}). By convenience, we set $\tta (0)\eqo 0$. Thus $\tta (n) \eqo  a_n$ by (\ref{aennedef}). To simplify notation, we set $\cH_s\eqo H_s (\tau_\infty)$, $s\ino \bbR_+$. By (\ref{heightmom}) and Lemma \ref{ascladd}, there exists a constant $c_1(\mu, \ttb)\ino (1, \infty)$ (that only depends on $\mu$ and $\ttb$) such that for all integers  $m\ino \bbN$ and $n\ino \bbN^*$, 
\begin{equation}
\label{tradVR} 
\bE \Big[  \big| \cH_{m+n}  -\!\!\!\!\!\!\!\! \!\! \inf_{\; \quad k\in \lgeo m, m+n \rgeo}\!\!\!\!\!\!\!\!\!\!\!   \cH_k  \; \big|^\ttb  \Big] \vee\bE \Big[  \big| \cH_{m}  -\!\!\!\!\!\!\!\! \!\! \inf_{\; \quad k\in \lgeo m, m+n \rgeo}\!\!\!\!\!\!\!\! \!\!\! \cH_k \;  \big|^\ttb  \Big]  \leq  c_1(\mu, \ttb) \tta (n)^\ttb \; .
\end{equation}

We first prove an upper bound for $|\cH_{s^\prime} \! -\! \cH_s|$ that is a simple consequence of the specific interpolation explained by (\ref{ctrvshght}) in Definition~\ref{Lukadef} $(\textbf{a})$ and that is used when $|s-s^\prime|$ is small: there exists 
$c_2(\mu, \ttb)\ino (1, \infty)$ (that only depends on $\mu$ and $\ttb$) such that for all $s,s^\prime \ino \bbR_+$,
\begin{equation}
\label{Liploc}
\bE \Big[\big|\cH_{s^\prime} \! -\! \cH_s \big|^\ttb \Big] \leq c_2(\mu, \ttb) \, |s^\prime \! -\! s|^\ttb \; .
\end{equation} 
\noi
\emph{Proof.} Let us suppose first that there is $m\ino \bbN$ such that 
$\frac{_1}{^2}m\leqo s \leqo s^\prime \leq \frac{_1}{^2}(m+1)$, and set $n\eqo\lfloor \frac{_1}{^2}m\rfloor$. 
Then, (\ref{ctrvshght}) implies that $|\cH_{s^\prime} \! -\! \cH_s | \eqo 2(s^\prime\! -\! s) |\cH_{\frac{m}{2}} - \cH_{\frac{m+1}{2}}| \leqo 2(s^\prime\! -\! s)(1+\cH_{n}\! -\! \inf_{k\in \lgeo n , n+1\rgeo} \cH_k)$. 

For all real valued r.v.s $X$, we use the notation $\lVert X\rVert_{\mathtt b}\! :=\! (\bE [ |X|^{\mathtt b} )^{1/\mathtt b}$, for the 
$L^{\mathtt b}$ norm of $X$.  
By (\ref{tradVR}), we get $\lVert \cH_{s^\prime} \! -\! \cH_s \rVert_\ttb \leq c (s^\prime \! -\! s) $ where 
$c\! := \! 2+ 2 c_1 (\mu, \ttb)^{1/\ttb} \tta(1)$. If $\lfloor 2s\rfloor  \leko \lfloor 2s^\prime \rfloor $, then we get 
\begin{eqnarray*}
\lVert \cH_{s^\prime} \! -\! \cH_s \rVert_\ttb & \leqo&  \lVert \cH_{s} \! -\! \cH_{\frac{1}{2}\lceil 2s \rceil }\rVert_\ttb + 
\lVert \cH_{s^\prime} \! -\! \cH_{\frac{1}{2}\lfloor 2s^\prime \rfloor} \rVert_\ttb+
\!\! \!\!\!\!\!\!\sum_{\quad \lceil  2s\rceil\leq k < \lfloor 2s^\prime \rfloor} \!\!\!\!\!\! \lVert \cH_{\frac{1}{2}(k+1)} \! -\! \cH_{\frac{1}{2}k} \rVert_\ttb \\
&\leq & c \big( \tfrac{1}{2}  \lceil  2s\rceil \! -\! s  +  s^\prime \! -\! \tfrac{1}{2} \lfloor 2s^\prime \rfloor+  \tfrac{1}{2}\lfloor 2s^\prime \rfloor \! -\! \tfrac{1}{2}\lceil  2s\rceil \big)= c(s^\prime \! -\! s),  
\end{eqnarray*}
which entails (\ref{Liploc}).  \cq 

\medskip

Recall the convention $\tta (0)\eqo 0$. We next prove the following bound: there exists 
$c_3(\mu, \ttb)\ino (1, \infty)$ (that only depends on $\mu$ and $\ttb$) such that for all real numbers $s^\prime\geqo s\geq 0$,  
\begin{equation}
\label{hhoolld}
\bE \Big[\big|\cH_{s^\prime} \! -\! \cH_s \big|^\ttb \Big] \leq c_3(\mu, \ttb) 
\big( 1 + \tta (\lfloor s^\prime\rfloor \! -\! \lfloor s \rfloor  ) \big)^\ttb .
\end{equation} 
\noi
\emph{Proof.} By (\ref{tradVR}) and  (\ref{Liploc}), we get 
\begin{eqnarray*} 
\lVert \cH_{s^\prime} \! -\! \cH_s \rVert_\ttb &\leq& \big| \! \big|  \cH_{s} \! -\! \cH_{\lfloor s\rfloor} \big| \! \big|_\ttb  +  \big| \! \big| \cH_{s^\prime} \! -\! \cH_{\lfloor s^\prime \rfloor}  \big| \! \big|_\ttb+  \big| \! \big|  \cH_{\lfloor s\rfloor}  -\!\!\!\!\!\!\!\! \!\! \inf_{\; \quad k\in \lgeo \lfloor s\rfloor, \lfloor s' \rfloor \rgeo}\!\!\!\!\!\!\!\!\! \!\!\! \cH_k  \; \,  \big| \! \big|_\ttb+  \big| \! \big| \cH_{\lfloor s^\prime \rfloor}  -\!\!\!\!\!\!\!\! \!\! \inf_{\; \quad k\in \lgeo \lfloor s\rfloor, \lfloor s^\prime \rfloor \rgeo}\!\!\!\!\!\!\!\!\!\!\! \!   \cH_k  \; \,  \big| \! \big|_\ttb\\
& \leq & 2c_2(\mu, \ttb)^{1/\ttb}+ 2 c_1(\mu, \ttb)^{1/\ttb}   \tta (\lfloor s^\prime\rfloor \! -\! \lfloor s \rfloor  ), 
 \end{eqnarray*} 
which easily implies (\ref{hhoolld}). 
 \cq 

\medskip

We next compare $\tta (\cdot) $ to power functions thanks to Potter's bounds. Let $s_0$ and $\epp$ as in the statement of the proposition. There exist $n_0 \ino \bbN^*$ and $c_4\ino(1,\infty)$ that only depend on $\mu$, $\epp$ and $s_0$ such that for all integer $n\geq n_0$ and all real number $s\ino [\frac{n_0}{n} , s_0]$,  
\begin{equation}
\label{controaa}
\frac{1}{c_4} s^\epp \leq \frac{L(n)}{L(ns)} \leq c_4\,  s^{-\epp} \quad \textrm{and thus} \quad \frac{\tta(ns)}{\tta(n)} \leq c_4 \, s^{1-\frac{1}{\alpha} -\epp} \; .
\end{equation}
\noi
\emph{Proof.} By Potter's bounds (see e.g.~Bingham, Goldies \& Teugel~\cite[Theorem 1.5.6]{BiGoTe}), there exists 
$n_0 \ino \bbN^*$ that depend on $L$, $\epp$ and $s_0$ such that for all real numbers $y\geqo x \geqo n_0 $, $L(x)/L(y) $ and $ L(y)/L(x)$ belong to $[\frac{_1}{^2} (y/x)^{-\epp}, 2 (y/x)^{\epp}]$. We fix $n\geqo n_0$. If $s\ino  [\frac{n_0}{n} ,1]$, then $L(n)/L(ns) \ino [\frac{_1}{^2} s^{\epp}, 2 s^{-\epp}]$. If $s\ino  [1, s_0]$, then $L(n)/L(ns) \ino [\frac{_1}{^2} s^{-\epp}, 2 s^{\epp}]$ and we observe that $2 s^{\epp}\leq 2s_0^{2\epp} s^{-\epp}$. This implies (\ref{controaa}) 
by taking $c_4\eqo 2s_0^{2\epp}$. \cq 

\medskip

We now complete the proof of (\ref{HHolder}): we fix $s, s^\prime \ino [0, s_0]$ such that $s^\prime \geqo s$ and we suppose that $n\geq n_0$ without loss of generality. We first assume that  $s^\prime \! -\! s \leqo \frac{n_0}{n}$. In this case we use (\ref{Liploc}) to get 
$\lVert \tfrac{1}{a_n} \cH_{ns^\prime} \! -\! \tfrac{1}{a_n} \cH_{ns} \rVert_\ttb \leq c_2 (\mu, \! \ttb)^{1/\ttb} n(s^\prime \! -\! s)/a_n $. Now observe that 
\begin{multline*}
 \tfrac{n}{a_n} |s^\prime\!  -\! s| \leqo n^{\frac{1}{\alpha}} \! L(n) \big(\tfrac{n_0}{n} \big)^{\frac{1}{\alpha}+\epp} |s^\prime\!  -\! s|^{1-\frac{1}{\alpha}-\epp}\leqo  n^{\frac{1}{\alpha}} c_4 L(n_0)  \big(\tfrac{n_0}{n} \big)^{-\epp}  \big(\tfrac{n_0}{n} \big)^{\frac{1}{\alpha}+\epp} |s^\prime\!  -\! s|^{1-\frac{1}{\alpha}-\epp}\\
 \eqo  c_4 L(n_0) n_0^{\frac{1}{\alpha}}  |s^\prime\!  -\! s|^{1-\frac{1}{\alpha}-\epp},
 \end{multline*}
which implies (\ref{HHolder}) with $c\eqo c_2 (\mu, \ttb) (c_4 L(n_0) n_0^{1/\alpha})^{\ttb}$. 

  We next assume that $s^\prime - s \geqo  \frac{n_0}{n}$. It easily implies that $\lfloor ns^\prime \rfloor \! -\! \lfloor ns\rfloor \geqo n_0$. Since $\lim_{n\to \infty} \tta (n)\eqo \infty$, 
$\sup_{n\geq n_0} \tta (n)^{-1} \! = : \! c_5 \leko \infty$. Thus by (\ref{controaa}), it is easy to check the following. 
\begin{eqnarray*}
\frac{1+ \tta (\lfloor ns^\prime \rfloor \! -\! \lfloor ns\rfloor)}{\tta (n)}& \leq & (1+c_5) \frac{\tta (\lfloor ns^\prime \rfloor \! -\! \lfloor ns\rfloor)}{\tta (n)} \leq (1+c_5) c_4  \big( s^\prime \! -\! s + \tfrac{1}{n} \big)^{1-\frac{1}{\alpha} -\epp} \\
& \leq & 2^{1-\frac{1}{\alpha} -\epp}\! (1+c_5) c_4 ( s^\prime \! -\! s )^{1-\frac{1}{\alpha} -\epp}. 
\end{eqnarray*}
Thanks to (\ref{hhoolld}), we get (\ref{HHolder}) with $c\eqo c_3 (\mu, \ttb) \big( 2^{1-\frac{1}{\alpha} -\epp}(1\! +\! c_5) c_4\big)^{\ttb} $. This easily completes the proof of (\ref{HHolder}) by taking $c$ equal to the largest of the two constants.  
Then (\ref{LbHolder}) is a direct consequence of Lemma~\ref{condKolcont} when $\gamma\leko 1\! -\! \frac{1}{\alpha} \! -\! \frac{1}{\ttb}$. When $\gamma\ino [1\! -\! \frac{1}{\alpha} \! -\! \frac{1}{\ttb},1 \! -\! \frac{1}{\alpha})$, we use the elementary inequality $x^\ttb\leqo 1\! +\! x^{2/(1-\frac{1}{\alpha}-\gamma)}$, $x\geq 0$, to get (\ref{LbHolder}) anyway. Then, (\ref{Lbtotheight}) is an easy consequence from (\ref{LbHolder}).  \cqfd

\begin{proposition}
\label{HolderWW} We fix $\alpha \ino (1, 2]$ and we assume that $\mu$ satisfies (\ref{hypostaintro}). Let $\nu$ be a supercritical offspring distribution whose mean is denoted by $\ttm_\nu$. 
Let $\ttb_0\ino \big( 1\! +\!  \frac{\alpha}{{\alpha-1}} , \infty)$. We assume that  $\sum_{k\in \bbN} k^{1+2\ttb_0} \nu(k)\leko \infty$. Let $\ttb \ino \big( 1 \! +\!  \tfrac{\alpha }{\alpha -1} , \ttb_0)$. 
Recall from (\ref{aennedef}) the definition of $(a_n)_{n\in \bbN^*}$. 
Let $s_0 \ino (1, \infty)$ and let $\gamma$ be a positive real number such that $\gamma \leko \tfrac{1}{2}\big(1 \! -\! \frac{1}{\alpha})\big(1 \! -\! \frac{1}{\ttb}) \! -\! \tfrac{1}{2\ttb} $ (there are such ones since $\ttb\geko 1\! +\! \tfrac{\alpha }{\alpha -1}\, $).  
Then there exists $c\ino (1, \infty)$ that only depends on $\mu, \nu , \ttb, s_0$ and $\gamma$ such that for all $n\ino \bbN^*$ and for $s,s^\prime\ino [0, s_0]$, 
\begin{equation}
\label{WWHolder}
 \textrm{$\bP$-a.s.~} \quad 
\bE \Big[ \big| \tfrac{1}{\sqrt{a_n}}\widehat{\cW}^*_{ns^\prime} (\tau_\infty) \! -\!\tfrac{1}{\sqrt{a_n}}\widehat{\cW}^*_{ns} (\tau_\infty)  \big|^{2\ttb} \big| (\bT, \boo) \Big] 
\leq c \, K_{\ttb} (\bT) \, \big| s^\prime \! -\! s\big|^{1+2\ttb\gamma} \; .
\end{equation}
where $K_\ttb (\bT)$ is the $\bP$-a.s.~finite and $(\bT, \boo)$-measurable r.v.~defined in Proposition \ref{quoscillW}. Moreover, 
conditionally given $(\bT, \boo)$, the laws of the processes 
$ (\tfrac{1}{\sqrt{a_n}}\widehat{\cW}^*_{ns} (\tau_\infty) )_{s\in \bbR_+}$, $n\ino \bbN$, are tight on $\bC^0 (\bbR_+, \bbR)$.  
\end{proposition}
\noi
\textbf{Proof.} To simplify the notation,
 we denote by $\mathcal G$ the $\sigma$-field generated by $(\bT, \boo)$, and by $\mathcal G'$ the $\sigma$-field generated by $((\bT, \boo), \tau_\infty)$. We denote by $C$ the contour process $C(\tau_\infty)$ and by $\cH$ the height process $H(\tau_\infty)$ and we also drop $\tau_\infty$ in the contour snake $\cW$ and the height snake $\cW^*$.
Although the proposition concerns the height snake, it is easier to consider first 
the endpoint process of contour snake $\widehat{\cW}$ conditionally given $\mathcal G'$.  

Let us fix integers $0\leqo k \leqo l  \leqo n_0$. We denote by $j$ the smallest integer of $\lgeo k , l\rgeo $ such that 
$C_j \eqo \inf_{s\in [k, l]} C_s $. By definition of BRWs and by definition of the endpoint process of the contour snake $\widehat{\cW}$, conditionally given $\mathcal G'$, $\widehat{\cW}_k-\widehat{\cW}_j$ 
(resp.~$\widehat{\cW}_l-\widehat{\cW}_j$) is distributed as $S_{X_{C_k} }\! -\!  S_{X_{C_j} }$ (resp.~$S_{X_{C_l}  }\! -\!  S_{X_{C_j}}$) where $\overline{X}_n\eqo (\bT, X_n)$, $n\ino \bbN$, stands for a $\ttm_\nu$-biased RW on $\bT$ starting at $X_0\eqo \boo$. Lemma \ref{quoscillW} applies and we $\bP$-a.s.~get  
\begin{eqnarray*}
\bE \big[ |\widehat{\cW}_l \!\!\! \!\!\! & - & \!\!\!  \!\!\! \widehat{\cW}_k |^{2\ttb}  \big| \mathcal G'\big] \leq 
2^{2\ttb -1}\big( \bE \big[ |\widehat{\cW}_k-\widehat{\cW}_j |^{2\ttb}  \big| \mathcal G'\big]   + \bE \big[ |\widehat{\cW}_l-\widehat{\cW}_j |^{2\ttb}  \big| \mathcal G'\big]  \big) \\
& \leq &2^{2\ttb -1}K_\ttb (\bT) \big( C_{k}  \big| C_{k}  \! -\! C_{j} \big|^{\ttb -1} \!  +  C_{l}   \big| C_{l}  \! -\! C_{j} \big|^{\ttb -1}\big)  \\
& \leq & 2^{2\ttb} K_\ttb (\bT) \,  \big(  \!\!\!  \!\! \max_{\quad r\in [0, n_0]} \!\!\!  \!\! C_r  \big)  \!\!  \max_{\; r, r'\in [k, l]} \!     | C_{r}  \! -\! C_{r'}|^{\ttb -1}.
\end{eqnarray*}

Let  $s, s'\ino [0, n_0]$ be such that $s'\geqo s$. Suppose first that $s, s'\ino [k, k+1]$. The specific 
interpolation of the contour snake implies that $| \widehat{\cW}_{s'} \! -\!  \widehat{\cW}_{s}| \eqo |s'\! -\! s| \,
 |\widehat{\cW}_{k} \! -\!  \widehat{\cW}_{k+1} | $. Thus,  
 $\bE \big[ |\widehat{\cW}_{s} \! -\!  \widehat{\cW}_{s'} |^{2\ttb}  \big| \mathcal G'\big] \leqo 2^{2\ttb} K_\ttb (\bT) \, (\max_{ r\in [0, n_0]}  C_r ) \,  |s'\! -\! s|^{\ttb-1}$, since $\max_{\; r, r'\in [k, k+1]}   | C_{r} \! -\! C_{r'}|\eqo 1$ and $2\ttb\geqo \ttb\! -\! 1$. 
 Now suppose that $s, s'\ino [0, n_0]$ are such that $\lceil s \rceil \leqo \lfloor s'\rfloor$. Thus, 
\begin{eqnarray} \label{bluered}
\bE \Big[ |\widehat{\cW}_s -\! \!\!\!\!\!\!\!  & &\!\!\!\! \! \! \!  \widehat{\cW}_{s'}    |^{2\ttb}  \Big| \mathcal G'\Big]  \nonumber \\
\!\!\! & \leq & \!\!\!  
3^{2\ttb -1}2^{2\ttb} K_\ttb (\bT)\,  \big(  \!\!\!  \!\! \max_{\quad r\in [0, n_0]} \!\!\!  \!\! C_r  \big)   \big( | \lceil s \rceil  \! -\! s|^{\ttb-1} \!\!  +  |s'\! -\!  \lfloor s'\rfloor|^{\ttb-1}\!\! + \!\!\! \max_{\; r, r'\in [\lceil s \rceil,  \lfloor s'\rfloor]} \!   | C_{r}  \! -\! C_{r'}|^{\ttb -1} \big) \nonumber \\
\!\!\! & \leq & \!\!\! 6^{2\ttb }  K_\ttb (\bT) \,  \big(  \!\!\!  \!\! \max_{\quad r\in [0, n_0]} \!\!\!  \!\! C_r  \big)  \!\!  \max_{\; r, r'\in [s, s']} \!    | C_{r}   \! -\! C_{r'}|^{\ttb -1}.
\end{eqnarray} 
This inequality easily extends to all $s, s^\prime \ino [0, n_0]$ such that $s\leqo s'$. 

Recall that $\cH$ and $\widehat{\cW}^* $ are derived from  $C$ and $\widehat{\cW} $ by the time-change $\phi_{\tau_\infty} \! :\! \bbR_+ \! \to \! \bbR_+$ that is the increasing continuous bijection defined in Remark~\ref{contord} $(\textbf{d})$. Namely, $\cH_s\eqo C_{\phi_{\tau_\infty} (s)} $ and $\widehat{\cW}^*_s \eqo \widehat{\cW}_{\phi_{\tau_\infty} (s)} $, $s\ino \bbR_+$. Then
for all $s_1, s', s\ino \bbR_+$ such that $s_1\geqo s'\geqo s$, by choosing $n_0\eqo \phi_{\tau_\infty} ( \lceil  s_1 \rceil )$ in (\ref{bluered}), we $\bP$-a.s.~get 
\begin{equation}
\label{toutcondi}
\bE \Big[ |\widehat{\cW}^*_s \! -\!  \widehat{\cW}^*_{s'}    |^{2\ttb}  \Big| \mathcal G'\Big]  \leq   6^{2\ttb }  K_\ttb (\bT)  \!\!\!   \max_{\: r\in [0, \lceil s_1\rceil]} \!\!\!  \!\! \cH_r     \max_{\; r, r'\in [s, s']} \!   | \cH_{r} \! -\! \cH_{r'} |^{\ttb -1} .
\end{equation}
We next  fix $s_0\ino (0, \infty)$ and $0\leko\gamma \leko \tfrac{1}{2}\big(1 \! -\! \frac{1}{\alpha})\big(1 \! -\! \frac{1}{\ttb}) \! -\! \tfrac{1}{2\ttb} $. Note that $\gamma_0\! :=\! \frac{1+2\ttb\gamma}{\ttb -1}\leko 1 \! -\! \frac{1}{\alpha}$. We set 
$Z_n\eqo \mathtt{Hol}_{\gamma_0, s_0} (\tfrac{1}{a_n} \cH_{n \, \cdot }  )$ and $\Gamma_{\! n} \eqo \max_{r\in [0,\lceil s_0\rceil]} 
\tfrac{1}{a_n} \cH_{nr}$. Then (\ref{toutcondi}) entails
$$ \bE \Big[ |\tfrac{1}{\sqrt{a_n}}\widehat{\cW}^*_{ns} \! -\!  \tfrac{1}{\sqrt{a_n}}\widehat{\cW}^*_{ns'}    |^{2\ttb}  \Big| \mathcal G'\Big]  \leq   6^{2\ttb }  K_\ttb (\bT) \, \Gamma_{\! n} Z_n^{\ttb-1} |s\! -\! s'|^{1+2\ttb \gamma} .$$
Since $\gamma_0 \leko 1 \! -\! \frac{1}{\alpha}$, Proposition~\ref{HolderH} applies: there is $c_1\ino (0, \infty)$, which 
only depends on $\ttb$, $\mu$, $s_0$ and $\gamma_0$ such that $\sup_{n\in \bbN^*} (\bE [\Gamma_{\! n}^\ttb] + \bE [Z_n^\ttb])\leq c_1$, which implies $\sup_{n\in\bbN^*} \bE [\Gamma_{\! n} Z_n^{\ttb-1} ] \leqo c_1$ by Hölder inequality. So we obtain
$$ \bE \Big[ |\tfrac{1}{\sqrt{a_n}}\widehat{\cW}^*_{ns} \! -\!  \tfrac{1}{\sqrt{a_n}}\widehat{\cW}^*_{ns'}    |^{2\ttb}  \Big| \mathcal G\Big]  \leq   6^{2\ttb }  K_\ttb (\bT) \, \bE \big[\Gamma_{\! n} Z^{\ttb-1}_n\big]  \,  |s\! -\! s'|^{1+2\ttb \gamma} \leq 6^{2\ttb } c_1  K_\ttb (\bT) \,  |s\! -\! s'|^{1+2\ttb \gamma} , $$ 
since $\tau_\infty$ is independent of $(\bT, \boo)$. This proves (\ref{WWHolder}). Conditional tightness is a consequence of Lemmas~\ref{condKolcont} and \ref{condtight}. \cqfd 

\subsection{Proof of Proposition \ref{harmocv} }
\label{pfharmosec}

We keep the notation and the general assumptions discussed in Section \ref{overviewsec}. 
In particular, we recall (\ref{remindcv}): $\mathscr Q_n \! \to \! ( \mathbf g_1, \mathbf d_1, X,H,H)$ weakly on $\bbR^2\! \times \! \bD (\bbR_+ , \bbR)\! \times \! \bC^0 (\bbR_+ , \bbR)^2$.  
By Skorokhod representation theorem, 
there are processes $\mathscr Q_n'$ with the same law as $\mathscr Q_n$  that converge almost surely to $( 
\mathbf g'_1, \mathbf d'_1, X', H', H')$, which has the same law as $(\mathbf g_1, \mathbf d_1 ,X, H, H)$. 
Without loss of generality, we can assume that $(\Omega, \ccF, \bP)$ is large enough to allow the almost sure convergence to take place. 
Then observe that $(\mathbf g'_1, \mathbf d'_1)$ is necessarily the first excursion interval of $H'$ longer than $1$, and 
$\mathbf g'_1$ and $ \mathbf d'_1$ are thus derived from $H'$ as in Remark \ref{gundundef}. 
Since the Lukasiewicz path (or the height process) of a forest completely encodes it, there actually exists a sequence of GW($\mu$)-forests $\tau_{n, \infty}$, $n\ino \bbN^*$, independent from the environment $(\bT, \boo)$ and such that 
\begin{equation}
\label{cQnpri}
\mathscr Q_n'=  \big(  \tfrac{1}{n} \frak g'_n ,\tfrac{1}{n} \frak d'_n, \tfrac{1}{b_n} V_{\lfloor n\cdot \rfloor} (\tau_{n, \infty}) ,   \tfrac{1}{a_n} H_{ n\cdot } (\tau_{n, \infty})  ,    \tfrac{1}{a_n} C_{2n\cdot} (\tau_{n, \infty})  \big)
\end{equation}
where $\frak g'_n$ and $\frak d'_n$ are derived from $\tau_{n, \infty}$ as $\frak g_n$ and $\frak d_n$ are from $\tau_\infty$ in (\ref{JGDdef}) (roughly speaking $\frak g'_n$ and $\frak d'_n$ are the endpoints of the first excursion of $H_{ n\cdot } (\tau_{n, \infty})$ longer than $n$). Namely we assume that 
\begin{equation} 
\label{skoroderie}
\textrm{$\bP$-a.s.}\qquad 
\mathscr Q_n' \underset{n\to \infty}{ -\!\!\! -\!\!\! \longrightarrow } ( \mathbf g'_1, \mathbf d'_1, X', H' , H')
\end{equation}
in $\bbR^2 \! \times \! \bD(\bbR_+, \bbR)\! \times \! \bC^0(\bbR_+,\bbR)^2$.

Then, for all $n\ino \bbN^*$, 
we introduce the process $( \widehat{\cW}^*_s (\tau_{n, \infty}))_{s\in \bbR_+}$ as a $\bbR$-valued process such that 
\begin{equation}
\label{skoroderie2}
\Big( \tau_{n, \infty} , \big(\widehat{\cW}^*_s  (\tau_{n, \infty}) \big)_{s\in \bbR_+} \Big) \overset{\textrm{(law)}}{=} \Big( \tau_{\infty} , \big(\widehat{\cW}^*_s (\tau_{\infty}) \big)_{s\in \bbR_+ } \Big), 
\end{equation} 
where $\widehat{\cW}^* (\tau_{\infty})$ is defined as previously. 
This section is mainly devoted to the proof of Proposition \ref{skomrgFou} below. Before proving it, we first show that it actually implies  
Proposition \ref{harmocv}. 
\begin{proposition}
\label{skomrgFou} 
We fix $\alpha \ino (1, 2]$. Let $\mu$ satisfy (\ref{hypostaintro}). Let $\nu$ be a supercritical offspring distribution whose mean is denoted by $\ttm_\nu$. We assume that there exists $\ttb \ino (2, \infty)$ such that 
$\sum_{k\in \bbN} k^\ttb \nu(k)\leko \infty$. Let $(\bT, \boo)$ be an invariant GW($\nu$)-tree. 
Let $\tau_{n,\infty}$, $n\ino \bbN$, be GW($\mu$)-forests that are independent from $(\bT, \boo)$ and 
such that (\ref{skoroderie}) holds true. Let $( \widehat{\cW}^*_s (\tau_{n, \infty}))_{s\in \bbR_+}$ be as in (\ref{skoroderie2}). We denote by $(W'_s)_{s\in \bbR_+}$ a one-dimensional Brownian snake starting at $0$ and whose lifetime process is $H'$.  
We recall from (\ref{vutchi}) the definition of $\sigma_\nu$. Then for all $p\ino \bbN^*$ and for all real numbers $0\leqo s_1\leko \ldots \leko s_p$ and $\lambda_1, \ldots, \lambda_p$, $\bP$-almost surely
\begin{equation}
\label{ascvmrghrm}
\lim_{n\to \infty}\bE \Big[\prod_{1\leq j\leq p} \exp \Big( i \tfrac{1}{\sigma_\nu \sqrt{a_n}} \lambda_j  \widehat{\cW}^*_{\! \lfloor ns_j\rfloor} (\tau_{n, \infty}) \Big)  
\, \Big| \, (\bT, \boo) , \tau_{n,\infty} \Big] = \bE \Big[\prod_{1\leq j\leq p} e^{i  \lambda_j  \widehat{W}'_{s_j}  }  
\, \Big| H' \Big] .
\end{equation}
\end{proposition}
\noi
\textbf{Proof.} The proof of Proposition~\ref{skomrgFou} starts below, after the proof of Proposition~\ref{harmocv}.  \cqfd 

\medskip

\noi
\textbf{Proof of Proposition \ref{harmocv}.} As already explained, the time-change (\ref{liminvtc}) and (\ref{intcCWW}) combined with Lemma~\ref{composko} implies that it sufficient to prove conditionally given $(\bT, \boo)$ that the following convergence 
\begin{equation}
\label{cvharmo3/2}
\Big(  \tfrac{1}{n} \frak g_n ,\tfrac{1}{n} \frak d_n ,\tfrac{1}{b_n} V_{\lfloor n \cdot \rfloor } (\tau_\infty)  ,   \tfrac{1}{a_n} H_{n\cdot } (\tau_\infty) ,   \tfrac{1}{ \sigma_\nu \sqrt{a_n}}\, \widehat{\cW}^*_{n\cdot }(\tau_\infty) \Big)
 \underset{n\to \infty}{-\!\!\!-\!\!\!-\!\!\!\longrightarrow} \big( \mathbf g_1, \mathbf d_1, X,  H , \widehat{W} \big)
\end{equation}
holds weakly on $\bbR^2 \! \times \!  \bD (\bbR_+, \bbR) \! \times \! \bC^0 (\bbR_+, \bbR)^2$ equipped with the product topology. 

We recall (\ref{skoroderie}) and we now prove that Proposition \ref{skomrgFou} implies Proposition \ref{harmocv}. 
First recall for all $s\geko 0$ that $\bP$-a.s.~$X'_{s-} \eqo X'_s$. By standard results on Skorokhod topology (see Lemma~\ref{Sko3} $(ii)$ below), it $\bP$-a.s.~implies that $\lim_{n\to \infty}  \tfrac{1}{b_n} V_{\lfloor ns\rfloor}(\tau_{n, \infty})\eqo X'_s$. 
By Proposition \ref{skomrgFou} and dominated convergence, we then get 
for all real numbers $0\leqo s_1\leko \ldots \leko s_p$ and 
$\xi , \xi', \lambda_1, \lambda'_1, \lambda''_1, \ldots, \lambda_p, \lambda'_p, \lambda_p''$,  
\begin{eqnarray}
\label{HcW*mrg}
\textrm{$\bP$-a.s.~}\quad \lim_{n\to \infty}\bE \Big[ e^{i\tfrac{1}{n}\xi \frak g_n+ i\tfrac{1}{n} \xi' \frak d_n}\!\! \prod_{1\leq j\leq p}  \! \!   \! \! \!\!   \! \!  & &  \! \!   \! \!   \! \! \!\!  
e^{
i\tfrac{1}{b_n} \lambda''_j V_{\lfloor ns_j\rfloor}(\tau_{ \infty}) + i \tfrac{1}{a_n} \lambda'_j  H_{ns_j} (\tau_\infty)+ i\tfrac{1}{\sigma_\nu \sqrt{a_n}} \lambda_j  \widehat{\cW}^*_{\! ns_j} (\tau_{ \infty}) } 
\, \Big| \, (\bT, \boo)  \Big]  \nonumber \\
&=& \bE \Big[  e^{i\xi \mathbf g_1+ i\xi' \mathbf d_1} \!\! \prod_{1\leq j\leq p} e^{i\lambda''_j X_{s_j} +i \lambda'_j H_{s_j} + i\lambda_j  \widehat{W}_{s_j}  }   \Big] .
\end{eqnarray}

To simplify notation, we denote by $\mathcal G$ the $\sigma$-field generated by $(\bT, \boo)$. We also denote by $\bD_{\! q}$ the space $\bD(\bbR_+, \bbR^q)$ of cadlag functions equipped with the Skorokhod topology, and we also recall the shorthand notation $\bC^0_1$ for $\bC^0(\bbR_+, \bbR)$ which is equipped with the topology of uniform convergence on every compact interval. Below, product spaces are equipped with product topology and if $E$ is a Polish space, we denote by $\cM_1(E)$ the space of its Borel probability measures equipped with the weak convergence. 
The following elementary 
lemma sets the notation and the basic results we need.  
\begin{lemma}
\label{Sko3} The following properties hold true. 
\begin{compactenum}
\item[$(i)$] The function $\jmath: ( \mathtt u,  \ttx, \mathtt h, \mathtt w )\ino \bbR^2 \! \times \! 
\bD_1 \! \times \! (\bC^0_1)^2 \!\! \mapsto \! 
\big(\mathtt u, (\ttx (s), \mathtt h(s), \mathtt w(s))_{s\in \bbR_+} \big)\ino  \bbR^2 \! \times \!  \bD_3 $ is continuous. 
\item[$(ii)$] Let $s_1, \ldots , s_p\ino \bbR_+$. The function $\phi_{s_1, \ldots, s_p} : \mathtt z\ino \bD_3 \mapsto (\ttz(s_1), \ldots, \ttz(s_p) ) \ino (\bbR^3)^p$ is continuous at $\ttz \ino \bD_3$ such that $\ttz (s_j-)\eqo \ttz(s_j)$, $1\leqo j\leqo p$. 
\item[$(iii)$] Let $k\ino \{1, 2, 3\}$. The function $\pi_k: \big( \mathtt u, (\ttz_1(s),  \ttz_2(s),  \ttz_3(s)  )_{s\in \bbR_+} \big) \ino \bbR^2 \! \times \!  \bD_3\mapsto \ttz_k(\cdot) \ino \bD_1$ is continuous. 
\item[$(iv)$] Let $k\ino \{1, 2, 3\}$. For all $\mathtt Q\ino \cM_1 (\bbR^2 \! \times \! \bD_3)$, we set $\mathtt Q^k\eqo \mathtt Q \circ \pi^{-1}_k \ino \cM_1(\bD_1)$. Then $\mathtt Q \mapsto \mathtt Q^k $ is continuous. 
 \end{compactenum}
\end{lemma}
\noi
\textbf{Proof.} For $(i)$, see e.g.~Jacod \& Shiryaev~\cite[Chapter VI, Proposition 2.2 (b) p.~338]{JaSh02}. For $(ii)$, see e.g.~\cite[Chapter VI, Proposition 2.1 (b5), p.~337]{JaSh02}. The last two points are immediate.  \cqfd 

\medskip

By (\ref{remindcv}), by the independence of $\tau_\infty$ from $(\bT, \boo)$ and by Proposition \ref{HolderWW}, we see that conditionally given $\mathcal G$ the laws of the processes 
$ \big( \tfrac{1}{n} \frak g_n ,\tfrac{1}{n} \frak d_n , \tfrac{1}{b_n} V_{\lfloor n\cdot \rfloor}(\tau_{\infty})  , \tfrac{1}{a_n}  H_{n\cdot } (\tau_\infty), \tfrac{1}{\sigma_\nu \sqrt{a_n}}  \widehat{\cW}^*_{\! n\cdot} (\tau_{ \infty}) )$
are tight on $\bbR^2 \! \times \!\bD_1 \! \times \! (\bC^0_1)^2$. By Lemma \ref{Sko3} $(i)$, it implies that conditionally given $\mathcal G$ the laws of the processes 
$Z^{_{(n)}}_{^{\! }}\! :=\! \big(  \tfrac{1}{n} \frak g_n ,\tfrac{1}{n} \frak d_n , \big( \tfrac{1}{b_n} V_{\lfloor ns\rfloor}(\tau_{\infty})  , \tfrac{1}{a_n}  H_{ns} (\tau_\infty), \tfrac{1}{\sigma_\nu \sqrt{a_n}}  \widehat{\cW}^*_{\! ns} (\tau_{ \infty}) )_{s\in \bbR_+} \big)$ are tight on $\bbR^2 \! \times \bD_3$. 

We next denote by 
$\omega \ino \Omega \mapsto Q_n (\omega, \cdot) \ino \cM_1(\bbR^2 \! \times \! \bD_3)$ a regular version of the law of $Z^{_{(n)}}_{^{\! }}$ conditionally given $\mathcal G$. We denote by $P_n$ the law of $\tfrac{1}{b_n} V_{\lfloor n\cdot \rfloor}(\tau_{\infty})$ under $\bP$, we denote by $P$ the law $X$ under $\bP$ and we denote by $Q_\infty$ the law of $\big( \mathbf g_1, \mathbf d_1, (X_s, H_s, \widehat{W}_s)_{s\in \bbR_+}\big)$ under $\bP$. 
We observe that for $\bP$-almost all $\omega$, $Q^1_n (\omega, \cdot) \eqo P_n (\cdot)$. This, combined with (\ref{HcW*mrg}), entails that there exists $\Omega_0 \ino \mathcal G$ such that $\bP(\Omega_0)\eqo 1$, such that for all 
$\omega \ino \Omega_0$, it holds that the $(Q_n (\omega, \cdot))_{n\in \bbN}$ are tight on $\bbR^2 \! \times \! \bD_3$, that 
$Q_n^1(\omega, \cdot)\eqo P_n (\cdot)$, and that 
 \begin{equation}
\label{reFourconv} \lim_{n \to \infty} \int_{\bbR^2 \! \times \! \bD_3 }  \!\! \!\! \!\! \!\! \!\! \!\!
e^{i \langle \xi, \mathtt u \rangle + i\langle \lambda_1, \mathtt z(s_1) \rangle + \ldots + i \langle \lambda_p, \mathtt z(s_p) \rangle }  Q_n (\omega, d\mathtt u d\ttz ) \!  =\!\! \int_{\bbR^2 \! \times \!  \bD_3} \!\! \!\!   \!\! \!\! \!\! \!\!  e^{ i \langle \xi, \mathtt u \rangle +i\langle \lambda_1, \mathtt z(s_1) \rangle + \ldots + i \langle \lambda_p, \mathtt z(s_p) \rangle } Q_\infty (d\mathtt u d\ttz ) . 
\end{equation}
for all $p\ino \bbN^*$, all $\xi \ino \bbQ^2$, all $s_1, \ldots, s_p\ino \bbQ\cap \bbR_+$ and all $\lambda_1, \ldots , \lambda_p \ino \bbQ^3$. 

We then fix $\omega\ino \Omega_0$ and we denote by $Q(\omega, \cdot) \ino \cM_1 (\bbR^2 \! \times \!  \bD_3)$ 
a weak limit-point of the $ Q_n(\omega, \cdot)$ (there is such one since the $Q_n (\omega, \cdot)$ 
are tight on $\bbR^2 \! \times \! \bD_3$). Namely, there exists an increasing sequence of integers 
$(n_k(\omega))_{k\in \bbN}$ such that $\lim_{k\to \infty}Q_{n_k (\omega)} (\omega, \cdot) \eqo Q(\omega, \cdot)$. 
By Lemma \ref{Sko3} $(iv)$, $Q^1_{n_k (\omega)} (\omega, \cdot)\eqo P_{n_k (\omega)} \to 
Q^1(\omega, \cdot) \eqo P$. Since $X$ has no fixed time discontinuity, it implies for all $s_1, \ldots, s_p\ino \bbR_+$ that $Q(\omega, d\mathtt u d\ttz )$-a.s.~$\ttz (s_j-)\eqo \ttz(s_j)$, $1\leqo j\leqo p$. Then by Lemma \ref{Sko3} $(ii)$, by weak-continuity of the Fourier transform of probability measures and by (\ref{reFourconv}), we get 
 $$ \int_{\bbR^2 \! \times \! \bD_3 }  \!\! \!\! \!\! \!\! \!\! \!\! e^{ i \langle \xi, \mathtt u \rangle +i\langle \lambda_1, \mathtt z(s_1) \rangle + \ldots + i \langle \lambda_p, \mathtt z(s_p) \rangle } Q(\omega, d\mathtt u  d\mathtt z)  = \int_{\bbR^2 \! \times \! \bD_3 }  \!\! \!\! \!\! \!\! \!\! \!\!e^{i \langle \xi, \mathtt u \rangle + i\langle \lambda_1, \mathtt z(s_1) \rangle + \ldots + i \langle \lambda_p, \mathtt z(s_p) \rangle } Q_\infty (d\mathtt u  d\mathtt z) $$ 
for all $p\ino \bbN^*$, all $\xi \ino \bbQ^2$, all $s_1, \ldots, s_p\ino \bbQ\cap \bbR_+$ and all $\lambda_1, \ldots , \lambda_p \ino \bbQ^3$. Thus, it also holds for all $\xi \ino \bbR^2$, all $s_1, \ldots, s_p\ino  \bbR_+$ and all $\lambda_1, \ldots , \lambda_p \ino \bbR^3$, by the continuity of the paths and dominated convergence. Injectivity of the Fourier transform of probability measures then entails that 
$Q (\omega, \cdot  )$ and $Q_\infty (\cdot )$ have the same finite-dimensional marginals, which implies that 
$Q (\omega, \cdot  )\eqo Q_\infty ( \cdot )$. This proves that for all $\omega\ino \Omega_0$, $Q_\infty (\cdot )$ is the only limit point of the $Q_n(\omega, \cdot )$: namely, $Q_n (\omega, \cdot) \! \to \! Q_\infty (\cdot )$, which easily entails (\ref{cvharmo3/2}) by Lemma \ref{Sko3} $(iii)$ and the fact that the topology induced by Skorokhod topology on the set $\bC^0_1$ is that of uniform convergence on every compact interval. This completes the proof of Proposition \ref{harmocv}.  \cqfd

\medskip

Before proceeding to the proof of Proposition \ref{skomrgFou}, we need to introduce some notation and to prove auxiliary results on finite-dimensional marginals of height processes and height snakes. 
 
\medskip

\noi
\textbf{Contour processes of real trees with a finite number of leaves.}  Let us first recall from Example~\ref{rltrcdng} the definition of the real tree $(T_h, d_h, \mathtt r_h)$ coded by a continuous function 
$h: [0, \zeta] \! \to \! \bbR$ and recall that $p_h \! :\!  [0, \zeta]\! \to \! T_h$ stands for the canonical projection. 
We discuss here how to encode 
the subtree $T'\! :=\! \bigcup_{1\leq j\leq p} \lgeo \mathtt r_h, p_h (s_j) \rgeo$ that is spanned by the points 
corresponding to the times $0\leqo s_1\leko \ldots \leko s_p \leko \zeta$ (here, recall from Definition~\ref{errtredef} that for all $\sigma, \sigma'\ino T_h$, $\lgeo \sigma, \sigma'\rgeo$ stands for the geodesic path joining $\sigma$ to $\sigma'$).  
The real tree $T'$ can be viewed as a discrete rooted ordered 
tree (its \emph{skeleton}) equipped with `heights' and if $h(s_j)\wedge h(s_{j+1}) \geko \min_{[s_i, s_{j+1}]} h$, $T'$ has exactly $p$ leaves (i.e., ~points $\sigma \ino T'\backslash \{ \mathtt r_h\}$ such that 
$T'\backslash \{ \sigma \}$ is connected). 
To encode $T'$, we introduce its contour function which represents 
the distance from the root of a particle visiting  $T'$ at unit speed and respecting the contour order of its skeleton. The contour function of $T'$ belongs to the set $\ccC_p$ of functions $\cC\ino \bC^0(\bbR_+, \bbR_+)$ such that 
\begin{compactenum}

\smallskip

\item[$(a)$] $\cC_0\eqo 0$ and $\zeta_\cC\eqo \sup\{ s\ino  \bbR_+: \cC_s \geko 0 \} \leko \infty$,

\smallskip

\item[$(b)$] $\cC$ is a broken line whose slopes are either equal to $1$ or equal to $-1$, that has $p$ local maxima.

\smallskip

\end{compactenum} 

\noi
Namely, there are $r_0\eqo \rho_0\eqo 0 \leko r_1 \leko \rho_1\leko \ldots \leko r_{p-1} \leko \rho_{p-1} \leko r_p \leko \rho_p = \zeta_\cC$ such that for all $s\ino \bbR_+$,
\begin{equation}
\label{rhorjdef}
\cC_s = \int_0^s\!\!\! \un_{J_+} (r) \, dr \, -\!\! \int_0^s \!\!\! \un_{J_-} (r) \, dr  \quad \textrm{where} \;  J_+ \eqo \bigcup_{1\leq j\leq p} [\rho_{j-1}, r_j]   \; \textrm{and} \; J_-\eqo [0, \zeta_\cC] \backslash J_+.
\end{equation}
For all $j\ino\lgeo 1,p\rgeo$, note that $r_j$ (resp.~$\rho_{j-1}$) is the time when the function $\cC$ reaches its $j$-th 
local maximum (resp.~its $j$-th local minimum). Then, observe that $\cC_{\rho_{j-1}}\eqo  \min_{[r_{j-1}, r_j]} \cC$. It shall be convenient to define $\ccC_0$ as the set reducing to the null function.

Let $p\ino\bbN^*$. Note that there is a one-to-one correspondence $F_p$ between the open subset $\mathcal D_p \eqo \big\{ (x_j)_{1\leq j\leq 2p-1}  \ino \bbR_+^{2p-1}: x_{2j-1} \! \wedge  x_{2j+1} \geko x_{2j}, \, 1\leqo j\leqo p\! -\! 1 \big\}$ of $\bbR_+^{2p-1}$ and $\ccC_p$, that is defined as follows: $\cC\eqo F_p(  (x_j)_{1\leq j\leq 2p-1} )$ if 
\begin{equation}
\label{profcC}
\forall j\ino \lgeo 1, p \rgeo, \quad x_{2j-1} = \cC_{r_j}   \quad  \textrm{ and } \quad x_{2j}\eqo \min_{[r_{j} , r_{j+1} ]} \cC\eqo \cC_{\rho_j}\: \textrm{ when $j<p$}.
\end{equation}
We easily check that \emph{$F_p$ and its inverse $F_p^{-1}$ are continuous when $\ccC_p$ is 
equipped with the topology of the uniform convergence on every compact interval.}

\medskip

\noi
\textbf{Decompositions at the first branching time.} Let $\cC \ino \ccC_{p}$ with $p\geqo 2$. Let $\zeta_\cC$, the $r_j$ and the $\rho_j$ be as in (\ref{rhorjdef}). We set 
\begin{equation}
\label{dec1rstbr}
m (\cC)\eqo \min_{[r_1, r_p]} \cC \; \, \,  \textrm{and} \; \, \,  q(\cC)\eqo -2+\#\{s\ino[0,\zeta_\cC] : \cC_s= m(\cC)\},
\end{equation}
and we list the elements of $\{s\ino[0,\zeta_\cC] : \cC_s= m(\cC)\}$ in increasing order as $\gamma_0\leko \gamma_1\leko\ldots\leko \gamma_{q(\cC)+1}$. Note that $\gamma_0\leko r_1\leko \gamma_1\leqo \gamma_{q(\cC)}\leko r_p\leko \gamma_{q(\cC)+1}$, and that the intervals $(\gamma_{k-1}, \gamma_k)$, $1\leqo k \leqo q(\cC)\! +\! 1$, are the connected components of $\{ s\ino \bbR_+ : \cC_s \geko m (\cC) \}$. Since $p\geqo 2$, it holds that $q(\cC)\geq 1$. Then, for all $k\ino\lgeo 1,q(\cC)\rgeo$ and all $s\ino\bbR_+$, we also set
\begin{equation}
\label{dec2suite}
\cC_s [k]^{^{_-}}\! \eqo \cC_{(\gamma_0 +s) \wedge \gamma_k }\!\!  - m (\cC)\quad\textrm{and}\quad \cC_s [k]^{^{_+}}\! \eqo \cC_{(\gamma_k +s) \wedge \gamma_{q(\cC)+1} }\!\!  - m (\cC).
\end{equation}
Now, recall $F_p$ from (\ref{profcC}). If we write $(x_j)_{1\leq j<2p}\! :=\! F_p^{-1}(\cC)$, then we observe that
\begin{equation}
\label{decsuite} 
\cC [k]^{^{_-}}\! \eqo F_{p_k(\cC)}\big((x_j\! -\! m(\cC))_{1\leq j< 2p_k(\cC)}\big) \quad\!\textrm{and}\!\quad \cC[k]^{^{_+}}\! \eqo F_{p-p_k(\cC)}\big((x_{2p_k(\cC)+j}\! -\! m(\cC))_{1\leq j< 2p-2p_k(\cC)}\big) , 
\end{equation}
where $p_k(\cC)=\#\{j\ino\lgeo 1,p\rgeo : r_j\leqo \gamma_k\}$. Here, note that $1\leqo p_k(\cC)\leko p$ and $\rho_{p_k(\cC)}\eqo\gamma_k$.
\smallskip
 
\noi 
In the special case where $p\eqo 1$, we set $m (\cC)\eqo \cC_{r_1}\eqo \max \cC$, $q(\cC)\eqo 1$ and $\cC[1]^{^{_-}}\! \eqo \cC[1]^{^{_+}}\!$ is the null function. For any $k\ino\lgeo 1,q(\cC)\rgeo$, we call the pair of functions $(\cC[k]^{^{_-}}\! ,\cC[k]^{^{_+}}\!)$ the \emph{$k$-th decomposition of $\cC$ at the first branching point}. 

\begin{lemma}
\label{cvfixskel} Let $p\ino \bbN^*$, and let $\cC, \cC^{_{(n)}}_{^{\! }}
\ino \ccC_{p}$, $n\ino \bbN$, be such that $\cC^{_{(n)}}_{^{\! }} \! \to \! \cC$ uniformly on every compact interval.  
Then, $m(\cC^{_{(n)}}_{^{\! }}) \!\!  \to \! m(\cC)$, and there exists a sequence of integers $k_n\ino\lgeo 1,q(\cC)\rgeo$, $n\ino\bbN$, such that:
\begin{compactenum}
\item[$(a)$] $\cC[k_n]^{^{_-}} \!, \cC^{_{(n)}}_{^{\! }}[1]^{^{_-}} \!
\ino \ccC_{p_{k_n}(\cC)}$ and $\cC[k_n]^{^{_+}} \!, \cC^{_{(n)}}_{^{\! }}[1]^{^{_+}} \!
\ino \ccC_{p-p_{k_n}(\cC)}$ when $n$ is large enough,
\smallskip
\item[$(b)$] $\cC^{_{(n)}}_{^{\! }}\! [1]^{^{_-}} \!\! - \cC[k_n]^{^{_-}} \!\! \to \! 0$ and $\cC^{_{(n)}}_{^{\! }}\! [1]^{^{_+}} \!\! - \cC[k_n]^{^{_+}} \!\! \to \! 0$ uniformly on every compact interval.
\end{compactenum}
\end{lemma} 
\noi
\textbf{Proof.} Note that $F^{-1}_{p} \! (\cC^{_{(n)}}_{^{\! }})\! =:\! (x^{_{(n)}}_{j})_{1\leq j\leq 2p-1} \! \to \! (x_{j})_{1\leq j\leq 2p-1}\! :=\! F_p^{-1} (\cC)$. Thus, 
$m(\cC^{_{(n)}}_{^{\! }}) \eqo 
\min_{1\leq j\leq 2p-1} x^{_{(n)}}_{j}\to \min_{1\leq j\leq 2p-1} x_j 
\eqo m(\cC)$. Let us set $p^{_{(n)}}_{^{\! }}\eqo p_1( \cC^{_{(n)}}_{^{\! }})$ to lighten the notation. 
We construct the sequence $(k_n)_{n\in\bbN}$ as follows: if $x_{2p^{(n)}} \! \neq\! m(\cC)$, then 
we set 
$k_n=1$; if $x_{2p^{(n)}} \! =\! m(\cC)$ then there is $k\ino\lgeo 1,q(\cC)\rgeo$ such that 
$p^{_{(n)}}_{^{\! }}\eqo p_k(\cC)$, and we set $k_n=k$. By definition, we have 
$x^{_{(n)}}_{2p^{(n)}}\! =\! m(\cC^{_{(n)}}_{^{\! }})$, so it follows that $|x_{2p^{(n)}} \! -\! m(\cC)|$ 
converges to $0$. But since the latter sequence can only take on a finite number of values, it is 
eventually stationary at $0$. When $n$ is large enough, it thus holds $ p_1( \cC^{_{(n)}}_{^{\! }})\eqo p_{k_n}(\cC)$. 
We complete the proof by (\ref{decsuite}) and by continuity of $F_{\! p_k(\cC)}$. \cqfd 

\medskip

\noi
\textbf{Finite-dimensional marginals of height snakes.} We next use the decompositions at the first branching point to compute the finite-dimensional marginals of Brownian snakes and BRWs. More precisely, we fix $p\ino \bbN^*$ and $\cC\ino \ccC_{p}$, and we denote by $(\mathtt W_s(\cC))_{s\in \bbR_+}$ the one-dimensional Brownian snake starting at $0$ whose lifetime process is $\cC$ (see Definition~\ref{Brosnadef}). 
Then we introduce the following notation for all $\lambda_1, \ldots, \lambda_p \ino \bbR$: 
\begin{equation}
\label{Fousnaky}
\ccL (\cC\, ; \, \lambda_1, \ldots , \lambda_p )= \bE \Big[\prod_{1\leq j\leq p} e^{i\lambda_j \widehat{\mathtt W}_{r_j}  (\cC)} \Big] 
\end{equation}
where $r_j$ stands for the time when $\cC$ reaches its $j$-th local maximum.
As an easy consequence of the snake property, for all $k\ino\lgeo 1,q(\cC)\rgeo$ with $q(\cC)$ as in (\ref{dec1rstbr}), we get 
\begin{equation}
\label{Brsnakdec}
\ccL (\cC\, ; \, \lambda_1, \ldots , \lambda_p )= e^{-\frac{_1}{^2} m(\cC) \overline{\lambda}^2} \ccL \big( \cC[k]^{^{_-}} ; \, (\lambda_j)_{1\leq j\leq p_k(\cC)} \big)\, \ccL \big( \cC[k]^{^{_+}} ; \, (\lambda_j)_{p_k(\cC)+1\leq j\leq p} \big) \; .
\end{equation}
where $(\cC[k]^{^{_-}}\! ,\cC[k]^{^{_+}}\!)$ is the $k$-th decomposition of $\cC$ at its first branching time as in (\ref{dec2suite}) and (\ref{decsuite}), 
where $m(\cC)\eqo \min_{[r_1, r_p]} \cC$,
and where $\overline{\lambda}\eqo \lambda_1+ \ldots+ \lambda_p$.

 Let us introduce similar notation for $\bT$-valued BRWs, where $(\bT, \boo)$ stands here for an invariant GW tree whose offspring distribution $\nu$ is supercritical. We denote by $\ttm_\nu$ its mean and we assume that there is $\ttb \ino (2, \infty)$ such that  $\sum_{k\in \bbN} k^\ttb \nu(k)\leko \infty$. 
To that end, we suppose that there exists $a\ino (0, \infty)$ and a discrete tree $t\ino \bbT$ such that  
\begin{equation} 
\label{hypocodcnt}
C_s (t)\eqo a\, \cC_{s/a} , \quad s\ino \bbR_+ \; 
 \end{equation}
where $C_\cdot(t)$ stands for the contour function of the rooted ordered tree $t$ as in Definition~\ref{contourdef} ($\textbf b$). 

We then fix $x\ino \bT$ and we denote by $(Y_u)_{u\in t}$ a $t$-indexed and $\bT$-valued $\ttm_\nu$-biased BRW starting at $Y_{\varnothing}$. Namely, conditionally given $(\bT, x)$, we $\bP$-a.s.~get 
$\bP\big(  \forall u\ino t, Y_u\eqo  y_u | (\bT, \boo)\big) \eqo \un_{\{y_\varnothing = x\}} \prod_{u\in t\backslash \{ \varnothing \}} p_{\bT, \ttm_\nu} (y_{\overleftarrow{u}}, y_u) $ for all $y_u\ino \bT$, $u\ino t$ (recall from (\ref{deflambRW}) the definition of the transition probabilities $p_{\bT, \ttm_\nu}( \cdot, \cdot) $ of the  $\ttm_\nu$-biased RW on $\bT$).    
Here we recall from (\ref{defSx}) that $(S_x)_{x\in \bT}$ stand for the harmonic coordinates of $(\bT, \boo)$. Then, we introduce the following notation for all $\lambda_1, \ldots, \lambda_p \ino \bbR$: 
\begin{equation}
\label{FoubrRWy}
\widetilde{\ccL}_{\bT, x} (a\, \cC_{\frac{\cdot}{^a}}\, ; \, \lambda_1, \ldots , \lambda_p )= \bE \Big[ \prod_{1\leq j\leq p} e^{i\lambda_j \big( S_{Y_{v(j)}} -S_{Y_{\varnothing}}\big) } \Big| (\bT, x) \Big] 
\end{equation}
where the vertex $v(j)\ino t$ 
is the $j$-th leaf of $t$ in increasing contour order.  As a consequence of the definition of BRWs, for all $k\ino\lgeo 1,q(\cC)\rgeo$ with $q(\cC)$ as in (\ref{dec1rstbr}), we
$\bP$-a.s.~get 
\begin{multline}
\label{brRWdec}
\widetilde{\ccL}_{\bT,x} (a\, \cC_{\frac{\cdot}{^a}}\, ; \, \lambda_1, \ldots , \lambda_p)= E_{\bT, x} \Big[ \, 
e^{i\overline{\lambda} \big( S_{X_{am(\cC)} }  -S_{X_0}\big)}\\
\times\widetilde{\ccL}_{\bT, X_{am(\cC)}}  \big( a\, \cC_{\frac{\cdot}{^a}} [k]^{^{_-}} ; \, (\lambda_j)_{1\leq j\leq p_k(\cC)} \big) \, \widetilde{\ccL}_{\bT, X_{am(\cC)}}  \big( a\, \cC_{\frac{\cdot}{^a}} [k]^{^{_+}} ; \, (\lambda_j)_{p_k(\cC)+1\leq j\leq p} \big) \, \Big] \; ,
\end{multline}
where $\overline{X}_{n}\eqo (\bT, X_n)$, $n\ino \bbN$, stands for a $\ttm_\nu$-biased RW starting at $X_0\eqo x$ (and recall that $P_{\bT, x} $ is the conditional law of $(\overline{X}_n)_{n\in \bbN}$ given $(\bT, x)$).  
Here we have set $\overline{\lambda}\eqo \lambda_1+ \ldots+ \lambda_p$ and we recall that $(\cC[k]^{^{_-}}\! ,\cC[k]^{^{_+}}\!)$ is the $k$-th decomposition of $\cC$ at its first branching time: $m(\cC)\eqo \min_{[r_1, r_p]} \cC$. 
Note that $am (\cC)$ is the height of the $\leq_{\mathtt{lex}}$-smallest vertex $u\ino t$ such that $k_u(t) \! \neq \! 1$ (which is also the first time that this vertex is visited in increasing contour order).  

One of the main steps of the proof of Proposition~\ref{skomrgFou} is the following lemma, which relies on decompositions at the first branching point and on 
the specific CLT in Proposition~\ref{fourcondharmo}. 
\begin{lemma}
\label{cvmargiform}
We keep the above assumptions and notation. We fix $p\ino \bbN^*$. Let $\cC, \cC^{_{(n)}}_{^{\! }}\ino \ccC_{p}$ and $r_n \ino \bbN$, $n\ino \bbN$. Here, $\overline{X}_{n}\eqo (\bT, X_n)$, $n\ino \bbN$, stands for a $\ttm_\nu$-biased RW starting at $X_0\eqo \boo$. We assume the following. 
\begin{compactenum}

\smallskip

\item[$(a)$] $\cC^{_{(n)}}_{^{\! }}\!\!  \to \cC$ uniformly on every compact interval. 

\smallskip
 
\item[$(b)$] For all $n\ino \bbN$, there exist $a_n \ino (0, \infty)$ and $t_n\ino \bbT$ such that 
$C_s (t_n)\eqo a_n\, \cC^{_{(n)}}_{s/a_n}$, $s\ino \bbR_+$.

\smallskip  
 
\item[$(c)$] There exists $r\ino \bbR_+$ such that $r_n/a_n \! \to \! r$. 

\smallskip
 
\end{compactenum}
Then for all $\lambda_1, \ldots, \lambda_p \ino \bbR$, it holds $\bP$-a.s.~that
\begin{equation}
\label{margiformcv}
\widetilde{\ccL}_{\bT,X_{ r_n}} (a_n\,  \cC^{_{(n)}}_{\frac{\cdot}{^{a_n}}}\, ; \, \tfrac{1}{\sigma_\nu\sqrt{a_n}} \lambda_1, \ldots ,  \tfrac{1}{\sigma_\nu\sqrt{a_n}} \lambda_p) \underset{n\to \infty}{-\!\!\! -\!\!\! -\!\!\! -\!\!\! \longrightarrow } \ccL (\cC\, ; \, \lambda_1, \ldots , \lambda_p ),
\end{equation}
where we recall from (\ref{vutchi}) the definition of $\sigma_\nu$. 
\end{lemma}
\noi
\textbf{Proof.} We argue recursively on $p$. If $p\eqo 1$, then the trees $t_n$ are lines and (\ref{margiformcv}) is a direct consequence of Proposition~\ref{fourcondharmo}. We then assume that $p\geqo 2$ and that the lemma holds true for all $p'\leko p$. Lemma~\ref{cvfixskel} yields that $m(\cC^{_{(n)}}_{^{\! } })\! \to \! m(\cC) $ and gives a specific sequence $k_n\ino\lgeo 1,q(\cC)\rgeo, n\ino\bbN$. We fix $k\ino\lgeo 1,q(\cC)\rgeo$ which is a subsequential limit of $(k_n)$, and denote by $(\cC^{_{(n)}}_{^{\! }}\! [1]^{^{_-}}\!,\cC^{_{(n)}}_{^{\! }}\! [1]^{^{_+}}\!)$ and $(\cC[k]^{^{_-}}\! ,\cC[k]^{^{_+}}\!)$ resp.~the first and the $k$-th decomposition at the first branching point of resp.~$\cC^{_{(n)}}_{^{\! }}\!$, $\cC$. 

Observe that 
$(r_n + a_n m(\cC^{_{(n)}}_{^{\! } }))/a_n \! \to\! r+ m(\cC)$. Also recall from Lemma~\ref{cvfixskel} that along the subsequence of integers $n$ such that $k_n=k$, $\cC^{_{(n)}}_{^{\! }}\! [1]^{^{_\pm}}\!\! \to \! \cC[k]^{^{_\pm}}\!$ 
uniformly on every compact interval. For all such $n$ large enough, the common number of local maxima of $\cC^{_{(n)}}_{^{\! }}\! [1]^{^{_\pm}}\!$ and $\cC[k]^{^{_\pm}}\!$ is smaller than $p$, so our recurrence assumption applies and we $\bP$-a.s.~get
\begin{align*}
\widetilde{\ccL}_{\, \bT, X_{\! r_n +a_n m(\cC^{(n)})}} \!  \Big( a_n\, \cC^{_{(n)}}_{\frac{\cdot}{^{a_n}}} [1]^{^{_-}} ;  \Big( \tfrac{\lambda_j}{\sigma_\nu\sqrt{a_n}}\Big)_{\! 1\leq j\leq p_1(\cC^{(n)})} \Big) &\xrightarrow[\substack{n\to \infty\\ k_n=k}]{}
\ccL \big( \cC[k]^{^{_-}} ; \, (\lambda_j)_{1\leq j\leq p_k(\cC)} \big) , \\
\widetilde{\ccL}_{\, \bT, X_{\! r_n +a_n m(\cC^{(n)})}} \!  \Big( a_n\, \cC^{_{(n)}}_{\frac{\cdot}{^{a_n}}} [1]^{^{_+}} ;  \Big( \tfrac{\lambda_j}{\sigma_\nu\sqrt{a_n}}\Big)_{\! p_1(\cC^{(n)})+1\leq  j\leq p} \Big) &\xrightarrow[\substack{n\to \infty\\k_n=k}]{} 
\ccL \big( \cC[k]^{^{_+}} ; \, (\lambda_j)_{p_k(\cC)+1\leq j\leq p } \big) .
\end{align*}
To simplify, we denote by $L_n$ (resp.~$L[k]$) the product of the left-hand (resp.~right-hand) sides of the two previous convergences. Let us also denote by $\ccF_{n}$ the $\sigma$-field generated by $(\overline{X}_m)_{0\leq m\leq n}$. Then, we $\bP$-a.s.~get $\lim_{n\to \infty, k_n=k}\bE \big[|L[k] \! -\! L_n| \, \big| \, \ccF_{r_n}\big]\eqo 0$ by the technical Lemma~\ref{exozarb} in Appendix.

Note that $L[k]$ is not random. Then, combining the decomposition (\ref{brRWdec}) with the Markov property for $\overline{X}$ at time $r_n$ implies $\bP$-a.s.~that 
\[\Big| \widetilde{\ccL}_{\, \bT,X_{r_n}} (a_n\, \cC^{_{(n)}}_{\frac{\cdot}{^{a_n}}}\,  ; \, \tfrac{1}{\sigma_\nu\sqrt{a_n}} \lambda_1, \ldots ,  \tfrac{1}{\sigma_\nu\sqrt{a_n}} \lambda_p)   - L[k]\, 
E_{\bT, X_{r_n}}\! \Big[ \exp \Big( i\tfrac{\overline{\lambda}}{\sigma_\nu \sqrt{a_n}} \big( S_{X'_{a_n m(\cC^{(n)})} }  \!\! -\! S_{X'_0}\big) \Big) \Big] \Big|\]
is smaller or equal to $\bE \big[|L_n \! -\! L[k]| \, \big| \, \ccF_{r_n}\big]$, where $\overline{X}'_n\eqo(\bT, X_n')$, $n\ino \bbN$, stands here for an auxiliary $\ttm_\nu$-biased RW. We then apply Proposition~\ref{fourcondharmo} to get $\bP$-a.s.~
\begin{eqnarray*}  E_{\bT, X_{r_n}}\! \Big[ \exp \Big( i\tfrac{\overline{\lambda}}{\sigma_\nu \sqrt{a_n}} \big( S_{X'_{a_n m(\cC^{(n)})} }  \!\! -\! S_{X'_0}\big) \Big) \Big] \!\! \!\! \!\! &= & \!\! \!\! \!\!  \bE \Big[\exp \Big( i\tfrac{\overline{\lambda}}{\sigma_\nu \sqrt{a_n}} \big( S_{X_{r_n +a_n m(\cC^{(n)})} }  \!\! -\! S_{X_{r_n}}\Big)  \Big| \ccF_{r_n}\Big]  \\
\!\!  &  \underset{n\to \infty}{ -\!\!\! -\!\!\! \longrightarrow }  &  \!\! \exp \big(\! -\! \tfrac{1}{2} m(\cC) \overline{\lambda}^2 \big).
\end{eqnarray*}
Therefore, we get that $\bP$-a.s.~
$$ \widetilde{\ccL}_{\, \bT,X_{r_n}} (a_n\, \cC^{_{(n)}}_{\frac{\cdot}{^{a_n}}} \, ; \, \tfrac{1}{\sigma_\nu\sqrt{a_n}} \lambda_1, \ldots ,  \tfrac{1}{\sigma_\nu\sqrt{a_n}} \lambda_p) \;   \xrightarrow[n\to \infty,\, k_n=k]{} \;  e^{-\frac{_1}{^2} m(\cC) \overline{\lambda}^2} L[k]= \ccL (\cC\, ; \, \lambda_1, \ldots , \lambda_p )$$
thanks to the decomposition (\ref{Brsnakdec}). Since $\{n\ino\bbN : k_n=k\}$, $1\leqo k\leqo q(\cC)$, is a finite partition of $\bbN$, this completes the proof of (\ref{margiformcv}), and thus of the lemma. \cqfd 

\medskip

\noi
\textbf{Proof of Proposition \ref{skomrgFou}.} To that end, we need to relate the previous results to the finite-dimensional 
marginals of the endpoint process of the height snake in harmonic coordinates and to the finite-dimensional marginals of the endpoint process of one-dimensional Brownian snake $\widehat{W}'$ whose lifetime process is $H'$ as in (\ref{skoroderie}).

Recall from (\ref{profcC}) the definition of $F_p$ for all $p\ino\bbN^*$. For all $h\ino \bC^0 (\bbR_+,\bbR_+)$ and for all $s_1\geko 0$ such that $h(s_1)\geko 0$, we set $\cC (h, s_1) \eqo F_1 (h(s_1))$. More generally, for all $p\ino \bbN^*$ with $p\geqo 2$ and for all real numbers $0\eqo s_0\leko s_1 \leko \ldots \leko s_p$ such that the following condition 
\begin{equation}
\label{pleaveshypo}
\forall j\ino\lgeo 1,p\! -\! 1\rgeo,\quad h(s_j)\wedge h(s_{j+1}) \geko \min_{[s_j, s_{j+1}]} h 
\end{equation}
is satisfied, we set
\begin{equation}
\label{Chsj}
F_p \Big( h(s_1), \!\! \min_{\; [s_1, s_2]} \!\! h,  h(s_{2}), \ldots,    h(s_{p-1}), \!\!\! \!\! \min_{\; \;\,  [s_{p-1}, s_p]} \!\! \!\!\!h\, ,  h(s_{p}) \Big)=: 
\cC (h; s_1, \ldots, s_p) \, \in \ccC_p.
\end{equation}
As explained at the beginning of the section, $\cC (h; s_1, \ldots, s_p)$ is the contour process of the subtree $\bigcup_{1\leq j\leq p} \lgeo \mathtt r_h, p_h (s_j) \rgeo$ in $(T_h, d_h, \mathtt r_h)$.

We know from Lemma~\ref{branchmass} that $(H';s_1,\ldots,s_p)$ $\bP$-a.s.~satisfies (\ref{pleaveshypo}). As an easy consequence of Definition~\ref{Brosnadef} of finite-dimensional marginals of Brownian snakes, we see that for all $\lambda_1, \ldots, \lambda_p \ino \bbR$, $\bP$-a.s.~
\begin{equation}
\label{linkbrosna}
\bE \Big[\prod_{0\leq j\leq p} e^{i\lambda_j \widehat{W}'_{s_j}  } \Big| (H'_s)_{s\in \bbR_+} \Big] =  \ccL \big( 
\cC (H'; s_1, \ldots, s_p)\, ; \, \lambda_1, \ldots , \lambda_p ) 
\end{equation}
where we recall from (\ref{Fousnaky}) the definition of the right member.

Similarly, in the discrete setting: we recall that the GW($\mu$)-forests trees $\tau_{n, \infty}$, $n\ino \bbN^*$, are 
such that the almost sure limit (\ref{skoroderie}) holds true and recall from (\ref{skoroderie2}) 
the definition of the joint law of $(\tau_{n, \infty} , \widehat{\cW}^* (\tau_{n,\infty}))$. The almost sure 
convergence (\ref{skoroderie}) $\bP$-a.s.~implies that 
\begin{equation}
\label{festocv}
\lim_{n\to \infty}  \tfrac{1}{a_n} H_{\lfloor ns_j \rfloor} (\tau_{n,\infty}) \eqo H'_{s_j} \quad \textrm{and} \quad  
\lim_{n\to \infty}  \tfrac{1}{a_n}  \!\! \!\!\!\!\!\!\!\!\!  \min_{_{\qquad [ \lfloor ns_j \rfloor, \lfloor ns_{j+1} \rfloor ]}} \!\!\!\!\!\!\!\!\!\!\!\! \!\!  H(\tau_{n, \infty}) 
=\min_{_{[s_j, s_{j+1}]}} H' . 
\end{equation}
Combined with Lemma~\ref{branchmass}, this $\bP$-a.s.~entails for all sufficiently large $n$, that (\ref{pleaveshypo}) is satisfied by $( a_n^{-1}H (\tau_{n, \infty}); \lfloor ns_1 \rfloor, \ldots , \lfloor ns_p \rfloor)$. This allows us to define 
\begin{equation}
\label{stepdefCn}
\cC^{_{(n)}}_{^{\! }}  = \cC \big( \frac{_{_1}}{^{^{a_n}}}H (\tau_{n, \infty}); \lfloor ns_1 \rfloor, \ldots , \lfloor ns_p \rfloor \big).
\end{equation}
As a consequence of the definition of the height snake $\widehat{\cW}^* (\tau_{n, \infty})$, we see that 
\begin{eqnarray}
\label{linkheightsna}
\bE \Big[\prod_{0\leq j\leq p} \exp \Big( i \tfrac{1}{\sigma_\nu \sqrt{a_n}} \lambda_j  \!\! \!\! \!\! \!\! & &\!\! \!\! \!\! \!\!  \widehat{\cW}^*_{\! \lfloor ns_j\rfloor} (\tau_{n, \infty}) \Big)  
\, \Big| \, (\bT, \boo) , \tau_{n,\infty} \Big]  \nonumber \\
&= & \widetilde{\ccL}_{\, \bT,\boo}  \Big( a_n \cC^{_{(n)}}_{\frac{\cdot}{^{a_n}}} \, ; \,
 \tfrac{1}{\sigma_\nu \sqrt{a_n}} \lambda_1, \ldots ,   \tfrac{1}{\sigma_\nu \sqrt{a_n}} \lambda_p \Big).
\end{eqnarray}
Moreover, since the function $F_p$ is continuous, we get from (\ref{festocv}) that  $\cC^{_{(n)}}_{^{\! }}\!\!\!  \to\! \cC (H'; s_1, \ldots, s_p) $ uniformly on every compact interval. 
Then Proposition~\ref{skomrgFou} follows by applying Lemma~\ref{cvmargiform} to (\ref{linkbrosna}) and (\ref{linkheightsna}).   \cqfd

\subsection{Proof of Proposition \ref{harmoclose}}
\label{harmoclosepfsec}
We keep the notation and the general assumptions discussed in Section \ref{overviewsec}. 
In particular, we recall that $\tau_{\infty}$ is a GW($\mu$)-forest corresponding to the sequence $(\tau^{_{(k)}}_{^{\!}})_{k\in \bbN^*}$ of i.i.d.~GW($\mu$)-trees (Definition \ref{GWfordef} $(\textbf{b})$): namely 
$\theta_{(k)}\tau_\infty \eqo \tau^{_{(k)}}_{^{\!}}$, for all $k\ino \bbN^*$. 
We recall that conditionally given the environment $(\bT, \boo)$ and the genealogy $\tau_\infty$, for all $k\ino \bbN^*$, 
$\fTheta_k\eqo \big( \tau^{_{(k)}}_{^{\,}}\!\! \! , \, \varnothing; \overline{Y}^{_{(k)}}_{u}\eqo (  \bT, Y^{_{(k)}}_{u}), u\ino \tau^{_{(k)}}_{^{\!}}  \big)$ are independent $\bT$-valued $\ttm_\nu$-biased BRWs such that $Y^{_{(k)}}_{\varnothing}\eqo \boo$. It is 
convenient to define a single $\tau_\infty$-indexed BRW $(Y_u)_{u\in \tau_\infty}$ by setting $Y_\varnothing \eqo \boo$ and  
\begin{equation}
\label{singlebrRW}
\forall k\ino \bbN^*, \; \forall u\ino \tau^{_{(k)}}_{^{\!}}\! , \quad Y_{(k)\ast u}:=Y^{_{(k)}}_{u} \; .
\end{equation}
We recall that $(u_n)_{n\in \bbN \cup \{ -1\}}$ are the vertices of $\tau_\infty$ in lexicographical order: namely, $u_{-1}\eqo \varnothing$, $u_0\eqo (1)$, etc. 
This section is devoted to the proof of the following result which implies Proposition \ref{harmoclose}. Before proving Proposition \ref{qharmoclose}, we show that it implies Proposition \ref{harmoclose}. 
\begin{proposition}
\label{qharmoclose} We keep the assumptions and the notation of Proposition \ref{harmoclose}. Then, there are constants 
$c_0,c'_0, \epp_0, \epp_1 \ino (0, \infty)$ which only depend on $\mu$, $\nu$ and $\ttb$, such that for all $n\ino \bbN^*$,
\begin{equation}
\label{quantclose}
\bP \Big( \!\!\!\! \max_{\quad m\in \lgeo 0, n\rgeo } \big| \tfrac{1}{\sigma^2_{\nu}} S_{Y_{u_m}} \! -\! \bbn  Y_{u_m}\bbn   \big| \, > \, c'_0 n^{-\epp_1} \sqrt{a_n} \,  \Big) \leq  c_0n^{-\epp_0} .
\end{equation}
\end{proposition}
\noi
\textbf{Proof.} The proof of Proposition \ref{qharmoclose} is given below, right after the proof of Proposition \ref{qharmoclose}. \cqfd 

\medskip

\noi
\textbf{Proof of Proposition \ref{harmoclose}.} We admit Proposition \ref{qharmoclose} and we prove that 
it implies Proposition \ref{harmoclose}. By (\ref{quantclose}), we first see that 
$$ \bE \Big[ \sum_{j\in \bbN} \bP \Big( \!\!\!\!\!  \max_{\quad m\in \lgeo 0, 2^j\rgeo } \! \big| \tfrac{1}{\sigma^2_{\nu}} S_{Y_{u_m}} \! -\! \bbn  Y_{u_m}\bbn   \big| > c'_0 2^{-\epp_1 j} \sqrt{a_{2^j}} \,\,  \Big| \,  (\bT, \boo)  \Big)\;  \Big] < \infty $$
which entails $\bP$-a.s.~that 
$\lim_{j\to \infty} \bP \big( \max_{m\in \lgeo 0, 2^j\rgeo } \big| \tfrac{1}{\sigma^2_{\nu}} S_{Y_{u_m}} \! -\! \bbn  Y_{u_m}\bbn   \big| \geko c'_0 2^{-\epp_1 j} \sqrt{a_{2^j}}  \big|  (\bT, \boo)  \big) \eqo 0$. 

 We next fix $s_0\ino (1, \infty)$. 
Since $n^{-\epp_1} \sqrt{a_n}$ is a regularly varying sequence, there is $c\ino (0, \infty)$ (which only depends on $\mu$, $\nu$ and $s_0$) such that $cn^{-\epp_1} \sqrt{a_n} \geqo  c'_0 2^{-\epp_1 j} \sqrt{a_{2^j}}$ for all integers $n$ such that 
$2^{j-1}\leqo \lceil ns_0 \rceil  \leq 2^j$. Since 
\begin{eqnarray*}
\bP \big( \!\!\!\!\!  \max_{\quad m\in \lgeo 0, \lceil ns_0 \rceil \rgeo } \! \big| \tfrac{1}{\sigma^2_{\nu}} S_{Y_{u_m}} \! -\! \bbn  Y_{u_m}\bbn   \big| \!\!\!\!\!  &>& \!\!\!\!\!   c n^{-\epp_1} \sqrt{a_n} \, \big| \,  (\bT, \boo)  \big) \\
&\leq &\bP \big( \!\!\!\!\!  \max_{\quad m\in \lgeo 0, 2^j\rgeo } \! \big| \tfrac{1}{\sigma^2_{\nu}} S_{Y_{u_m}} \! -\! \bbn  Y_{u_m}\bbn   \big| > c'_02^{-\epp_1 j} \sqrt{a_{2^j}} \, \big| \,  (\bT, \boo)  \big) ,
\end{eqnarray*}
we $\bP$-a.s.~get 
$\lim_{n\to \infty} \bP \big( \max_{m\in \lgeo 0, \lceil ns_0 \rceil  \rgeo } \big| \tfrac{1}{\sigma^2_{\nu}} S_{Y_{u_m}} \! -\! \bbn  Y_{u_m}\bbn   \big| \geko  cn^{-\epp_1 } \sqrt{a_{n}}  \big|  (\bT, \boo)  \big) \eqo 0$. Thanks to the specific interpolation 
defining the contour  and height snakes, we next get 
$$ \max_{\quad s\in [0, s_0] }\! \big| \tfrac{1}{\sigma^2_{\nu}} \widehat{\cW}^*_{ns} (\tau_\infty) \! -\!   \widehat{W}^*_{ns} (\tau_\infty)\big| \leq \!\!\!\!\!  \max_{\quad m\in \lgeo 0, \lceil ns_0 \rceil \rgeo } \! \big| \tfrac{1}{\sigma^2_{\nu}} S_{Y_{u_m}} \! -\! \bbn  Y_{u_m}\bbn   \big|  $$
which completes the proof of Proposition \ref{harmoclose}. \cqfd 

\medskip

\noi
\textbf{Proof of Proposition \ref{qharmoclose}.} To simplify notation, we set  
$$ \Gamma_{\! n}  \eqo \max_{s\in [0, n]} H_s (\tau_\infty) \eqo \max_{m\in \lgeo 0, n\rgeo} (|u_m|\!  -\! 1) \quad \textrm{and} \quad \widetilde{\Gamma}_{\! n}= \!  \max_{m\in \lgeo 0, n\rgeo}  \big| \bbn Y_{u_m} \bbn  \big| \; .$$
We first prove for all $p, q\ino \bbN^*$ that 
\begin{equation}
\label{amplispace}
\bP \big(  \widetilde{\Gamma}_{\! n} \geq q \big) \leq   \bP \big(  \Gamma_{\! n} \geq p \big) + 8 n p^3 \exp \Big( \! \! -\! \tfrac{(q-1)^2}{2p} \Big) \; .  
\end{equation}
\emph{Proof.} Let $\overline{X}_{\! n} \eqo (\bT, X_n)$, $n\ino \bbN$, be a $\ttm_\nu$-biased RW on $\bT$ such that $X_0\eqo \boo$. By definition of BRWs, we get 
$$ \bP \big(  \widetilde{\Gamma}_{\! n} \geq q \big| (\bT, \boo) , \tau_\infty  \big) \leq \sum_{1\leq m\leq n} P_{\bT, \boo} \big( \big| \bbn X_{|u_m|-1} \bbn \big| \geqo q  \big) \leq n P_{\bT, \boo} \big( \max_{0 \leq k \leq \Gamma_{\! n} } 
\big| \bbn X_{k} \bbn \big| \geqo q \, \big|\, \Gamma_{\! n} \big) .$$
By Lemma \ref{PZCV}, we get $ \un_{\{  \Gamma_{\! n}  \leq p \}}  \bP \big(  \widetilde{\Gamma}_{\! n} \geqo q \big|  \tau_\infty  \big)  \leqo  8 n p^3 \exp \big( \! \! -\! \tfrac{(q-1)^2}{2p} \big) $, which entails (\ref{amplispace}). \cq 

\smallskip

As in Lemma \ref{harmcoo4} $(ii)$, we next fix $\beta \ino (\tfrac{1}{2} \! +\! \tfrac{1}{2\ttb} , 1)$, which is specified later and we defined the r.v.~$R_\beta$ as in (\ref{LoRatiodef}). We then have
\begin{equation}
\label{LoRatioineq}
 \forall p\ino \bbN, \quad \big| \tfrac{1}{ \sigma_\nu^2} S_{\boo(p)} \! -\! \bbn \boo (p)\bbn   \big| \leq R_\beta  (1+ p^{\beta} )\quad \textrm{and} \quad \bP (R_\beta \geqo \lambda) \leq c_1 \lambda^{-2\ttb}, \quad \lambda \ino (0, \infty)
\end{equation}
where $c_1 \ino (0, \infty)$ is a constant that only depends on $\nu$, $\ttb$ and $\beta$. We then prove that if 
$2\widetilde{\Gamma}_{\! n} \leq q $ and if $R_\beta \leqo \lambda $, then 
\begin{equation}
\label{badcase}
\Big( \!\!\!\! \max_{\quad m\in \lgeo 0, n\rgeo } \big| \tfrac{1}{\sigma^2_{\nu}} S_{Y_{u_m}} \! -\! \bbn  Y_{u_m}\bbn   \big|  > 2\lambda (1+ q^\beta) \Big) \Longrightarrow \Big( \exists m\ino \lgeo  0, n\rgeo : u_m \ino B_{q, \lambda , \beta}\Big)
\end{equation}
 where $B_{q,  \lambda, \beta}$ is the set of vertices of $\bT\backslash L_\boo$ defined in Lemma \ref{badharm}.  

\medskip

\noi
\emph{Proof.} Suppose that $m\ino \lgeo  0, n\rgeo$ is such that 
$\alpha_n:= \big| \tfrac{1}{\sigma^2_{\nu}} S_{Y_{u_m}} \! -\! \bbn  Y_{u_m}\bbn   \big| \geko  2\lambda (1+ q^\beta) $. Since $R_\beta \leqo \lambda $, we get $\alpha_n \geqo \lambda (1+ q^\beta)+ R_\beta (1+ q^\beta)$ and thus 
$$  \big| \tfrac{1}{\sigma^2_{\nu}} (S_{Y_{u_m}}\!\! -\! S_{\boo \wedge Y_{u_m} } \big) \! -\! 
\big(\bbn  Y_{u_m}\bbn \! -\! \bbn  \boo \wedge Y_{u_m}\bbn \big)  \big| +  
\big| \tfrac{1}{\sigma^2_{\nu}} S_{\boo \wedge Y_{u_m} }\! -\! 
 \bbn  \boo \wedge Y_{u_m} \bbn \big| \geqo \alpha_n   \geko  \lambda (1+ q^\beta)+ R_\beta (1+ q^\beta) .$$
Note that there exists $u_l\ino \lgeo 0, n\rgeo$ such that $\boo \wedge Y_{u_m}\eqo Y_{u_l}$. Since $\widetilde{\Gamma}_{\! n} \leqo q/2\leqo q$, there exists $j\ino \lgeo 0 , q\rgeo $ such that 
$ \boo \wedge Y_{u_m}\eqo \boo (j)$ and by definition of $R_\beta$,  we get 
$$\big| \tfrac{1}{\sigma^2_{\nu}} S_{\boo \wedge Y_{u_m} }\! -\! 
 \bbn  \boo \wedge Y_{u_m} \bbn \big| \leq R_\beta (1+ j^\beta)\leq R_\beta (1+ q^\beta) $$
and thus  
\begin{equation}
\label{awayspine}
\big| \tfrac{1}{\sigma^2_{\nu}} (S_{Y_{u_m}}\!\! -\! S_{\boo \wedge Y_{u_m} } \big) \! -\! 
\big(\bbn  Y_{u_m}\bbn \! -\! \bbn  \boo \wedge Y_{u_m}\bbn \big)  \big|  \geko \lambda (1+ q^\beta)\; .
\end{equation}
Now observe that $ \bbn  Y_{u_m}\bbn \! -\! \bbn  \boo \wedge Y_{u_m}\bbn = \bbn  Y_{u_m}\bbn \! -\! \bbn  Y_{u_l}\bbn \leq 2\widetilde{\Gamma}_{\! n} \leq q$. Thus (\ref{awayspine}) implies that $u_m\ino B_{q, \lambda, \beta}$, which completes the proof of (\ref{badcase}). \cq

\smallskip

We next set $V_m \eqo V_m(\tau_\infty)$ and $I_m \eqo \min_{l\in \lgeo 0, m\rgeo } V_l $, where $V(\tau_\infty)$ is the Lukasiewicz path of $\tau_\infty$. As an easy consequence of the concatenation (\ref{concate}) in Definition \ref{Lukadef}, the vertex $u_n$ `belongs' to $\tau^{_{(k)}}_{^{\!}}$ where $k\eqo -I_n+1$ (here we mean that $u_n \eqo (k)\ast v$ for some 
$v\ino  \tau^{_{(k)}}_{^{\!}}$). We next prove that there is a constant $c_2 \ino (0, \infty)$, which only depends on 
$\nu$, $\ttb$ and $\beta$ such that for all $p, q, r\ino \bbN^*$ and all $\lambda \ino (0, \infty)$, 
\begin{eqnarray}
\label{summarybad}
\bP \Big( \!\!\!\!\!\!  \max_{\quad m\in \lgeo 0, n\rgeo } \big| \tfrac{1}{\sigma^2_{\nu}} S_{Y_{u_m}} \!\!\!\!\!\! \! & -& \!\!\!\!\!   \bbn  Y_{u_m}\bbn   \big| \geko  2\lambda (1+ q^\beta) \Big) \nonumber \\
&\leq &\bP \big( \Gamma_{\! n} \geq p \big) +  8 n p^3 e^{- \tfrac{(q-1)^2}{ 8p}}+
  \bP (-I_n \geqo r ) + c_2 q r \lambda^{-2\ttb} .
\end{eqnarray}
\noi
\emph{Proof.} Observe that 
\begin{eqnarray*}
\bP \big( \exists m\ino \lgeo  0, n\rgeo : u_m \ino B_{q, \lambda , \beta}\big) & \leq & 
\bP \big(\exists k \leqo \! -I_n\!+\!1, \, \exists 
u\ino  \tau^{_{(k)}}_{^{\!}}: Y_u \ino B_{q, \lambda , \beta}  \big)  \\
& \leq & \bP (-I_n \geq r) + \sum_{1\leq k \leq r} \bP \big( \exists 
u\ino  \tau^{_{(k)}}_{^{\!}}: \! Y_u \ino B_{q, \lambda , \beta}  \big)\\
& \leq &  \bP (-I_n \geq r) + c_2 q r \lambda^{-2\ttb} .
\end{eqnarray*} 
by Lemma \ref{badharm}. Then we get (\ref{summarybad}) thanks to (\ref{amplispace}), (\ref{LoRatioineq}), and (\ref{badcase}). \cq 

\smallskip

We recall from (\ref{bnBiGoTe}) the definition of the slowly varying function $L$ such that $b_n\eqo n^{1/\alpha} L(n)$ and we also recall that $a_n \eqo n/b_n \eqo n^{1-\frac{1}{\alpha}}/ L(n)$. 
Recall from (\ref{Lbtotheight}) in Proposition \ref{HolderH} that for all $\ttb_1\geko 1$, there exists $c_3\ino (0, \infty)$ that only depends on $\mu$ and $\ttb_1$ such that $\bE \big[ \Gamma_{n}^{\ttb_1} ]\leq c_3 (a_n)^{\ttb_1}$. Therefore, 
\begin{equation}
\label{heightboouu}
\forall n, p\ino \bbN^*, \quad \bP \big(  \Gamma_{n}\geq p \big) \leq c_3 (a_n/p)^{\ttb_1}\; .
\end{equation}
We next prove that there exists $c_4 \ino (0, \infty)$ that only depends on $\mu$ such that 
\begin{equation}
\label{infboouu}
\forall n, r\ino \bbN^*, \quad\bP \big( -I_n\geq r \big) \leq 3e^{-c_4 r a_n/n }.
\end{equation}
\noi
\emph{Proof.} Recall the notation $T_{-r}\eqo \inf \{ m\ino \bbN: V_m\eqo -r\}$. Observe that $\bP (-I_n \geq r)\eqo \bP (T_{-r} \leqo n)$. By Markov inequality, $\bP (T_{-r} \leqo n) \leqo e\bE [ \exp (-\tfrac{1}{n} T_{-r} )]$. We next use the fact that $T_{-r}$ 
is a sum of $r$ i.i.d.~r.v.s distributed as $T_{-1}$. Then, $\bE [ \exp (-\tfrac{1}{n} T_{-r} )]\eqo \varphi_\mu (e^{-1/n} )^r$, where we recall from Lemma \ref{numbtau} that $\varphi_\mu$ stands for the generating function of $T_{-1}$. By (\ref{numtauber}) in the same lemma, there is $c_4$ such that $c_4 (n^{1/\alpha} L(n))^{-1} \leqo -\log  \varphi_\mu (e^{-1/n} )$, which entails (\ref{infboouu}). \cq

\medskip

 Since we assume that $\ttb > \frac{2\alpha}{\alpha -1}$, then $\big( \tfrac{1}{2}+ \tfrac{1}{2\ttb}\big) +  \tfrac{1}{2\ttb}  \tfrac{\alpha +1}{\alpha-1} \leko  1$. Thus we can find $\beta \ino \big(\tfrac{1}{2}+ \tfrac{1}{2\ttb} , 1 \big) $ and $\epp \ino (0, \infty)$ such that 
\begin{equation}
\label{powertuning}
\beta + \tfrac{1}{2\ttb}  \tfrac{\alpha +1}{\alpha-1} (1+ \epp) + \epp + \tfrac{3\epp}{1+ 2\epp} \leko 1 . 
\end{equation} 
Then by Potter's bounds (see e.g.~Bingham, Goldies \& Teugel~\cite[Theorem 1.5.6]{BiGoTe}), there exists $c_5 \ino (1, \infty)$ that only depends on $\mu$ and $\epp$ such that 
\begin{equation}
\label{BiGoTeappro}
\forall n \ino \bbN^*, \quad  c^{-1}_5 n^{(1-\frac{_1}{\alpha})(1- \epp)  } \leq a_n \eqo   \tfrac{1}{L(n)}  n^{1-\frac{_1}{^\alpha}}  \leq c_5  n^{(1-\frac{_1}{\alpha})(1+ \epp)  }. 
\end{equation}
We then set 
$$ q= n^{\frac{1}{2}(1-\frac{_1}{\alpha})(1+ 2\epp)  }, \quad p= q^{2 \cdot\frac{{1+ \frac{3}{2}\epp}}{{1+2\epp}}}, \quad  r= q^{\frac{2}{{\alpha-1}}} 
\quad \textrm{and} \quad  \lambda= q^{ \frac{1}{{2\ttb}}  \frac{{\alpha +1}}{{\alpha-1}} (1+ \epp) } \; .
$$
Then we check that $\frac{1}{p}a_n \leqo c_5 n^{-\frac{1}{2}(1-\frac{_1}{\alpha}) \epp}$ 
and that $\frac{1}{n} r a_n \geqo  \frac{1}{c_5 }  n^{\frac{1}{2} \epp}$. We fix $\ttb_1\geko 1$, and we check that there are positive constants $c_{6}, \ldots, c_{11}$ that only depend on $\mu$, $\nu$, $\ttb$, $\beta$, $\ttb_1$ and $\epp$ such that  
\begin{eqnarray*} 
\mathtt{RHS} (\ref{summarybad}) \leq   c_6 n^{-\frac{1}{2}(1-\frac{_1}{\alpha}) \ttb_1 \epp} + c_7 n^{c_8} \exp (-c_9 
n^{\frac{1}{2} \epp(1-\frac{_1}{^\alpha}) } ) + 3 \! \exp (-c_{10} n^{\frac{1}{2} \epp}  ) + c_{11} q^{-\frac{\alpha+1}{\alpha-1} \epp}.
\end{eqnarray*}
Here we used (\ref{heightboouu}) and (\ref{infboouu}). Thus there are $c_0\ino (0, \infty)$ and $\epp_0\ino (0, 1)$, which only depend on $\mu$, $\nu$, $\ttb$, $\beta$, $\ttb_1$ and $\epp$, such that $\mathtt{RHS} (\ref{summarybad}) \leqo c_0n^{-\epp_0}$. 

Next, by (\ref{powertuning}) and (\ref{BiGoTeappro}), 
there are positive constants $c_{12}, c_{13}$, which only depend on $\mu$, $\nu$, $\ttb$, $\beta$, $\ttb_1$ and $\epp$, 
such that  
$$ 2\lambda (1+ q^\beta)  \leqo c_{12} q^{\beta + \frac{1}{{2\ttb}}  \frac{{\alpha +1}}{{\alpha-1}} (1+ \epp) } \! \leqo  c_{12} q^{1-\epp -\tfrac{3\epp}{1+ 2\epp}}\!  \leqo  c_{12} q^{-\epp} q^{\tfrac{1-\epp}{1+ 2\epp}} \! \eqo  c_{12} q^{-\epp} n^{\frac{1}{2}(1-\frac{_1}{\alpha})(1-\epp)} \! \leqo c_{13} q^{-\epp} a_n^{\! \frac{_1}{^2}} .$$
This proves that $ 2\lambda (1+ q^\beta)\leqo c'_0n^{-\epp_1} \sqrt{a_n}$ where we have set 
$\epp_1\eqo \frac{1}{2}(1\! -\! \frac{1}{\alpha})(1+ 2\epp) \epp $ and $c'_0\eqo c_{13}$. This implies that $\mathtt{LHS} (\ref{quantclose}) \leqo \mathtt{LHS} (\ref{summarybad}) \leqo \mathtt{RHS} (\ref{summarybad}) \leqo c_0n^{-\epp_0}$, which completes the proof of (\ref{quantclose}).  
\cqfd

\section{Proof of Theorem \ref{QumetlimbrRW}}
\label{QumetlimbrRWpfsec} 
\subsection{Limit of the $\tau_n$-indexed RW and its snakes.}
\label{limtaunbrRWsec}

Before stating the main result of the section, let us introduce some notation. 
We fix $\alpha \ino (1, 2]$ and we fix an offspring distribution $\mu$ that satisfies (\ref{hypostaintro}). 
Here $\nu$ is a supercritical offspring distribution whose mean is denoted by $\ttm_\nu$. We assume that there is $\ttb\ino ( \frac{2\alpha}{\alpha-1}, \infty)$ such that $\sum_{k\in \bbN} k^{1+2\ttb} \nu(k)\leko \infty$. 
Let $(\bT, \boo)$ be an invariant GW($\nu$)-tree as in Definition \ref{definvRW}. 

Let $\tau': \Omega \to \bbT$ be an (auxiliary) a.s.~finite random ordered rooted tree. Recall that $V (\tau')$ is its  Lukasiewicz path, that 
$H (\tau')$ is its height process (Definition \ref{Lukadef}) and that $C(\tau')$ is its contour process (Definition \ref{contourdef}). 
We then denote by 
$\fTheta (\tau') \eqo \big(\tau' ,\varnothing ;  \overline{Y}_{\! u} (\tau') \eqo (\bT, Y_u (\tau')) , u\ino \tau' \big)$ a $\tau'$-indexed and $\bT$-valued $\ttm_\nu$-biased BRW such that conditionally given $(\bT, \boo)$ and $\tau'$, $\fTheta (\tau')$ has the 
law $Q^{(\bT, \boo)}_{(\tau', \varnothing)}$ as in Definition \ref{biasedbfRWdef} with $\lambda\eqo \ttm_\nu$.  
We associate with $\fTheta (\tau') $ two $\bbR$-valued and $\tau'$-indexed BRWs
$$ M (\tau')  := \big( \bbn Y_{u} (\tau') \bbn \big)_{u\in \tau'}  \quad \textrm{and} \quad \mathcal M(\tau') 
:= \big( S_{Y_{u} (\tau')} \big)_{u\in \tau'} \; , $$
where we recall from (\ref{defSx}) the definition of the harmonic coordinates $(S_x)_{x\in \bT}$. 
According to Definition \ref{brwlkdef}, we introduce the following four $\bC^0(\bbR_+, \bbR)$-valued continuous processes:  
\begin{compactenum}
\item[$-$] $\big( W_s(\tau';  \cdot)\big)_{\! s\in \bbR_+}$ is the contour snake of the BRW  
$M(\tau')$; 
\item[$-$] $\big(W_s^*(\tau'; \cdot)\big)_{\! s\in \bbR_+}$ is the height snake of the BRW $M(\tau')$; 
 \item[$-$] $\big(\cW_s(\tau' ; \cdot)\big)_{\! s\in \bbR_+}$ is the contour snake of the BRW $\mathcal M (\tau') $; 
\item[$-$] $\big(\cW_s^*(\tau' ; \cdot)\big)_{\! s\in \bbR_+}$ is the height snake of the BRW $\mathcal M (\tau') $. 
\end{compactenum}
In this section, the auxiliary tree $\tau'$ shall be distributed as $\tau_n$ and  $\tau_{\geq n}$, which are defined as follows:  
\begin{equation}
\label{defratr}
\tau_{\geq n} \overset{\textrm{(law)}}{=} \tau \quad \textrm{under $\bP (\, \cdot \, |\, \# \tau \geqo n)$} \quad \textrm{and} \quad  \tau_{n} \overset{\textrm{(law)}}{=} \tau \quad \textrm{under $\bP (\, \cdot \, | \, \# \tau = n)$,} 
\end{equation}
where $\tau$ is a GW($\mu$)-tree (Definition \ref{GWfordef} $(\textbf{a})$). We note that the definition of $\tau_n$ only makes sense when $\bP (\# \tau\eqo n) \geko 0$, which is the case for all sufficiently large $n$ since $\mu$ is assumed to be aperiodic.

We finally introduce  the processes $(X'', H'', W'')$ and $(X', H', W')$ defined on $(\Omega, \ccF, \bP)$, whose laws 
are described as follows.  
\begin{compactenum}

\smallskip

\item[$-$]  The process $ (X''\! , H'') $ is distributed as $(X, H(X))$ under $\bN (dX \, | \, \zeta \geko 1)$ and conditionally given $H''\! $, $W''$ is the one-dimensional Brownian snake with lifetime process $H''$ (see Proposition \ref{Holdsnake}).

\smallskip

\item[$-$]The process $ (X'\! , H') $ is distributed as $(X, H(X))$ under $\bN (dX \, | \, \zeta \eqo 1)$ and conditionally given $H'\! $, $W'$ is the one-dimensional Brownian snake with lifetime process $H'$ (see Proposition \ref{Holdsnake}).  

\smallskip

\end{compactenum}

\noi
The goal of this section is to derive from Theorem~\ref{cvsnakes} the following result that is the main argument of the proof of Theorem~\ref{QumetlimbrRW}. 
\begin{theorem}
\label{cvsnacondi} We fix $\alpha \ino (1, 2]$. Let $\mu$ satisfy (\ref{hypostaintro}). 
Let $\tau_n$ be as in (\ref{defratr}) (i.e., ~namely, a GW($\mu$)-tree conditioned to have $n$ vertices). 
Let $\nu$ be a supercritical offspring distribution whose mean is denoted by $\ttm_\nu$. We assume that there is  $\ttb\ino ( \frac{2\alpha}{\alpha-1}, \infty)$ such that $\sum_{k\in \bbN} k^{1+2\ttb} \nu(k)\leko \infty$.
Let $(\bT, \boo)$ be an invariant GW($\nu$)-tree that is independent from $\tau_n$. 
We recall the $a_n$ and $b_n$ from (\ref{remindcv}) and we recall 
from (\ref{vutchi}) the definition of $\sigma_\nu$. We also recall the notation $\mathscr Q_n'$ from (\ref{VHCGDcv}) in Theorem \ref{VHCcvstable} $(iv)$. We keep the above notation. 
Then, conditionally given $(\bT, \boo)$, the following convergence 
\begin{eqnarray}
\label{snakescv_prime}
\Big(\mathscr Q_n'  ,  \tfrac{1}{ \sigma_\nu \! \sqrt{a_n}} \widehat{\cW}_{2n\cdot }(\tau_n),   \tfrac{1}{ \sigma_\nu \! \sqrt{a_n}}\widehat{\cW}^*_{n\cdot }(\tau_n)   \!\!  \!\!  \!\!  \!\!  \!\!   \!\! \! & , &\!   \!\!   \!\!  \!\!  \!\!  \!\!  \!\!  \tfrac{\sigma_\nu}{ \sqrt{a_n}} \widehat{W}_{2n\cdot }(\tau_n),\!  \tfrac{\sigma_\nu}{ \sqrt{a_n}} \widehat{W}^*_{n\cdot }(\tau_n)\Big)  \nonumber  \\
&\underset{n\to \infty}{-\!\!\!-\!\!\!-\!\!\!\longrightarrow}& \big( X', H', H', \widehat{W}',  \widehat{W}' , \widehat{W}',  
\widehat{W} '\big)
\end{eqnarray}
a.s.~holds weakly on $\bD([0, 1], \bbR) \! \times \! \bC ([0, 1], \bbR)^6$ equipped with the product topology. 
\end{theorem}
\noi
\textbf{Proof.} We set $Z''_n \eqo \big( \tfrac{1}{b_n} V_{\lfloor n\cdot \rfloor } (\tau_{\geq n})$, $ \tfrac{1}{a_n} H_{n\cdot }(\tau_{\geq n})$, $\tfrac{1}{\sigma_\nu\sqrt{a_n}} \widehat{\cW}^*_{n\cdot }(\tau_{\geq n})$, $\tfrac{\sigma_\nu}{ \sqrt{a_n}} \widehat{W}^*_{n\cdot }(\tau_{\geq n}) \big)$ for all $n\ino \bbN^*$ and we also set 
$Z''$ $=$  $(X'', H'', \widehat{W}'', \widehat{W}'')$. We first prove that 
conditionally given $(\bT, \boo) $, the following convergence 
\begin{equation}
\label{truecv}
Z''_n \xrightarrow[n\to \infty]{} Z''
\end{equation}
holds weakly on $\bD (\bbR_+, \bbR) \times \bC^0 (\bbR_+, \bbR)^3$. 
\emph{Indeed}, by easy arguments, Theorem~\ref{cvsnakes} entails that conditionally on $(\bT, \boo)$, the following convergence 
\begin{eqnarray}
\label{extractcv}
\Big( \tfrac{1}{b_n} V_{\lfloor ( \frak g_n +n \cdot ) \wedge \frak d_n\rfloor } (\tau_\infty) \!\!\!\! \! &,& \!\!\!\!  \!
 \tfrac{1}{a_n} H_{( \frak g_n + n\cdot ) \wedge \frak d_n } (\tau_\infty), \tfrac{1}{ \sigma_\nu\sqrt{a_n}} \widehat{\cW}^*_{( \frak g_n + n\cdot ) \wedge \frak d_n}(\tau_\infty), \tfrac{\sigma_\nu}{ \sqrt{a_n}} \widehat{W}^*_{( \frak g_n + n\cdot ) \wedge \frak d_n}(\tau_\infty)\Big) \nonumber \\
\!\!  \!\! & & \!\! \!\! \underset{n\to \infty}{-\!\!\!-\!\!\!-\!\!\!\longrightarrow} \Big( X_{(\mathbf g_1+ \cdot )\wedge \mathbf d_1},  H_{(\mathbf g_1+ \cdot )\wedge \mathbf d_1},   \widehat{W}_{(\mathbf g_1+ \cdot )\wedge \mathbf d_1},  \widehat{W}_{(\mathbf g_1+ \cdot )\wedge \mathbf d_1} \Big)
\end{eqnarray} 
holds weakly on $\bD (\bbR_+, \bbR) \! \times \! \bC^0 (\bbR_+, \bbR)^3$. Then, as a consequence of the definition of $\frak g_n$ and $\frak d_n$, the left member of (\ref{extractcv}) has the same joint law with $(\bT, \boo)$ as $(Z''_n, (\bT, \boo))$.  
Similarly, the definition of $\mathbf g_1$ and $\mathbf d_1$ implies that the right member of (\ref{extractcv}) is distributed as $Z''$ (see Remark~\ref{gundundef}). Thus, (\ref{extractcv}) is equivalent to (\ref{truecv}). \cq

\medskip

To complete the proof of the theorem, we next want to use Lemma \ref{absClem}. To that end, we fix $s\ino (0, 1)$ and we set 
$s_n \eqo \lfloor ns \rfloor$. We also set 
\begin{compactenum}

\smallskip

\item[$-$] $ Z''_{n} (s)  \eqo \big( \tfrac{1}{b_n} V_{\lfloor n \cdot \rfloor \wedge s_n } ( \tau_{\geq n} ),  
\tfrac{1}{a_n} H_{(n\cdot ) \wedge s_n} (  \tau_{\geq n}  ) , 
\tfrac{1}{\sigma_\nu \sqrt{a_n}} \widehat{\cW}^*_{ (n\cdot ) \wedge s_n} (  \tau_{\geq n} ) , 
\tfrac{\sigma_\nu}{ \sqrt{a_n}} \widehat{W}^*_{ (n\cdot ) \wedge s_n} (  \tau_{\geq n} ) \big)$,

\smallskip

\item[$-$] $  Z'_{n} (s) \eqo \big( \tfrac{1}{b_n} V_{\lfloor n \cdot \rfloor \wedge s_n } ( \tau_{ n} ),  
\tfrac{1}{a_n} H_{(n\cdot ) \wedge s_n} (  \tau_{n}  ) , 
\tfrac{1}{\sigma_\nu \sqrt{a_n}} \widehat{\cW}^*_{ (n\cdot ) \wedge s_n} (  \tau_{n} ) , 
\tfrac{\sigma_\nu}{ \sqrt{a_n}} \widehat{W}^*_{ (n\cdot ) \wedge s_n} (  \tau_{n} ) \big)$,

\smallskip

\item[$-$]  $Z'' (s)\eqo (X''_{\cdot \wedge s}, H''_{\cdot \wedge s} ,  \widehat{W}''_{\cdot \wedge s}, \widehat{W}''_{\cdot \wedge s})$ and $Z' (s) \eqo (X'_{\cdot \wedge s}, H'_{\cdot \wedge s} , \widehat{W}'_{\cdot \wedge s}, \widehat{W}'_{\cdot \wedge s})$.

\smallskip

\end{compactenum} 
We easily deduce from (\ref{abscont}) that $\bE [F(Z'' (s)) D_s (X''_s) ]\eqo \bE [F(Z'(s)) ]$ for all measurable nonnegative functions $F$. We next deduce from (\ref{truecv}) that conditionally given $(\bT, \boo)$, $Z''_n (s) \! \to \! Z''(s)$ weakly on the appropriate spaces. 
Next denote by $(u_k)_{0\leq l< \# \tau_{\geq n}}$ the vertices of $\tau_{\geq n}$ listed in lexicographical order. 
We observe that the subtree $\{ u_l \, ; 0 \leqo l \leqo s_n \}$ is measurable with respect to $V_{\cdot \wedge s_n}  (\tau_{\geq n} )$. Thus, $Z''_n (s)$ is a measurable function of $V_{\cdot \wedge s_n}  (\tau_{\geq n} )$ and of an independent r.v.~$U_n$. Consequently Lemma \ref{absClem} applies and conditionally given $(\bT, \boo)$ the joint convergence 
\begin{equation}
\label{truecvbis}
Z'_n(s)  \xrightarrow[n\to \infty]{} Z' (s) 
\end{equation}
holds weakly on $\bD (\bbR_+, \bbR) \times \bC^0 (\bbR_+, \bbR)^3$.

We next derive the joint convergence of the related contour process and contour snakes via the time-change discussed in 
Remark~\ref{contord} $(\textbf{d})$. Namely, there is an increasing continuous bijection $\phi_{\tau_n} \! :\! [0, n] \! \to \! [0, 2n]$ such that $H_r (\tau_n)\eqo C_{\phi_{\tau_n} (r)} (\tau_n)$, $r\ino [0, n]$. 
Moreover, deduce from (\ref{controphit}) that 
$\max_{r\in [0, 1]} |\frac{1}{2n}\phi_{\tau_n} (nr) \! -\! r | \leqo 3\tfrac{a_n}{n} \max_{s\in [0, 1]} \tfrac{1}{a_n} H_{nr} (\tau_n) $, which tends to $0$ in $\bP$-probability since $a_n/n \! \to \! 0$. 
Therefore, if for all $s\ino [0, 1]$ we set $\varphi_n (s)\eqo \frac{1}{n} \phi^{-1}_{\tau_n} (2ns)$ (here, $\phi^{-1}_{\tau_n}$ stands for the inverse of $\phi_{\tau_n} $), then we get

\smallskip

\noi
$(i)$  $\lim_{n\to \infty}\max_{s\in [0, 1] } \big| \varphi_n (s) \! -\! s \big| = 0$ in $\bP$-probability, 

\smallskip

\noi
$(ii)$ $C_{2ns} (\tau_n)\eqo H_{n\varphi_n(s)} (\tau_n)$, 
$ \widehat{\cW}_{2ns}(\tau_n)\eqo  
 \widehat{\cW}^*_{n\varphi_n (s)} (\tau_n)$,  $ \widehat{W}_{2ns}(\tau_n)\eqo   \widehat{W}^*_{n\varphi_n (s)} (\tau_n)$ for all $s\ino [0, 1]$. 

\smallskip

\noi
This combined with Lemma \ref{composko} shows that (\ref{truecvbis}) implies 
for all $s\ino (0, 1)$ that conditionally given $(\bT, \boo)$ the joint convergence 
\begin{equation}
\label{tillscv} \Big( Z'_n (s)  , \tfrac{1}{a_n} C_{2n (\cdot \wedge s)  } (\tau_n),  \tfrac{1}{\sigma_\nu \sqrt{a_n}}   \widehat{\cW}_{2n (\cdot \wedge s) }(\tau_n) , 
\tfrac{\sigma_\nu}{\sqrt{ a_n}}  \widehat{W}_{2n (\cdot \wedge s) }(\tau_n) \Big)   \!\!  \underset{n\to \infty}{-\!\!\!-\!\!\!\longrightarrow}  (Z' (s) , H'_{\cdot \wedge s}, \widehat{W}'_{\cdot \wedge s} ,\widehat{W}'_{\cdot \wedge s} ) 
\end{equation}
holds weakly on $\bD (\bbR_+, \bbR) \! \times \! \bC^0 (\bbR_+, \bbR)^6$. 

To complete the proof of the theorem, we next use a symmetry property of contour processes and contour snakes. More precisely, denote by $\widetilde{\tau}_n$ the mirror image of $\tau_n$. Namely there exists a height preserving and $\leq_{\mathtt{lex}}$-increasing bijection $u\ino \tau_n\mapsto \widetilde{u} \ino \widetilde{\tau}_n$ such that $\widetilde{(v_k)}= \widetilde{v}_{2(n-1)-k}$, where $(v_k)_{0\leq k \leq 2(n-1)}$  and $(\widetilde{v}_k)_{0\leq k \leq 2 (n-1)}$ stand for the vertices of resp.~$\tau_n$ and $\widetilde{\tau}_n$ in increasing contour order. Then, for all $s\ino [0, n]$, we get 
$$ C_{2(n -s)} (\tau_{n}) \eqo C_{ 2(s -1)_+ } (\widetilde{\tau}_{n}), \quad \widehat{\cW}_{2(n -s)} (\tau_{n}) \eqo \widehat{\cW}_{ 2(s -1)_+ } (\widetilde{\tau}_{n}), \quad \widehat{W}_{2(n -s)} (\tau_{n}) \eqo \widehat{W}_{ 2(s -1)_+ } (\widetilde{\tau}_{n}) . $$
Since $\widetilde{\tau}_n$ and $\tau_n$ have the same law, (\ref{tillscv}) then implies for all $s\ino (0, 1)$ that conditionally given $(\bT, \boo)$, the joint convergence 
$$  \Big(  \tfrac{1}{\sigma_\nu \sqrt{a_n}}   \widehat{\cW}_{2n ( (1-\cdot )\wedge s)  }(\tau_n) , 
\tfrac{\sigma_\nu}{\sqrt{a_n}}  \widehat{W}_{2n ( (1-\cdot )\wedge s) }(\tau_n) \Big)    \!  \underset{n\to \infty}{-\!\!\!-\!\!\!\longrightarrow}   ( \widehat{W}'_{\cdot \wedge s} ,\widehat{W}'_{\cdot \wedge s} ) \; .$$
holds weakly on $ \bC^0 (\bbR_+, \bbR)^2$. 
If $s\ino (\frac{_1}{^2}, 1)$, the previous limit combined with (\ref{tillscv}) imply that conditionally given $(\bT, \boo)$ the laws of the processes $\tfrac{1}{\sigma_\nu \sqrt{a_n}}   \widehat{\cW}_{2n \cdot   }(\tau_n) $ and 
$\tfrac{\sigma_\nu}{\sqrt{ a_n}}  \widehat{W}_{2n \cdot  }(\tau_n) $ are tight on $\bC ([0, 1], \bbR)$. 
Then recall from Theorem~\ref{VHCcvstable} $(iv)$ that the laws of the processes $\tfrac{1}{b_n}V_{\lfloor n\cdot \rfloor } (\tau_n)$, $\tfrac{1}{a_n} H_{n\cdot} (\tau_n)$ and $\tfrac{1}{a_n} C_{2n\cdot} (\tau_n)$ are tight on resp.~$\bD ([0, 1], \bbR)$ and $\bC ([0, 1], \bbR)$.
Since $s$ can be chosen arbitrarily close to $1$ in (\ref{tillscv}), the processes that are weak limits of the previous ones have 
necessarily the same finite-dimensional marginals as $(X', H', H'  , \widehat{W}', \widehat{W}')$. It therefore 
shows that conditionally given $(\bT, \boo)$ the following convergence 
 \begin{equation}
\label{snakescv3/4}
\Big(\mathscr Q_n', \tfrac{1}{ \sigma_\nu \! \sqrt{a_n}} \widehat{\cW}_{2n\cdot }(\tau_n),   \tfrac{1}{ \sigma_\nu \! \sqrt{a_n}}\widehat{W}_{2n\cdot }(\tau_n)  \Big)\underset{n\to \infty}{-\!\!\!-\!\!\!-\!\!\!\longrightarrow} \big( X', H', H', \widehat{W}',  \widehat{W}' \big)
\end{equation}
holds weakly on $\bD([0, 1], \bbR) \! \times \! \bC ([0, 1], \bbR)^4$. Now observe that for all $s\ino [0, 1]$, 
$$ \widehat{\cW}^*_{ns}(\tau_n)\eqo  
 \widehat{\cW}_{ 2n\varphi^{-1}_n (s)} (\tau_n) \quad \textrm{and} \quad  \widehat{W}^*_{ns}(\tau_n)\eqo   
 \widehat{W}_{ 2n\varphi^{-1}_n (s)} (\tau_n)\; .$$ 
where $\varphi^{-1}_n \! : \! [0, 1] \! \to [0, 1]$ stands for the inverse of $\varphi_n$. 
From the above-mentioned uniform convergence of the $\varphi_n$, we easily prove that 
$\lim_{n\to \infty} $ $\max_{s\in [0, 1] } \big| \varphi^{-1}_n (s) \! -\! s \big|$ $\eqo 0$ in $\bP$-probability and we finally derive the convergence (\ref{snakescv_prime}) from (\ref{snakescv3/4})  and Lemma~\ref{composko}. 
This completes the proof of the theorem.  \cqfd 

\subsection{Proof of Theorem \ref{QumetlimbrRW}}
\label{reQumetlimbrRWpfsec}

We keep the notation and the assumptions of Theorem \ref{cvsnacondi}. In particular $\tau_n $ has the same law as a GW($\mu$)-tree $\tau$ under $\bP (\, \cdot \, | \, \# \tau\eqo n)$ and $\fTheta (\tau_n) \eqo \big(\tau_n ,\varnothing ;  \overline{Y}_{\! u}  \eqo (\bT, Y_u ) , u\ino \tau_n \big)$ is a $\tau_n$-indexed and $\bT$-valued $\ttm_\nu$-biased BRW such that conditionally given $(\bT, \boo)$ and $\tau_n$, $\fTheta (\tau_n)$ has the 
law $Q^{(\bT, \boo)}_{(\tau_n, \varnothing)}$  in Definition \ref{biasedbfRWdef} with $\lambda \eqo \ttm_\nu$. Recall that $\tau_n$ is independent of $(\bT, \boo)$, which is an invariant GW($\nu$)-tree. We set 
$$ \ttm_n \eqo \sum_{v \in \tau_n } \delta_v , \quad \mathcal R_n = \big\{ Y_v \, ; \, v\ino \tau_n \big\} \quad \textrm{and} \quad \ttm_{\mathtt{occ}}^{(n)}= \sum_{v\in \tau_n } \delta_{Y_v} . $$
We want to prove conditionally given $(\bT, \boo)$ 
a joint convergence of the following rescaled pointed measured metric spaces equipped with their graph distances $d_{\mathtt{gr}}$: 
$$ \ftau_n := \big( \tau_n , \tfrac{1}{a_n} d_{\mathtt{gr}} , \varnothing , \tfrac{1}{n}\ttm_n \big)  \quad \textrm{and} \quad  \bcR_n := \big( \mathcal R_n ,  \tfrac{\sigma_\nu}{\sqrt{a_n}} d_{\mathtt{gr}} , \boo,  \tfrac{1}{n}  \ttm_{\mathtt{occ}}^{(n)}\big)$$
weakly with respect to the Gromov--Hausdorff--Prokhorov convergence.

  To that end, we first prove a convergence of the snake-trees associated with the contour snake of the relative heights of $\fTheta(\tau_n)$. Namely, recall that $W (\tau_n; \cdot)$ stands for the contour snake associated with the $\bbR$-valued BRW $\big( \bbn Y_{v} (\tau_n) \bbn \big)_{v\in \tau_n}$ as in Definition \ref{brwlkdef}. To simplify notation, we set for all $s\ino [0,1]$ and $r\ino \bbR_+$, 
$$ C^{(n)}_s \eqo \tfrac{1}{a_n}C_{2ns} (\tau_n) , \quad W^{(n)}_s (r)= \tfrac{\sigma_\nu}{\sqrt{a_n}} W_{2ns}( \tau_n; a_n r) .$$  
Then $(W^{(n)}_{s} (\cdot) )_{s\in [0, 1]}$ is a snake with lifetime process $(C^{(n)}_{s})_{s\in [0, 1]}$ (as in Definition \ref{defC0sna}).
We recall the notation $d_{C^{(n)} , W^{(n)}} $ for the snake-pseudometric associated with $(C^{(n)} , W^{(n)} )$ 
(Definition \ref{defsnkme}). We also recall from (\ref{djhczkk}) that $d_{C^{(n)}}$ stands for the tree-pseudometric  associated with $C^{(n)}$. By combining Theorem \ref{cvsnacondi}, Proposition \ref{endvssna} and Lemma~\ref{Hermagor}, we get 
conditionally given $(\bT, \boo)$ the following joint convergence 
\begin{equation}
\label{limdistn}
\big( d_{C^{(n)}} , d_{C^{(n)} , W^{(n)}} )  \underset{n\to \infty}{-\!\!\!-\!\!\!-\!\!\!\longrightarrow} (d_H , d_{H, W} ) 
\end{equation}
weakly on $\bC ([0, 1]^2, \bbR)^2$. Here $H$ is the $\alpha$-stable height process under $\bN (\, \cdot \, | \, \zeta \eqo 1)$ and $W$ is a one-dimensional Brownian snake with lifetime process $H$ as in Proposition \ref{Holdsnake}. 
Then to simplify, we denote by 
$\widetilde{\ftau}_n \eqo  (T_{C^{(n)}},  d_{C^{(n)}} , \mathtt r_{C^{(n)}} , \ttm_{C^{(n)}})$ the real tree coded by the function $C^{(n)}$ as in Example \ref{rltrcdng}. By Remark~\ref{interpo1}, we get that 
\begin{equation}
\label{treclcl}
\bdelta_{\mathtt{GHP}} \big( \ftau_n,  \widetilde{\ftau}_n  \big)  \underset{n\to \infty}{-\!\!\!-\!\!\!-\!\!\!\longrightarrow} 0\; .
\end{equation}
 To simplify, we also denote the snake-tree coded by the snake-pseudometric $d_{C^{(n)} , W^{(n)}}$ by 
$$ \bcT^{(n)}_{\! \!\mathtt{sna}} = \big( \cT^{(n) }_{ \mathtt{sna}},  d^{(n)}_{ \mathtt{sna}} , \mathtt r^{(n) }_{ \mathtt{sna}}, \ttm^{(n) }_{ \mathtt{sna}}  \big) .$$
(See Definition \ref{notasnkm}). 
By Proposition \ref{groweak}, the limits (\ref{limdistn}) and (\ref{treclcl}) imply that conditionally given $(\bT, \boo)$ the joint convergence 
\begin{equation}
\label{limspcn}
\big( \ftau_n\,  , \, \bcT^{(n)}_{\!\! \mathtt{sna}} \big) \underset{n\to \infty}{-\!\!\!-\!\!\!-\!\!\!\longrightarrow} \big( (T_H, d_H, \mathtt r_H, \ttm_H) , (T_{H, W}, d_{H, W}, \mathtt r_{H, W},  \ttm_{H, W})  \big)  
\end{equation}
holds weakly with respect to Gromov--Hausdorff--Prokhorov convergence. Here $ (T_H, d_H, \mathtt r_H, \ttm_H)$ is the real tree coded by $H$ which is the $\alpha$-stable tree and $ (T_{H, W}, d_{H, W}, \mathtt r_{H, W},  \ttm_{H, W}) $ is the snake-tree coded by the snake-pseudometric $d_{H, W}$ (Definition \ref{notasnkm}) which is the $\alpha$-stable Brownian cactus.

We next prove that $\bcR_n$ is close to $ \bcT^{(n)}_{\! \!\mathtt{sna}}$ thanks to the results discussed in Section \ref{trebrRWsnasec}. 
To that end, we introduce the following. 

\smallskip

\noi
$-$ We denote by $(\widetilde{\tau}_n, d_{\widetilde{\tau}_n})$ 
the real tree obtained by joining adjacent vertices in $\tau_n$ by a unit length segment (recall that we suppose here that $\tau_n \! \subset \! \widetilde{\tau}_n$ so that $d_{\widetilde{\tau}_n}$ extend the graph distance). 

\smallskip

\noi
$-$ Recall from (\ref{explotilde}) that $\widetilde{v} \! : \! [0, 2n] \! \to \! \widetilde{\tau}_n$ stands for the contour exploration of $\widetilde{\tau}_n$ (namely $(\widetilde{v} (k))_{0\leq k\leq 2 (n-1)}$ are the vertices of $\tau_n$ in contour order and  $\widetilde{v}$ is a 
`linear' extension of it). 

\smallskip

\noi
$-$ We denote by $(\widetilde{\bT}, d_{\widetilde{\bT}})$ the real tree obtained by joining adjacent vertices in $\bT$ by a unit length segment. We also assume that $\bT \! \subset \! \widetilde{\bT}$ so that $d_{\widetilde{\bT}}$ extends the graph distance. 

\smallskip

\noi
$-$ Recall from (\ref{Yexten}) that $\widetilde{Y}\! : \! \widetilde{\tau}_n\! \to \! \widetilde{\bT}$ is the `linear' extension of the BRW $(Y_v)_{v\in \tau_n}$.

\smallskip

We set $\widetilde{\bcR}_n \eqo (\widetilde{\cR}_n, \tfrac{\sigma_\nu}{\sqrt{a_n}}d_{\widetilde{\bT}}, \boo, \tfrac{1}{2n}\widetilde{\ttm}_{\mathtt occ} )$ where $\widetilde{\cR}_n \eqo \{  \widetilde{Y}_{\sigma} \, ; \, \sigma\ino \widetilde{\tau}_n \} \eqo \{  \widetilde{Y}_{\widetilde{v} (s)} \, ; \, s\ino [0, 2n] \}$ and $\widetilde{\ttm}_{\mathtt{occ}}$ is the occupation measure of $s\mapsto \widetilde{Y}_{\widetilde{v} (s)}$ as defined by (\ref{moccudef}). We then set for all $s,s'\ino [0, 1]$:
$$ d_n (s,s')\eqo \tfrac{\sigma_\nu}{\sqrt{a_n}}  d_{\widetilde{\bT}} \big( \widetilde{Y}_{\widetilde{v} (2ns)},   \widetilde{Y}_{\widetilde{v} (2ns')}\big)\; , $$
which is a tree-pseudometric. 
Recall from (\ref{ncdngRR}) that $\widetilde{\bcR}_n$ is isometric to the pointed measured compact real tree yielded by the pseudometric $d_n$. Then (\ref{interpospa}) implies that conditionally given $(\bT, \boo)$, we have the weak convergence
\begin{equation}
\label{ranclcl}
\bdelta_{\mathtt{GHP}} \big( \bcR_n,  \widetilde{\bcR}_n  \big)  \underset{n\to \infty}{-\!\!\!-\!\!\!-\!\!\!\longrightarrow} 0\; .
\end{equation}

We next prove that $d_n$ and the snake-pseudometric $d_{C^{(n)}, W^{(n)}}$ are close thanks to 
Lemma \ref{graphvssna}, which applies here thanks to Lemma~\ref{critGWimpl}. We thus get the following bound:
$$\bP \Big(   \max_{s,s'\in [0, 1]} \big| d_n (s,s') \! -\! d_{C^{(n)}, W^{(n)}}(s,s') \big| >2 \sigma_\nu a_n^{\! -\frac{_1}{^4}} \,  \Big|  \, (\bT, \boo) \Big) \leq \ttm_\nu^2  n^3 \! \exp\!  \big(\! -\! a_n^{1/4} \log \ttm_\nu \big).$$
Since $a_n$ is regularly varying with a positive exponent, this thus entails that conditionally given $(\bT, \boo)$ the following joint convergence 
\begin{equation}
\label{limdistn+}
\big( d_{C^{(n)}} , d_{C^{(n)} , W^{(n)}}, d_n )  \underset{n\to \infty}{-\!\!\!-\!\!\!-\!\!\!\longrightarrow} (d_H , d_{H, W},  d_{H, W} ) 
\end{equation}
holds weakly on  $\bC([0, 1]^2, \bbR)^3$. By Proposition \ref{groweak}, it entails that 
conditionally given $(\bT, \boo)$ the following joint convergence 
\begin{equation}
\label{limspcn+}
\big( \ftau_n\,  , \, \widetilde{\bcR}_n \big) \underset{n\to \infty}{-\!\!\!-\!\!\!-\!\!\!\longrightarrow} \big( (T_H, d_H, \mathtt r_H, \ttm_H) , (T_{H, W}, d_{H, W}, \mathtt r_{H, W},  \ttm_{H, W})  \big)  
\end{equation}
holds weakly with respect to Gromov--Hausdorff--Prokhorov convergence. By (\ref{ranclcl}), we thus have proved that 
conditionally given $(\bT, \boo)$, the following joint convergence 
\begin{equation}
\label{resumetr}
\big( \ftau_n\,  ,   \bcR_n \big) \underset{n\to \infty}{-\!\!\!-\!\!\!-\!\!\!\longrightarrow} \big( (T_H, d_H, \mathtt r_H, \ttm_H) , (T_{H, W}, d_{H, W}, \mathtt r_{H, W},  \ttm_{H, W})  \big)  
\end{equation}
holds weakly with respect to Gromov--Hausdorff--Prokhorov convergence, which is the desired result. \cqfd

\appendix 

\section{Proof of Lemma \ref{branchmass}}
\label{branchmasspfsec}

We fix $\alpha \ino (1, 2]$ and we recall that $X$ is an $\alpha$-stable spectrally positive Lévy process whose law is characterized by (\ref{Xlaw}). Recall that $H$ is the corresponding height process as defined in (\ref{Hlimit}). We also introduce the notation $S_s\eqo \sup_{r\in [0, s]} X_r $, $s\ino \bbR_+$, for the supremum process of $X$. As a Lévy process, $X$ enjoys a so-called \emph{duality property}: for all $s\ino(0,\infty)$, the time-reversed process $(\widehat{X}^{(s)}_r)_{r\in[0,s]}$ defined by
\begin{equation}
\label{reverse-time_X}
\widehat{X}_s^{(s)}= X_s\quad\textrm{and}\quad \widehat{X}_r^{(s)}= X_s \! -\! X_{(s-r)-},\quad r\ino[0,s),
\end{equation}
has the same law as $(X_r)_{r\in[0,s]}$. Note that $\widehat{X}^{(s)}$ is independent of the shifted process $(X_r^{(s)})_{s\in\bbR_+}$, which is also distributed as $X$ and defined by 
\begin{equation}
\label{shifted_X}
X_r^{(s)}=X_{s+r}-X_s,\quad r\in\bbR_+.
\end{equation}
Then, $H_s$ is $\bP$-a.s.~the local time at level $0$ at time $s$ of the process $\widehat{S}^{(s)}\! -\! \widehat{X}^{(s)}$, where $\widehat{S}^{(s)}\eqo \sup_{[0,r]}\widehat{X}^{(s)}$ for all $r\ino[0,s]$. Our proof of Lemma~\ref{branchmass} relies on this connection, which we explain in greater detail below.
\medskip

Elementary results of fluctuation theory (see e.g.~Bertoin \cite{Be} Chapter VI.1 and VII.1) yield that $S\! -\! X$ is a strong Markov process in $\bbR_+$, and that $0$ is regular for itself and $(0, \infty)$ with respect to this Markov process. Thus, there exists a Markovian local time at $0$ for $S\! -\! X$, which is a continuous and nondecreasing process $L\eqo (L_s)_{s\in \bbR_+}$. Moreover, as proved by D.~\& Le Gall~\cite[Lemma~1.1.3]{DuLG02}, we can choose $L$ such that the following limit holds in $\bP$-probability for any $s\ino\bbR_+$:
\begin{equation}
\label{Llimit}
L_s:=\lim_{\varepsilon\to 0} \frac{1}{\epp}\int_0^s {\bf 1}_{\{ S_r <  \varep+ X_r \}}\,dr.
\end{equation}

Now, for $s\geqo 0$, let us denote by $\widehat{L}^{(s)}=(\widehat{L}_r^{(s)})_{r\in[0,s]}$ the local time at $0$ of $\widehat{S}^{(s)}\! -\! \widehat{X}^{(s)}$, where $\widehat{X}^{(s)}$ is the time-reversed process as in (\ref{reverse-time_X}). Comparing (\ref{Llimit}) with (\ref{Hlimit}) $\bP$-a.s.~entails that $H_s\eqo \widehat{L}_s^{(s)}$, so $H_s$ and $L_s$ have the same law by the duality property. We also denote by $H^{(s)}\eqo (H_r^{(s)})_{r\in\bbR_+}$ the height process of the shifted process $X^{(s)}$ as in (\ref{shifted_X}). Then, for all $s,r\ino(0,\infty)$, D.~\& Le Gall~\cite[Lemma~1.4.5]{DuLG02} provide the two following identities that we will use to prove Lemma~\ref{branchmass}:
\begin{equation}
\label{variation_H_loctime}
H_{s+r}-\inf_{[s,s+r]}H = H_r^{(s)} \quad\textrm{ and }\quad H_s-\inf_{[s,s+r]}H= \widehat{L}_{s\wedge \widehat{T}}^{(s)}\quad \textrm{$\bP$-a.s.},
\end{equation}
where $\widehat{T}\eqo \inf\{t\ino[0,s] : \widehat{X}_t^{(s)}\geko - \inf_{[0,r]}X^{(s)}\}$ (with the convention $\inf\varnothing\eqo\infty$).
\medskip

 \noi
\textbf{Proof of Lemma~\ref{branchmass}.} Let $s,r\ino(0,\infty)$ and let $J$ be a r.v.~that is distributed as $-\inf_{[0,r]}X$ and independent of $X$. We set $T\eqo \inf\{t\ino[0,s] : X_t\geko J\}$. Thanks to (\ref{variation_H_loctime}), we exactly need to show that $\bP$-a.s.~$L_r\geko 0$ and $L_{s\wedge T}\geko 0$. Since $0$ is regular for itself with respect to $S\! -\! X$, standard results about local times ensure that $\bP$-a.s, $L_t\geko 0$ for all $t\ino(0,\infty)$. Thus, we only have to prove that $\bP$-a.s~$T\geko 0$, which holds if and only if $J\geko 0$. But it is well-known that $\bP$-a.s.~$\inf_{[0,r]}X\leko0$ (see e.g.~Bertoin \cite[Theorem~VII.1 and Corollary~VII.5]{Be}), which concludes the proof. \cqfd

\section{Proof of Lemma \ref{absClem}}
\label{absClempfsec}

We keep the notation and the assumptions of Theorem \ref{VHCcvstable} and of Lemma \ref{absClem}. 
To simplify notation, 
denote by $\bD_{ 1}$ the space $\bD(\bbR_+, \bbR)$. Product spaces are equipped with product topology and if $E''$ is a Polish space, we denote by $\cM_1(E'')$ the space of its Borel probability measures equipped with the topology of weak convergence. 

We denote by $q_n (\cdot, d \ttu)   : \Omega \! \to \! \cM_1 (E)$ a regular version of the conditional law of $U_n$ given $\ccG$, We also denote by $\mathtt Q_{\geq n} (d \ttv) $ and $\mathtt Q_{ n} (d\ttv)$  the laws of resp.~$ \tfrac{1}{b_n}V_{ \lfloor n\cdot \rfloor  \wedge r_n} (\tau_{\geq n})$ and  $ \tfrac{1}{b_n}V_{ \lfloor n\cdot \rfloor  \wedge r_n} (\tau_{n})$. Note that $ \mathtt Q_{ n} (d\ttv)\eqo D^{_{(n)}}_{^{r_n}} ( b_n\ttv (r'_n) ) \, \mathtt Q_{\geq n} (d \ttv) $ by Theorem \ref{VHCcvstable} $(ii)$, where we have set here $r'_n\eqo r_n/n$. For all $\omega\ino \Omega$, we next 
denote by  $Q_{\geq n} (\omega, d \ttv d\ttz )$ and $Q_{n} (\omega, d \ttv d\ttz )$ the law of $(\ttv, Z_n (\ttv, \ttu))$ under 
resp.~$\mathtt Q_{ \geq n} (d\ttv)\otimes q_n (\omega, d \ttu) $ and $\mathtt Q_{ n} (d\ttv)\otimes q_n (\omega, d \ttu) $. 
Since $U_n$ and $\ccG$ are independent from 
$\tau_{\geq n} $ and $\tau_n$, we easily check that 
 $Q_{\geq n} (\cdot, d \ttv d\ttz ) : \Omega \! \to \! \cM_1 (\bD_1 \! \times \! E')$  and  $Q_{n} (\cdot, d \ttv d\ttz ) : \Omega \! \to \! \cM_1 (\bD_1 \! \times \! E')$ are  regular versions of the conditional law of resp.~$M_{\geq n}$ in (\ref{supenncnd}) and $M_{n}$ in (\ref{eqenncnd}) given $\ccG$. 
 We also denote by $Q(d \ttv d\ttz)$ the law of $(X^{_{(r)}}_{^{\!}}, Z)$ and Assumption (\ref{supenncnd}) means that there is $\Omega_0\ino \ccG$ such that $\bP (\Omega_0)\eqo 1$ and such that for all $\omega\ino \Omega_0$, $Q_{\geq n} (\omega, d\ttv d\ttz) \! \to \! Q ( d\ttv d\ttz) $ in $\cM_1 (\bD_1 \! \times \! E')$. 
Let $F$ be a real-valued continuous function on $\bD_1 \! \times \! E'$ that is bounded by $\lVert F\rVert_\infty$. We fix $\omega\ino \Omega_0$. By the construction above, we first get 
$$ u_n := \int_{\bD_1 \times  E'} \!\!\!\!\!\!\!\! \!\! \! Q_n ( \omega, d\ttv d\ttz) \, F(\ttv, \ttz) 
 \eqo   \int_{\bD_1  \times  E'} \!\!\! \!\!\! \!\! \!\! \!  Q_{\geq n} ( \omega, d\ttv d\ttz)\,  F(\ttv, \ttz) D^{{(n)}}_{{r_n}} \big( b_n\ttv (r'_n) \big) \; .$$
Then for all $c\ino (0, \infty)$, 
\begin{eqnarray}
\label{truncc}
 \Big| u_n -   \int_{\bD_1  \times E'}  \!\!\! \!\!\! \!\! \!\! \!  Q_{\geq n} ( \omega, d\ttv d\ttz)  \!\!\! \!\!\! \!\!  & &  \!\!\! \!\!\! \!\! 
F(\ttv, \ttz) \big( c\! \wedge\!  D^{_{(n)}}_{^{r_n}} \big( b_n\ttv (r'_n) \big) \big)   \Big|  \\
& \leq &  \lVert F\rVert_{ \infty} \,  \bE \Big[ D^{_{(n)}}_{ r_n} \! ( V_{r_n} (\tau_{\geq n} )) \, \un_{\{ D^{_{(n)}}_{r_n} ( V_{r_n} (\tau_{\geq n} )) \geq c\}} \Big] =: \eta_n (c) \; .\nonumber
\end{eqnarray}   
We next fix $x,y \ino (0, \infty)$ with $x\leqo y$ and we set $\epp_n (x,y) \! := \! \max_{p \in \lgeo \lfloor b_n x\rfloor, \lfloor b_n y \rfloor \rgeo } \big| D^{_{(n)}}_{^{r_n}} (p) \! -\! D_r (p/b_n)  \big| $. We recall from (\ref{limlocgned}) in Theorem~\ref{VHCcvstable} $(iii)$ that 
$\epp_n (x,y)  \to 0$ as $n\! \to \! \infty$. Then we get 
\begin{eqnarray}
\label{substt}
\Big|   \int_{\bD_1  \times E'}  \!\!\! \!\!\! \!\! \!\! \!  Q_{\geq n} ( \omega, d\ttv d\ttz) 
F(\ttv, \ttz) \big( c\,  \wedge \!   \!\!\! \!   \!\! \! & &  \!\! \!  \!\!\! \!  D^{_{(n)}}_{^{r_n}} \big( b_n\ttv (r'_n) \big) \big)  -
 \int_{\bD_1  \times E'}  \!\!\! \!\!\! \!\! \!\! \!  Q_{\geq n} ( \omega, d\ttv d\ttz)  
F(\ttv, \ttz) \big( c \! \wedge \! D_r \big( \ttv (r'_n) \big) \big) 
\Big| \nonumber \\
& \leq &  \lVert F\rVert_{_{ \infty}}  \epp_n (x,y) + 2c  \lVert F\rVert_{_{ \infty}}   \bP \big( \tfrac{1}{b_n}V_{r_n} (\tau_{\geq n})  \notin \! (x,y)\big) . 
\end{eqnarray}  
Since $c\wedge D_r (\cdot) $ is bounded and since $X^{_{(r)}}_{^{\!}}$ is almost surely continuous at $r$, we get 
 $$ \lim_{n\to \infty}  \int_{\bD_1  \times E'}  \!\!\! \!\!\! \!\! \!\! \!  Q_{\geq n} ( \omega, d\ttv d\ttz)  
F(\ttv, \ttz) \big( c \! \wedge \! D_r \big( \ttv (r'_n) \big) \big) =  \int_{\bD_1  \times E'}  \!\!\! \!\!\! \!\! \!\! \!  Q ( d\ttv d\ttz)  
F(\ttv, \ttz) \big( c \! \wedge \! D_r \big( \ttv (r) \big) \big) . $$
If we set $u_\infty\eqo  \int_{\bD_1  \times E'}   Q ( d\ttv d\ttz)  
F(\ttv, \ttz)  D_r ( \ttv (r) )$, then we have proved for all $c, x,y\ino (0, \infty)$ that 
\begin{multline*}
\limsup_{n\to \infty}|u_n \! -\! u_\infty| \leq \sup_{n\in \bbN^*} \eta_n (c) +2c \lVert F\rVert_{_{\infty}}  
\limsup_{n\to \infty} \bP \big( \tfrac{1}{b_n}V_{r_n} (\tau_{\geq n})  \notin \! (x,y)\big) \\
+ \lVert F\rVert_{_{\infty}} 
\bE \Big[ D_r \big( X^{_{(r)}}_{^{r}}\! \big) \un_{\big\{  D_r \big( X^{_{(r)}}_{^{r}} \! \big)\geq c \big\}}\Big]. 
\end{multline*} 
Note that $\limsup_{n\to \infty} \bP \big( \tfrac{1}{b_n}V_{r_n} (\tau_{\geq n}) \notin \! (x,y) \big) \to 0$ as $x\! \to \! 0$ and $y\! \to \! \infty$ 
since $\tfrac{1}{b_n}V_{r_n} (\tau_{\geq n})  \! \to \! X^{_{(r)}}_{^{r}} \ino(0,\infty)$ by (\ref{supenncnd}) and an elementary argument. 
It then implies that $u_n \! \to \! u_\infty$ because  that $\lim_{c\to \infty}\sup_{n\in \bbN^*} \eta_n (c)\eqo 0 $ by  (\ref{Scheff}) Theorem \ref{VHCcvstable} $(iii)$. \cqfd

\section{Proof of Proposition \ref{condiTCL}}
\label{pfCLTmarti}

We first prove the following lemma that is used in the proof of Proposition \ref{condiTCL} and also 
in Section \ref{pfharmosec} in the proof of Proposition \ref{harmocv}. 
\begin{lemma}
\label{exozarb} Let $(\mathcal G_n)_{n\in \bbN}$ be a filtration on $(\Omega,\ccF)$. Let $c\ino (0, \infty)$ be a constant and let 
$(R_n)_{n\in \bbN}$ be a sequence of nonnegative r.v.s such that for all $n\ino \bbN$, 
$\bP$-a.s.~$R_n \leq c$ and $\bP$-a.s.~$R_n\! \to \! 0$. Let $(r_n)_{n\in \bbN}$ be any sequence of integers. 
Then $\bP$-a.s.~$\bE [ R_n | \mathcal G_{r_n} ] \! \to \! 0$.    
\end{lemma}
\noi
\textbf{Proof.} For all $p\ino \bbN$, we set $\overline{R}_p \eqo \sup_{n\geq p} R_n$. We observe that 
$\limsup_{n\to \infty}\bE [R_n | \mathcal G_{r_n} ]\leqo \sup_{k\in \bbN} \bE [\overline{R}_p | \mathcal G_k]$ a.s.~for all $p\ino \bbN$. So we only need to prove that a.s.~$\lim_{p\rightarrow \infty} \sup_{k\in \bbN} \bE [\overline{R}_p | \mathcal G_k]\eqo 0$. To that end denote by 
$\mathcal G_{\infty}$ the $\sigma$-field generated by $\bigcup_{k\in \bbN} \mathcal G_k$ and observe, by Doob's $L^1$-martingale convergence, that a.s.~$\lim_{k\to \infty}\bE [\overline{R}_p| \mathcal G_k ] \eqo 
\bE [\overline{R}_p| \mathcal G_\infty ] $. Moreover by conditional dominated convergence we also get a.s.~$\lim_{p\to \infty}\bE [\overline{R}_p| \mathcal G_k ] \eqo 0$ for all $k\ino \bbN\cup \{ \infty\}$. Then there exists $\Omega_0\ino \ccF$ such that $\bP (\Omega_0)\eqo 1$ and such that for all integers 
$p' \geq p$ and $k\ino \bbN\cup \{ \infty\}$ and all $\omega\ino \Omega_0$, 
$\bE [ \overline{R}_{p'}| \mathcal G_k] (\omega)\leqo \bE [ \overline{R}_{p}| \mathcal G_k] (\omega)$, 
$\lim_{k\to \infty}\bE [\overline{R}_p| \mathcal G_k ]  (\omega) \eqo 
\bE [\overline{R}_p| \mathcal G_\infty ] (\omega)$ and $\lim_{p\to \infty}\bE [\overline{R}_p| \mathcal G_k ] (\omega) \eqo 0$.
We fix $\omega\ino \Omega_0$, $p\ino \bbN$ and $\epp\ino (0, 1)$. There is $k_{p, \epp} (\omega) \ino \bbN$ such that 
$\sup_{k\geq k_{p, \epp} (\omega)} \bE [\overline{R}_p| \mathcal G_k ]  (\omega) \leqo \epp +\bE [\overline{R}_p| \mathcal G_\infty ]  (\omega)$. Thus for all $p'\geqo p$, 
$\sup_{k\in \bbN} \bE [\overline{R}_{p'}| \mathcal G_k ]  (\omega)\leq \max_{k\leq k_{p, \epp} (\omega)} \bE [\overline{R}_{p'}| \mathcal G_k ]  (\omega) + \epp +\bE [\overline{R}_p| \mathcal G_\infty ]  (\omega)$. Therefore, 
$\limsup_{p'\to \infty} \sup_{k\in \bbN} \bE [\overline{R}_{p'}| \mathcal G_k ]  (\omega)\leq \epp + \bE [\overline{R}_p| \mathcal G_\infty ]  (\omega)\to \epp$ as $p\!\to \! \infty$. Since $\epp$ can be chosen arbitrarily small, this shows that $\lim_{p'\to \infty} \sup_{k\in \bbN} \bE [\overline{R}_{p'}| \mathcal G_k ]  (\omega)\eqo  0$ for all $\omega\ino \Omega_0$, which entails the desired result. \cqfd

\medskip

We now prove Proposition~\ref{condiTCL}. To that end, we recall the following: for all $n\ino \bbN$, $(\cF_{n, k})_{k\in \bbN}$ is a filtration on $(\Omega, \ccF)$; 
$(D_{n, k})_{k\in \bbN^*}$ is a sequence of real valued r.v.s; $(s_n)_{n\in \bbN}$ is a sequence of integers and $s\in\bbR_+$. We introduce the following notation 
$$ M_{n, k}\eqo D_{n,1}+ \ldots + D_{n,k}, \quad Q_{n,k} \eqo \bE [D_{n,k}^2| \cF_{n, k-1}], \quad V_{n, k}\eqo Q_{n, 1} + \ldots + Q_{n,k} , $$
with $M_{n, 0}\eqo V_{n, 0}\eqo 0$, and we make the following assumptions. 
\begin{compactenum}

\smallskip

\item[$(a)$] For all $n,k$, $D_{n, k}$ is $\cF_{n,k}$-measurable, $\bE [ D_{n,k}^2] \leko \infty$ and $\bP$-a.s.~$\bE [D_{n, k} | \cF_{n, k-1}]\eqo 0$.
\smallskip

\item[$(b)$] There exist a filtration $(\mathcal G_j)_{j\in \bbN}$ on $(\Omega, \ccF)$ and a sequence of integers $(r_n)_{n\in \bbN}$ such that 
\smallskip

$\cF_{n, 0}\eqo \mathcal G_{r_n}$, $n\ino \bbN$. 
\item[$(c)$] For all $\epp \ino (0, 1)$, $\bP$-a.s.~$\lim_{n\to \infty} 
\sum_{1\leq k\leq s_n}\bE [D_{n,k}^2\un_{\{| D_{n, k}|>\epp  \}} | \cF_{n, k-1} ]= 0$. 
\smallskip

\item[$(d')$] For all $\epp \ino (0, 1)$, $\bP$-a.s.~$\lim_{n\to \infty} \bE \big [\un_{\{ |V_{n, s_n} -s| > \epp \}} \big|\cF_{n, 0} \big]\eqo 0$. 
\end{compactenum}
\begin{remark}
\label{ddprime}
First observe that 
\emph{Assumptions~$(a$-$d)$ in Proposition~\ref{condiTCL} imply $(d')$.} 
Indeed, Assumption~$(d)$ in Proposition~\ref{condiTCL} implies  for all $\epp \ino (0, 1)$ that 
$\bP$-a.s.~$R_n\! := \! \un_{\{ |V_{n, s_n} -s| > \epp \}} \to 0$ and Assumption~$(b)$ allows us to apply Lemma~\ref{exozarb}, which entails $(d')$. \cq 
\end{remark}

So we prove Proposition~\ref{condiTCL} if we prove the following. 
\begin{proposition}
\label{TCLbis} We keep the previous notation and we assume $(a$-$c)$ and $(d')$. Then for all $\lambda\ino \bbR$, $\bP$-almost surely, $\lim_{n\to \infty} \bE [ \exp (i\lambda M_{n, s_n}) | \cF_{n, 0} ] \eqo \exp (-\frac{_1}{^2} s\lambda^2)$. 
\end{proposition}
\noi
\textbf{Proof.} We first prove the proposition under the additional assumption that there exists $c\ino (0, \infty)$ such that $\bP (V_{n, s_n} \leqo c) \eqo 1$. We rely on the following standard inequalities that hold for all $a, b\ino \bbR_+$, all $x\ino \bbR$ and all $z\ino \bbC$: 
\begin{equation}
\label{standineq}
\!\! \!\! \!\! \sup_{\quad y\in [-a, a]} \!\! \!\!  \!\!  |e^{by}\! -\! 1| \leqo e^{ba}\! -\! 1, \quad \big|e^{ix}\! -\! 1 \! -\! ix  +\! \frac{_{_1}}{^{^2}}x^2 | \leqo |x^2|\! \wedge \! (\frac{_{_1}}{^{^6}} |x|^3) , \quad |e^z\! -\! 1\! -\! z| \leqo  \frac{_{_1}}{^{^2}}|z|^2e^{(\mathtt{Re} (z))_+} .
\end{equation}
We first observe that $\alpha_n  \eqo \bE [ e^{i\lambda M_{n, s_n}} | \cF_{n, 0} ]\! -\!  e^{-\frac{_1}{^2} s\lambda^2}\eqo \beta_n + \gamma_n$ where 
$$\beta_n \eqo \bE  \big[ e^{i\lambda M_{n, s_n}} \big( 1 \! -\! e^{ \frac{_1}{^2}\lambda^2 (V_{n, s_n}- s)} \big)  \big| \cF_{n, 0} \big] \quad \textrm{and} \quad \gamma_n\eqo  \bE  \big[ e^{-\frac{_1}{^2} s\lambda^2} \big( e^{i\lambda M_{n, s_n} + \frac{_1}{^2}\lambda^2 V_{n, s_n} } \! -\! 1 \big)  \big| \cF_{n, 0} \big] .$$
We fix $\epp \ino (0, 1)$ such that $\frac{_{_1}}{^{^3}} |\lambda| \epp \leko 1$ and we bound $\beta_n$ and $\gamma_n$. 
Since $V_{n, s_n}$ is bounded by the constant $c$, the first inequality in (\ref{standineq}) implies 
\begin{equation}
\label{boundbeta}
|\beta_n| \leq e^{\frac{_1}{^2}\lambda^2\epp } \! -\! 1 + \big( 1+ e^{\frac{_1}{^2}\lambda^2(c+ s)} \big)  \bE \big [\un_{\{ |V_{n, s_n} -s| > \epp \}} \big|\cF_{n, 0} \big]\; .
\end{equation}
Thus, by $(d')$, for all $\epp \ino (0, 1)$ we a.s.~get $\limsup_{n\to \infty} |\beta_n| \leqo 
e^{\frac{_1}{^2}\lambda^2\epp } \! -\! 1$ and thus a.s.~$\beta_n \! \to \! 0$. 
 
We next observe that 
\begin{eqnarray*}
e^{\frac{_1}{^2}\lambda^2 s} \gamma_n &= &\!\!\!\!\!\!\!\! \sum_{\quad 1\leq k\leq s_n} \!\!\!\!\!\! \bE \Big[ e^{i\lambda M_{n, k-1} +\frac{_1}{^2}\lambda^2 V_{n,k} } \bE \big[e^{i\lambda D_{n,k}} \! -\! e^{-\frac{_1}{^2}\lambda^2 Q_{n, k}} \big| \cF_{n, k-1} \big]  \Big| \cF_{n, 0} \Big]  \\
&=&  \!\!\!\!\!\!\!\! \sum_{\quad 1\leq k\leq s_n} \!\!\!\!\!\! \bE \big[ e^{i\lambda M_{n, k-1} +\frac{_1}{^2}\lambda^2 V_{n,k} } \big( \delta_{n,k}^{_{(1)}} + \delta_{n,k}^{_{(2)}} \big) \big| \cF_{n, 0} \big], \quad \textrm{where}
\end{eqnarray*}
\[
\delta_{n,k}^{_{(1)}} = \bE \big[e^{i\lambda D_{n,k}} \! -\! 1 -i\lambda D_{n,k}  + \tfrac{_1}{^2} \lambda^2 D^2_{n, k} \big| \cF_{n, k-1} \big] \quad \textrm{and} \quad   \delta_{n,k}^{_{(2)}}\eqo 1 \! -\!  \tfrac{_1}{^2} \lambda^2 Q_{n, k} - e^{-\frac{_1}{^2}\lambda^2 Q_{n, k}} ,\]
since $\lambda M_{n, k-1} +\frac{_1}{^2}\lambda^2 V_{n,k}$ is $\cF_{n, k-1}$-measurable, since a.s.~$\bE [D_{n, k} \big| \cF_{n, k-1}]\eqo 0$ and by definition of $Q_{n, k}$. We first bound $\delta_{n,k}^{_{(1)}}$ by use of the second inequality in (\ref{standineq}). To that end we set 
$Z\eqo   \lambda^2 D^2_{n, k} (1\wedge  \frac{_{_1}}{^{^6}} |\lambda D_{n, k}|)$. Since 
$\frac{_{_1}}{^{^3}} |\lambda| \epp \leko 1$, we observe that 
$Z\leqo  \lambda^2 D^2_{n, k} \un_{\{  |D_{n, k} |> \epp\}}+ \frac{_1}{^6} |\lambda|^3 \epp D^2_{n, k}  $.  
Therefore $|\delta_{n,k}^{_{(1)}}|\leqo  \lambda^2 \bE \big[ D^2_{n, k} \un_{\{  |D_{n, k} |> \epp\}}\big| \cF_{n, k-1} \big]+ \frac{_1}{^6} |\lambda|^3 \epp Q_{n,k}$. By the third inequality in (\ref{standineq}), we get $ |\delta_{n,k}^{_{(2)}}| \leqo \tfrac{_1}{^8} \lambda^4 Q^2_{n,k}$. Then observe that $Q_{n,k} \leq \epp^2 +\bE \big[ D^2_{n, k} \un_{\{  |D_{n, k} |> \epp\}}\big| \cF_{n, k-1} \big]$. Thus  $ |\delta_{n,k}^{_{(2)}}| \leqo \tfrac{_1}{^8} \lambda^4 Q_{n,k} \epp^2+  \tfrac{_1}{^8} \lambda^4 c \bE \big[ D^2_{n, k} \un_{\{  |D_{n, k} |> \epp\}}\big| \cF_{n, k-1} \big] $, since $Q_{n, k} \leqo V_{n, s_n} \leqo c$. Thus, 
\begin{align} 
\label{gbnd}
 e^{-\frac{_1}{^2}\lambda^2 c } e^{\frac{_1}{^2}\lambda^2 s} |\gamma_n|  &\leq  \!\!\!\!\!\!\!\! \sum_{\quad 1\leq k\leq s_n} \!\!\!\!\!\! \bE \Big[ e^{\frac{_1}{^2}\lambda^2 (V_{n,k}-c) } \big( |\delta_{n,k}^{_{(1)}}| +   |\delta_{n,k}^{_{(2)}}| \big) \Big| \cF_{n, 0} \Big] \nonumber \\
&\leq   \tfrac{1}{24}|\lambda|^3 \epp \big(4+ 3|\lambda| \epp\big) \bE [V_{n, s_n}  | \cF_{n, 0} ]+  \tfrac{1}{8}\lambda^2 (8+ c\lambda^2) \bE [ L_{n,s_n} (\epp) | \cF_{n, 0} ]
\end{align}
where we have set $L_{n,k} (\epp)\eqo \sum_{1\leq j\leq k}\bE [D_{n,j}^2\un_{\{| D_{n, j}|>\epp  \}} | \cF_{n, k-1} ]$. Arguing as in Remark~\ref{ddprime}, Assumptions~$(b)$ and $(c)$ allow us to apply Lemma~\ref{exozarb}, which entails that $\lim_{n\to \infty} \bE [L_{n, s_n} (\epp) | \cF_{n, 0} ] \eqo 0$. Moreover, since $V_{n, s_n} \leqo c$, for all $\epp\ino (0, 1)$ we a.s.~get $\limsup_{n\to \infty} |\gamma_n| \leq 
\frac{1}{{24}} c e^{-\frac{_1}{^2}\lambda^2 (s-c)}|\lambda|^3 \big(4+ 3|\lambda| \epp\big) \epp $ and thus a.s.~$\gamma_n \! \to \! 0$. We have proved that a.s.~$ \lim_{n\to\infty }\bE [ e^{i\lambda M_{n, s_n}} | \cF_{n, 0} ] \eqo e^{-\frac{_1}{^2}\lambda^2 s}$ under the assumption that $\bP (V_{n, s_n} \leqo c)\eqo 1$ for all $n\in \bbN$. 

  Let us remove this assumption. We fix $c\ino (1+ s, \infty)$ and we set $\widetilde{D}_{n, k}\eqo D_{n, k} \un_{\{ V_{n, k} \leq c \}}$, $\widetilde{M}_{n, k}\eqo\sum_{1\leq j \leq k} \widetilde{D}_{n, j}$, $\widetilde{V}_{n, k}\eqo \sum_{1\leq j \leq k} \bE [\widetilde{D}_{n, j}^2|\cF_{n, j-1} ]$, $\widetilde{L}_{n, k} (\epp)\eqo \sum_{1\leq j\leq k}\bE [\widetilde{D}_{n,j}^2\un_{\{| \widetilde{D}_{n, j}|>\epp  \}} | \cF_{n, k-1} ]$ and $T_n\eqo \max \{ k\ino \lgeo 0,  s_n\rgeo : V_{n, k}\leqo c  \}$. We easily check that the $\widetilde{D}_{n, k}$ satisfy $(a)$ since $V_{n,k}$ is $\cF_{n, k-1}$-measurable. Then, observe that 
$\widetilde{V}_{n, k} \eqo V_{n, k\wedge T_n} $ (which implies $\bP \big( \widetilde{V}_{n, s_n}\leqo c\big)\eqo 1$) 
and that $\widetilde{L}_{n, k} (\epp)\eqo L_{n, k\wedge T_n} (\epp)$. 
Therefore the $\widetilde{D}_{n, k}$ satisfy $(c)$ because $\widetilde{L}_{n, s_n} (\epp)\leqo L_{n, s_n} (\epp)\! \to \! 0$ a.s.
Next, note that for all $\epp\ino (0, 1)$, a.s.~
$$\bE \big[ \un_{\{ |\widetilde{V}_{n, s_n} -s | >\epp \}} | \cF_{n, 0}\big] \leq \bE \big[ \un_{\{T_n < s_n  \}} | \cF_{n, 0}\big] +\bE \big[ \un_{\{ |V_{n, s_n} -s | >\epp \}} | \cF_{n, 0}\big] \leq 2\bE \big[ \un_{\{ |V_{n, s_n} -s | >\epp \}} | \cF_{n, 0}\big] $$
 since $T_n \leko s_n$ implies that $V_{n, s_n} \geko c\geko s+1$. This proves that 
 the $\widetilde{D}_{n, k}$ satisfy $(d')$ and thus we get a.s.~$ \lim_{n\to\infty }\bE [ e^{i\lambda \widetilde{M}_{n, s_n}} | \cF_{n, 0} ] \eqo e^{-\frac{_1}{^2}\lambda^2 s}$, by the previous arguments. But observe that 
 \begin{eqnarray*}
 \Big| \bE \big[  e^{i\lambda \widetilde{M}_{n, s_n}} \big| \cF_{n, 0} \big] \! -\! \bE \big[ 
  e^{i\lambda M_{n, s_n}} \big| \cF_{n, 0} \big] \Big|& \leq & 
 \bE \Big[  \big| e^{i\lambda \widetilde{M}_{n, s_n}}\! -\! e^{i\lambda M_{n, s_n}} \big| \un_{\{ T_n < s_n \}} \Big| \cF_{n, 0} \Big] \\
  &\leq&   2\bE \big[ \un_{\{ |V_{n, s_n} -s | >\epp \}} | \cF_{n, 0}\big] \underset{n\to\infty}{-\!\!\! -\!\!\! \longrightarrow} 0 
\end{eqnarray*}
$\bP$-a.s.~which implies the desired result. \cqfd

%
%
%{\small
%\bibliographystyle{acm}
%
%\bibliography{Refsna2}
%
%}
\end{document}